%% file: PaperDraft.tex
\documentclass[11pt]{article}

\usepackage{color, amsmath, amssymb, amsthm, graphicx, bbm, footnote, mathtools, bm, enumitem, subcaption, dsfont}

\usepackage[top=1in, bottom=1in, left=1in, right=1in]{geometry}

\usepackage[toc,page,header]{appendix}
\usepackage{titletoc}
\usepackage{titling}
\allowdisplaybreaks

\usepackage{textcomp}

\usepackage{algorithm} 
\usepackage{algpseudocode}
\usepackage{booktabs,tabularx,multirow,threeparttable}

\usepackage{setspace}
\usepackage{adjustbox}
\usepackage[colorlinks,allcolors=blue]{hyperref}
\usepackage[round]{natbib}  
\usepackage{authblk}
\usepackage[para]{footmisc}  

\usepackage{verbatim,fix-cm,tikz}
\usetikzlibrary{calc,arrows.meta,shapes.geometric}

\usepackage{apptools}
\AtAppendix{\counterwithin{theorem}{section}}
\AtAppendix{\counterwithin{remark}{section}}

\newtheorem{theorem}{Theorem}
\newtheorem{definition}{Definition}
\newtheorem{proposition}[theorem]{Proposition}

\newtheorem{lemma}[theorem]{Lemma}
\newtheorem{remark}{Remark}
\newtheorem{assump}{Assumption}

\DeclareMathOperator*{\argmax}{arg\,max}
\DeclareMathOperator*{\argmin}{arg\,min}
\DeclarePairedDelimiter\ceil{\lceil}{\rceil}

\newcommand{\norm}[1]{\left|\left| #1 \right|\right|}

\renewcommand{\hat}{\widehat}
\renewcommand{\tilde}{\widetilde}

\title{\textbf{Gradient-Guided Density Peak Clustering}}
\author{Yikun Zhang$^{1,2,\ast}$ \;\text{and}\;  Yen-Chi Chen$^{3,\dagger}$}
\date{\small	
	\textit{$^1$Department of Statistics, University of Chicago\\
		$^2$NSF-Simons AI Institute for the Sky (SkAI Institute)}\\
		$^{\ast}$\href{mailto:yikunz@uchicago.edu}{yikunz@uchicago.edu}\; \\
		\textit{$^3$Department of Statistics, University of Washington}\\
	$^{\dagger}$\href{mailto:yenchic@uw.edu}{yenchic@uw.edu} \\~\\	
	\today}
\begin{document}
	\maketitle

\begin{abstract}
Density peak clustering (DPC) connects each observation to its nearest neighbor of higher density and identifies cluster centers as high-density observations with unusually large nearest neighbor uphill shifts. The resulting uphill paths from observations to cluster centers, however, can be irregular and unstable in low-density regions, making the clustering assignments sensitive to local perturbations and obscuring the population geometry of the DPC graph. In this paper, we introduce \emph{gradient-guided density peak clustering} (GGDPC), which performs a gradient ascent step before each nearest neighbor uphill search. We develop a stability theory that relates the GGDPC graph to the gradient ascent flow of the population density. In particular, we establish consistency of GGDPC under five complementary criteria: recovery of local modes, adjusted Rand index, dendrogram (cluster tree), path length, and waterfall measure. 
Together, these results provide new statistical, geometric, and topological interpretations of DPC-type clustering algorithms.
\end{abstract}

\section{Introduction}
\label{sec:Intro}

Density peak clustering (DPC; \citealt{rodriguez2014clustering}) is a simple and widely used density-based clustering method. Given density values or estimates of them at the observations, DPC connects each observation, except for the sample global mode, to its nearest neighbor among observations of higher density. The length of this directed edge is the \emph{1-nearest-neighbor (1NN) uphill distance} of the observation. Cluster centers (or density modes) are then identified as observations with unusually large 1NN uphill distances, together with high density values, through a decision diagram; see the bottom left panel of \autoref{fig:old_faithful} for an illustration. Removing the outgoing edges from these cluster centers thus forms a partition of observations according to the resulting connected components of the graph; see the top left panel of \autoref{fig:old_faithful}. This clustering construction combines elements of mode clustering \citep{fukunaga1975estimation,chacon2012clusters,chacon2015population,menardi2016review,chen2016comprehensive,arias2023unifying,arias2025clustering}, density level set clustering \citep{rinaldo2010generalized,steinwart2011adaptive,rinaldo2012stability}, and hierarchical clustering \citep{jain1988algorithms,nielsen2016hierarchical} in a single algorithmic procedure.

%

DPC offers several practical advantages. First, it does not require a pre-specified number of clusters, and its decision diagram gives an interpretable mechanism for center selection. Second, once density values are available at the observations, the clustering is built from simple 1NN searches rather than by iteratively estimating the gradient and running a full gradient ascent process from every observation \citep{cheng1995mean,comaniciu2002mean}. These features have made DPC attractive in a broad range of applications; see \cite{wei2023overview,wang2024density} for recent reviews.


\begin{figure}[!th]
	\captionsetup[subfigure]{justification=centering}
	\begin{subfigure}[t]{0.49\linewidth}
		\centering
		\includegraphics[width=1\linewidth]{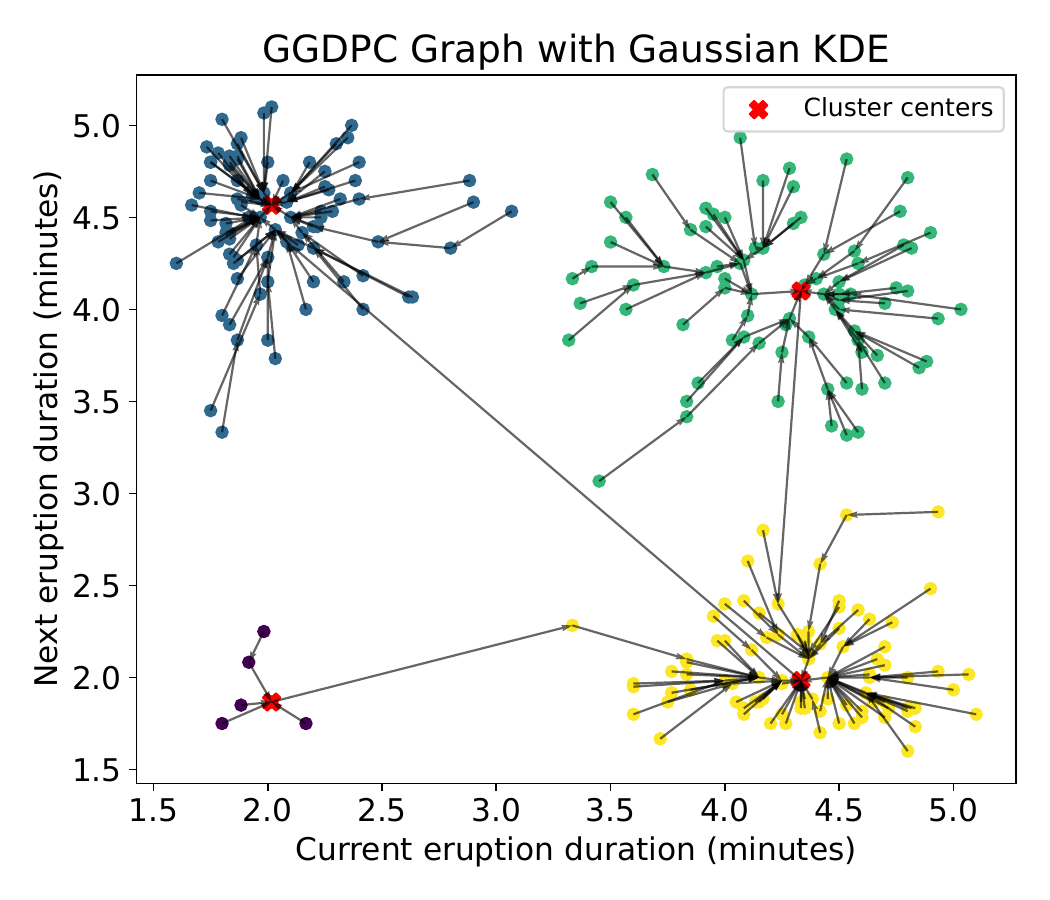}
	\end{subfigure}
	\hfil
	\begin{subfigure}[t]{0.49\linewidth}
		\centering
		\includegraphics[width=1\linewidth]{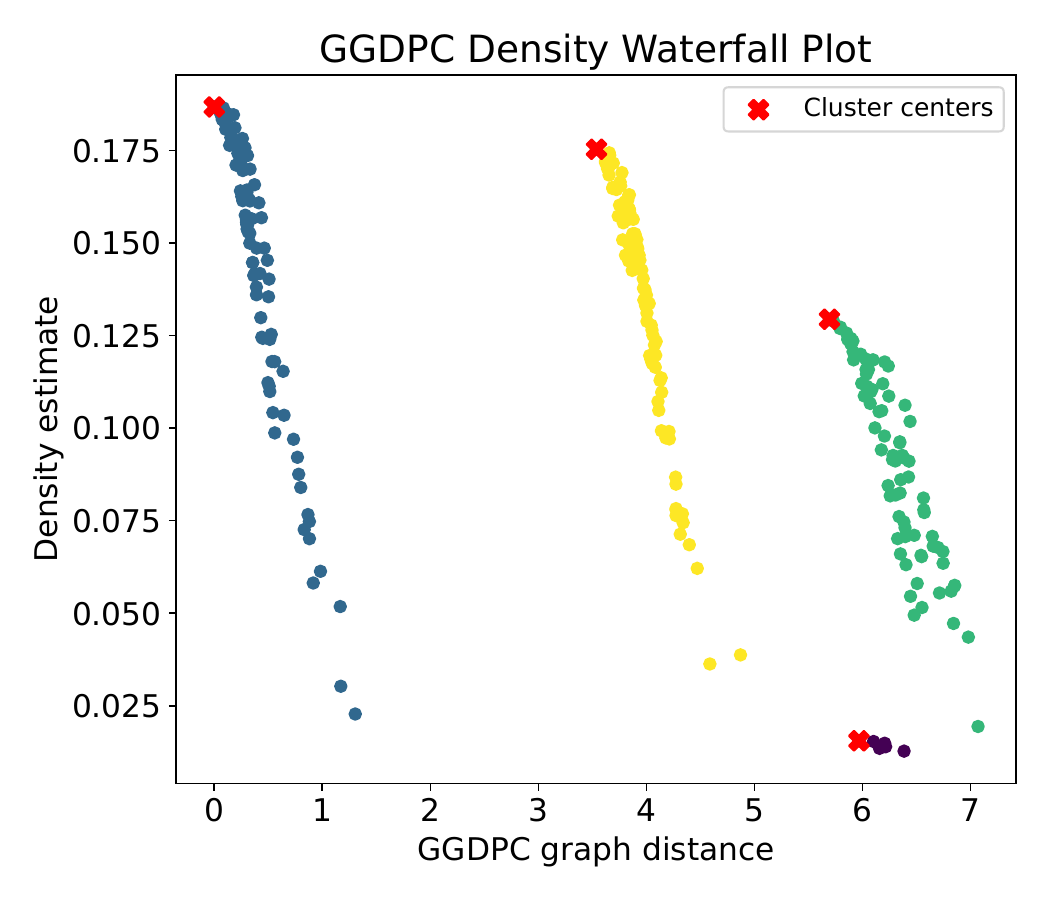}
	\end{subfigure}
	\begin{subfigure}[c]{0.49\linewidth}
		\centering		\includegraphics[width=1\linewidth]{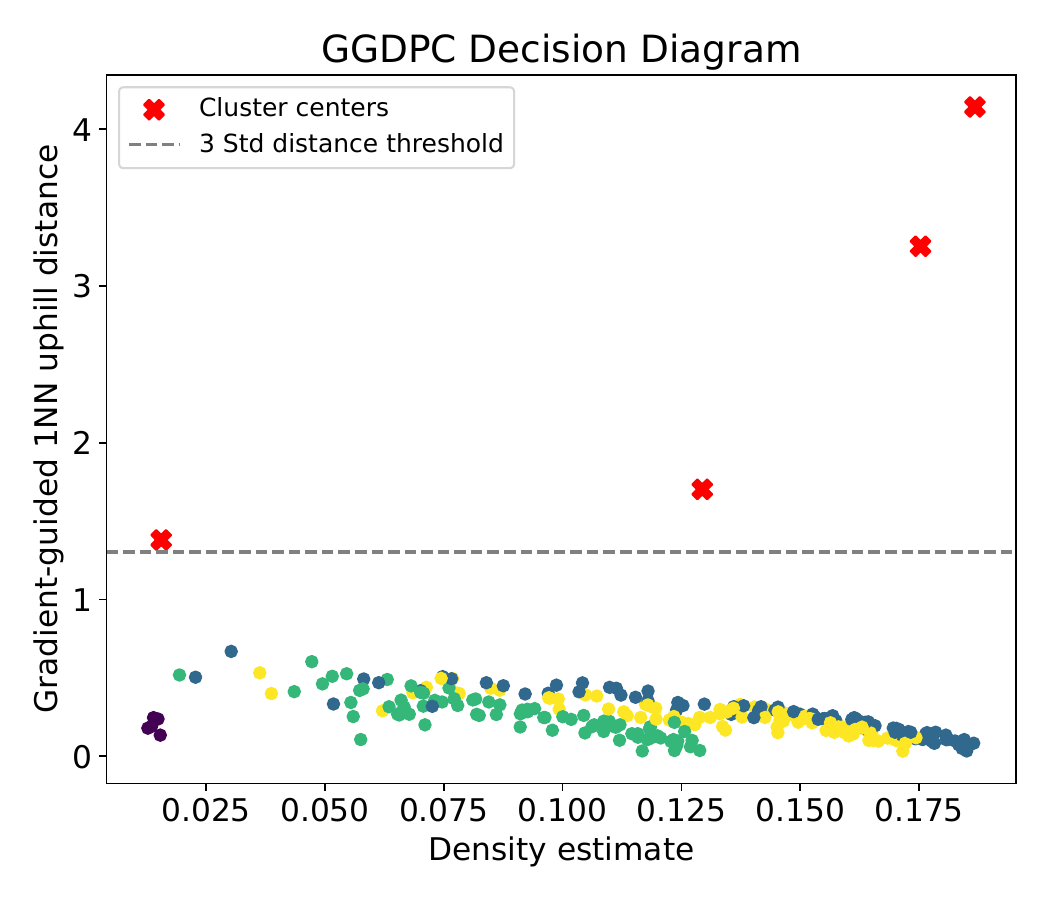}
	\end{subfigure}
	\hfil
	\begin{subfigure}[c]{0.49\linewidth}
		\centering		\includegraphics[width=1\linewidth]{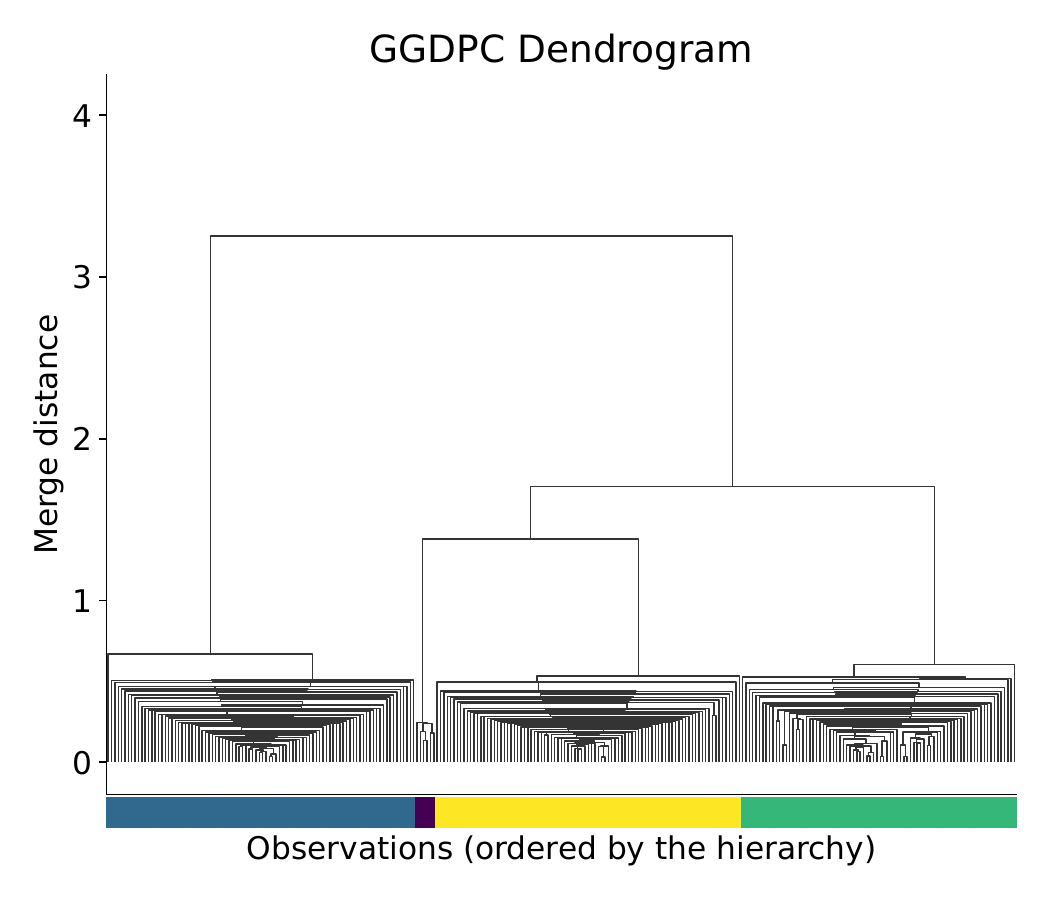}
	\end{subfigure}
	\caption{Four complementary summaries of GGDPC applied to pairs of consecutive eruption durations in the Old Faithful data. \textbf{Top left (GGDPC graph):} the directed graph constructed by GGDPC on the observations. Red crosses mark the selected cluster centers and colors indicate the resulting clusters.
		\textbf{Top right (density waterfall plot):} estimated density plotted against the global GGDPC graph distance. The colored branches trace the density profiles associated with the selected centers and produce the characteristic waterfall pattern.
		\textbf{Bottom left (decision diagram):} the gradient-guided 1NN uphill distance plotted against estimated density. The gray dashed line shows the distance threshold used in this illustration, equal to the sample mean uphill distance plus three sample standard deviations.
		\textbf{Bottom right (GGDPC dendrogram):} the hierarchy obtained by thresholding gradient-guided 1NN uphill distances. The vertical axis records merge distance and the color strip indicates the selected clustering.}
	\label{fig:old_faithful}
\end{figure}

Nevertheless, the 1NN uphill rule that makes DPC computationally attractive is also the source of a fundamental instability. When an observation is away from density modes, its 1NN of higher density need not lie near the local gradient ascent direction. Consequently, a single perturbation can redirect all downstream observations in that branch of the DPC graph to a different cluster center. This phenomenon is often described as the \emph{domino effect} or \emph{chain reaction} in the literature \citep{xie2016robust,seyedi2019dynamic}; see also Section 5.3 in \cite{deng2025density}.
Empirically, it can produce irregular boundaries between adjacent clusters; see \autoref{fig:gmm_graph} for an illustration under a symmetric two-component Gaussian mixture. More importantly for statistical theory, the same phenomenon makes the full DPC path from a general starting point to a density mode difficult to compare with a stable population object.
Existing theoretical analyses of DPC and the closely related quick shift algorithm \citep{vedaldi2008quick} provide important guarantees for density mode estimation and related local clustering properties around the modal regions \citep{jiang2017consistency,jiang2018quickshift++,verdinelli2018analysis,tobin2023theoretical}, but the asymptotic behavior of the entire 1NN uphill path from an arbitrary starting point remains much less understood.

\begin{figure}[!t]
	\captionsetup[subfigure]{justification=centering}
	\begin{subfigure}[t]{0.49\linewidth}
		\centering
		\includegraphics[width=1\linewidth]{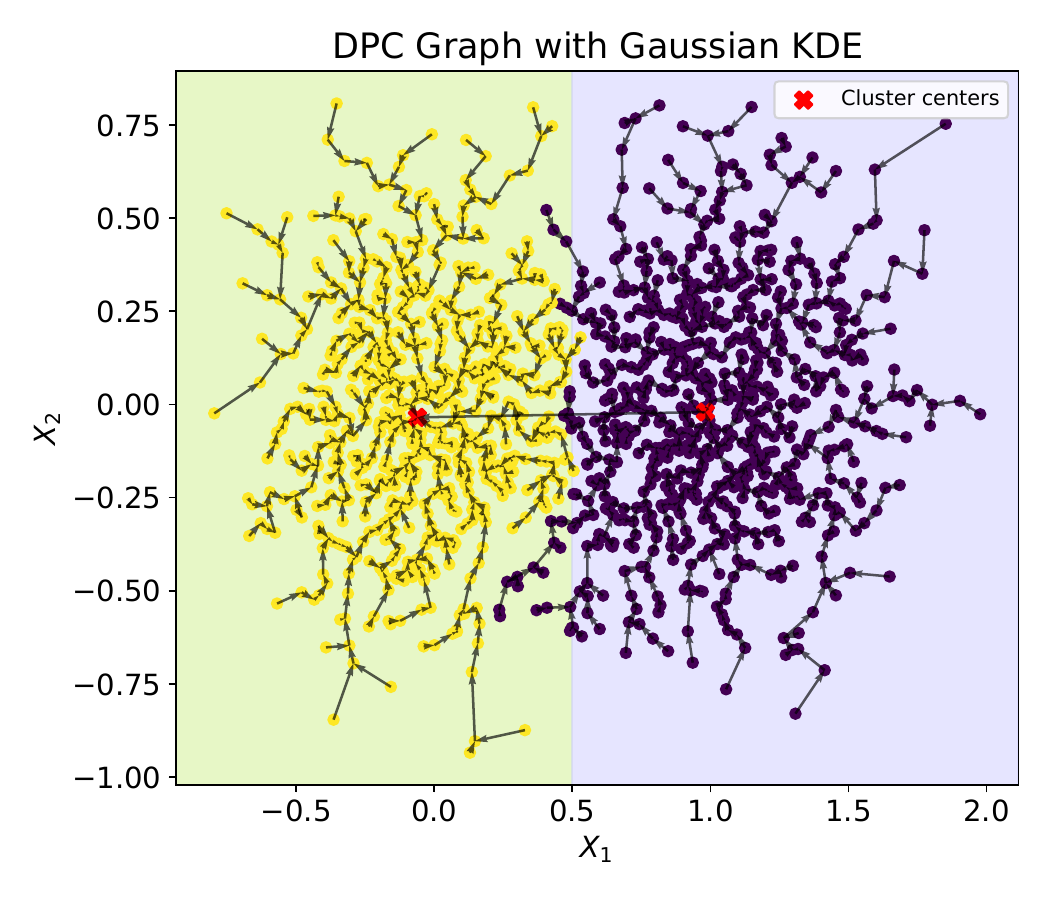}
	\end{subfigure}
	\hfil
	\begin{subfigure}[t]{0.49\linewidth}
		\centering
		\includegraphics[width=1\linewidth]{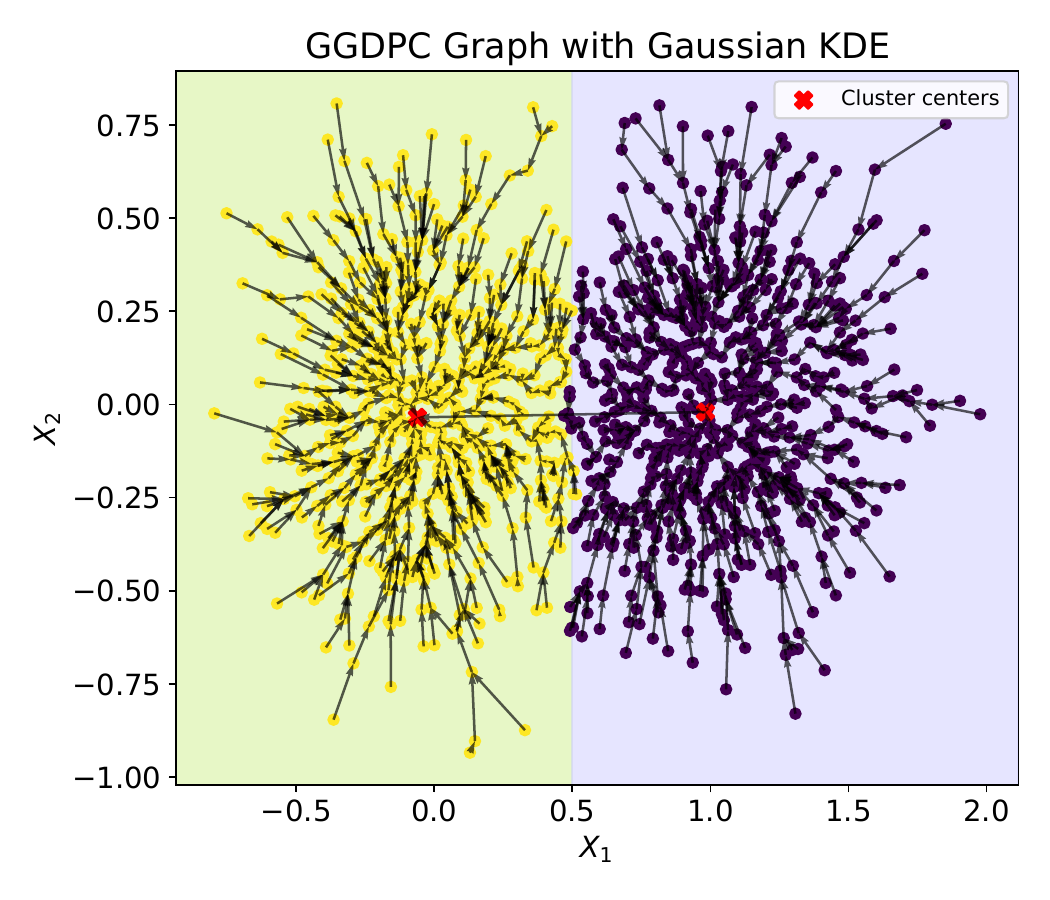}
	\end{subfigure}
	\caption{Comparison of DPC (left) and GGDPC (right) for $n=1500$ observations generated from the Gaussian mixture $0.5 \cdot \mathcal{N}(\bm{\mu}_1,0.09 \bm{I}_2) + 0.5 \cdot \mathcal{N}(\bm{\mu}_2,0.09 \bm{I}_2)$, where $\bm{\mu}_1=(0,0)^T$, $\bm{\mu}_2=(1,0)^T$, and $\bm{I}_2\in \mathbb{R}^{2\times 2}$ is the identity matrix. By symmetry, the population separatrix lies at $x_1=0.5$. The DPC graph exhibits more branches and cluster assignments that cross this boundary, whereas GGDPC better follows the local gradient ascent geometry and reduces such propagation.}
	\label{fig:gmm_graph}
\end{figure}

\subsection{Main Contributions}

To address this instability, we propose \emph{gradient-guided density peak clustering} (GGDPC), a minimal modification of DPC. Before searching for the 1NN of higher density, GGDPC takes a one-step gradient ascent from each observation and performs the 1NN uphill search around that update. The gradient step supplies a stable local direction, while the subsequent 1NN search keeps the GGDPC procedure on the observed data cloud and preserves the computational advantage of DPC. Furthermore, near a density mode where the gradient is small, GGDPC retains the large 1NN uphill shifts for distinguishing those well-separated cluster centers. Thus, GGDPC regularizes the geometry of the ascending path while preserving the simple and interpretable graph structure of DPC.

\autoref{fig:old_faithful} summarizes four complementary views of GGDPC on consecutive eruption durations from the Old Faithful Geyser data \citep{azzalini1990look}. Besides the directed clustering graph and the decision diagram inherited from DPC, we design two new informative plots for GGDPC. The first is a \emph{density waterfall plot}, which displays estimated density against GGDPC graph distance to the sample global mode and reveals the density profile around each cluster center. The second is a \emph{GGDPC dendrogram}, obtained by thresholding gradient-guided 1NN uphill distances and following how graph components merge as the threshold increases. These new plots also motivate two of the population objects developed in our theory.
After introducing GGDPC in \autoref{sec:GGDPC}, we establish the following consistency and stability properties of GGDPC under regularity conditions.
\begin{enumerate}[leftmargin=1.5em,label=\textbf{\arabic*.}]
	\item \textbf{Convergence of GGDPC clustering:} In \autoref{sec:GGDPC_conv_clu}, we establish consistency of the separated density modes selected by GGDPC and prove convergence of the resulting clustering assignments to the population modal partition under the adjusted Rand index (ARI).
	
	\item \textbf{Stability of the GGDPC dendrogram:} In \autoref{sec:clu_tree}, we view the GGDPC graph across distance thresholds as a cluster tree or dendrogram and define its population counterpart, the \emph{modal distance dendrogram}, whose merge heights are determined by distances from non-global modes to their upper level sets. We then prove convergence of the empirical GGDPC dendrogram to this population dendrogram under the Gromov-Hausdorff distance.
	
	\item \textbf{Path-length stability of GGDPC:} In \autoref{sec:stability}, we show that GGDPC ascending paths approximate the geometry of the population gradient ascent flow. Specifically, we establish convergence of the GGDPC path length to the corresponding gradient ascent flow length, uniformly over almost every starting point in the density support.
	
	\item \textbf{Consistency of GGDPC graph distance and density waterfalls:} In \autoref{subsec:conv_graph_dist}, we prove convergence of the GGDPC graph distance to a deterministic population limit that combines gradient ascent flow lengths within modal basins with modal projection distances across basins.
	In \autoref{subsec:den_water}, we further establish convergence of the empirical density waterfall measure to its two-dimensional population analogue under the Wasserstein-1 distance.
\end{enumerate}
Notably, we also derive explicit convergence rates for the above theoretical results. As the target becomes increasingly geometrically refined, the analysis becomes more demanding and the corresponding rates generally become slower. These results provide a new route toward the pathwise consistency question for 1NN hill-climbing procedures, including the quick shift algorithm, raised in Section 6 of \cite{arias2025clustering}.

\subsection{Other Related Work}

The methodological literature on DPC is extensive; see \cite{wei2023overview,wang2024density,deng2025density} for recent reviews. Many variants of DPC alter the density estimate, neighborhood construction, assignment rule, or selection criterion for cluster centers to improve its robustness and mitigate the domino effect; see, for example, \cite{xie2016robust,li2018comparative,jiang2019novel,seyedi2019dynamic,hou2020density}. Our objective is different. Rather than introducing another local correction solely for empirical robustness, we modify the 1NN uphill search so that the resulting graph admits a direct comparison with a smooth population dynamical system.

GGDPC is also closely related to mean shift, a widely used mode clustering method that assigns observations to clusters according to the modes reached by iterative gradient-ascent updates \citep{cheng1995mean,comaniciu2002mean,li2007note,carreira2015review,arias2016errata}. Our theory further connects to the literature on clustering stability and density cluster trees. Stability under perturbations has long been used to study and validate clustering procedures \citep{lange2004stability,ben2006sober,von2010clustering}, while consistency of density-based cluster trees has been studied through Hartigan-type separation and related metrics \citep{hartigan1981consistency,hartigan1985statistical,chaudhuri2010rates,rinaldo2010generalized,rinaldo2012stability,eldridge2015beyond}. For our dendrogram analysis, we adopt the metric viewpoint of \cite{carlsson2010characterization}, under which a dendrogram is represented by an ultrametric and different dendrograms are compared through the Gromov-Hausdorff distance. The modal distance dendrogram induced by GGDPC, however, is distinct from the usual density cluster tree \citep{stuetzle2003estimating,stuetzle2010generalized}, because its merge heights are determined by geometric modal projection distances rather than by density levels at saddle points.


\section{Problem Setup and Background}
\label{sec:problem_setup}

For any $\bm{x}\in \mathbb{R}^d$, let $\norm{\bm{x}}$ denote the usual Euclidean norm and $B(\bm{x},r):=\left\{\bm{y}\in \mathbb{R}^d: \norm{\bm{x}-\bm{y}} < r\right\}$. For any matrix $A \in \mathbb{R}^{m\times n}$, we define $\norm{A}_{\max} = \max_{j,k}|A_{jk}|$ with $A_{jk}$ being its $(j,k)$ entry. We say that a function $f:\mathbb{R}^d \to \mathbb{R}$ is $C^M$ if it is $M$-times continuously differentiable. For a smooth function $f$, we let $\partial^{[\alpha]} f =\frac{\partial^{|[\alpha]|}}{\partial x_1^{\alpha_1} \cdots \partial x_d^{\alpha_d}} f$ denote the partial derivative of $f$ under a multi-index $[\alpha] = \left(\alpha_1,...,\alpha_d\right)$ with $\alpha_j\geq 0$ and $|[\alpha]|=\sum_{j=1}^d \alpha_j$. We write its gradient and Hessian matrix at $\bm{x}\in \mathbb{R}^d$ as $\nabla f(\bm{x})$ and $\nabla^2 f(\bm{x})$, respectively. We also denote the supremum norm of $f$ by $\norm{f}_{\infty}= \sup_{\bm{x}\in \mathbb{R}^d} |f(\bm{x})|$. If $\bm{f}:\mathbb{R}^d\to \mathbb{R}^k$ is vector-valued, then $\norm{\bm{f}}_{\infty} = \sup_{\bm{x}\in \mathbb{R}^d} \max_{1\leq j\leq k} |\bm{f}_j(\bm{x})|$. Throughout the theoretical results, however, supremum norms comparing an estimator with the probability density function or one of its derivatives are taken over its support.

For any two probability measures $P_1,P_2$ on $\left(\mathbb{R}^d, \mathcal{B}\right)$ with finite $q$-th moments, where $\mathcal{B}$ is the Borel $\sigma$-field on $\mathbb{R}^d$ and $q\geq 1$, their Wasserstein-$q$ distance is defined by
$$\mathrm{Wass}_q(P_1,P_2) := \inf_{P_J\in \mathcal{J}(P_1,P_2)} \left[\int_{\mathbb{R}^d\times \mathbb{R}^d} \norm{\bm{x}-\bm{y}}^q dP_J(\bm{x},\bm{y})\right]^{\frac{1}{q}},$$
where $\mathcal{J}(P_1,P_2)$ consists of all joint distributions for $(\bm{X},\bm{Y})$ that have marginal distributions $P_1$ and $P_2$. 
For two nonempty sets $A_1,A_2\subset \mathbb{R}^d$, let $\mathrm{Leb}(A_1)$ and $\mathrm{Leb}(A_2)$ denote their Lebesgue measures, respectively, and let
$$\mathrm{Haus}(A_1,A_2) = \max\left\{\sup_{\bm{x}\in A_1} d(\bm{x},A_2),\, \sup_{\bm{y}\in A_2} d(\bm{y},A_1)\right\}$$
denote their Hausdorff distance, where $d(\bm{x},A_1) = \inf_{\bm{y}\in A_1} \norm{\bm{x}-\bm{y}}$.

We write $a_n \lesssim b_n$ (or equivalently $b_n \gtrsim a_n$) if $a_n$ is bounded above by a constant multiple of $b_n$, and $a_n \asymp b_n$ if both $a_n \lesssim b_n$ and $a_n \gtrsim b_n$ hold. We also use the standard asymptotic Landau notation throughout the paper. For deterministic sequences ${h_n}$ and ${g_n}$ with $g_n > 0$, we write $h_n = O(g_n)$ if $\frac{|h_n|}{g_n}$ is bounded for all sufficiently large $n$, and $h_n = o(g_n)$ if $\frac{|h_n|}{g_n} \to 0$ as $n \to \infty$. For a random sequence $X_n$, $X_n=o_P(g_n)$ means that $\frac{X_n}{g_n}$ converges to 0 in probability, while $X_n=O_P(g_n)$ indicates that $\frac{X_n}{g_n}$ is bounded in probability as $n\to \infty$.

Let $\mathbb{X}_n=\left\{\bm{X}_1,...,\bm{X}_n \right\}$ be a random sample of independent and identically distributed (i.i.d.) observations from a distribution $P$ on $\mathbb{R}^d$. 
We impose the following regularity conditions throughout the paper.

\begin{assump}[Regular Morse density]
	\label{assump:density_reg} \leavevmode \noindent
	\begin{enumerate}[label=(\alph*)]
		\item The distribution $P$ is supported on a compact set $\mathcal{C} \subset \mathbb{R}^d$, and has a Lebesgue density $p$ satisfying $\inf_{\bm{x}\in \mathcal{C}} p(\bm{x})\geq p_{\min}>0$. Moreover, there exist constants $C_{\mathcal{C}},\epsilon_0>0$ such that $\mathrm{Leb}\left(\mathcal{C}\cap B(\bm{x},\epsilon) \right) \geq C_{\mathcal{C}} \cdot \mathrm{Leb}\left(B(\bm{x},\epsilon) \right)$ for every $\bm{x}\in \mathcal{C}$ and $0< \epsilon \leq \epsilon_0$.
		
		\item The restriction of $p$ to $\mathcal{C}$ admits a three-times continuously differentiable extension to an open neighborhood of $\mathcal{C}$, with derivatives uniformly bounded up to third order in that neighborhood.
		
		\item The extension of $p$ has finitely many critical points in the neighborhood of $\mathcal{C}$, and they all lie in the interior of $\mathcal{C}$. In addition, $\nabla^2 p(\bm{x})$ is non-singular at every critical point in $\mathcal{C}$, and the local modes have pairwise distinct density values.
		
		\item The set $\mathcal{C}$ is forward invariant for the vector field $\nabla p$, \emph{i.e.}, every solution of ${\bm{\gamma}}_{\bm{x}}'(t)=\nabla p(\bm{\gamma}_{\bm{x}}(t))$ with $\bm{\gamma}_{\bm{x}}(0) = \bm{x}\in \mathcal{C}$ remains in $\mathcal{C}$ for all $t\geq 0$.
	\end{enumerate}
\end{assump}

The compact support condition in Assumption~\ref{assump:density_reg}(a) is not essential to our theory and could be relaxed under suitable tail and localization conditions. The additional thickness condition rules out arbitrarily sharp cusps and lower-dimensional structures at the boundary of $\mathcal{C}$, which is a standard regularity condition in support estimation \citep{cuevas1990pattern,cuevas1997plug}. Assumption~\ref{assump:density_reg}(b) imposes standard smoothness on $p$ through a smooth extension and is thus compatible with a compactly supported distribution whose density is bounded away from zero on its support. Assumption~\ref{assump:density_reg}(c) requires the relevant critical points of $p$ in a neighborhood of $\mathcal{C}$ to be finite and non-degenerate. In particular, the extension is a Morse function on that neighborhood \citep{milnor1963morse,banyaga2004lectures}. The assumption of distinct modal densities is not needed for GGDPC path-length stability in \autoref{sec:stability}. Finally, Assumption~\ref{assump:density_reg}(d) makes the gradient ascent flow of $p$ intrinsic to $\mathcal{C}$, so that the smooth extension outside the support does not affect the clustering geometry.

\subsection{Density Peak Clustering}
\label{subsec:dpc_review}

Throughout the paper, ties in (estimated or population) density values and 1NN searches are resolved by a fixed strict ordering of the observations. Under Assumption~\ref{assump:density_reg}(a), such ties occur with probability zero, but this convention makes all finite-sample quantities below well-defined.

The classical density peak clustering (DPC) method introduced in \cite{rodriguez2014clustering} takes estimated density values $\hat{p}$ on $\mathbb{X}_n=\left\{\bm{X}_1,...,\bm{X}_n\right\}$ as input and defines 
$$\tilde{\Phi}_n(\bm{x}) \in \argmin\limits_{\bm{X}_i\in \mathbb{X}_n} \left\{\norm{\bm{X}_i-\bm{x}}: \hat{p}(\bm{X}_i) > \hat{p}(\bm{x})\right\}$$ 
as the 1NN of $\bm{x}\in \mathbb{R}^d$ with higher estimated density. The map $\bm{x}\mapsto \tilde{\Phi}_n(\bm{x})$ is undefined when no observation ranks above $\bm{x}$ and, under the above tie-breaking convention, is unique whenever it is well-defined. The corresponding 1NN uphill distance from $\bm{x}$ is given by
\begin{equation}
	\label{1NN_DPC_dist}
	\tilde{w}(\bm{x}) = 
	\begin{cases}
		\infty & \text{ if } \tilde{\Phi}_n(\bm{x}) \text{ is undefined},\\
		\norm{\tilde{\Phi}_n(\bm{x}) - \bm{x}} & \text{ otherwise}.
	\end{cases}
\end{equation}
Following \cite{tobin2023theoretical}, we describe the DPC algorithm through a directed acyclic graph, also referred to as a $\Delta$-tree in \cite{li2018comparative}.

Let $\tilde{G}$ denote the DPC graph with vertex set $V(\tilde{G})=\mathbb{X}_n$. Its edge set $E(\tilde{G})$ contains $\bm{X}_i \to \bm{X}_j$ whenever $\bm{X}_j = \tilde{\Phi}_n(\bm{X}_i)$ for $i,j=1,...,n$ and $i\neq j$. Each vertex $\bm{X}_i$ is assigned the weight $\tilde{w}(\bm{X}_i)$ defined in \eqref{1NN_DPC_dist}.
Then, DPC identifies the cluster centers as $\tilde{\mathcal{M}}=\left\{\bm{X}_i\in V(\tilde G): \tilde{w}(\bm{X}_i) > \lambda\right\}$ for a threshold $\lambda >0$. 
Removing the outgoing edges from the selected centers partitions $\tilde{G}$ into connected components, which define the resulting clusters. Alternative thresholding rules for cluster centers that combine uphill distances with density estimates are also applicable; see \autoref{app:thres_beyond_dist} for further discussion.
DPC may also designate observations with low estimated densities as noise. For instance, if $q_{\alpha}$ denotes the empirical lower $\alpha$-quantile of $\left\{\hat{p}(\bm{X}_i): i=1,...,n\right\}$ for $\alpha \in (0,1)$, the noise set can be defined as $\left\{\bm{X}_i\in V(\tilde G): \hat{p}(\bm{X}_i) < q_{\alpha}\right\}$.

Importantly, the DPC construction does not require $\hat{p}$ to be a consistent estimator of $p$. It is sufficient for $\hat{p}$, or its limiting function $\bar{p}$, to preserve the ordering induced by $p$, \emph{i.e.}, $\hat{p}(\bm{x}) > \hat{p}(\bm{y})$ or $\bar{p}(\bm{x}) > \bar{p}(\bm{y})$ whenever $p(\bm{x}) > p(\bm{y})$.

\subsection{Density Mode Clustering}
\label{subsec:mode_clu}

Given any $\bm{x}\in \mathcal{C}$, the gradient ascent flow (or integral curve) of a differentiable density $p$ starting at $\bm{x}$ is the function $\bm{\gamma}_{\bm{x}}:[0,\infty) \to \mathbb{R}^d$ defined by the ordinary differential equation
\begin{equation}
	\label{gradient_flow}
	\bm{\gamma}_{\bm{x}}'(t)=\nabla p(\bm{\gamma}_{\bm{x}}(t)),\qquad \bm{\gamma}_{\bm{x}}(0) = \bm{x}.
\end{equation}
The destination of the gradient ascent flow starting at $\bm{x}$ is defined as $\mathrm{dest}(\bm{x}) = \lim_{t\to \infty} \bm{\gamma}_{\bm{x}}(t)$.
Under Assumption~\ref{assump:density_reg}, \eqref{gradient_flow} is well-defined on $[0,\infty)$, with $\mathrm{dest}(\bm{x})$ being a critical point of $p$; see Section 9.3 in \cite{hirsch2012differential}. Moreover, $\mathrm{dest}(\bm{x}) \in \mathcal{M}$ for almost every $\bm{x} \in \mathcal{C}$ except for a set with Lebesgue measure zero, where $\mathcal{M}$ denotes the set of local modes of $p$. Consequently, density mode clustering defines population clusters through the \emph{basins of attraction} $\mathcal{C}_1,...,\mathcal{C}_{|\mathcal{M}|}$ for the $|\mathcal{M}|$ local modes, which are defined by
\begin{equation}
	\label{basin_attraction}
	\mathcal{C}_j = \left\{\bm{x}\in \mathcal{C}: \mathrm{dest}(\bm{x}) = \bm{m}_j\right\} \quad \text{ for } \quad \bm{m}_j\in \mathcal{M}.
\end{equation}
In the sequel, $\partial \mathcal{C}_j$ denotes the boundary of $\mathcal{C}_j$ relative to the support $\mathcal{C}$, which is also known as the \emph{separatrix} of gradient ascent flow \eqref{gradient_flow}. For $\bm{x}\in \mathbb{R}^d$ and $r>0$, let $d(\bm{x},\partial \mathcal{C}_a) = \inf_{\bm{y}\in \partial \mathcal{C}_a} \norm{\bm{x} - \bm{y}}$ and $\mathcal{C}_j\ominus r:= \left\{\bm{x}\in \mathcal{C}_j: d(\bm{x},\partial \mathcal{C}_j) \geq r\right\}$.

In practice, $p$ is commonly estimated by the kernel density estimator (KDE; \citealt{parzen1962estimation,scott2015multivariate}) as:
\begin{equation}
	\label{KDE}
	\hat{p}(\bm{x}) = \frac{1}{nh^d} \sum_{i=1}^n K\left(\frac{\norm{\bm{x} - \bm{X}_i}}{h}\right),
\end{equation}
where $h>0$ is a smoothing bandwidth parameter and $K:[0,\infty)\to [0,\infty)$ is a kernel profile satisfying $\int_{\mathbb{R}^d} K(\norm{\bm{u}})\, d\bm{u}=1$ and $\int_{\mathbb{R}^d} \norm{\bm{u}}^2 K(\norm{\bm{u}})\, d\bm{u} <\infty$. The set $\hat{\mathcal{M}}$ of estimated local modes from $\hat{p}$ can be obtained using the mean shift algorithm \citep{fukunaga1975estimation,comaniciu2002mean}, which iterates the following formula until convergence:
$$\bm{x}^{(0)} = \bm{x}, \qquad \bm{x}^{(k+1)} = \frac{\sum_{i=1}^n \bm{X}_i K\left(\frac{\norm{\bm{x}^{(k)} - \bm{X}_i}}{h}\right)}{\sum_{j=1}^n K\left(\frac{\norm{\bm{x}^{(k)} - \bm{X}_j}}{h}\right)} \quad \text{ for } k=0,1,....$$
Under standard conditions, the mean shift iteration converges to an estimated local mode $\hat{\bm{m}}_j\in \hat{\mathcal{M}}$ \citep{cheng1995mean,li2007note,ghassabeh2013convergence,ghassabeh2015sufficient}. The iteration can be viewed as a gradient ascent procedure with an adaptive step size and approximates the gradient ascent flow of $\hat{p}$ as the step size tends to 0 \citep{arias2016errata}. Accordingly, we define the sample basins of attraction by
\begin{equation}
	\label{basin_attraction_est}
	\hat{\mathcal{C}}_j = \left\{\bm{x}\in \mathbb{R}^d: \hat{\mathrm{dest}}(\bm{x}) = \hat{\bm{m}}_j\right\} \quad \text{ for } \quad \hat{\bm{m}}_j\in \hat{\mathcal{M}},
\end{equation}
where $\hat{\mathrm{dest}}(\bm{x})$ denotes the destination of the gradient ascent flow of $\hat{p}$ starting at $\bm{x}$. The observations in $\mathbb{X}_n=\left\{\bm{X}_1,...,\bm{X}_n\right\}$ are thus clustered according to their destination local modes. 

\section{Gradient-Guided Density Peak Clustering}
\label{sec:GGDPC}

While DPC is computationally efficient and scalable, its statistical analysis is complicated by the irregular and unstable behavior of the 1NN uphill paths underlying the algorithm. In this section, we propose a more stable variant, termed \emph{gradient-guided density peak clustering} (GGDPC). Inspired by density mode clustering in \autoref{subsec:mode_clu}, GGDPC inserts a gradient ascent step before each 1NN uphill search.

Let $\hat{p}$ and $\hat{g}$ be estimators of $p$ and its gradient $\nabla p$, respectively. When $\hat{p}$ is differentiable, a natural choice is $\hat{g}=\nabla\hat{p}$. For a general GGDPC implementation, $\hat{g}$ can be any consistent estimator of $\nabla p$. 
For any $\bm{x} \in \mathbb{R}^d$, we define the gradient-guided 1NN uphill point by
\begin{equation}
	\label{gg_1NN_uphill}
	\hat{\Phi}_n(\bm{x}) \in \argmin_{\bm{X}_i\in \mathbb{X}_n} \left\{\norm{\bm{X}_i - \left(\bm{x} + \eta_n \hat{g}(\bm{x})\right)}: \hat{p}(\bm{X}_i) > \hat{p}(\bm{x}) \right\},
\end{equation}
where $\eta_n>0$ is a step size parameter depending on the sample size $n$. As in DPC, $\hat{\Phi}_n(\bm{x})$ is undefined when no observation ranks above $\bm{x}$ under $\hat{p}$ and the strict ordering rule. The corresponding gradient-guided 1NN uphill distance is
\begin{equation}
	\label{1NN_GGDPC_dist}
	\hat{w}_n(\bm{x}) = 
	\begin{cases}
		\infty & \text{ if }\hat{\Phi}_n(\bm{x}) \text{ is undefined},\\
		\norm{\hat{\Phi}_n(\bm{x}) - \bm{x}} & \text{ otherwise}.
	\end{cases}
\end{equation}
A key feature of $\hat{\Phi}_n(\bm{x})$ is that the higher-density constraint is imposed relative to the \emph{original} point $\bm{x}$ rather than the gradient ascent update $\bm{x}+\eta_n\hat{g}(\bm{x})$. This choice is important for two reasons. First, near the global mode, there may be no observations with estimated densities higher than that at the gradient ascent update. Second, near a local mode, imposing the higher-density constraint at the updated point may direct multiple nearby observations toward different higher-density regions, creating spurious cluster centers within the same modal region.
Thus, the gradient ascent update in \eqref{gg_1NN_uphill} provides a stable search direction, while the density ordering remains anchored at the original observation.

Analogous to DPC, GGDPC induces a directed acyclic graph $G$ with vertex set $V(G)=\mathbb{X}_n$. For each observation $\bm{X}_i$ that is not a sample global mode, we add the directed edge $\bm{X}_i \to \bm{X}_j$ into its edge set $E(G)$ whenever $\bm{X}_j = \hat{\Phi}_n(\bm{X}_i)$ and $i\neq j$. We also associate $\bm{X}_i$ with the weight $\hat{w}_n(\bm{X}_i)$ defined in \eqref{1NN_GGDPC_dist}. For a threshold $\lambda>0$, GGDPC identifies the cluster centers as $\hat{\mathcal{M}}=\left\{\bm{X}\in V(G): \hat{w}_n(\bm{X}) > \lambda\right\}$. Removing the outgoing edges from these cluster centers yields the truncated graph $G_{\lambda}$.
Each remaining observation is assigned to a cluster center by following its directed path in $G_{\lambda}$. Algorithm~\ref{algo:ggdpc} summarizes the proposed GGDPC procedure.

\begin{algorithm}[t]
\caption{The Gradient-Guided Density Peak Clustering (GGDPC) Algorithm}
\label{algo:ggdpc}
\begin{algorithmic}
	\State {\bf Input:} Data $\mathbb{X}_n=\left\{\bm{X}_1,...,\bm{X}_n\right\}$, density estimator $\hat{p}$, gradient estimator $\hat{g}$, step size $\eta_n >0$, and threshold value $\lambda >0$.
	\begin{enumerate}
		\item Compute $\hat{p}(\bm{X}_i)$ and $\hat{g}(\bm{X}_i)$ for each observation $\bm{X}_i\in \mathbb{X}_n$.
		
		\item For each $\bm{X}_i$ that is not a sample global mode, derive its gradient-guided 1NN uphill point 
		$$\hat{\Phi}_n(\bm{X}_i) \in \argmin_{\bm{X}_j\in \mathbb{X}_n} \left\{\norm{\bm{X}_j - \left(\bm{X}_i + \eta_n \hat{g}(\bm{X}_i)\right)}: \hat{p}(\bm{X}_j) > \hat{p}(\bm{X}_i) \right\},$$
		and add $\bm{X}_i \to \hat{\Phi}_n(\bm{X}_i)$ to the edge set $E(G)$.
		
		\item Compute the gradient-guided 1NN uphill distance
		$$\hat{w}_n(\bm{X}_i) = 
		\begin{cases}
			\infty & \text{ if } \bm{X}_i \text{ is the sample global mode},\\
			\norm{\hat{\Phi}_n(\bm{X}_i) - \bm{X}_i} & \text{ otherwise},
		\end{cases} \quad \text{ for each } \bm{X}_i\in \mathbb{X}_n.$$
		
		\item Set $\hat{\mathcal{M}}=\left\{\bm{X}_i\in \mathbb{X}_n: \hat{w}_n(\bm{X}_i) > \lambda\right\}$. 
		For each $\bm{X}_i$, follow the directed edges in $E(G)$ until reaching a cluster center $\hat{\bm{m}}_j\in \hat{\mathcal{M}}$, and assign $\bm{X}_i$ to the corresponding cluster $\hat{\mathcal{C}}_j$.
	\end{enumerate}
	\State {\bf Output:} Cluster sets $\hat{\mathcal{C}}_1,...,\hat{\mathcal{C}}_{|\hat{\mathcal{M}}|}$.
\end{algorithmic}
\end{algorithm}

As illustrated in \autoref{fig:gmm_graph}, the gradient ascent step directs the subsequent 1NN search toward higher-density regions and stabilizes the resulting uphill paths, particularly in low-density regions. Near local modes, where the gradient is small, GGDPC behaves similarly to DPC. Thus, GGDPC alleviates irregular cross-boundary propagation while preserving the robust cluster-center identification mechanism of DPC.

For several theoretical results below, we use the following regularity condition on $\hat{p}$ and $\hat{g}$, which is satisfied, for instance, by the KDE in \eqref{KDE} under an appropriate differentiable kernel. It facilitates explicit convergence rates and uniform control of the GGDPC updates.
In \autoref{app:proof_ggdpc_estimated}, we provide a weaker condition (Assumption~\ref{assump:sample_uphill}) that is sufficient for stability of the sample GGDPC path in \autoref{subsec:stability_sample}. That weaker condition allows the gradient estimator $\hat{g}$ to differ from $\nabla\hat{p}$ and $\hat{p}$ to be non-differentiable.

\begin{assump}[Differentiability of the density estimator]
	\label{assump:sample_den_diff}
	There is a fixed open neighborhood of $\mathcal{C}$ on which the density estimator $\hat{p}$ is twice continuously differentiable, $\hat{g}=\nabla \hat{p}$, and its partial derivatives are bounded in probability up to the second order.
\end{assump}

\section{Convergence of GGDPC Clustering}
\label{sec:GGDPC_conv_clu}

In this section, we study the convergence of GGDPC clustering from two complementary perspectives. We first establish consistency of the cluster centers selected by GGDPC, and then analyze the agreement between the resulting clustering assignments and the population modal partition using the ARI.

\subsection{Modal Consistency}
\label{subsec:mode_consistency}

Under Assumption~\ref{assump:density_reg}(c), let 
$$\mathcal{M}=\left\{\bm{x}\in \mathcal{C}: \nabla p(\bm{x})=0, \rho_{\max}(\nabla^2 p(\bm{x})) < 0\right\} = \left\{\bm{m}_1,...,\bm{m}_{|\mathcal{M}|}\right\}$$ 
denote the set of local modes of $p$, where $\rho_{\max}(\nabla^2 p(\bm{x}))$ is the largest eigenvalue of $\nabla^2 p(\bm{x})$. Under Assumption~\ref{assump:density_reg}(c), we assume without loss of generality that $p(\bm{m}_1) > \cdots > p(\bm{m}_{|\mathcal{M}|})$. Specifically, $\bm{m}_1$ is the unique global mode.

For each non-global mode $\bm{m}_j\in \mathcal{M}$, we define its strict upper level set within $\mathcal{C}$ by
$$\mathcal{U}_j := \mathcal{U}_{\bm{m}_j}=\left\{\bm{x}\in \mathcal{C}: p(\bm{x}) > p(\bm{m}_j)\right\}$$
and denote its closure by $\overline{\mathcal{U}}_j$. The projection set of $\bm{m}_j$ onto $\overline{\mathcal{U}}_j$ is
\begin{equation}
	\label{mode_proj_upper}
	\Pi_{\overline{\mathcal{U}}_j}(\bm{m}_j) := \argmin_{\bm{x}\in \overline{\mathcal{U}}_j} \norm{\bm{x} - \bm{m}_j}.
\end{equation}
We define the population uphill shift distance from $\bm{m}_j$ to the higher-density region $\overline{\mathcal{U}}_j$ by
$$\psi_j:=w(\bm{m}_j) = \begin{cases}
	\infty & \text{ if } \bm{m}_j \text{ is a global mode of } p,\\
	d(\bm{m}_j,\overline{\mathcal{U}}_j) & \text{ otherwise}.
\end{cases}$$
The scalar $\psi_j$ is well-defined even when the projection $\Pi_{\overline{\mathcal{U}}_j}(\bm{m}_j)$ consists of more than one point. In particular, when the projection set is a singleton, $\psi_j=\norm{\bm{m}_j - \Pi_{\overline{\mathcal{U}}_j}(\bm{m}_j)}$ for $j\geq 2$.
Then, we define the $\lambda$-separated upper mode set of $p$ for a fixed threshold $\lambda>0$ by
$$\mathcal{M}_{\lambda} = \left\{\bm{m}_j\in \mathcal{M}: \psi_j > \lambda\right\}.$$
For the GGDPC graph, we recall from \eqref{1NN_GGDPC_dist} that $\hat{\mathcal{M}}_{\lambda} = \left\{\bm{X}_i\in \mathbb{X}_n: \hat{w}_n(\bm{X}_i) > \lambda\right\}$ is the corresponding set of GGDPC cluster centers, or equivalently the sinks of the truncated graph $G_{\lambda}$ at threshold $\lambda$.

\begin{theorem}[Consistency of GGDPC modes]
\label{thm:modal_const}
Suppose that Assumption~\ref{assump:density_reg} holds. If $\eta_n = o(1)$, $\norm{\hat{p}-p}_{\infty}=o_P(1)$, and $\norm{\hat{g} - \nabla p}_{\infty} = o_P(1)$, then for a fixed level $\lambda>0$ satisfying $\lambda\notin \left\{\psi_j:1< j\leq |\mathcal{M}|\right\}$, with probability tending to one, $\left|\mathcal{M}_{\lambda}\right| = \left|\hat{\mathcal{M}}_{\lambda}\right|$ and 
$$\mathrm{Haus}\left(\mathcal{M}_{\lambda}, \hat{\mathcal{M}}_{\lambda}\right) = O_P\left(\left(\frac{\log n}{n}\right)^{\frac{1}{d}} + \norm{\hat{p} - p}_{\infty}^{\frac{1}{2}}\right).$$
If, in addition, Assumption~\ref{assump:sample_den_diff} holds, then the Hausdorff distance rate sharpens to
$$\mathrm{Haus}\left(\mathcal{M}_{\lambda}, \hat{\mathcal{M}}_{\lambda}\right) = O_P\left(\left(\frac{\log n}{n}\right)^{\frac{1}{d}} + \min\left\{\norm{\hat{p} - p}_{\infty}^{\frac{1}{2}},\, \norm{\hat{g} - \nabla p}_{\infty}\right\}\right).$$
\end{theorem}

The proof of \autoref{thm:modal_const} is in \autoref{app:proof_modal_const}. Similar modal consistency results for DPC and related algorithms are established in \cite{jiang2017consistency,verdinelli2018analysis,tobin2023theoretical}. The first rate in \autoref{thm:modal_const} separates the sample coverage error $q_n=\left(\frac{\log n}{n}\right)^{\frac{1}{d}}$ from the localization error $\norm{\hat{p} - p}_{\infty}^{\frac{1}{2}}$ induced by density estimation. If $\hat{p}$ is constructed by a (boundary-corrected) KDE \eqref{KDE} under Assumption~\ref{assump:density_reg}(b) and standard kernel regularity conditions, then $\norm{\hat{p} - p}_{\infty}=O_P\left(\left(\frac{\log n}{n}\right)^{\frac{2}{d+4}}\right)$ and $\mathrm{Haus}\left(\mathcal{M}_{\lambda}, \hat{\mathcal{M}}_{\lambda}\right) = O_P\left(\left(\frac{\log n}{n}\right)^{\frac{1}{d+4}}\right)$, which is \emph{not} minimax optimal for mode estimation. However, under the additional Assumption~\ref{assump:sample_den_diff} and a (boundary-corrected) KDE with a differentiable kernel, $\norm{\hat{g} - \nabla p}_{\infty}=O_P\left(\left(\frac{\log n}{n}\right)^{\frac{2}{d+6}}\right)$ and $\mathrm{Haus}\left(\mathcal{M}_{\lambda}, \hat{\mathcal{M}}_{\lambda}\right) = O_P\left(\left(\frac{\log n}{n}\right)^{\min\left\{\frac{2}{d+6}, \frac{1}{d}\right\}}\right)$. When $d\leq 6$, this agrees with the canonical minimax rate $O_P\left(\left(\frac{\log n}{n}\right)^{\frac{2}{d+6}}\right)$ up to a logarithmic factor \citep{romano1988weak,arias2022estimation}. 


\subsection{Convergence of the Adjusted Rand Index}

When reference labels are available, the clustering agreement is commonly assessed by the adjusted Rand index (ARI; \citealt{hubert1985comparing}). Let $\mathcal{P}(\mathbb{X}_n)$ denote the collection of all partitions of $\mathbb{X}_n=\left\{\bm{X}_1,...,\bm{X}_n\right\}$. Given the density $p$ and the basins of attraction $\mathcal{C}_1,...,\mathcal{C}_{|\mathcal{M}|}$ in \eqref{basin_attraction} under Assumption~\ref{assump:density_reg}, the population modal partition of $\mathbb{X}_n$ is defined as $\mathcal{P}_n^*:= \left\{\mathbb{X}_n\cap \mathcal{C}_1,...,\mathbb{X}_n\cap \mathcal{C}_{|\mathcal{M}|}\right\}$. Then, for any estimated clustering $\hat{\mathcal{P}}_n \in \mathcal{P}(\mathbb{X}_n)$, the ARI is defined by
\begin{equation}
	\label{ARI}
	\mathrm{ARI}\left(\hat{\mathcal{P}}_n, \mathcal{P}_n^*\right) = \frac{2(N_{tp} N_{tn} - N_{fp} N_{fn})}{(N_{tp} + N_{fn}) (N_{fn} + N_{tn}) + (N_{tp} + N_{fp}) (N_{fp} +N_{tn})},
\end{equation}
where $N_{tp}$ denotes the number of observation pairs assigned to the same cluster under both $\hat{\mathcal{P}}_n$ and $\mathcal{P}_n^*$; $N_{tn}$ denotes the number assigned to different clusters under both partitions; $N_{fp}$ denotes the number assigned to the same cluster under $\hat{\mathcal{P}}_n$ but to different clusters under $\mathcal{P}_n^*$; and $N_{fn}$ denotes the number assigned to different clusters under $\hat{\mathcal{P}}_n$ but to the same cluster under $\mathcal{P}_n^*$. When $\hat{\mathcal{P}}_n$ and $\mathcal{P}_n^*$ coincide up to relabeling, $\mathrm{ARI}\left(\hat{\mathcal{P}}_n, \mathcal{P}_n^*\right)=1$.

To establish consistency of the GGDPC assignments with the population modal partition, we impose an additional geometric assumption. Let $\Pi_{\partial\mathcal{C}_a}(\bm{x}) := \argmin_{\bm{s}\in\partial\mathcal{C}_a} \norm{\bm{x}-\bm{s}}$ denote the projection set onto $\partial \mathcal{C}_a$.

\begin{assump}[Normal repulsion from the separatrix]
	\label{assump:sep_repulsion}
	The boundary $\partial\mathcal{C}_a$ is a finite union of stable manifolds of saddle points whose Hessian matrices have exactly one positive and $d-1$ negative eigenvalues, and is a $C^2$ hypersurface away from finitely many critical points. Define its regular part by
	$$\mathcal S_{\rm reg} = \left\{\bm{s}\in \partial\mathcal{C}_a: \nabla p(\bm{s})\neq 0 \text{ and } \partial\mathcal{C}_a \text{ is locally a } C^2 \text{ hypersurface near }\bm{s}\right\}.$$
	There exist constants $r_{\mathcal S}, \rho_{\mathcal S}>0$ such that, whenever $\bm{x} \in \mathcal{C}_a$ satisfies $0 < d(\bm{x},\partial \mathcal{C}_a) < r_{\mathcal{S}}$ and at least one closest boundary point belongs to $\mathcal S_{\rm reg}$, the projection $\Pi_{\partial\mathcal{C}_a}(\bm{x})$ is a singleton. Moreover,
	$$\nu(\bm{s})^T\nabla^2p(\bm{s})\nu(\bm{s})\geq \rho_{\mathcal S}$$
	for every $\bm{s}\in\mathcal S_{\rm reg}$ outside sufficiently small fixed neighborhoods of the boundary saddle points, where $\nu(\bm{s})$ denotes the unit normal vector to $\partial\mathcal{C}_a$ at $\bm{s}$ pointing into $\mathcal{C}_a$.
\end{assump}

Assumption~\ref{assump:sep_repulsion} imposes a local repulsion condition along the regular part of the separatrix. In particular, it requires the Hessian of $p$ to be strictly positive in the inward normal direction. Lemma~\ref{lem:sep_repulsion} extends this local repulsion through neighborhoods of boundary saddle points using the stable manifold theorem. Lemma~\ref{lem:flow_separation_sep} later in \autoref{subsec:stability_popul} converts this local property into a global lower bound on the distance of a gradient flow trajectory from the separatrix. We also provide a two-Gaussian mixture example in \autoref{app:sep_repulsion_example} for an illustration, though this assumption holds more generally. A related condition is used in \cite{chen2017statistical} to study the stability of stable and unstable manifolds of a density function.

\begin{theorem}[Convergence of GGDPC under the ARI]
	\label{thm:GGDPC_ARI}
	Suppose that Assumptions~\ref{assump:density_reg},
	\ref{assump:sample_den_diff}, and
	\ref{assump:sep_repulsion} hold for every modal basin of attraction. Let $|\mathcal{M}|\geq 2$ and $\hat{\mathcal{P}}_{n,\lambda}$ be the partition of $\mathbb{X}_n$ produced by GGDPC under a fixed threshold $\lambda \in \left(0, \min_{2\leq j \leq |\mathcal{M}|} \psi_j\right)$. If $\eta_n + \frac{q_n}{\eta_n}=o(1)$, $\norm{\hat{p} - p}_{\infty}=o_P(1)$, and $\norm{\hat{g}-\nabla p}_{\infty}= o_P(1)$ with $q_n=\left(\frac{\log n}{n}\right)^{\frac{1}{d}}$, then
	$$1 - \mathrm{ARI}\left(\hat{\mathcal{P}}_{n,\lambda}, \mathcal{P}_n^*\right) = O_P\left(\eta_n + \frac{q_n}{\eta_n} + \norm{\hat{g}-\nabla p}_{\infty}\right).$$
\end{theorem}

The proof of \autoref{thm:GGDPC_ARI} is in \autoref{app:ggdpc_ari_proof}. Balancing the first two terms in the rate gives $\eta_n\asymp \sqrt{q_n}$. In particular, when $\hat{g}$ is constructed using (boundary-corrected) KDE \eqref{KDE} with bandwidth $h_n \asymp \left(\frac{\log n}{n}\right)^{\frac{1}{d+6}}$, 
\begin{equation}
	\label{ARI_rate_kde}
	1 - \mathrm{ARI}\left(\hat{\mathcal{P}}_{n,\lambda}, \mathcal{P}_n^*\right) = O_P\left(\left(\frac{\log n}{n}\right)^{\min\left\{\frac{1}{2d}, \frac{2}{d+6}\right\}}\right),
\end{equation}
and the sample discretization of the GGDPC path dominates the rate as $O_P\left(\left(\frac{\log n}{n}\right)^{\frac{1}{2d}}\right)$ when $d \geq 3$.

\section{GGDPC Dendrogram}
\label{sec:clu_tree}

Varying the GGDPC distance threshold $\lambda$ produces a nested family of partitions of the sample $\mathbb{X}_n$, and hence a natural hierarchical representation of the clustering structure. In this section, we formalize this representation as the \emph{GGDPC dendrogram}. We show that it is equivalent to a single linkage dendrogram constructed from the weighted GGDPC graph and establish its convergence, under the Gromov-Hausdorff distance, to a population \emph{modal distance dendrogram}. Although this population hierarchy is induced by the same density $p$ as the usual density cluster tree, the two hierarchies need not have the same merge topology.

\subsection{Definition and Computation}
\label{subsec:dendro_def}

Recall that $\mathcal{P}(\mathbb{X}_n)$ is the collection of all partitions of $\mathbb{X}_n$. For a fixed threshold $\lambda \geq 0$, let $G_{\lambda}$ be the truncated GGDPC graph obtained by removing every outgoing edge from an observation $\bm{X}_i$ with $\hat{w}_n(\bm{X}_i) > \lambda$. Equivalently, $G_{\lambda}$ retains precisely those directed edges $\bm{X}_i \to \hat{\Phi}_n(\bm{X}_i)$ for which $\hat{w}_n(\bm{X}_i) \leq \lambda$.
We define the \emph{GGDPC dendrogram} $(\mathcal{T}_G,\mathbb{X}_n)$ through a function $\mathcal{T}_G: [0,\infty) \to \mathcal{P}(\mathbb{X}_n)$ by letting $\mathcal{T}_G(\lambda)$ be the partition of $\mathbb{X}_n$ induced by the connected components of $G_{\lambda}$, where edge directions are ignored when defining connectedness.
In particular, each component corresponds to observations whose directed paths terminate at the same cluster center $\hat{\bm{m}}_j \in \hat{\mathcal{M}}$ of $G_{\lambda}$. Since $E(G_{\lambda}) \subset E(G_{\lambda'})$ whenever $0\leq \lambda\leq \lambda'$, the partitions become successively coarser as $\lambda$ increases, so $\mathcal{T}_G$ forms a hierarchy in the following sense.

\begin{proposition}[Hierarchy of the GGDPC dendrogram]
\label{prop:hier_cluster_tree}
For any $0\leq \lambda \leq \lambda' <\infty$, the partition $\mathcal{T}_G(\lambda)$ refines $\mathcal{T}_G(\lambda')$. Equivalently, for every $A \in \mathcal{T}_G(\lambda)$, there exists a unique $A' \in  \mathcal{T}_G(\lambda')$ such that $A \subseteq A'$.
\end{proposition}

Proposition~\ref{prop:hier_cluster_tree} follows directly from the tree structure of the GGDPC graph $G$. A more general version is established in Proposition~\ref{prop:general_thres_tree}, where edges are thresholded according to a scoring function that may depend jointly on the 1NN uphill distance $\hat{w}_n(\bm{X})$ and the estimated density $\hat{p}(\bm{X})$. The same construction applies to the original DPC graph. In particular, the GGDPC and DPC dendrograms can be interpreted as a proximity dendrogram in the definition of Chapter 3.2 in \cite{jain1988algorithms}; see also Section 3.1 of \cite{carlsson2010characterization}\footnote{Rigorously, we assume that the observations in $\mathbb{X}_n$ are distinct, which holds almost surely under the absolute continuity of $P$.}.

There are two equivalent ways to construct the GGDPC dendrogram from the graph $G$.

\begin{itemize}
	\item {\bf Divisive Approach:} As suggested in \autoref{sec:GGDPC}, we start from the full graph $G$ at $\lambda \geq \max_{\bm{X}_i\neq \bm{X}^*} \hat{w}_n(\bm{X}_i)$ based on \eqref{gg_1NN_uphill} and decrease $\lambda$ toward 0, where $\bm{X}^*$ is the root node with no outgoing edge. Whenever $\lambda$ passes below the weight of an edge $\bm{X}_i\to \bm{X}_j$, that edge is removed from $G_{\lambda}$. Since the underlying undirected graph of $G$ is a tree, removing each retained edge splits one connected component into two, thereby producing the successive bifurcations of the dendrogram.
	
	\item {\bf Agglomerative Approach:} Choose $d_{\star} > d_{\max} = \max_{1\leq i,j \leq n} \norm{\bm{X}_i - \bm{X}_j}$ and define a symmetric dissimilarity matrix $D\in \mathbb{R}^{n\times n}$ with $D_{ii}=0$ for all $i=1,...,n$ and 
	\begin{equation}
		\label{dist_mat_single_link}
		D_{ij} =D_{ji} = \begin{cases}
			\norm{\bm{X}_i -\bm{X}_j} & \text{ if }\bm{X}_i \text{ and } \bm{X}_j \text{ are adjacent in } G,\\
			d_{\star} & \text{ otherwise}.
		\end{cases}
	\end{equation}
	Thus, after ignoring edge directions, $D$ records the GGDPC edge weights for adjacent observations and assigns a common value $d_{\star}$ larger than every such edge weight to nonadjacent pairs. Applying single linkage clustering to $D$ then yields the same hierarchy as thresholding the GGDPC graph directly.
\end{itemize}

The following proposition formalizes this equivalence, whose proof is in \autoref{app:ggdpc_single_link}.

\begin{proposition}[GGDPC dendrogram via single linkage clustering]
	\label{prop:ggdpc_dendro}
	Let $\mathcal{T}_{\rm SL}(\lambda)$ denote the partition obtained by cutting the single linkage dendrogram induced by $D$ in \eqref{dist_mat_single_link} at height $\lambda$. Then,
	$$\mathcal{T}_{\rm SL}(\lambda)=\mathcal{T}_G(\lambda) \qquad \text{ for every } \lambda \geq 0.$$
	Consequently, the single linkage cluster tree induced by $D$ is identical to the GGDPC dendrogram $(\mathcal{T}_G,\mathbb{X}_n)$.
\end{proposition}


\subsection{Stability of the GGDPC Dendrogram}
\label{subsec:dendro_stability}

We study convergence of the GGDPC dendrogram through its associated ultrametrics under the Gromov-Hausdorff distance \citep{gromov1987hyperbolic,burago2001course}. Recall that an ultrametric $u: \mathbb{X}\times \mathbb{X} \to \mathbb{R}_+$ is a valid metric satisfying the strengthened triangle inequality $\max\left\{u(\bm{x},\bm{z}), u(\bm{z},\bm{y})\right\}\geq u(\bm{x},\bm{y})$ for any $\bm{x},\bm{y},\bm{z}\in \mathbb{X}$.
By the standard correspondence between dendrograms and ultrametrics in Theorem 9 of \cite{carlsson2010characterization}, we define an ultrametric over observations in $\mathbb{X}_n$ by
\begin{equation}
	\label{ultrametric_dendro}
	u(\bm{X}_i,\bm{X}_j):= u_{\mathbb{X}_n}(\bm{X}_i,\bm{X}_j) = \min\left\{r\geq 0: \bm{X}_i,\bm{X}_j \text{ belong to the same block of } \mathcal{T}_G(r)\right\}.
\end{equation}
In particular, $u_{\mathbb{X}_n}(\bm{X}_i,\bm{X}_i) = 0$. Let
$\mathrm{Path}_G(\bm X_i,\bm X_j)$ denote the edge set of the
unique undirected path in $G$ joining $\bm X_i$ and $\bm X_j$. Proposition~\ref{prop:ggdpc_dendro} implies that 
$$u_{\mathbb{X}_n}(\bm{X}_i,\bm{X}_j) = \max\left\{\hat{w}_n(\bm{X}_k): \left(\bm{X}_k, \hat{\Phi}_n(\bm{X}_k)\right) \in \mathrm{Path}_G(\bm X_i,\bm X_j)\right\}$$
when $\bm{X}_i\neq \bm{X}_j$.
For two GGDPC dendrograms $(\mathcal{T}_G, \mathbb{X}_n)$ and $(\mathcal{T}_{G'}, \mathbb{Y}_m)$ on data samples $\mathbb{X}_n$ and $\mathbb{Y}_m$ respectively, their corresponding ultrametrics $u_{\mathbb{X}_n}$ and $u_{\mathbb{Y}_m}$ can be constructed according to \eqref{ultrametric_dendro}. We define the distance between the dendrograms by
$$d_{\rm den}\left((\mathcal{T}_G, \mathbb{X}_n), (\mathcal{T}_{G'}, \mathbb{Y}_m)\right) = \mathrm{GH}\left((\mathbb{X}_n, u_{\mathbb{X}_n}),(\mathbb{Y}_m, u_{\mathbb{Y}_m})\right),$$
where $\mathrm{GH}(\cdot,\cdot)$ denotes the Gromov-Hausdorff distance; see \autoref{app:proof_conv_ggdpc_dendro} for its definition.

We next construct the population counterpart of the GGDPC dendrogram. The following assumption guarantees that each non-global mode $\bm{m}_j$ has a unique and non-degenerate projection \eqref{mode_proj_upper} onto its closed upper level set $\overline{\mathcal{U}}_j$.

\begin{assump}[Non-degenerate modal projection]
	\label{assump:unique_proj}
	For every non-global mode $\bm{m}_j\in \mathcal{M}$, the projection $\bm{z}_j:=\Pi_{\overline{\mathcal{U}}_j}(\bm{m}_j)$ is unique and satisfies the following conditions.
	\begin{enumerate}[label=(\alph*)]
		\item There exists a mode $\bm{m}_{\pi(j)} \in \mathcal{M}$ such that $\bm{z}_j \in \mathcal{C}_{\pi(j)}$, $\nabla p(\bm{z}_j)\neq \bm{0}$, and $d(\bm{z}_j,\partial\mathcal{C}_{\pi(j)})>0$.
		
		\item Let $\mu_j >0$ be the Lagrange multiplier determined by $\bm{z}_j-\bm{m}_j=\mu_j \nabla p(\bm{z}_j)$ under (a). There exists a constant $\kappa_j>0$ such that 
		$$\bm{v}^T\left[I_d - \mu_j \nabla^2 p(\bm{z}_j)\right]\bm{v} \geq \kappa_j$$ 
		for every $\bm{v}\in \mathbb{R}^d$ with $\norm{\bm{v}}=1$ and $\bm{v}^T \nabla p(\bm{z}_j)=0$, where $I_d\in \mathbb{R}^{d\times d}$ is the identity matrix.
	\end{enumerate}
\end{assump}

The uniqueness of $\bm{z}_j=\Pi_{\overline{\mathcal{U}}_j}(\bm{m}_j)$ in Assumption~\ref{assump:unique_proj} rules out exact distance ties between distinct connected components of the higher-density region $\overline{\mathcal{U}}_j$. Such ties may arise, for example, under symmetry of $p$, but are unstable under generic asymmetric perturbations. Assumption~\ref{assump:unique_proj}(a) further ensures that the projection $\bm{z}_j$ does not lie at a critical point of $p$ or on the stable manifold of a non-modal critical point. Notice that $\bm{z}_j$ cannot be another local mode in $\mathcal{M}$ when the modal heights $p(\bm{m}_1),...,p(\bm{m}_{|\mathcal{M}|})$ are distinct. 
The multiplier relation in Assumption~\ref{assump:unique_proj}(b) follows from the first-order condition for the constrained minimization
$$\min_{\bm{y}\in \mathcal{C}\setminus B(\bm{m}_j,r_j)} \frac{1}{2}\norm{\bm{y}-\bm{m}_j}^2 \quad \text{ subject to } \quad p(\bm{y}) \geq p(\bm{m}_j)$$
for some $r_j >0$, together with $\nabla p(\bm{z}_j)\neq 0$ in Assumption~\ref{assump:unique_proj}(a). The second-order condition in Assumption~\ref{assump:unique_proj}(b) is the genuine geometric restriction that excludes any flat tangential region at the projection $\bm{z}_j$. Without this condition, convergence of the 1NN uphill update $\hat{\bm{z}}_{j,n} = \hat{\Phi}_n(\hat{\bm{m}}_{j,n})$ of the sample local mode $\hat{\bm{m}}_{j,n}$ to the population projection $\bm{z}_j$ of the true local mode $\bm{m}_j$ can be arbitrarily slow, as opposed to our rate in Lemma~\ref{lem:emp_mode_proj} of \autoref{app:lem_stab_emp_proj}.
This is a mild non-degeneracy condition, and the two-Gaussian mixture example in \autoref{app:sep_repulsion_example} again provides a concrete setting in which it holds.

Recall from \autoref{subsec:mode_consistency} that $\mathcal{M}=\left\{\bm{m}_1,...\bm{m}_{|\mathcal{M}|}\right\}$ denotes the set of local modes of $p$, and we assume that $p(\bm{m}_1)> \cdots > p(\bm{m}_{|\mathcal{M}|})$. For each $j\geq 2$, $\bm{z}_j=\Pi_{\overline{\mathcal{U}}_j}(\bm{m}_j)$ and $\psi_j=\norm{\bm{m}_j-\bm{z}_j}$ in \eqref{mode_proj_upper}. Under Assumption~\ref{assump:unique_proj}, there is a unique index
$\pi(j)$ such that $\bm z_j\in \mathcal C_{\pi(j)}$.

We define a weighted directed graph $G_{\mathcal{M}}$ with vertex set $\mathcal{M}$ and edges $\bm{m}_j \to \bm{m}_{\pi(j)}$ for $j=2,...,|\mathcal{M}|$, where the edge from $\bm{m}_j$ has weight $\psi_j$. Since $p(\bm z_j)=p(\bm m_j)$, $\nabla p(\bm z_j)\neq \bm{0}$, and the gradient flow from $\bm z_j$ converges to $\bm m_{\pi(j)}$, we know that $p(\bm{m}_{\pi(j)}) > p(\bm{m}_j)$. Thus, the modal density strictly increases along every directed edge, and $G_{\mathcal{M}}$ is a rooted tree with root $\bm{m}_1$.

Let $\mathcal{P}(\mathcal{M})$ be the collection of all partitions of $\mathcal{M}$. For $\lambda \geq 0$, we define a mapping $\mathcal{T}_{\rm mode}:=\mathcal{T}_{G_{\mathcal{M}}}$ from $[0,\infty)$ to $\mathcal{P}(\mathcal{M})$ by letting  $\mathcal{T}_{\rm mode}(\lambda)$ be the partition of $\mathcal{M}$ obtained by retaining the edges of $G_{\mathcal{M}}$ with weights at most $\lambda$. We call $\mathcal{T}_{\rm mode}$ the \emph{modal distance dendrogram} (or cluster tree). Let
$\mathrm{Path}_{G_{\mathcal M}}(\bm m_j,\bm m_k)$ denote the
edge set of the unique undirected path in $G_{\mathcal{M}}$ joining $\bm m_j$ and $\bm m_k$.
Its associated ultrametric is defined by $u_{\mathcal{M}}(\bm{m}_j,\bm{m}_j)=0$ and
\begin{align*}
u_{\mathcal{M}}(\bm{m}_j,\bm{m}_k) &= \max\left\{\psi_{\ell}: \left(\bm{m}_{\ell},\bm{m}_{\pi(\ell)}\right) \in \mathrm{Path}_{G_{\mathcal M}}(\bm m_j,\bm m_k)\right\}.
\end{align*}

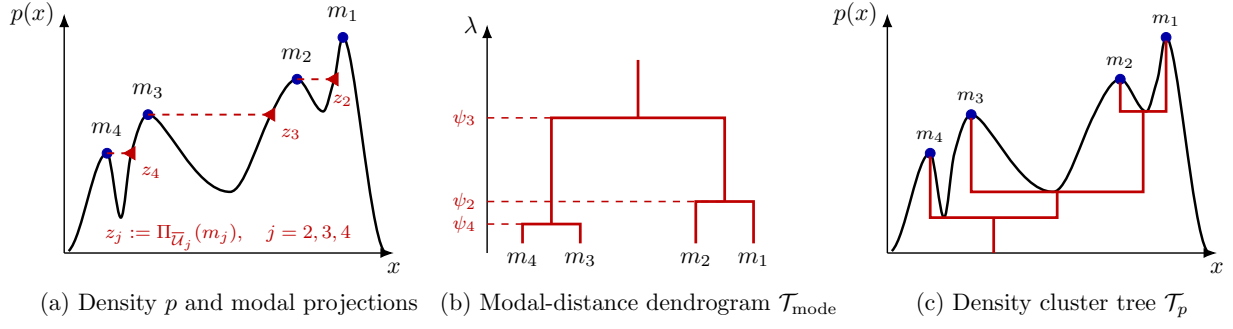
\begin{figure}[t]
\centering
\resizebox{\textwidth}{!}{%
	\input{Figures/Cluster_Tree.tikz}
}
\caption{Illustration of a one-dimensional density $p$ for which the modal distance dendrogram $\mathcal{T}_{\rm mode}$ and the density cluster tree $\mathcal{T}_p$ have different merge topologies. The modal projection distances satisfy $\psi_4<\psi_2 \ll\psi_3$, so $\mathcal{T}_{\rm mode}$ in panel (b) first joins the nearby pairs $(\bm{m}_1,\bm{m}_2)$ and $(\bm{m}_3,\bm{m}_4)$ before joining the two pairs. The modal density values satisfy $p(\bm{m}_1) > \cdots > p(\bm{m}_4)$, so the density cluster tree in panel (c) joins $\bm{m}_1,...,\bm{m}_4$ consecutively.
}
\label{fig:Tmode_Tp}
\end{figure}

\begin{remark}
The modal distance dendrogram $\mathcal T_{\rm mode}$ need not have the same merge topology as the density cluster tree $\mathcal T_p$ associated with the same density $p$; see \autoref{fig:Tmode_Tp} for an illustrative example. The distinction arises because the merge heights of $\mathcal T_{\rm mode}$ are determined by the geometric projection distances $\psi_j,j=2,...,|\mathcal{M}|$, whereas those of the usual density cluster tree are determined by density levels of $p$ at its saddle points. 
\end{remark}

We now establish convergence of the GGDPC dendrogram $\mathcal{T}_G$ to the modal distance dendrogram $\mathcal{T}_{\rm mode}$ under the Gromov-Hausdorff distance as $n\to \infty$.

\begin{theorem}[Gromov-Hausdorff convergence of the GGDPC dendrogram]
\label{thm:conv_ggdpc_dendro}
Suppose that Assumptions~\ref{assump:density_reg}, \ref{assump:sample_den_diff}, and \ref{assump:unique_proj} hold. Let $a_n\to \infty$ be any deterministic sequence. Assume further that $\eta_n=o(1)$, $\frac{a_n q_n}{\eta_n} = o(1)$, $\norm{\hat{p} - p}_{\infty}=o_P(1)$, $a_n\norm{\hat{g}-\nabla p}_{\infty}= o_P(1)$, and $\norm{\nabla^2\hat{p} - \nabla^2 p}_{\infty}=o_P(1)$ with $q_n=\left(\frac{\log n}{n}\right)^{\frac{1}{d}}$. Then, 
\begin{align*}
	d_{\rm den}\left((\mathcal{T}_G,\mathbb{X}_n), (\mathcal{T}_{\rm mode}, \mathcal{M})\right) &= O_P\left(\eta_n + a_n\left[\frac{q_n}{\eta_n} + \norm{\hat{g}-\nabla p}_{\infty}\right]\right).
\end{align*}
\end{theorem}

The proof of \autoref{thm:conv_ggdpc_dendro} is in \autoref{app:proof_conv_ggdpc_dendro}. The limiting behavior can be understood through the 1NN uphill distance $\hat{w}_n(\bm{x})=\norm{\bm{x} - \hat{\Phi}_n(\bm{x})}$ in \eqref{1NN_GGDPC_dist}. As $n\to \infty$, the 1NN uphill distance associated with non-modal observations vanishes, whereas the shifts associated with local modes (or cluster centers) converge to the corresponding population modal projection distances. Consequently, the small-scale branches of the empirical dendrogram collapse, while its persistent upper level structure converges to the modal distance dendrogram. This behavior is illustrated in the bottom right panel of \autoref{fig:old_faithful}. Under the choice $\eta_n\asymp \sqrt{q_n}$,  with $a_n\to\infty$ arbitrarily slowly, and a (boundary-corrected) KDE for $\hat{g}$ with bandwidth $h_n \asymp \left(\frac{\log n}{n}\right)^{\frac{1}{d+6}}$, 
$$d_{\rm den}\left((\mathcal{T}_G,\mathbb{X}_n), (\mathcal{T}_{\rm mode}, \mathcal{M})\right) = O_P\left(\left(\frac{\log n}{n}\right)^{\min\left\{\frac{1}{2d}, \frac{2}{d+6}\right\}}\right)$$
up to an arbitrarily slowly diverging factor. Thus, the GGDPC
dendrogram converges at the same polynomial rate as the ARI in
\eqref{ARI_rate_kde}.

\section{Stability of GGDPC Paths}
\label{sec:stability}

In this section, we study the path-length stability of GGDPC, \emph{i.e.}, whether the length of its 1NN uphill path from any starting point to a density mode converges to the length of the corresponding population gradient ascent flow from the same starting point. 

Let $\mathcal{C}_a$ be a basin of attraction of $\bm{m}^*:=\bm{m}_j \in \mathcal{M}$. For a starting point $\bm{x}\in \mathcal{C}_a$, we define the path length of the population gradient ascent flow to $\bm{m}^*$ by $L(\bm{x}):= \int_0^\infty \norm{\nabla p(\bm{\gamma}_{\bm{x}}(t))}\,dt$. We also define the maximum distance from a point in the closure $\overline{\mathcal{C}}_a$ of $\mathcal{C}_a$ to its nearest observation by
\begin{equation}
	\label{max_1nn_dist}
	\zeta_n:=\sup_{\bm{x}\in \overline{\mathcal{C}}_a}\min_{1\leq i\leq n}\norm{\bm{x}-\bm{X}_i}.
\end{equation}
Under Assumption~\ref{assump:density_reg}(a), Lemma~\ref{lem:1NN_sup_bnd} shows that $\zeta_n = O_P(q_n)$ with $q_n:=\left(\frac{\log n}{n}\right)^{\frac{1}{d}}$.
We first study an oracle version of GGDPC in which both $p$ and $\nabla p$ are known, and then turn to the statistical procedure based on the estimated density $\hat{p}$ and gradient $\hat{g}$.

\subsection{Oracle GGDPC Path Length}
\label{subsec:stability_popul}

Suppose that $p$ and $\nabla p$ are known with $\hat p=p$ and $\hat g=\nabla p$ in GGDPC. Compared with \eqref{gg_1NN_uphill}, the oracle version of the gradient-guided 1NN uphill point is defined by
\begin{equation}
	\label{gg_1NN_uphill_popul}
	\Phi_n(\bm{x}) \in \argmin_{\bm{X}_i\in \mathbb{X}_n} \left\{\norm{\bm{X}_i- (\bm{x}+\eta_n \nabla p(\bm{x}))}: p(\bm{X}_i) > p(\bm{x}) \right\},
\end{equation}
and the associated shifted vector is $\bm{S}_n(\bm{x}):=\Phi_n(\bm{x}) - \bm{x}$. For a starting point $\bm{x}\in \mathcal{C}_a$, the oracle GGDPC iterative path evolves according to
$$\bm{Y}_0^{(n)}=\bm{x},
\qquad \bm{Y}_{k+1}^{(n)}=\Phi_n(\bm{Y}_k^{(n)}).$$
Let 
\begin{equation}
\label{oracle_ggdpc_stop_time}
T_n:=\inf\left\{k\geq 0:\Phi_n(\bm Y_k^{(n)})\text{ is undefined or does not belong to } \mathcal C_a\right\}.
\end{equation}
Thus, $\bm Y_{T_n}^{(n)}$ is the terminal vertex of the path within $\mathcal{C}_a$, and its outgoing edge, if any, is excluded.
Write $L_n(\bm{x}):=\sum_{k=0}^{T_n-1}\norm{\bm{S}_n(\bm{Y}_k^{(n)})}$.
We derive an ``almost uniform'' path-length stability result of GGDPC over the entire basin of attraction $\mathcal{C}_a$, allowing the starting point to approach the basin boundary $\partial\mathcal{C}_a$ as $n\to \infty$. The main difficulty is to ensure that the oracle GGDPC path remains in the same basin of attraction as its population gradient-flow trajectory. This requires additional regularity conditions controlling how the gradient flow separates from the basin boundary $\partial\mathcal{C}_a$.

\begin{assump}[$C^1$-linearization]
\label{assump:C1_linear}
For every saddle point $\bm{s}$ whose stable manifold intersects $\partial\mathcal{C}_a$, the gradient vector field $\nabla p$ is $C^1$-linearizable in a neighborhood of $\bm{s}$.
\end{assump}

The stable manifold of a critical point $\bm{s}\in \mathcal{C}$ consists of all points in $\mathcal{C}$ whose gradient ascent flows converge to $\bm{s}$. Assumption~\ref{assump:C1_linear} requires smooth linearization only in neighborhoods of saddle points that may lie on basin boundaries. A vector field is said to be $C^1$-linearizable near a critical point if, after a local continuously differentiable change of coordinates, its flow behaves like a linear vector field.
Formal definitions of stable and unstable manifolds and smooth linearization, together with some sufficient conditions for Assumption~\ref{assump:C1_linear}, are provided in \autoref{app:tech_dynamic}. The two-Gaussian mixture example in \autoref{app:sep_repulsion_example} also satisfies this linearization condition.

\begin{lemma}[Gradient flow separation from the separatrix]
	\label{lem:flow_separation_sep}
	Suppose that Assumptions~\ref{assump:density_reg}, \ref{assump:sep_repulsion}, and \ref{assump:C1_linear} hold. Then, there exist constants $C_{\mathcal{S}}\in(0,1)$ and $r_0>0$ such that, for every $0<r\leq r_0$ and every  $\bm{x}\in\mathcal{C}_a$ satisfying $d(\bm{x},\partial \mathcal{C}_a)\geq r$, we have that $$d(\bm{\gamma}_{\bm{x}}(t),\partial\mathcal{C}_a)\geq C_{\mathcal{S}}\cdot r
	\quad \text{ for all } \quad t\geq 0.$$
	Consequently, $U_{\frac{C_{\mathcal{S}}r}{2}}(\bm{x})
	:=\left\{\bm{y}\in \mathcal{C}:\inf_{t\geq 0}\norm{\bm{y}-\bm{\gamma}_{\bm{x}}(t)}\leq \frac{C_{\mathcal{S}}\cdot r}{2}\right\}
	\subset\mathcal{C}_a$.
\end{lemma}

The proof of Lemma~\ref{lem:flow_separation_sep} is in \autoref{app:proof_flow_sep}. We now present the path-length stability result toward the separatrix by allowing the initial distance $d(\bm{x},\partial\mathcal{C}_a)$ from the separatrix to shrink with $n$.

\begin{theorem}[Stability of the oracle GGDPC path]
	\label{thm:GGDPC_path_full}
	Suppose that Assumptions~\ref{assump:density_reg}, \ref{assump:sep_repulsion}, and \ref{assump:C1_linear} hold. Let $\delta_n \downarrow 0$ be some deterministic sequence, and let $C_{\mathcal{S}}>0$ be the constant in Lemma~\ref{lem:flow_separation_sep}. If $\eta_n+ \frac{q_n}{\eta_n}=o(\delta_n)$ and $\frac{q_n\log n}{\eta_n^{d+1}}=o(1)$ with $q_n = \left(\frac{\log n}{n}\right)^{\frac{1}{d}}$, then
	$$\sup_{\bm x\in\mathcal C_a\ominus(2\delta_n/C_{\mathcal{S}})} \left|L_n(\bm{x})-L(\bm{x})\right| =O_P\left(\frac{q_n}{\eta_n}\left|\log\left(\frac{q_n}{\eta_n}\right)\right| +\frac{q_n}{\eta_n\delta_n} + \frac{\eta_n}{\delta_n} + \frac{q_n\log n}{\eta_n^{d+1}} \right).$$
\end{theorem}

The proof of \autoref{thm:GGDPC_path_full} is in \autoref{app:ggdpc_path_full_proof}. The term $\frac{q_n \log n}{\eta_n^{d+1}}$ comes from the modal core region $B\left(\bm{m}^*, \frac{Cq_n}{\eta_n}\right)$ for some constant $C>0$, where the 1NN sampling error is no longer negligible relative to the gradient ascent update. When $\delta_n$ is fixed, we obtain a faster convergence rate $O_P\left(\eta_n + \frac{q_n}{\eta_n}\left|\log\left(\frac{q_n}{\eta_n}\right)\right| + \frac{q_n\log n}{\eta_n^{d+1}} \right)$. The additional terms $\frac{q_n}{\eta_n\delta_n} + \frac{\eta_n}{\delta_n}$ in \autoref{thm:GGDPC_path_full} arise from controlling the GGDPC path near boundary saddle points. Choosing $\eta_n \asymp \left(q_n \log n\right)^{\frac{1}{d+2}}$ and $\delta_n \asymp \left(q_n \log n\right)^{\frac{1}{2(d+2)}}$ gives $\eta_n + \frac{q_n}{\eta_n} = o(\delta_n)$ and 
$$\sup_{\bm x\in\mathcal C_a\ominus(2\delta_n/C_{\mathcal{S}})}\left|L_n(\bm{x})-L(\bm{x})\right|=O_P\left(\left(q_n \log n\right)^{\frac{1}{2(d+2)}}\right) = O_P\left(n^{-\frac{1}{2d(d+2)}} \left(\log n\right)^{\frac{d+1}{2d(d+2)}}\right).$$

\subsection{Stability of the Sample GGDPC Path Length}
\label{subsec:stability_sample}

We now return to the statistical GGDPC procedure, where the density function $p$ and gradient $\nabla p$ are estimated by $\hat{p}$ and $\hat{g}$, respectively.
For any $\bm{x}\in\mathcal{C}_a$, we recall from \eqref{gg_1NN_uphill} that the gradient-guided 1NN uphill shift is defined as $\hat{\bm{S}}_n(\bm{x})=\hat{\Phi}_n(\bm{x})-\bm{x}$. The (sample) GGDPC path is then given by
\begin{equation}
\label{sample_ggdpc_path}
\hat{\bm Y}_0^{(n)}=\bm{x}, \quad \hat{\bm Y}_{k+1}^{(n)}=\hat{\Phi}_n(\hat{\bm Y}_k^{(n)}).
\end{equation}
Let $\hat{T}_n$ be defined analogously to $T_n$ in \eqref{oracle_ggdpc_stop_time}, with $\hat{\Phi}_n$ in place of $\Phi_n$. The sample GGDPC path length is $\hat{L}_n(\bm{x}) = \sum_{k=0}^{\hat{T}_n-1} \norm{\hat{\bm{S}}_n(\hat{\bm Y}_k^{(n)})}$.

\begin{theorem}[Stability of the sample GGDPC path]
\label{thm:GGDPC_estimated}
Suppose that Assumptions~\ref{assump:density_reg}, \ref{assump:sample_den_diff}, \ref{assump:sep_repulsion}, and \ref{assump:C1_linear} hold. Assume that, for a deterministic sequence $\delta_n\downarrow 0$,  $\eta_n+\frac{q_n}{\eta_n}=o(\delta_n)$, $\delta_n\left|\log\left(\frac{q_n}{\eta_n}\right)\right|=O(1)$, $\frac{q_n\log n}{\eta_n^{d+1}}=o(1)$, $\norm{\hat{p} - p}_{\infty}=o_P(1)$, $\norm{\hat{g} - \nabla p}_{\infty} = o_P(\delta_n)$, and $\norm{\nabla^2 \hat{p} - \nabla^2 p}_{\infty}=o_P(1)$ with $q_n = \left(\frac{\log n}{n}\right)^{\frac{1}{d}}$.
Then, 
$$\sup_{\bm x\in\mathcal C_a\ominus(2\delta_n/C_{\mathcal S})} \left|\hat{L}_n(\bm{x})-L(\bm{x})\right|
=O_P\left(\frac{q_n}{\eta_n}\left|\log\left(\frac{q_n}{\eta_n}\right)\right| +\frac{q_n}{\eta_n\delta_n} + \frac{\eta_n}{\delta_n} + \frac{q_n\log n}{\eta_n^{d+1}} +\frac{\norm{\hat{g} - \nabla p}_{\infty}}{\delta_n}\right),$$
where $C_{\mathcal{S}}>0$ is the constant in Lemma~\ref{lem:flow_separation_sep}.
\end{theorem}

The proof of \autoref{thm:GGDPC_estimated} is in \autoref{app:proof_ggdpc_estimated}. Compared with the oracle GGDPC version in \autoref{thm:GGDPC_path_full}, the convergence rate of the sample GGDPC path incurs an additional term $\frac{\norm{\hat{g} - \nabla p}_{\infty}}{\delta_n}$, which reflects the effect of gradient estimation. To make the rate explicit, let $\norm{\hat{g} - \nabla p}_{\infty}=O_P(s_n)$ for a deterministic sequence $s_n\downarrow 0$ and choose $\eta_n \asymp \left(q_n \log n\right)^{\frac{1}{d+2}}$ and $\delta_n \asymp \sqrt{\left(q_n \log n\right)^{\frac{1}{d+2}} + s_n}$. Provided that $\delta_n\left|\log\left(\frac{q_n}{\eta_n}\right)\right|=O(1)$, as holds whenever $s_n$ decays at a polynomial rate, these choices satisfy the conditions in \autoref{thm:GGDPC_estimated} and yield that  
$$\sup_{\bm x\in\mathcal C_a\ominus(2\delta_n/C_{\mathcal S})} \left|\hat L_n(\bm x)-L(\bm x)\right| = O_P\left(\sqrt{\left(q_n \log n\right)^{\frac{1}{d+2}} + s_n}\right).$$
Under Assumption~\ref{assump:density_reg}(b), together with standard kernel and support regularity conditions \citep{gine2002rates,einmahl2005uniform,chacon2011asymptotics}, we can optimally estimate $\hat g=\nabla \hat{p}$ using a (boundary-corrected) KDE with bandwidth $h_n\asymp\left(\frac{\log n}{n}\right)^{\frac{1}{d+6}}$, so that
$\norm{\hat{g}- \nabla p}_{\infty} = O_P(s_n)$ with $s_n=\left(\frac{\log n}{n}\right)^{\frac{2}{d+6}}$.
Consequently, we obtain that 
\begin{equation}
\label{path_length_rate_kde}
\sup_{\bm x\in\mathcal C_a\ominus(2\delta_n/C_{\mathcal S})}\left|\hat L_n(\bm{x})-L(\bm{x})\right| =
O_P\left(\left[\left(\frac{(\log n)^{d+1}}{n}\right)^{\frac{1}{d(d+2)}} + \left(\frac{\log n}{n}\right)^{\frac{2}{d+6}}\right]^{\frac{1}{2}}\right).
\end{equation}
In particular, gradient estimation determines the rate when $d = 1$, while the sample discretization of the GGDPC path determines the rate when $d>1$. The relatively slow rate in \eqref{path_length_rate_kde} is partly driven by sample discretization within modal neighborhoods, where the population gradient $\nabla p$ vanishes and thus provides little directional guidance for the 1NN uphill updates. In \autoref{sec:discussion}, we discuss a simple projection modification that maps observations sufficiently close to an estimated local mode directly to that mode, substantially improving the convergence rate of the GGDPC path length.


\section{Convergence of GGDPC Graph Distance and Density Waterfall}
\label{sec:conv_graph_dist}

As discussed in \autoref{subsec:dpc_review} and \autoref{sec:GGDPC}, the full DPC or GGDPC graph forms a directed acyclic graph (more precisely, an arborescence) oriented toward the sample global mode $\bm{X}^* \in \mathbb{X}_n$. In this section, we establish convergence of the GGDPC graph distance to a deterministic population analogue. This population quantity alternates between gradient flow lengths within the basins of attraction of local modes and modal projection distances between basins. It also provides a natural coordinate for summarizing the GGDPC graph and, when combined with the estimated density, yields a two-dimensional representation of the density landscape associated with each cluster.

\subsection{Stability of the GGDPC Graph Distance}
\label{subsec:conv_graph_dist}

Recall from \autoref{sec:GGDPC} that the full GGDPC graph $G$ is oriented toward the sample global mode $\bm{X}^*$. For any $\bm{X}_i\neq \bm{X}^*$, let $\bm{X}_i\equiv \bm{X}_{i_1} \to \bm{X}_{i_2} \to \cdots \to \bm{X}_{i_{T_i}} \equiv \bm{X}^*$ denote its unique directed path in $G$, where $\bm{X}_{i_{\ell +1}} = \hat{\Phi}_n(\bm{X}_{i_{\ell}})$ by \eqref{gg_1NN_uphill}. We define its GGDPC graph distance to the root $\bm{X}^*$ by
$$\hat{d}_G(\bm{X}_i) = \sum_{\ell=1}^{T_i-1} \norm{\bm{X}_{i_{\ell+1}} - \bm{X}_{i_\ell}}.$$
For any $\bm{x}\in \mathcal{C}$, let $i^*(\bm{x}) \in \argmin_{1\leq i\leq n}\norm{\bm{x}-\bm{X}_i}$ with ties resolved by the fixed convention introduced earlier, and define
\begin{equation}
	\label{GGDPC_graph_dist}
	\hat{d}_G(\bm{x}) = \norm{\bm{x} - \bm{X}_{i^*(\bm{x})}} + \hat{d}_G\!\left(\bm{X}_{i^*(\bm{x})}\right).
\end{equation}

\begin{figure}[t]
	\centering
	\resizebox{0.8\linewidth}{!}{%
		\input{Figures/LM1_figure.tikz}%
	}
	\caption{Schematic illustration of a population GGDPC path from $\bm{x}$ to the global mode $\bm{m}_1$ through the sequence $\bm{m}_{j_1}, \bm{m}_{j_2},\bm{m}_{j_3}\equiv \bm{m}_1$, denoted by blue stars. The cyan and green regions are the upper level sets associated with $\bm{m}_{j_1}$ and $\bm{m}_{j_2}$, respectively. Red triangles denote the modal projections $\bm{z}_{j_\ell} = \Pi_{\overline{\mathcal{U}}_{j_\ell}}(\bm{m}_{j_\ell}), \ell=1,2$. Dashed curves represent gradient flows within their basins of attraction, while solid segments indicate the modal projections. The population GGDPC graph distance is the sum of the lengths of these gradient flow curves and projection segments.}
	\label{fig:ggdpc-path}
\end{figure}
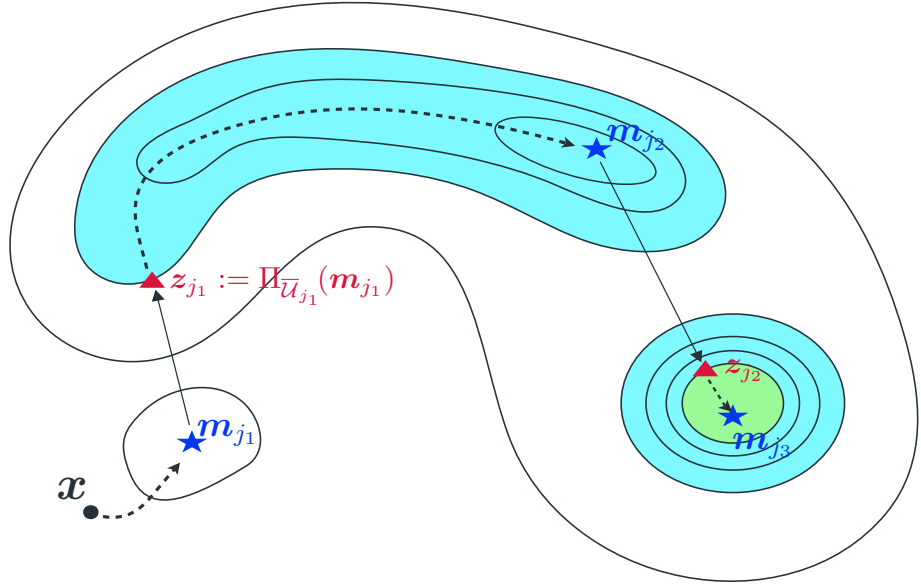

We next construct the population analogue to which $\hat{d}_G(\bm{x})$ converges as $n\to \infty$. 
Consider $\bm{x}\in \mathcal{C}_j$ whose gradient ascent flow converges to $\bm{m}_j \in \mathcal{M}$. Under Assumption~\ref{assump:unique_proj}, if $j\neq 1$, the projection $\bm{z}_j=\Pi_{\overline{\mathcal{U}}_j}(\bm{m}_j)$ lies in the basin $\mathcal{C}_{\pi(j)}$. Since the gradient ascent flow from $\bm{z}_j$ converges to a mode of strictly higher density, repeatedly applying the map $j\mapsto \pi(j)$ generates a finite sequence
\begin{equation}
\label{mode_seq}
j_1=j,\qquad j_{\ell+1}=\pi(j_\ell), \qquad \bm{m}_{j_{T_j}}=\bm{m}_1.
\end{equation}
For $\bm{x}\in \mathcal{C}_j$, write $T_{\bm{x}}:= T_j$. For a truncated graph at threshold $\lambda$, we define $T_{\bm x,\lambda}$ analogously as the first index at which the modal sequence reaches a mode whose outgoing modal projection distance exceeds $\lambda$.
Thus, the population GGDPC path alternates between a gradient ascent flow within a modal basin of attraction and a projection from the resulting local mode to its upper level set; see \autoref{fig:ggdpc-path} for an illustration. We define the population GGDPC graph distance by
\begin{equation}
	\label{GGDPC_graph_dist_pop}
	\bar{d}_G(\bm{x}) = L(\bm{x}) + \sum_{\ell=1}^{T_{\bm{x}}-1} \left[\norm{\bm{m}_{j_\ell} - \Pi_{\overline{\mathcal{U}}_{j_\ell}}(\bm{m}_{j_\ell})} + L\left(\Pi_{\overline{\mathcal{U}}_{j_\ell}}(\bm{m}_{j_\ell}) \right) \right], \qquad \bm{x}\in \mathcal{C}_j,
\end{equation}
where we recall from \autoref{sec:stability} that $\bm{x}\mapsto L(\bm{x})$ is the path length of the population gradient ascent flow from $\bm{x}$ to a local mode.

Let  $\mathcal{N}_{\mathcal{C}}:=\left\{\bm{x}\in\mathcal{C}:\operatorname{dest}(\bm{x})\notin\mathcal{M}\right\}$. Under Assumption~\ref{assump:density_reg}, $\mathcal{N}_{\mathcal{C}}$ is a finite union of stable manifolds of non-modal critical points of $p$ and has Lebesgue measure zero. Hence, we define $\bar{d}_G$ by \eqref{GGDPC_graph_dist_pop} on $\mathcal{C} \setminus \mathcal{N}_{\mathcal{C}}$ and assign an arbitrary finite value to $\bar{d}_G(\bm{x})$ when $\bm{x}\in \mathcal{N}_{\mathcal{C}}$.

Then, under Assumptions~\ref{assump:density_reg} and
\ref{assump:unique_proj}, the population graph distance is essentially bounded as $\operatorname*{ess\,sup}\limits_{\bm{x}\in \mathcal{C}} \bar{d}_G(\bm{x})<\infty$; see Proposition~\ref{prop:finite_graph_distance} in \autoref{app:graph_finite_proof}.

\begin{theorem}[Convergence of the GGDPC graph distance]
\label{thm:GGDPC_graph_dist}
Suppose that Assumptions~\ref{assump:density_reg}, \ref{assump:C1_linear}, \ref{assump:sample_den_diff}, \ref{assump:sep_repulsion}, and \ref{assump:unique_proj} hold for every modal basin of attraction. Let $\mathcal S_{\rm full}:=\bigcup_a\partial\mathcal C_a$. 
Assume that, for a deterministic sequence $\delta_n\downarrow 0$,  $\eta_n+\frac{q_n}{\eta_n}=o(\delta_n)$, $\delta_n\left|\log\left(\frac{q_n}{\eta_n}\right)\right|=O(1)$, $\frac{q_n\log n}{\eta_n^{d+1}}=o(1)$, $\norm{\hat{p} - p}_{\infty}=o_P(1)$, $\norm{\hat{g} - \nabla p}_{\infty} = o_P(\delta_n)$, and $\norm{\nabla^2 \hat{p} - \nabla^2 p}_{\infty}=o_P(1)$ with $q_n = \left(\frac{\log n}{n}\right)^{\frac{1}{d}}$. Then, 
\begin{align*}
	\sup_{\bm x\in\mathcal C\setminus\mathcal S_{\rm full}^{2\delta_n/C_{\mathcal S}}} \left|\hat{d}_G(\bm{x}) - \bar{d}_G(\bm{x}) \right| &= O_P\left(\frac{q_n}{\eta_n}\left|\log\left(\frac{q_n}{\eta_n}\right)\right| +\frac{q_n}{\eta_n\delta_n} + \frac{\eta_n}{\delta_n} + \frac{q_n\log n}{\eta_n^{d+1}} +\frac{\norm{\hat{g} - \nabla p}_{\infty}}{\delta_n} \right),
\end{align*}
where $\mathcal S_{\rm full}^{r}=\{\bm x\in\mathcal C:d(\bm x,\mathcal S_{\rm full})<r\}$ and $C_{\mathcal S}$ is the minimum of the separation constants in Lemma~\ref{lem:flow_separation_sep} over the finitely many modal basins.
\end{theorem}

The proof of \autoref{thm:GGDPC_graph_dist} is in \autoref{app:GGDPC_graph_dist_proof}. If $\delta_n \equiv \delta$ is a fixed constant, then Assumptions~\ref{assump:C1_linear} and \ref{assump:sep_repulsion} are unnecessary, and the convergence rate improves accordingly. The same rate also holds for the truncated GGDPC graph distance (or local graph distance) $\hat{d}_{G_{\lambda}}$ for any $\lambda >0$ satisfying $\lambda \notin \left\{\psi_j: 1< j \leq |\mathcal{M}|\right\}$. Its population analogue $\bar{d}_{G_{\lambda}}$ is defined similarly as \eqref{GGDPC_graph_dist_pop}, except that the modal sequence \eqref{mode_seq} stops at the first mode $\bm{m}_{j_{T_{\bm{x}},\lambda}}$ whose outgoing projection distance satisfies $\psi_{j_{T_{\bm{x}},\lambda}} > \lambda$. When $\hat{p}$ and $\hat{g}$ are constructed via (boundary-corrected) KDEs with bandwidth $h_n \asymp \left(\frac{\log n}{n}\right)^{\frac{1}{d+6}}$, the rate in \autoref{thm:GGDPC_graph_dist} is identical to \eqref{path_length_rate_kde} under the associated choices of $\eta_n$ and $\delta_n$.

\subsection{Density Waterfalls}
\label{subsec:den_water}

We propose the \emph{global density waterfall plot} as the scatter plot 
$$\left\{\left(\hat{d}_G(\bm{X}_i), \hat{p}(\bm{X}_i)\right), i=1,...,n\right\},$$
with the GGDPC graph distance on the $x$-axis and the estimated density on the $y$-axis. Along each modal branch, the density typically decreases as the graph distance from the global mode increases, producing a waterfall-like pattern.
Similarly, for a fixed threshold $\lambda >0$, the \emph{local density waterfall plot} is obtained by replacing $\hat{d}_G$ with the truncated graph distance $\hat{d}_{G_{\lambda}}$.

To study the asymptotic stability of these plots, we first represent the global density waterfall by its empirical measure
\begin{equation}
	\label{waterfall_measure_emp_est}
	\hat{Q}_n(A) = \frac{1}{n}\sum_{i=1}^n \mathds{1}\left\{\left(\hat{d}_G(\bm{X}_i), \hat{p}(\bm{X}_i)\right) \in A\right\} \quad \text{ for any Borel set } A \in \mathcal{B}(\mathbb{R}^2).
\end{equation}
Its population analogue is the distribution of $\left(\bar{d}_G(\bm{X}), p(\bm{X})\right)$ for $\bm{X}\sim P$. Namely, we define the population global waterfall measure $Q$ by
\begin{equation}
	\label{waterfall_measure_pop}
	Q(A) = \mathbb{P}\left(\left(\bar{d}_G(\bm{X}), p(\bm{X})\right) \in A\right) \quad \text{ for any Borel set } A \in \mathcal{B}(\mathbb{R}^2).
\end{equation}
For a fixed $\lambda>0$, we analogously define the population and empirical local waterfall measures by
$$Q_\lambda(A)= \mathbb{P}\left(\left(\bar{d}_{G_\lambda}(\bm{X}),p(\bm{X})\right)\in A\right), \qquad \hat Q_{n,\lambda}(A) = \frac{1}{n} \sum_{i=1}^n \mathds{1}\left\{
\left(\hat{d}_{G_{\lambda}}(\bm{X}_i),\hat{p}(\bm{X}_i)\right)\in A\right\}.$$
The following theorem establishes convergence of $\hat Q_n$ to $Q$ under the Wasserstein-1 distance, thereby providing a population-level stability guarantee for the density waterfall representation. The same conclusion applies to the local waterfall measures $\hat Q_{n,\lambda}$ and $Q_\lambda$ for any fixed $\lambda>0$ satisfying $\lambda\notin\{\psi_j:1<j\leq|\mathcal M|\}$.

\begin{theorem}[Wasserstein-1 convergence of the density waterfall]
	\label{thm:wasser_waterfall}
	Suppose that Assumptions~\ref{assump:density_reg}, \ref{assump:sample_den_diff}, \ref{assump:sep_repulsion}, \ref{assump:unique_proj}, and \ref{assump:C1_linear} hold for every modal basin of attraction. Assume that, for a deterministic sequence $\delta_n\downarrow 0$,  $\eta_n+\frac{q_n}{\eta_n}=o(\delta_n)$, $\delta_n\left|\log\left(\frac{q_n}{\eta_n}\right)\right|=O(1)$, $\frac{q_n\log n}{\eta_n^{d+1}}=o(1)$, $\norm{\hat{p} - p}_{\infty}=o_P(1)$, $\norm{\hat{g} - \nabla p}_{\infty} = o_P(\delta_n)$, and $\norm{\nabla^2 \hat{p} - \nabla^2 p}_{\infty}=o_P(1)$ with $q_n = \left(\frac{\log n}{n}\right)^{\frac{1}{d}}$. Then, 
	\begin{align*}
		&\mathrm{Wass}_1(\hat{Q}_n, Q) \\
		&= O_P\left(\delta_n + \frac{q_n}{\eta_n}\left|\log\left(\frac{q_n}{\eta_n}\right)\right| +\frac{q_n}{\eta_n\delta_n} + \frac{\eta_n}{\delta_n} + \frac{q_n\log n}{\eta_n^{d+1}} +\frac{\norm{\hat{g} - \nabla p}_{\infty}}{\delta_n} +\norm{\hat{p} -p}_{\infty}+ \frac{\log n}{\sqrt{n}}\right).
	\end{align*}
\end{theorem}

The proof of \autoref{thm:wasser_waterfall} is in \autoref{app:waterfall_proof}. The high-level proof idea is to split the observations in $\mathbb{X}_n$ based on their distances to the full separatrix. Let 
$$\mathcal{C}_n^{\rm far} = \mathcal{C} \setminus \mathcal{S}_{\rm full}^{\frac{2\delta_n}{C_{\mathcal{S}}}}, \qquad \mathcal{C}_n^{\rm close} = \mathcal{S}_{\rm full}^{\frac{2\delta_n}{C_{\mathcal{S}}}},$$
where $C_{\mathcal{S}}$ is the separation constant in Lemma~\ref{lem:flow_separation_sep}.  For observations in $\mathcal{C}_n^{\rm far}$, their contributions to the Wasserstein-1 discrepancy are controlled by the uniform convergence of the GGDPC graph distance in \autoref{thm:GGDPC_graph_dist}, together with the density estimation error $\norm{\hat{p} - p}_{\infty}$. For observations in $\mathcal{C}_n^{\rm close}$, Lemma~\ref{lem:int_control_sep} controls the aggregate empirical graph distance, while the $O_P(\delta_n)$ mass of this region controls the population contribution, yielding an overall $O_P(\delta_n)$ term. Thus, the proof may be viewed as an analogue of an $\epsilon$-contamination decomposition \citep{huber1964robust,huber1965robust}, with observations
near the separatrix forming a vanishing contamination component.

To make the rate in \autoref{thm:wasser_waterfall} more explicit, suppose that $\norm{\hat g-\nabla p}_\infty=O_P(s_n)$ and $\norm{\hat p-p}_\infty=O_P(e_n)$ for deterministic sequences $s_n,e_n\downarrow 0$. Again, choosing $\eta_n \asymp \left(q_n \log n\right)^{\frac{1}{d+2}}$ and $\delta_n \asymp \sqrt{\left(q_n \log n\right)^{\frac{1}{d+2}} + s_n}$ with $\delta_n\left|\log\left(\frac{q_n}{\eta_n}\right)\right|=O(1)$ satisfies the rate conditions in \autoref{thm:wasser_waterfall}, and 
$$\mathrm{Wass}_1(\hat{Q}_n, Q) = O_P\left(\left[\left(q_n \log n\right)^{\frac{1}{d+2}} + s_n\right]^{\frac{1}{2}} + e_n + \frac{\log n}{\sqrt{n}}\right).$$
Typically, $\frac{\log n}{\sqrt{n}} \lesssim e_n \lesssim \left[\left(q_n \log n\right)^{\frac{1}{d+2}} + s_n\right]^{\frac{1}{2}}$, so the Wasserstein-1 convergence rate of the density waterfall empirical measure inherits the same trade-off between the sample discretization and gradient estimation as the GGDPC path-length convergence rate in \autoref{thm:GGDPC_estimated}. In particular, under KDEs on $\hat{p}$ and $\hat{g}$, it coincides with \eqref{path_length_rate_kde}.

\section{Empirical Studies}
\label{sec:emp_studies}

We have illustrated GGDPC and its associated clustering visualizations using the Old Faithful geyser data, which record 272 eruptions of the Old Faithful geyser in Yellowstone National Park. In \autoref{fig:old_faithful}, we apply GGDPC to pairs of consecutive eruption durations. Since ground-truth labels are unavailable for this real-world dataset, its clustering structure can only be assessed qualitatively through visualization. 

In this section, we quantitatively compare GGDPC with several DPC-type methods using the two-component Gaussian mixture considered in \autoref{fig:gmm_graph} across different sample sizes. The competing methods include (i) the original DPC; (ii) DPC-KNN-PCA \citep{du2016study}, which estimates density via $k$-nearest neighbors in a principal component space; (iii) SNN-DPC \citep{liu2018shared}, which modifies DPC using the number of shared nearest neighbors between observations; (iv) DPC-CE \citep{guo2022density}, which combines DPC with graph-based connectivity estimation to improve the identification of cluster centers; (v) DPC-DLP \citep{seyedi2019dynamic}, which uses $k$-nearest-neighbor density estimation together with a graph-based dynamic label propagation; and (vi) DPC-MDNN \citep{wang2025improved}, which defines nearest neighbors and density estimates using manifold distances.

\begin{table}[t]
\centering
\begin{threeparttable}
	\setlength{\tabcolsep}{8pt}
	\begin{tabular}{cccccc}
		\toprule
		Sample Size & 500 & 1000 & 2000 & 5000 & 8000 \\
		\midrule
		\multicolumn{6}{c}{\textbf{Adjusted Rand Index (ARI)}} \\
		\midrule
		DPC
		& 0.898 & 0.920 & 0.935 & 0.946 & 0.953 \\
		DPC-KNN-PCA
		& \textbf{0.935} & \textbf{0.939}
		& 0.943 & 0.945 & 0.948 \\
		SNN-DPC
		& 0.881 & 0.871 & 0.872 & 0.874 & 0.878 \\
		DPC-CE
		& 0.840 & 0.886 & 0.911 & 0.928 & 0.941 \\
		DPC-DLP
		& 0.341 & 0.494 & 0.632 & 0.772 & 0.808 \\
		DPC-MDNN
		& 0.211 & 0.077 & 0.010 & 0.000 & 0.000 \\
		GGDPC
		& 0.916 & 0.938 & \textbf{0.953}
		& \textbf{0.963} & \textbf{0.967} \\
		\midrule
		\multicolumn{6}{c}{\textbf{Running Time (seconds)}} \\
		\midrule
		DPC
		& 0.039 & 0.100 & 0.347 & 1.858 & 4.627 \\
		DPC-KNN-PCA
		& \textbf{0.020} & 0.075 & 0.313 & 2.239 & 6.096 \\
		SNN-DPC
		& 0.167 & 1.111 & 9.922 & 159.913 & 632.394 \\
		DPC-CE
		& 0.728 & 1.802 & 4.896 & 20.686 & 45.763 \\
		DPC-DLP
		& 0.074 & 0.337 & 1.634 & 14.996 & 41.864 \\
		DPC-MDNN
		& 0.262 & 1.128 & 5.770 & 56.383 & 180.123 \\
		GGDPC
		& 0.028 & \textbf{0.067} & \textbf{0.203}
		& \textbf{1.061} & \textbf{2.601} \\
		\bottomrule
	\end{tabular}
\end{threeparttable}
\vspace{2mm}
\caption{Clustering performance and computational efficiency of different DPC-type methods on the two-component Gaussian mixture across sample sizes. Larger ARI values indicate better clustering performance, while smaller running times indicate greater computational efficiency.}
\label{tab:gmm_results}
\end{table}

For DPC and GGDPC, we select observations whose 1NN uphill distances exceed the sample mean by at least four sample standard deviations as cluster centers. For the other DPC-type methods, we provide the true number of clusters, namely two, whenever required by the corresponding procedure. Moreover, for methods involving $k$-nearest-neighbor searches, we select $k$ to maximize the ARI relative to the ground-truth labels. This oracle tuning deliberately favors the competing methods. Their remaining tuning parameters are set to the default values recommended by the corresponding implementations, as varying these parameters does not materially change the conclusions. \autoref{tab:gmm_results} reports the ARI \eqref{ARI} and running time over 1000 Monte Carlo replications. All the code in this paper is available at \url{https://github.com/zhangyk8/GGDPC}.


Overall, GGDPC achieves a competitive balance between clustering accuracy and computational efficiency. It attains among the highest ARI values while requiring substantially less computation than most competing DPC-type methods, despite using no ground-truth information for parameter tuning. For a fair comparison, the reported GGDPC running times exclude the optional dendrogram construction.
The poor performance of DPC-MDNN in this experiment appears to arise from its single linkage construction. When the two Gaussian components overlap, the resulting neighborhood graph can contain chains of nearby observations connecting the two components, causing single linkage to merge most observations into a single cluster. Replacing single linkage by average or complete linkage may alleviate this behavior, but such modifications fall outside the DPC-MDNN procedure proposed in \citet{wang2025improved}.

\section{Discussion}
\label{sec:discussion}

In summary, we propose GGDPC as a gradient-guided modification of DPC that improves the empirical stability of its ascending paths while admitting a direct population-level interpretation. 
We establish convergence under five complementary notions of consistency and stability: local modes (\autoref{thm:modal_const}), ARI (\autoref{thm:GGDPC_ARI}), dendrogram (\autoref{thm:conv_ggdpc_dendro}), graph distance (\autoref{thm:GGDPC_graph_dist}), and density waterfall (\autoref{thm:wasser_waterfall}). 
The resulting convergence rates reveal the hierarchy of statistical and geometric difficulty: modal consistency requires the weakest control, followed by ARI and dendrogram convergence, whereas the graph distance and density waterfall consistency require substantially finer control of the GGDPC paths. 
We conclude by discussing several directions for future research.


{\bf 1. Improved rates of convergence:}
The convergence rates derived in \autoref{sec:stability} and several downstream results exhibit a relatively strong dependence on the ambient dimension, partly due to sample discretization within modal neighborhoods. One possible refinement is to modify GGDPC so that every observation lying in an estimated modal neighborhood $B\left(\hat{\bm m}_j,\frac{a_nq_n}{\eta_n}\right)$, where $a_n\to\infty$ arbitrarily slowly, is mapped directly to the corresponding estimated local mode $\hat{\bm m}_j$. Under this modification, the rate in \autoref{thm:GGDPC_path_full} improves to $O_P\left(\eta_n + \frac{q_n}{\eta_n}\left|\log\left(\frac{q_n}{\eta_n}\right)\right| + \frac{q_n}{\eta_n\delta_n}\right)$. For example, taking $\eta_n\asymp q_n^{\frac{1}{3}}$ and $\delta_n \asymp q_n^{\frac{1}{3}} \log n$ yields the rate $O_P\left(\left(\frac{\log n}{n}\right)^{\frac{1}{3d}} \log n\right)$ with corresponding improvements in the downstream convergence results. 

{\bf 2. Minimax theory for separatrices and gradient flow length.}
While minimax theory for gradient and local mode estimation is well developed \citep{stone1982optimal,arias2022estimation}, much less is known about separatrices and gradient flow lengths. Unlike local modes or pointwise gradient values, these are inherently nonlocal geometric quantities. In particular, separatrices depend on the global organization of the gradient field, while flow lengths accumulate information along entire trajectories. Consequently, local estimation errors may propagate over extended regions, creating additional challenges for establishing sharp upper and lower bounds.


{\bf 3. Waterfall plots as visualization tools for local signatures.}
The waterfall plot introduced in this paper, \emph{i.e.}, a scatter plot of one-dimensional distances against (estimated) densities, is not specific to GGDPC or DPC. More generally, it pairs a one-dimensional notion of distance to a cluster representative with a density value, providing a compact summary of local cluster structure in multivariate data. For mean shift clustering, the GGDPC graph distance can be replaced by the accumulated mean shift path length to the limiting local mode. For $k$-means clustering \citep{macqueen1967some,lloyd1982least}, it can instead be replaced by the distance to the assigned cluster centroid. 
The resulting waterfall plots may provide useful insights into within-cluster morphology when direct visualization of the original data is difficult.


\section*{Acknowledgments}

YZ was supported in part by YC's NSF grant DMS-2141808. YC is supported by NSF grants DMS-1952781, 2112907, 2141808, and NIH U24-AG07212. 

\section*{Declaration of the Use of AI-Assisted Technologies}

The authors used OpenAI's ChatGPT (GPT-5.6 Sol) to assist with proof exploration, exposition, language editing, and the refinement of simulation code. All mathematical arguments, results, and computational outputs were independently reviewed and verified by the authors. The authors take full responsibility for the content of the paper, including any remaining errors or omissions.

\vspace{5mm}
\singlespacing
\bibliography{DenPeakClu_bib}

\newpage 
\onehalfspacing

\begin{center}
{\LARGE \textbf{Supplementary Materials to ``Gradient-Guided Density Peak Clustering''}}\\~\\

\end{center}

\appendix

\renewcommand{\thepage}{S\arabic{page}}
\setcounter{page}{1}

\noindent{\bf\LARGE Contents}

\startcontents[sections]
\printcontents[sections]{l}{1}{\setcounter{tocdepth}{2}}

\vspace{5mm}

\section{Technical Concepts in Dynamical Systems}
\label{app:tech_dynamic}

This section collects several concepts in dynamical systems used in our analysis of GGDPC paths. In particular, all notions of dynamical systems are defined under the $C^3$ extension of the density $p:\mathbb{R}^d \to \mathbb{R}$ under Assumption~\ref{assump:density_reg}(b), which we continue to denote by $p$.

\begin{definition}[Stable and unstable manifolds]
	Let $\bm{s}$ be a critical point of the density $p$. Its stable manifold under the gradient ascent flow is
	$$W^s(\bm{s})=\left\{\bm{x}\in \mathcal{C}:\lim_{t\to \infty}\bm{\gamma}_{\bm{x}}(t)=\bm{s}\right\},$$
	and its unstable manifold $W^u(\bm{s})$ consists of points whose backward gradient flow trajectories converge to $\bm{s}$ as $t\to -\infty$, \emph{i.e.},
	$$W^u(\bm{s})=\left\{\bm{x}\in \mathcal{C}: \bm{\gamma}_{\bm{x}}(t) \text{ is defined for all sufficiently negative } t \text{ and } \lim_{t\to -\infty}\bm{\gamma}_{\bm{x}}(t)=\bm{s}\right\}.$$
\end{definition}

By the stable manifold theorem, $W^s(\bm{s})$ and $W^u(\bm{s})$ are locally $C^2$ embedded manifolds in $\mathbb{R}^d$ near $\bm{s}$. Since the Jacobian of the gradient vector field $\nabla p:\mathbb{R}^d \to \mathbb{R}^d$ at $\bm{s}$ is $\nabla^2 p(\bm{s})$, their tangent spaces $T_{\bm{s}}W^s(\bm{s})$ and $T_{\bm{s}}W^u(\bm{s})$ at $\bm{s}$ are the negative- and positive-eigenspaces of $\nabla^2 p(\bm{s})$, respectively. By the non-degeneracy of $\nabla^2 p(\bm{s})$ under Assumption~\ref{assump:density_reg}(c), $T_{\bm{s}}W^u(\bm{s}) \oplus T_{\bm{s}}W^s(\bm{s}) = \mathbb{R}^d$.

We next introduce the concept of smooth linearization used in Assumption~\ref{assump:C1_linear}. Let $f:U \to \mathbb{R}^d$ be a $C^M$ vector field on an open set $U \subset \mathbb{R}^d$. We denote by $\varphi_f^t$ its local flow, \emph{i.e.}, $\varphi_f^t(\bm{x}_0)$ is the solution at time $t$ of $\bm{x}'(t)=f(\bm{x}(t))$ with $\bm{x}(0)=\bm{x}_0$ whenever the solution remains in $U$.

\begin{definition}[$C^N$-conjugation]
	\label{def:conjugate}
	Let $f$ and $g$ be two $C^M$ vector fields with $f(\bm{0}) = g(\bm{0})=\bm{0}$, and let $1 \leq N \leq M$ be an integer. We say that $f$ and $g$ are \emph{$C^N$-conjugate near $\bm{0}$} if there exist neighborhoods $V_1,V_2 \subset \mathbb{R}^d$ of $\bm{0}$ and a $C^N$-diffeomorphism $\Phi:V_1\to V_2$ with $\Phi(\bm{0}) = \bm{0}$ such that
	$$\Phi\!\left(\varphi_f^t(\bm{x})\right) = \varphi_g^t\!\left(\Phi(\bm{x})\right)$$
	whenever both sides are defined and remain in the corresponding neighborhoods.
	The mapping $\Phi$ is called $C^N$-conjugation between $\bm{x}'(t)=f(\bm{x}(t))$ and $\bm{y}'(t)=g(\bm{y}(t))$.
\end{definition}

\begin{definition}[$C^N$-linearization]
	Let $f$ be a $C^M$ vector field with $f(\bm{0})=\bm{0}$ and
	$$f(\bm{x}) = A\bm{x}+F(\bm{x}), \qquad A=Df(\bm{0}),$$
	where $F(\bm{0})=\bm{0}$ and $D F(\bm{0})=0$ is its Jacobian. We say that $f$ admits a \emph{$C^N$-linearization} near $\bm{0}$ if it is $C^N$-conjugate to the linear vector field $\bm{y}'(t)=A \bm{y}(t)$.
\end{definition}


A standard sufficient route to smooth linearization is through the Sternberg (non-resonance) condition \citep{sternberg1957local,sell1985smooth,barreira2007holder}. Let $\rho_1,...,\rho_d$ be the eigenvalues of a hyperbolic matrix $A\in \mathbb{R}^{d\times d}$. We say that $A$ satisfies the \emph{Sternberg non-resonance condition of order $N\geq 2$} if 
$$\rho_i - \sum_{j=1}^d \alpha_j \rho_j \neq 0, \qquad 2\leq \sum_{j=1}^d \alpha_j \leq N$$
for $i=1,...,d$ and nonnegative integers $\alpha_1,...,\alpha_d$. 
In our setting, the gradient vector field $\nabla p$ is $C^2$ under Assumption~\ref{assump:density_reg}(b), and its linearization at a critical point $\bm{s}$ is determined by the Hessian matrix $\nabla^2p(\bm{s})$. In particular, Sternberg-type results provide sufficient nonresonance conditions on the hyperbolic Hessian $\nabla^2 p(\bm{s})$ for the $C^1$-linearization required in Assumption~\ref{assump:C1_linear}. 

Another sufficient condition for $C^1$-linearization is Hartman's spectral
condition \citep{hartman1960local}; see also Section~1 of
\citet{chicone2000linearization}. Let $\bm{s}$ be a hyperbolic saddle point, and suppose that the negative and positive eigenvalues of
$\nabla^2p(\bm{s})$ lie, respectively, in $[-\alpha_L,-\alpha_R]$ and $[\beta_L,\beta_R]$, where $\alpha_L,\alpha_R,\beta_L,\beta_R>0$ with
$\alpha_R\leq\alpha_L$ and $\beta_L\leq\beta_R$.
Let $\mu,\nu>0$ denote the H\"older spectral exponents associated with the stable and unstable parts of the linearized flow, respectively. Then, Hartman's spectral condition states that
$$\alpha_L-\alpha_R<\mu\beta_L, \qquad
\beta_R-\beta_L<\nu\alpha_R.$$
If this condition holds, then a $C^2$ nonlinear vector field is
$C^1$-linearizable in a neighborhood of the hyperbolic saddle point.


%

\section{GGDPC Dendrograms Under General Edge Scoring Rules}
\label{app:thres_beyond_dist}

The default GGDPC dendrogram in \autoref{subsec:dendro_def} is constructed by thresholding each directed edge according to its 1NN uphill distance $\hat{w}_n(\bm{X}_i)$. In this section, we consider a more general construction in which the thresholding rule for each edge may depend jointly on $\hat{w}_n(\bm{X}_i)$ and the estimated density $\hat{p}(\bm{X}_i)$ at the source vertex. For every observation $\bm{X}_i$ that is not the sample global mode, let $E_i = (\bm{X}_i, \hat{\Phi}_n(\bm{X}_i)) \in E(G)$ be a directed edge in the GGDPC graph $G$. We associate $E_i$ with the pair $\left(\hat{w}_n(\bm{X}_i), \hat{p}(\bm{X}_i)\right) \in [0,\infty)^2$. Consider a fixed scoring function $s:[0,\infty) \times [0,\infty) \to [0,\infty)$ that assigns a decision score $s(E_i):= s\left(\hat{w}_n(\bm{X}_i), \hat{p}(\bm{X}_i)\right)$ to the edge $E_i$ in $G$. Then, for a threshold $\lambda \geq 0$, we define the truncated graph
$$G_{\lambda}^{(s)} = \left(\mathbb{X}_n, E_{\lambda}^{(s)}\right),\qquad E_{\lambda}^{(s)} :=\left\{E_i \in E(G): s(E_i)\leq \lambda\right\}.$$
Equivalently, the outgoing edge from $\bm{X}_i$ is removed whenever $s(E_i) > \lambda$. Let $\mathcal T_{G^{(s)}}(\lambda)$ denote the partition of $\mathbb{X}_n$ induced by the connected components of
$G_{\lambda}^{(s)}$, where edge directions are ignored when defining
connectedness.

The following proposition shows that any fixed edge-scoring rule induces
a valid hierarchy.

\begin{proposition}
	\label{prop:general_thres_tree}
Let $s:[0,\infty)^2 \to [0,\infty)$ be any fixed scoring function. For every $0\leq \lambda\leq \lambda' < \infty$, the partition $\mathcal{T}_{G^{(s)}}(\lambda)$ refines $\mathcal{T}_{G^{(s)}}(\lambda')$. Equivalently, for every $A\in \mathcal{T}_{G^{(s)}}(\lambda)$, there exists a unique $A'\in \mathcal{T}_{G^{(s)}}(\lambda')$ such that $A \subseteq A'$.
\end{proposition}

\begin{proof}[Proof of Proposition~\ref{prop:general_thres_tree}]
By the definition of $G_{\lambda}^{(s)}$, $s(E_i)\leq \lambda$ implies that $s(E_i)\leq \lambda'$ whenever $\lambda\leq \lambda'$. Hence, $E_{\lambda}^{(s)}\subseteq  E_{\lambda'}^{(s)}$ and $G_{\lambda}^{(s)}$ is a spanning subgraph of $G_{\lambda'}^{(s)}$.
	
Let $A$ be a connected component of $G_{\lambda}^{(s)}$, \emph{i.e.}, $A\in \mathcal{T}_{G^{(s)}}(\lambda)$, and choose any $\bm{X}_i,\bm{X}_j\in A$. Then, there exists an undirected path between $\bm{X}_i$ and $\bm{X}_j$ through edges in $E_{\lambda}^{(s)}$, ignoring their directions. Since $E_{\lambda}^{(s)}\subseteq  E_{\lambda'}^{(s)}$, $\bm{X}_i,\bm{X}_j$ lie in a single connected component $A'$ of $G_{\lambda'}^{(s)}$. Thus, $A\subseteq A'$.
	
Because connected components form a partition, there is a unique component of $G_{\lambda'}^{(s)}$ containing $A$.
\end{proof}

Notably, Proposition~\ref{prop:general_thres_tree} requires no monotonicity, continuity, or smoothness of the scoring function $s$. While it looks surprising, the reason is that $s$ assigns a fixed score to each edge $E_i$ and changing $\lambda$ only changes which edges are retained in $G$. However, additional structure on $s$ may be desirable from a clustering perspective. One natural choice is to require $s$ to be non-decreasing with respect to each of its coordinates:
$$\hat{w}_n(\bm{X}_i)\leq \hat{w}_n(\bm{X}_j), \, \hat{p}(\bm{X}_i) \leq \hat{p}(\bm{X}_j) \implies s(\hat{w}_n(\bm{X}_i), \hat{p}(\bm{X}_i)) \leq s(\hat{w}_n(\bm{X}_j), \hat{p}(\bm{X}_j)).$$
Such scoring functions include the product thresholding rule $\hat{w}_n(\bm{X}_i) \cdot \hat{p}(\bm{X}_i) > \lambda$ in the original DPC paper \citep{rodriguez2014clustering}, its dimensionless alternative $\left[\hat{w}_n(\bm{X}_i) \right]^d \hat{p}(\bm{X}_i)> \lambda$, as well as the robust linear regression on $\left\{\left(\hat{w}_n(\bm{X}_i), \hat{p}(\bm{X}_i)\right): i=1,...,n\right\}$ for determining the thresholding value in \cite{verdinelli2018analysis}.

\subsection{Practical Choice of the Edge Score}

Throughout this paper, we recommend the distance-only score $s(\hat{w}_n(\bm{X}_i), \hat{p}(\bm{X}_i)) = \hat{w}_n(\bm{X}_i)$, because it has several advantages. First, the threshold parameter and the dendrogram heights retain the direct geometric interpretation of the 1NN uphill distance. Second, the resulting dendrogram based on the distance-only score is expressed on the same distance scale as the GGDPC decision diagram; see the bottom two panels of \autoref{fig:old_faithful}. Third, this distance-only scoring choice leads to our stability theory for the GGDPC dendrogram in \autoref{subsec:dendro_stability}. Although Proposition~\ref{prop:general_thres_tree} guarantees a valid dendrogram for any fixed scoring function $s$, the convergence theory in \autoref{thm:conv_ggdpc_dendro} does not automatically extend to general scores without additional assumptions on $s$ and its population analogue.

\subsection{Optional Density Screening}

In applications, very low-density observations may have unusually large uphill distances and therefore generate undesirable small branches of
the dendrogram. A simple screening rule can prevent such observations from becoming cluster centers. 
Let $q_{\alpha}$ be the empirical lower $\alpha$-quantile of $\left\{\hat{p}(\bm{X}_i): i=1,...,n\right\}$ for a prespecified level $\alpha\in (0,1)$, for example with $\alpha=0.05$. Then, we define a scoring function as $s_\alpha(w,p)= w\cdot \mathds{1}\{p\geq q_{\alpha}\}$. Equivalently, for the edge starting at $\bm{X}_i$, we have that
$$s_{\alpha}(E_i):= s_{\alpha}\left(\hat{w}_n(\bm{X}_i), \hat{p}(\bm{X}_i)\right) = \hat{w}_n(\bm{X}_i) \cdot \mathds{1}\left\{\hat{p}(\bm{X}_i) \geq q_{\alpha}\right\}.$$
If $\hat p(\bm X_i)<q_\alpha$, then $s_\alpha(E_i)=0$, so its outgoing
edge is retained for every $\lambda\geq 0$. Consequently, such an
observation cannot become a cluster center solely because of a large
uphill distance. This screening rule attaches low-density observations
to the remaining graph rather than labeling them as noise, and should
therefore be distinguished from explicit noise-detection procedures at the end of \autoref{subsec:dpc_review}.

\section{Proof of Theorem~\ref{thm:modal_const}}
\label{app:proof_modal_const}

We begin with a uniform sample coverage lemma and then prove the consistency of GGDPC cluster centers with the population local modes in \autoref{thm:modal_const}.

\subsection{A Uniform Sample Coverage Lemma}

\begin{lemma}[Uniform sample coverage]
\label{lem:1NN_sup_bnd}
Suppose that Assumption~\ref{assump:density_reg}(a) holds, and the density $p$ is bounded on its support $\mathcal{C}$. Then,
$$\sup_{\bm{x}\in \mathcal{C}}\min_{i=1,...,n} \norm{\bm{x}-\bm{X}_i} = O\left(\left(\frac{\log n}{n}\right)^{\frac{1}{d}}\right) \quad \text{ almost surely}.$$
The same conclusion holds with $\mathcal{C}$ replaced by any of its closed subset.
\end{lemma}

\begin{proof}[Proof of Lemma~\ref{lem:1NN_sup_bnd}]
Let $r_n = C_0\left(\frac{\log n}{n}\right)^{\frac{1}{d}}$ for some sufficiently large constant $C_0>0$. Since the support $\mathcal{C}$ of the density $p$ is compact under Assumption~\ref{assump:density_reg}(a), there exists an $\frac{r_n}{2}$-net $\left\{\bm{z}_1,...,\bm{z}_{N_n}\right\} \subset \mathcal{C}$ such that
$$\mathcal{C} \subset \bigcup_{j=1}^{N_n} B\left(\bm{z}_j, \frac{r_n}{2}\right), \quad N_n \leq \frac{C_1}{r_n^d},$$
where $C_1>0$ is a constant depending only on $\mathcal{C}$ and $d$. Notice that if every ball $B\left(\bm{z}_j, \frac{r_n}{2}\right)$ contains at least one observation from $p$, then every $\bm{x}\in \mathcal{C}$ lies within distance $r_n$ of some observation, so $\sup_{\bm{x}\in \mathcal{C}}\min_{i=1,...,n} \norm{\bm{x}-\bm{X}_i} \leq r_n$. Thus, 
$$\left\{ \sup_{\bm{x}\in \mathcal{C}}\min_{i=1,...,n} \norm{\bm{x}-\bm{X}_i} > r_n\right\} \subset \bigcup_{j=1}^{N_n} \left\{B\left(\bm{z}_j, \frac{r_n}{2}\right) \cap \left\{\bm{X}_1,...,\bm{X}_n \right\} =\emptyset \right\}.$$
It suffices to upper bound the probability of the event on the right-hand side.

Now, for each $j$, by Assumption~\ref{assump:density_reg}(a), we have that
\begin{align*}
	\mathbb{P}\left(\bm{X}_1 \in B\left(\bm{z}_j, \frac{r_n}{2}\right) \right) &\geq p_{\min} \left|\mathcal{C} \cap B\left(\bm{z}_j, \frac{r_n}{2}\right) \right| \geq p_{\min} C_{\mathcal{C}}\left(\frac{r_n}{2}\right)^d := C_2 \cdot r_n^d
\end{align*}
for some constant $C_2 > 0$.
Hence, 
$$\mathbb{P}\left(B\left(\bm{z}_j, \frac{r_n}{2}\right) \cap \left\{\bm{X}_1,...,\bm{X}_n \right\} =\emptyset \right) \leq \left(1- C_2 r_n^d\right)^n \leq e^{-C_2n r_n^d}.$$
By the union bound, we derive that
\begin{align*}
	\mathbb{P}\left(\sup_{\bm{x}\in \mathcal{C}}\min_{i=1,...,n} \norm{\bm{x}-\bm{X}_i} > r_n\right) &\leq \sum_{j=1}^{N_n} \mathbb{P}\left(B\left(\bm{z}_j, \frac{r_n}{2}\right) \cap \left\{\bm{X}_1,...,\bm{X}_n \right\} =\emptyset \right) \\
	&\leq \frac{C_1 e^{-C_2n r_n^d}}{r_n^d}.
\end{align*}
Choosing $r_n=C_0\left(\frac{\log n}{n}\right)^{\frac{1}{d}}$, the above display becomes
\begin{align*}
	\mathbb{P}\left(\sup_{\bm{x}\in \mathcal{C}}\min_{i=1,...,n} \norm{\bm{x}-\bm{X}_i} > r_n\right) &\leq \frac{C_1 e^{-C_2n r_n^d}}{r_n^d} \leq \frac{C_1 \cdot n^{1-C_2C_0^d}}{C_0^d \log n}.
\end{align*}
Now, we pick $C_0>0$ large enough so that $C_2 C_0^d > 3$, so
$$\sum_{n=1}^{\infty} \mathbb{P}\left(\sup_{\bm{x}\in \mathcal{C}}\min_{i=1,...,n} \norm{\bm{x}-\bm{X}_i} > r_n \right) < \infty.$$
By the Borel-Cantelli theorem, we know that 
$$\sup_{\bm{x}\in \mathcal{C}}\min_{i=1,...,n} \norm{\bm{x}-\bm{X}_i} \leq r_n= C_0 \left(\frac{\log n}{n}\right)^{\frac{1}{d}} \quad \text{ almost surely}$$
for all large $n$. The result follows.
\end{proof}

\subsection{Main Proof of \autoref{thm:modal_const}}

\begin{proof}[Proof of \autoref{thm:modal_const}]
Let $q_n = \left(\frac{\log n}{n}\right)^{\frac{1}{d}}$. In the sequel, all comparisons involving $\hat p$ use the fixed strict tie-breaking ordering. We also work on the probability-one event supplied by the standing non-tie condition for the values $\hat p(\bm X_i)$.

{\bf Step 1: Sample local modes and their localization.} We first derive an upper bound for the maximal pairwise distance between the sample and true local modes. Under Assumption~\ref{assump:density_reg}(c), we know that each $\bm{m}_j\in \mathcal{M}$ is non-degenerate. Then, there exist constants $r_j>0$ and $0< c_j\leq C_j < \infty$ for each $\bm{m}_j$ such that
\begin{equation}
	\label{local_model_quadra}
	c_j\norm{\bm{x} -\bm{m}_j}^2 \leq p(\bm{m}_j) - p(\bm{x}) \leq C_j \norm{\bm{x} - \bm{m}_j}^2
\end{equation}
for any $\bm{x}\in B(\bm{m}_j,r_j)$. By shrinking these radii $r_1,...,r_{|\mathcal{M}|}$ if necessary, we can assume that \begin{equation}
	\label{rad_bnd}
	2\max_{j=1,...,|\mathcal{M}|} r_j < \lambda.
\end{equation}
Let $\bm{X}_{(j)} \in \argmin_{1\leq i\leq n} \norm{\bm{X}_i -\bm{m}_j}$ be the closest observation to $\bm{m}_j$. Since $B(\bm{m}_j,q_n) \subset \mathcal{C}$ when $n$ is sufficiently large, our argument in Lemma~\ref{lem:1NN_sup_bnd} shows that
\begin{equation}
	\label{1nn_mode}
	\max_{j=1,...,|\mathcal{M}|}\norm{\bm{X}_{(j)} - \bm{m}_j} \leq C_0 q_n
\end{equation}
with probability tending to one for some large constant $C_0 >0$. For each $B(\bm{m}_j,r_j)$, by \eqref{1nn_mode}, $\mathbb{X}_n\cap B(\bm{m}_j,C_0 q_n)$ is non-empty, and we define the sample local mode as 
$$\hat{\bm{m}}_{j,n} \in \argmax_{\bm{X}_i\in \mathbb{X}_n\cap B(\bm{m}_j,r_j)} \hat{p}(\bm{X}_i),$$
which is unique under the fixed tie-breaking ordering with probability tending to one. Also, \eqref{1nn_mode} implies that $\bm{X}_{(j)} \in B(\bm{m}_j,r_j)$ when $n$ is sufficiently large.
Then, $\hat{p}(\hat{\bm{m}}_{j,n}) \geq \hat{p}(\bm{X}_{(j)})$ and
\begin{align*}
	p(\bm{m}_j) - p(\hat{\bm{m}}_{j,n}) &\leq p(\bm{m}_j) - p(\bm{X}_{(j)}) + 2\norm{\hat{p} - p}_{\infty}\\
	&\leq C_j\norm{\bm{m}_j- \bm{X}_{(j)}}^2 + 2\norm{\hat{p} - p}_{\infty}.
\end{align*}
By \eqref{local_model_quadra} again, we obtain that
$$c_j\norm{\hat{\bm{m}}_{j,n} - \bm{m}_j}^2 \leq p(\bm{m}_j) - p(\hat{\bm{m}}_{j,n}) \leq C_j\norm{\bm{X}_{(j)} - \bm{m}_j}^2 + 2\norm{\hat{p} - p}_{\infty}.$$
As a result, 
\begin{equation}
	\label{mode_distance}
	\max_{j=1,...,|\mathcal{M}|} \norm{\hat{\bm{m}}_{j,n} - \bm{m}_j} \leq \bar{C} \left(q_n + \norm{\hat{p} - p}_{\infty}^{\frac{1}{2}} \right)
\end{equation}
with probability tending to one, for some constant $\bar{C} >0$ independent of $n$.

{\bf Step 2: Uniform availability of nearby higher-density observations away from the modes.}
Suppose now that $\hat{g}=\nabla \hat{p}$ on the chosen mode neighborhoods $B(\bm{m}_j,r_j)$ for $j=1,...,|\mathcal{M}|$. By \eqref{mode_distance}, the line segment between $\bm{X}_{(j)}$ and $\hat{\bm{m}}_{j,n}$ lies in $B(\bm{m}_j,r_j)$ when $n$ is sufficiently large. By the differentiability of $\hat{p}$, we know that
\begin{align*}
	0 &\leq \hat{p}(\hat{\bm{m}}_{j,n}) - \hat{p}(\bm{X}_{(j)}) \\
	&= p(\hat{\bm m}_{j,n})-p(\bm X_{(j)}) +\int_0^1
	\left(\hat g-\nabla p\right) \left(\bm X_{(j)}+t(\hat{\bm m}_{j,n}-\bm X_{(j)})\right)^T \left[\hat{\bm m}_{j,n}-\bm X_{(j)}\right] dt.
\end{align*}
Hence, $p(\bm{X}_{(j)}) - p(\hat{\bm{m}}_{j,n}) \leq \norm{\hat{g} - \nabla p}_{\infty} \norm{\hat{\bm{m}}_{j,n} - \bm{X}_{(j)}}$. By \eqref{local_model_quadra}, we derive that
\begin{align*}
	c_j\norm{\hat{\bm{m}}_{j,n} - \bm{m}_j}^2 &\leq p(\bm{m}_j) - p(\hat{\bm{m}}_{j,n}) \\
	&\leq p(\bm{m}_j) - p(\bm{X}_{(j)}) + \norm{\hat{g} - \nabla p}_{\infty} \norm{\hat{\bm{m}}_{j,n} - \bm{X}_{(j)}} \\
	&\leq C_j\norm{\bm{X}_{(j)} - \bm{m}_j}^2 + \norm{\hat{g} - \nabla p}_{\infty} \norm{\bm{X}_{(j)} - \bm{m}_j} + \norm{\hat{g} - \nabla p}_{\infty} \norm{\hat{\bm{m}}_{j,n} - \bm{m}_j}\\
	&\stackrel{\text{(i)}}{\leq} C_j'\left( q_n^2 + \norm{\hat{g} - \nabla p}_{\infty}^2\right),
\end{align*}
where (i) absorbs a sufficiently small multiple of $\norm{\hat{\bm{m}}_{j,n} - \bm{m}_j}^2$ into the left-hand side. As a result,
$$\max_{j=1,...,|\mathcal{M}|} \norm{\hat{\bm{m}}_{j,n} - \bm{m}_j} \leq \bar{C} \left(q_n + \norm{\hat{g} - \nabla p}_{\infty} \right).$$
Combining with \eqref{mode_distance}, we obtain under the extra differentiability of $\hat{p}$ that
\begin{equation}
	\label{mode_distance2}
	\max_{j=1,...,|\mathcal{M}|} \norm{\hat{\bm{m}}_{j,n} - \bm{m}_j} \leq \bar{C} \left(q_n + \min\left\{\norm{\hat{p} - p}_{\infty}^{\frac{1}{2}},\, \norm{\hat{g} - \nabla p}_{\infty}\right\} \right).
\end{equation}

{\bf Step 3: No observation other than a sample local mode is a sink of the truncated graph $\hat{\mathcal{M}}_{\lambda}$.}
We choose a fixed $\epsilon_\lambda>0$ sufficiently small that 
$$4\epsilon_\lambda<\lambda,
\qquad
2\max_{1\leq j\leq|\mathcal M|}r_j+4\epsilon_\lambda<\lambda,$$
after shrinking the modal radii $r_j$ if necessary. Consider the set $\mathcal{C} \setminus\bigcup_{j=1}^{|\mathcal M|} B^o(\bm m_j,r_j/2)$, where none of its points are local modes of $p$.  
By Assumption~\ref{assump:density_reg}, there exist fixed constants
$c_\lambda,v_\lambda>0$ such that, for every $\bm x\in \mathcal{C} \setminus\bigcup_{j=1}^{|\mathcal M|} B^o(\bm m_j,r_j/2)$, there is
a measurable set $A(\bm x)\subset\mathcal C\cap B(\bm x,\epsilon_\lambda)$ satisfying
\begin{equation}
	\label{uniform_fixed_uphill}
	P\left(\bm{X}\in A(\bm x)\right)\geq v_\lambda,
	\qquad
	\inf_{\bm y\in A(\bm x)}\left[p(\bm y)-p(\bm x)\right]\geq c_\lambda.
\end{equation}
To see this, if $\nabla p(\bm x)\neq 0$, we can select $A(\bm{x})$ by following the population gradient flow in a fixed sufficiently small step. If $\bm x$ is a non-modal critical point, we can also move along the eigenvector direction associated with the positive eigenvalue of $\nabla^2 p(\bm x)$. By the forward-invariance condition and the differentiability of $p$,  
$$P\left(\bm{X}\in A(\bm x)\right) \geq p_{\min}\cdot \mathrm{Leb}(A(\bm{x}) \cap \mathcal{C} \cap B(\bm{x},\epsilon_\lambda))\gtrsim C_{\mathcal{C}} \cdot  \mathrm{Leb}(B(\bm{y},\epsilon_\lambda'))$$
for some $\bm{y} \in A(\bm{x})$ along the above direction and $\epsilon_\lambda' \in (0,\epsilon_\lambda)$.
The uniformity follows from a finite subcover of $\mathcal{C} \setminus\bigcup_{j=1}^{|\mathcal M|} B^o(\bm m_j,r_j/2)$.

Conditionally on $\bm X_i=\bm x\in \mathcal{C} \setminus\bigcup_{j=1}^{|\mathcal M|} B^o(\bm m_j,r_j/2)$, the probability that none of the other $n-1$ observations falls in $A(\bm x)$ is at most $(1-v_\lambda)^{n-1}\leq \exp\left[-(n-1)v_{\lambda}\right]$. A union bound over $i=1,\ldots,n$ 
shows that, with probability $n\exp\left[-(n-1)v_{\lambda}\right]$ tending to one, every observation in $\mathcal{C} \setminus\bigcup_{j=1}^{|\mathcal M|} B^o(\bm m_j,r_j/2)$ has another observation $\bm X_\ell \in A(\bm X_i)$. On the
event $\norm{\hat p-p}_\infty<\frac{c_\lambda}{3}$, such an observation
satisfies that
$$\hat p(\bm X_\ell)- \hat p(\bm X_i) \geq p(\bm X_\ell) - p(\bm X_i) - 2\norm{\hat{p} - p}_{\infty} \geq \frac{c_{\lambda}}{3}> 0.$$ 
Hence, $\bm{X}_\ell$ is admissible in the definition of $\hat{\Phi}_n(\bm{X}_i)$ and, by minimality of the GGDPC update,
$$\hat w_n(\bm X_i) \leq \norm{\bm X_\ell-\bm X_i} +2\eta_n\norm{\hat g(\bm X_i)} \leq \epsilon_\lambda+o_P(1)<\lambda$$
uniformly over $\bm X_i\in \mathcal{C} \setminus\bigcup_{j=1}^{|\mathcal M|} B^o(\bm m_j,r_j/2)$, because
$\sup_{\mathcal C}\norm{\hat g}=O_P(1)$.

Now consider $\bm X_i\in B(\bm m_j,r_j)$ with
$\bm X_i\neq\hat{\bm m}_{j,n}$. By definition of the sample local
mode and the fixed strict ordering,
$\hat p(\hat{\bm m}_{j,n})>\hat p(\bm X_i)$, so
$\hat{\bm m}_{j,n}$ is admissible. Consequently,
\[
\hat w_n(\bm X_i)
\leq \norm{\hat{\bm m}_{j,n}-\bm X_i}
+2\eta_n\norm{\hat g(\bm X_i)}
\leq 2r_j+o_P(1)<\lambda
\]
uniformly. Since the balls $B(\bm m_j,r_j/2)$ and $\mathcal{C} \setminus\bigcup_{j=1}^{|\mathcal M|} B^o(\bm m_j,r_j/2)$ cover
$\mathcal C$, we conclude that
$\hat{\mathcal{M}}_{\lambda} \subset \left\{\hat{\bm{m}}_{1,n},...,\hat{\bm{m}}_{|\mathcal{M}|,n}\right\}$ with probability tending to one.

{\bf Step 4: Classification of the local sample modes.} 
Since $\lambda\notin \left\{\psi_j:1< j\leq |\mathcal{M}|\right\}$, there exists $A_{\lambda} >0$ such that $\left|\psi_j - \lambda\right| > 4A_{\lambda}$ for all $j\geq 2$. We now consider two different cases: (a) $\psi_j< \lambda$ and (b) $\psi_j > \lambda$, for the inclusion of $\hat{\bm{m}}_{j,n}$ in $\hat{\mathcal{M}}_{\lambda}$ or not.

{\bf Case (a) $\psi_j< \lambda$:} By the definition of
$\psi_j$, there exists $\bm{y}_j\in \mathcal{U}_j$ such that $\norm{\bm{y}_j - \bm{m}_j} < \lambda - 3A_{\lambda}$. Since $p(\bm{y}_j) > p(\bm{m}_j)$ and $p$ is continuous, there exists a ball $B(\bm{y}_j,s_j)$ of positive probability such that $p(\bm{y}) > p(\bm{m}_j)$ and $\norm{\bm{y} - \bm{m}_j} < \lambda - 2A_{\lambda}$ for all $\bm{y} \in B(\bm{y}_j,s_j)$. By Assumption~\ref{assump:density_reg}(a), with probability tending to one, $B(\bm{y}_j,s_j)$ contains an observation, say $\bm{X}_{j,n}$, so that
$$p(\bm{X}_{j,n}) > p(\bm{m}_j) \geq p(\hat{\bm{m}}_{j,n}).$$
By \eqref{mode_distance}, we further have that
$$\hat{p}(\bm{X}_{j,n}) - \hat{p}(\hat{\bm{m}}_{j,n}) \geq p(\bm{X}_{j,n}) - p(\hat{\bm{m}}_{j,n}) - 2\norm{\hat{p} - p}_{\infty} > 0, \qquad \norm{\hat{\bm{m}}_{j,n} - \bm{X}_{j,n}} < \lambda - A_{\lambda}$$
when $n$ is sufficiently large. Consequently, 
$$\hat{w}_n(\hat{\bm{m}}_{j,n}) \leq \norm{\bm{X}_{j,n} - \hat{\bm{m}}_{j,n}} + 2\bar{C}_1 \eta_n\left(1+ \norm{\hat{g} - \nabla p}_{\infty}\right) < \lambda$$
with probability tending to one. Thus, $\psi_j < \lambda$ implies that $\hat{\bm{m}}_{j,n} \notin \hat{\mathcal{M}}_{\lambda}$.

{\bf Case (b) $\psi_j > \lambda$:} We claim that with probability tending to one, every observation $\bm{X}_i$ satisfying $\hat{p}(\bm{X}_i) > \hat{p}(\hat{\bm{m}}_{j,n})$ also satisfies $\norm{\bm{X}_i - \hat{\bm{m}}_{j,n}} > \lambda$.

Suppose that the claim is false. Then, along a subsequence, there would exist observations $\bm Y_n$ such that
$$\hat{p}(\bm{Y}_n) > \hat{p}(\hat{\bm m}_{j,n}), \qquad \norm{\bm{Y}_n-\hat{\bm m}_{j,n}}\leq\lambda.$$
By compactness of $\mathcal{C}$, we can further find a subsequence such that $\bm{Y}_n \to \bm{y} \in \mathcal{C}$. By \eqref{mode_distance}, $\hat{\bm{m}}_{j,n} \to \bm{m}_j$ as $n\to \infty$. The uniform consistency of $\hat{p}$ under $\norm{\hat{p} - p}_{\infty}$ and the continuity of $p$ imply that
$$p(\bm{y})\geq p(\bm{m}_j), \qquad \norm{\bm{y}-\bm{m}_j}\leq\lambda,$$
where the first inequality uses $\norm{\hat{p} - p}_{\infty} \to 0$ as $n\to \infty$ with probability tending to one.
Moreover, $\bm Y_n\notin B(\bm{m}_j,r_j)$, because
$\hat{\bm m}_{j,n}$ maximizes $\hat p$ over the observations in $B(\bm{m}_j,r_j)$. Hence, $\bm{y}\neq \bm{m}_j$. Now, if $p(\bm{y}) > p(\bm{m}_j)$, then $\bm{y}\in \overline{\mathcal{U}}_j$. If $p(\bm{y})=p(\bm{m}_j)$, then $\bm{y}\in \overline{\mathcal{U}}_j$ and it cannot be a local mode because distinct modes have distinct density values under Assumption~\ref{assump:density_reg}. Therefore, $\bm{y} \in\overline{\mathcal U}_j$ and 
$$\psi_j = d(\bm{m}_j,\overline{\mathcal{U}}_j) \leq \norm{\bm{y}-\bm{m}_j} \leq \lambda,$$
contradicting $\psi_j>\lambda$. The claim thus follows, and $\psi_j > \lambda$ implies that $\hat{\bm{m}}_{j,n} \in \hat{\mathcal{M}}_{\lambda}$.

Furthermore, by the condition that $\norm{\hat{p} - p}_{\infty}=o_P(1)$, $\hat{\bm{m}}_{1,n}$ is the global sample maximizer of $\hat p$ with probability tending to one, so $\hat{w}_n(\hat{\bm{m}}_{1,n})=\infty$.

{\bf Step 5: Conclusions about one-to-one correspondence between $\hat{\mathcal{M}}_{\lambda}$ and $\mathcal{M}_{\lambda}$ as well as their Hausdorff distance.} Combining the results that $\hat{\mathcal{M}}_{\lambda} \subset \left\{\hat{\bm{m}}_{1,n},...,\hat{\bm{m}}_{|\mathcal{M}|,n}\right\}$ with probability tending to one, $\psi_j < \lambda \implies \hat{\bm{m}}_{j,n} \notin \hat{\mathcal{M}}_{\lambda}$, and $\psi_j > \lambda \implies \hat{\bm{m}}_{j,n} \in \hat{\mathcal{M}}_{\lambda}$, we conclude that
\begin{equation}
	\label{est_mode_cor}
	\hat{\mathcal M}_\lambda = \left\{\hat{\bm m}_{j,n}:\psi_j>\lambda\right\}
\end{equation}
with probability tending to one. This proves that $|\hat{\mathcal M}_\lambda| =|\mathcal M_\lambda|$. Additionally, \eqref{mode_distance} or \eqref{mode_distance2} together with \eqref{est_mode_cor} imply that 
\begin{align*}
	&\mathrm{Haus}\left(\mathcal M_\lambda, \hat{\mathcal M}_\lambda \right) \leq \max_{j:\psi_j > \lambda} \norm{\hat{\bm m}_{j,n}-\bm{m}_j} \\
	&= 
	\begin{cases}
		O_P\left(q_n + \min\left\{\norm{\hat{p} - p}_{\infty}^{\frac{1}{2}},\, \norm{\hat{g} - \nabla p}_{\infty}\right\} \right), & \text{ if }\hat{p} \text{ is continuously differentiable and } \hat{g}=\nabla \hat{p},\\ 
		O_P\left(q_n + \norm{\hat{p} - p}_{\infty}^{\frac{1}{2}} \right), & \text{ otherwise}.
	\end{cases}
\end{align*}
The results thus follow.
\end{proof}

\section{Proof of Theorem~\ref{thm:GGDPC_ARI}}
\label{app:ggdpc_ari_proof}

We begin with an example illustrating Assumptions~\ref{assump:sep_repulsion}, \ref{assump:unique_proj}, and \ref{assump:C1_linear} as well as establish a local repulsion lemma from the separatrix, without requiring Assumption~\ref{assump:C1_linear}. Then, we leverage this result to prove \autoref{thm:GGDPC_ARI}.

\subsection{Example: A Two-Gaussian Mixture for Assumptions~\ref{assump:sep_repulsion}, \ref{assump:unique_proj}, and \ref{assump:C1_linear}}
\label{app:sep_repulsion_example}

Consider a $d$-dimensional Gaussian mixture model
\begin{equation}
	\label{GMM2}
	\pi_1 \cdot \mathcal{N}(-\bm{\mu}, \sigma^2\bm{I}_d) + \pi_2 \cdot \mathcal{N}(\bm{\mu}, \sigma^2\bm{I}_d),
\end{equation}
where $\pi_1+\pi_2=1$ and $\bm{\mu}=(\mu,0,...,0)^T\in \mathbb{R}^d$ with $\mu >0$. It density is given by 
\begin{align*}
	p(\bm{x}) &= \frac{\pi_1}{(2\pi\sigma^2)^{\frac{d}{2}}} \exp\left[-\frac{(x_1+\mu)^2 + \norm{\bm{x}_{-1}}^2}{2\sigma^2}\right] + \frac{\pi_2}{(2\pi\sigma^2)^{\frac{d}{2}}} \exp\left[-\frac{(x_1-\mu)^2 + \norm{\bm{x}_{-1}}^2}{2\sigma^2}\right] \\
	&= \exp\left(-\frac{\norm{\bm{x}_{-1}}^2}{2\sigma^2}\right) p(x_1,\bm{0}),
\end{align*}
with $\bm{x}=(x_1,\bm{x}_{-1})^T\in \mathbb{R}^d$.
Suppose that $p$ is bimodal with an index-$(d-1)$ saddle point $\bm{s}=(s,0,...,0)^T \in \mathbb{R}^d$ with $-\mu < s < \mu$ between the two local modes. Since $\nabla p(\bm{s})=\bm{0}$, we know that
\begin{equation}
	\label{sad_eq}
	\frac{\partial}{\partial x_1} p(\bm{s}) = 0 \quad \implies \quad 
	\pi_1 \exp\left(-\frac{2s\mu}{\sigma^2}\right)
	(s+\mu) = \pi_2(\mu-s).
\end{equation}
Moreover, $\nabla_{\bm x_{-1}}p(s,\bm x_{-1})
=-\frac{p(s,\bm x_{-1})}{\sigma^2} \cdot \bm{x}_{-1}$. Therefore, $\mathcal S = \left\{ \bm{x}\in\mathbb R^d:x_1=s \right\}$ is invariant under the gradient ascent flow and forms the stable manifold (or separatrix) of $\bm{s}$ between two modal basins.

Now, let $\bm{e}_1=(1,0,...,0)^T\in \mathbb{R}^d$, which is normal to $\mathcal{S}$. Moreover, at the saddle point $\bm{s}$, we know from \eqref{sad_eq} that
$$\bm{e}_1^T \nabla^2 p(\bm{s}) \bm{e}_1 = \frac{p(\bm{s})}{\sigma^2} \left( \frac{\mu^2-s^2}{\sigma^2}-1 \right).$$
Since $\bm{s}$ is an index-$(d-1)$ saddle point, its unique normal
eigenvalue is positive so that $\frac{\mu^2-s^2}{\sigma^2}>1$.
For an arbitrary $\bm{s}'=(s,\bm{x}_{-1})\in\mathcal S$,
\begin{align*}
	\bm{e}_1^T\nabla^2 p(\bm{s}')\bm{e}_1&= \frac{\partial^2p}{\partial x_1^2}(s,\bm{x}_{-1})\\
	&= \exp\left(-\frac{\norm{\bm{x}_{-1}}^2}{2\sigma^2}\right)
	\frac{\partial^2p}{\partial x_1^2}(s,\bm 0)\\
	&=\frac{p(s,\bm{x}_{-1})}{\sigma^2}
	\left( \frac{\mu^2-s^2}{\sigma^2}-1\right) >0.
\end{align*}
Consequently, on every compact subset of the regular separatrix,
the normal Hessian is uniformly bounded away from zero, illustrating
the normal repulsion condition in Assumption~\ref{assump:sep_repulsion}.

In particular, the Gaussian mixture in \autoref{fig:gmm_graph}
corresponds, after recentering the first coordinate, to
$\pi_1=\pi_2=1/2$, $\mu=1/2$, and $\sigma^2=0.09$. In this case, $s=0$, or equivalently the separatrix in the original coordinates is
$x_1=0.5$, and
$$\frac{\mu^2-s^2}{\sigma^2}-1 =\frac{0.25}{0.09}-1 =\frac{16}{9}>0.$$

Furthermore, Assumption~\ref{assump:unique_proj} holds for this bimodal Gaussian mixture \eqref{GMM2} whenever the two modal heights are distinct, \emph{e.g.}, $\pi_1\neq \pi_2$. Indeed,
for the non-global mode $\bm m_-=(a,\bm 0)$, we assume without loss of
generality that $a<s<b$, where $\bm s=(s,\bm 0)$ is the saddle point and
$\bm m_+=(b,\bm 0)$ is the global mode. There is a unique $z\in(s,b)$ such that $p(z,\bm{0})=p(a,\bm{0})$. Since
$p(t,\bm x_{-1})\leq p(t,\bm{0})$, the unique projection of $\bm m_-$ onto
$\overline{\mathcal U}_-$ is $\bm z_-=(z,\bm 0)$. Moreover,
$\bm z_-\in\mathcal C_+$, $\nabla p(\bm z_-)\neq 0$, and
$d(\bm z_-,\partial\mathcal C_+)=z-s>0$. Finally, since
$\mu_-=\frac{z-a}{\frac{\partial}{\partial x_1}p(z,\bm{0})}>0$, then every unit vector $\bm v$ satisfying
$\bm v^T\nabla p(\bm z_-)=0$ has $v_1=0$, and hence
$$\bm v^T\left[I_d-\mu_-\nabla^2p(\bm z_-)\right]\bm v = 1+\mu_-\cdot \frac{p(\bm z_-)}{\sigma^2}>1.$$
Therefore, Assumption~\ref{assump:unique_proj} is satisfied.

Finally, Assumption~\ref{assump:C1_linear} also holds under this example. At the saddle point $\bm{s}=(s,\bm{0})$, the Hessian matrix satisfies
$$\nabla^2 p(\bm{s}) = \begin{pmatrix}
\rho_u & & & \\
& -\rho_s & & \\
& & \ddots & \\
& & & -\rho_s
\end{pmatrix},$$
where
$\rho_u = \frac{p(\bm{s})}{\sigma^2} \left(\frac{\mu^2-s^2}{\sigma^2}-1 \right)>0$ and $\rho_s=\frac{p(\bm{s})}{\sigma^2}>0$.
Thus, there is a single unstable eigenvalue, while all stable eigenvalues are identical. Since $p$ is $C^\infty$, Hartman's spectral condition (recall \autoref{app:tech_dynamic}) for $C^1$-linearization is automatically satisfied. Therefore, the gradient vector field $\nabla p$ is thus $C^1$-linearizable in a neighborhood of $\bm{s}$, verifying Assumption~\ref{assump:C1_linear}.

\subsection{Local Repulsion Lemma}

\begin{lemma}[Local repulsion from the separatrix without linearization]
\label{lem:sep_repulsion}
Suppose that Assumptions~\ref{assump:density_reg} and
\ref{assump:sep_repulsion} hold. For every modal basin
$\mathcal{C}_a$, there exist constants $r_0,C_0,C_1,C_2>0$
together with a $C^2$ function $H_a: \mathcal{S}_{\rm part}^{r_0} \cap \mathcal{C}_a \to \mathbb{R}$ such that 
$$C_1 \cdot d(\bm x,\partial\mathcal C_a) \leq H_a(\bm{x})
\leq C_2 \cdot d(\bm x,\partial\mathcal C_a),$$
and
$$\nabla H_a(\bm x)^T\nabla p(\bm x) \geq C_0 H_a(\bm{x})$$
whenever $\bm x\in \mathcal{S}_{\rm part}^{r_0} \cap \mathcal{C}_a$ and
$d(\bm x,\partial\mathcal C_a)<r_0$, where $\mathcal{S}_{\rm part}^{r_0} = \left\{\bm{x}\in \mathcal{C}: d(\bm{x}, \partial C_a) < r_0\right\}$.
\end{lemma}

\begin{proof}[Proof of Lemma~\ref{lem:sep_repulsion}]
We first consider a regular point $\bm{s}\in \partial \mathcal{C}_a$ separated from the finitely many boundary saddle points. By Assumption~\ref{assump:sep_repulsion},
$\partial\mathcal C_a$ is locally a $C^2$ hypersurface, so let
$H_a$ be its (signed) distance function, which is positive on $\mathcal{C}_a$. For any $\bm x\in \mathcal{S}_{\rm part}^{r_0} \cap \mathcal{C}_a$ and
$d(\bm x,\partial\mathcal C_a)<r_0$, we write
$$\bm x=\bm s+r\nu(\bm s),\qquad r=H_a(\bm x)>0,$$
where $\nu(\bm{s})$ is the inward unit normal. Then, we know from the invariance of the separatrix under the gradient flow that $\nu(\bm s)^T \nabla p(\bm s)=0$. By Taylor's expansion,
\begin{align*}
	\nabla H_a(\bm x)^T\nabla p(\bm x) &= \nu(\bm s)^T
	\left[\nabla p(\bm s) +r\nabla^2p(\bm s)\nu(\bm s) +O(r^2) \right]\\
	&=r\cdot \nu(\bm s)^T\nabla^2p(\bm s)\nu(\bm s)
	+O(r^2)\\
	&\geq \frac{\rho_{\mathcal S}}{2} H_a(\bm x),
\end{align*}
where the last inequality follows from Assumption~\ref{assump:sep_repulsion} after possibly decreasing $r_0$.

Now, if $\bm{s}$ is a boundary saddle, then the stable manifold theorem gives a $C^2$ local stable manifold $W^s_{\rm loc}(\bm s)$ of dimension $d-1$.
Choose a $C^2$ coordinate system or local chart $\Psi(\bm x)=(u,\bm v)$
so that $W^s_{\rm loc}(\bm s)=\{u=0\}$.
Let $F(\bm{z})=D\Psi(\Psi^{-1}(\bm{z}))\nabla p\left(\Psi^{-1}(\bm{z})\right)$ denote the transformed vector
field. Writing $F(u,\bm{v}) = \left(F_1(u,\bm{v}), F_{-1}(u,\bm{v})\right)^T$, then the invariance of the stable manifold implies that $F_1(0,\bm v)=0$.
Thus, 
$$F_1(u,\bm v)=a(u,\bm v)\cdot u$$ 
for some function $a$. In particular, $a(0,\bm{0})=\frac{\partial F_1}{\partial u}(0,\bm{0})=\rho_+(\bm s)>0$, which is the unique positive eigenvalue of $\nabla^2p(\bm s)$. Consequently,
after shrinking the neighborhood around $\bm{s}$, we know that $a(u,\bm v)\geq\frac{\lambda_+(\bm s)}{2}$.
Taking $H_a(\bm x)$ to be the signed first coordinate of $\Psi(\bm x)=(u,\bm{v})$ on the basin side yields that
$$\nabla h(\bm x)^T\nabla p(\bm x) =a(\Psi(\bm x))h(\bm x) \geq \frac{\lambda_+(\bm s)}{2} h(\bm x).$$
Since $\Psi$ is a local $C^2$ diffeomorphism, $H_a(\bm x)$ is
comparable to $d(\bm x,\partial\mathcal C_a)$. The result thus follows from compactness of the separatrix.
\end{proof}

\subsection{Main Proof of \autoref{thm:GGDPC_ARI}}

\begin{proof}[Proof of \autoref{thm:GGDPC_ARI}]
Under Assumption~\ref{assump:density_reg}, let $\mathcal{S}_{\rm full} := \bigcup_{a=1}^{|\mathcal{M}|} \partial \mathcal{C}_a$ be all the boundaries of basins of attraction. For every $\bm{x}\in \mathcal{C}\setminus \mathcal{S}_{\rm full}$, let $cl(\bm{x}) \in \left\{1,...,|\mathcal{M}|\right\}$ denote its population cluster label so that $cl(\bm{x}) = a$ whenever $\bm{x}\in \mathcal{C}_a$. Since
$\mathcal S_{\rm full}$ is a finite union of stable
manifolds with dimension less than or equal to $d-1$, it has Lebesgue measure zero. Hence, $cl(\bm X_i)$ is well-defined almost surely for every observation.

We divide the proof into four steps. All the constants denoted by $C_*$ below are fixed and independent of $n$.

{\bf Step 1: Uniform one-step GGDPC approximation near the separatrix.} We consider point $\bm{y}$ satisfying $d(\bm{y}, \partial \mathcal{C}_a) \geq A\left(\eta_n + \frac{q_n}{\eta_n} + \norm{\hat{g}-\nabla p}_{\infty}\right) \to 0$ and $d(\bm{y},\partial \mathcal{C}_a) < r_0$, where $q_n=\left(\frac{\log n}{n}\right)^{\frac{1}{d}}$, $r_0>0$ is defined as Lemma~\ref{lem:sep_repulsion}, and $A_1>0$ is some sufficiently large but fixed constant to be specified below. Let $\bm{z} = \bm{\gamma}_{\bm{y}}(\eta_n)$. By Assumption~\ref{assump:density_reg}, $\bm{z}\in \mathcal{C}_a$, and uniformly over the region of interest for $\bm{y}$, we have that
\begin{equation}
	\label{ARI_exact_flow}
	\bm{z}=\bm y+\eta_n\nabla p(\bm y)+ O\left(\eta_n^2\norm{\nabla p(\bm y)}\right).
\end{equation}
By the definition of $\zeta_n$ in \eqref{max_1nn_dist}, there exists an observation $\bm{X}(\bm{y}) \in \mathbb{X}_n$ satisfying $\norm{\bm{X}(\bm{y}) - \bm{z}}\leq \zeta_n$. We claim that $\bm X(\bm y)$ is admissible in the definition of $\hat\Phi_n(\bm y)$ in \eqref{gg_1NN_uphill} when $A_1$ is sufficiently large. By Assumption~\ref{assump:sample_den_diff} and Taylor's expansion, 
\begin{align*}
	\hat p(\bm X(\bm y))-\hat p(\bm y)
	&=\hat g(\bm y)^T\left[\bm X(\bm y)-\bm y\right] + O\left(\norm{\bm X(\bm y)-\bm y}^2\right)\\
	&= \nabla p(\bm{y})^T \left[\bm X(\bm y)-\bm y\right] + \left[\hat g(\bm y) -\nabla p(\bm{y})\right]^T\left[\bm X(\bm y)-\bm y\right] + O\left(\norm{\bm X(\bm y)-\bm y}^2\right)\\
	&\stackrel{\text{(i)}}{\geq} \eta_n\norm{\nabla p(\bm y)}^2 -C_1\eta_n \norm{\hat{g} - \nabla p}_{\infty}\norm{\nabla p(\bm y)} -C_1\zeta_n\left(\norm{\nabla p(\bm y)}+ \norm{\hat{g} - \nabla p}_{\infty}\right)\\
	&\quad -C_1\eta_n^2\norm{\nabla p(\bm y)} \left(\norm{\nabla p(\bm y)}+\norm{\hat{g} - \nabla p}_{\infty} \right) -C_1 \zeta_n^2,
\end{align*}
where (i) follows from \eqref{ARI_exact_flow}. Since $$\nabla p(\bm{y}) \gtrsim d(\bm{y}, \partial \mathcal{C}_a) \geq A_1\left(\eta_n + \frac{q_n}{\eta_n} + \norm{\hat{g}-\nabla p}_{\infty}\right),$$ 
we know that every negative term above, after division by $\eta_n\norm{\nabla p(\bm{y})}^2$, is bounded by $O\left(\frac{1}{A}\right) + o(1)$. Thus, after choosing $A_1>0$ sufficiently large and then $n$ sufficiently large, we know that $\hat p(\bm X(\bm y)) >\hat p(\bm y)$ and $\bm{X}(\bm{y})$ is thus admissible. By the minimality of $\hat{\Phi}_n(\bm{y})$, 
\begin{align*}
	\norm{\hat\Phi_n(\bm y)- \left[\bm y+\eta_n\hat g(\bm y) \right]}
	&\leq \norm{\bm X(\bm y)-\left[\bm y+\eta_n\hat g(\bm y)\right]} \\
	&\stackrel{\text{(ii)}}{\leq} \zeta_n + \eta_n\norm{\hat{g}-\nabla p}_{\infty} + C_2 \eta_n^2 \norm{\nabla p(\bm{y})},
\end{align*}
where (ii) follows from \eqref{ARI_exact_flow}. Therefore,
\begin{equation}
	\label{ARI_update_representation}
	\hat\Phi_n(\bm y)
	=
	\bm y+\eta_n\nabla p(\bm y)+\bm R_n(\bm y),
\end{equation}
where 
$$\norm{\bm R_n(\bm y)}
\leq
C_2\left[q_n+\eta_n\norm{\hat{g}-\nabla p}_{\infty}+ \eta_n^2\norm{\nabla p(\bm y)}\right]\leq C_3\eta_n\left[\eta_n + \frac{q_n}{\eta_n} + \norm{\hat{g}-\nabla p}_{\infty}\right].$$

{\bf Step 2: Discrete basin invariance away from a shrinking tube.} By Lemma~\ref{lem:sep_repulsion} and \eqref{ARI_update_representation}, there exists a distance-like function $H_a: \mathcal{S}_{\rm part}^{r_0} \cap \mathcal{C}_a \to \mathbb{R}$ such that $H_a(\bm{y}) \geq A_2 \left[\eta_n + \frac{q_n}{\eta_n} + \norm{\hat{g}-\nabla p}_{\infty}\right]$, where $A_2>0$ will be chosen sufficiently large, and 
\begin{align*}
	H_a(\hat\Phi_n(\bm y))
	&= H_a(\bm y) +\eta_n\nabla H_a(\bm y)^T\nabla p(\bm y) +\nabla H_a(\bm y)^T\bm R_n(\bm y) +O\left(\norm{\eta_n\nabla p(\bm y)+\bm R_n(\bm y)}^2\right)\\
	&\stackrel{\text{(iii)}}{\geq} (1+C_0\eta_n) H_a(\bm{y}) - C_4 \eta_n \left[\eta_n + \frac{q_n}{\eta_n} + \norm{\hat{g}-\nabla p}_{\infty}\right] \\
	&\stackrel{\text{(iv)}}{\geq} A_2 \left[\eta_n + \frac{q_n}{\eta_n} + \norm{\hat{g}-\nabla p}_{\infty}\right],
\end{align*}
where (iii) uses Lemma~\ref{lem:sep_repulsion} and the bound on $\norm{\bm{R}_n(\bm{y})}$ in \eqref{ARI_update_representation}, while (iv) follows by choosing $A_2>0$ to be sufficiently large. Thus, the one-step GGDPC update cannot cross the separatrix when it lies in the boundary neighborhoods and remain at distance of order $\eta_n + \frac{q_n}{\eta_n} + \norm{\hat{g}-\nabla p}_{\infty}$ and larger. Since \eqref{ARI_update_representation} holds when $d(\bm{y},\partial \mathcal{C}_a) \geq r_0$, this non-crossing property holds in the entire GGDPC path. In other word, there exists a fixed constant $A_3>0$ such that every iteration $\hat{\bm{Y}}_k^{(n)}$ in \eqref{sample_ggdpc_path} of the truncated GGDPC path lies in $\mathcal{C}_a$ whenever the initial point $\bm{x}$ satisfies $d(\bm{x},\partial \mathcal{C}_a) > A_3\left[\eta_n + \frac{q_n}{\eta_n} + \norm{\hat{g}-\nabla p}_{\infty}\right]$ with probability tending to one.

Now, since $0< \lambda < \min_{2\leq j \leq |\mathcal{M}|} \psi_j$, \autoref{thm:modal_const} implies that $\hat{\mathcal{M}}_{\lambda}=\left\{\hat{\bm m}_{1,n},...,\hat{\bm m}_{|\mathcal{M}|,n}\right\}$ with $\hat{\bm m}_{a,n} \in \argmax_{\bm X_i\in\mathbb X_n\cap B(\bm m_a,r_a)} \hat p(\bm X_i)$ for some small radius $r_a>0$ and $\hat{\bm m}_{a,n} \in \mathcal{C}_a$. Then, for any observation $\bm{X}_i \in \mathcal{C}_a$ satisfying $d(\bm X_i,\mathcal S_{\rm full}) > A_3\left[\eta_n + \frac{q_n}{\eta_n} + \norm{\hat{g}-\nabla p}_{\infty}\right]$, the truncated GGDPC path visiting $\bm{X}_i$ remains in $\mathcal{C}_a$, so
$$\hat{cl}_n(\bm{X}_i) = cl(\bm{X}_i)$$
with probability tending to one, where $\hat{cl}_n$ denotes the GGDPC cluster label after pairing $\hat{\bm{m}}_{a,n}$ with $\bm{m}_a$ as in \autoref{thm:modal_const}.

{\bf Step 3: Controlling the fraction of mis-clustered observations.}
Based on the result in Step 2, on an event whose probability tends to
one,
\begin{equation}
	\label{ARI_misclassified_tube}
	M_n:=\sum_{i=1}^n \mathds{1}
	\left\{ \hat{cl}_n(\bm X_i)\neq cl(\bm X_i)\right\} \leq \sum_{i=1}^n
	\mathds{1}\left\{d(\bm X_i,\mathcal S_{\rm full}) \leq A_3\left[\eta_n + \frac{q_n}{\eta_n} + \norm{\hat{g}-\nabla p}_{\infty}\right]\right\}.
\end{equation}

Since $\mathrm{Leb}
\left\{\bm x\in\mathcal C: d(\bm x,\mathcal S_{\rm full})\leq r
\right\} \leq C r$, we have that
\begin{align*}
	&\mathbb{E}\left[\frac{1}{n}\sum_{i=1}^n
	\mathds{1}\left\{d(\bm X_i,\mathcal S_{\rm full}) \leq A_3\left[\eta_n + \frac{q_n}{\eta_n} + \norm{\hat{g}-\nabla p}_{\infty}\right]\right\}\right] \\
	&= \mathbb{P}\left(d(\bm{X},\mathcal{S}_{\rm full}) \leq A_3 \left[\eta_n + \frac{q_n}{\eta_n} + \norm{\hat{g}-\nabla p}_{\infty}\right]\right) \\
	&\leq C_5 A_3 \left[\eta_n + \frac{q_n}{\eta_n} + \norm{\hat{g}-\nabla p}_{\infty}\right].
\end{align*}
By Markov's inequality and \eqref{ARI_misclassified_tube},
\begin{equation}
	\label{ARI_pointwise_error}
	\frac{M_n}{n} =O_P\left(\eta_n + \frac{q_n}{\eta_n} + \norm{\hat{g}-\nabla p}_{\infty}\right).
\end{equation}

{\bf Step 4: Conversion from observation-wise error to ARI.}
Let $N=\binom{n}{2}$ and use the notation $N_{tp},N_{tn},N_{fp},N_{fn}$ from \eqref{ARI}. Notice that after pairing the estimated clusters with their corresponding population modal basins, every disagreeing pair
contains at least one mis-clustered observation. Consequently,
$$N_{fp}+N_{fn} \leq M_n(n-M_n)+{\binom{M_n}{2}} \leq M_n(n-1).$$
It follows from \eqref{ARI_pointwise_error} that
\begin{equation}
	\label{ARI_pair_rate}
	\frac{N_{fp}+N_{fn}}{N}=O_P\left(\eta_n + \frac{q_n}{\eta_n} + \norm{\hat{g}-\nabla p}_{\infty}\right).
\end{equation}
Let $A_n:=N_{tp}+N_{fn}$ and $B_n:=N_{tp}+N_{fp}$
so that $A_n$ and $B_n$ are the numbers of pairs assigned to the same cluster under the population and GGDPC partitions, respectively. 
A direct algebraic rearrangement of the ARI definition \eqref{ARI} gives
\begin{equation}
	\label{ARI_exact_identity}
	1-\mathrm{ARI} = 1- \frac{2(N_{tp} N_{tn} - N_{fp} N_{fn})}{A_n(N-B_n) + B_n(N-A_n)} = \frac{N(N_{fp}+N_{fn})}{A_n(N-B_n) + B_n(N-A_n)}.
\end{equation}

Since each modal basin contains an open neighborhood of its mode and
$p\geq p_{\min}>0$ on $\mathcal C$ and $|\mathcal{M}|\geq 2$, we know that $P(\bm X\in\mathcal C_a) \in (0,1)$.
By the law of large numbers, 
$$\frac{A_n}{N} \stackrel{P}{\to} \sum_{a=1}^{|\mathcal{M}|} P(\bm X\in\mathcal C_a)^2.$$
Furthermore, $|A_n-B_n|= |N_{fn}-N_{fp}| \leq N_{fp}+N_{fn}$,
so \eqref{ARI_pair_rate} and $\eta_n + \frac{q_n}{\eta_n} + \norm{\hat{g}-\nabla p}_{\infty} \to 0$ imply that
$$\frac{B_n}{N} \stackrel{P}{\to} \sum_{a=1}^{|\mathcal{M}|} P(\bm X\in\mathcal C_a)^2.$$
Therefore,
\begin{align*}
	&\frac{A_n}{N}\left(1-\frac{B_n}{N}\right) +
	\frac{B_n}{N} \left(1-\frac{A_n}{N}\right)\stackrel{P}{\to}
	2\left[\sum_{a=1}^{|\mathcal{M}|} P(\bm X\in\mathcal C_a)^2\right]\left[1 - \sum_{a=1}^{|\mathcal{M}|} P(\bm X\in\mathcal C_a)^2\right]>0.
\end{align*}
Combining this with \eqref{ARI_exact_identity} and
\eqref{ARI_pair_rate} yields
$$1- \operatorname{ARI} \left(\hat{\mathcal P}_{n,\lambda}, \mathcal P_n^*\right) =O_P\left(\frac{N_{fp}+N_{fn}}{N}\right)=
O_P\left(\eta_n + \frac{q_n}{\eta_n} + \norm{\hat{g}-\nabla p}_{\infty}\right).$$
This completes the proof.
\end{proof}

\section{Proof of Theorem~\ref{thm:conv_ggdpc_dendro}}
\label{app:proof_conv_ggdpc_dendro}

Before proving Proposition~\ref{prop:ggdpc_dendro} and \autoref{thm:conv_ggdpc_dendro}, we introduce the notation and definitions needed for the arguments. We also establish a stability lemma that bounds the discrepancy between the 1NN uphill update $\hat{\bm{z}}_{j,n}=\hat{\Phi}_n(\hat{\bm{m}}_{j,n})$ from the empirical local mode $\hat{\bm{m}}_{j,n}$ and the projection $\bm{z}_j=\Pi_{\overline{\mathcal{U}}_j}(\bm{m}_j)$ of the associated population local mode $\bm{m}_j$.

An \emph{ultrametric} on a set $\mathbb{X}$ is a function
$u:\mathbb{X}\times \mathbb{X}\to\mathbb{R}_+$ satisfying, for all
$\bm{x},\bm{x}',\bm{x}''\in\mathbb{X}$,
\begin{enumerate}[label=(\roman*)]
	\item $u(\bm{x},\bm{x}')=0$ if and only if $\bm{x}=\bm{x}'$;
	\item $u(\bm{x},\bm{x}')=u(\bm{x}',\bm{x})$;
	\item $u(\bm{x},\bm{x}') \leq \max\left\{ u(\bm{x},\bm{x}''), u(\bm{x}'',\bm{x}')\right\}$.
\end{enumerate}
The last property is known as the \emph{strong triangle inequality} and implies the usual triangle inequality. Hence, every ultrametric is a metric.

Given two nonempty sets $A$ and $B$, a subset $R \in A\times B$ is called a \emph{correspondence} between $A$ and $B$ if both coordinate projections are surjective. That is, (i) for any $a\in A$, there exists $b\in B$ such that $(a,b)\in R$, and (ii) for any $b\in B$, there exists $a\in A$ such that $(a,b)\in R$. Let $\mathcal{R}(A,B)$ denote the collection of all possible correspondences between $A$ and $B$; see Section 5.1 in \cite{carlsson2010characterization} for some examples of correspondences. 

\begin{definition}
	\label{def:gromov_hausdorff}
	For two compact metric spaces $(\mathbb{X}, d_{\mathbb{X}})$ and $(\mathbb{Y}, d_{\mathbb{Y}})$, their \emph{Gromov-Hausdorff} distance is defined by
	\begin{equation}
		\label{Gromov_Hausdorff}
		\mathrm{GH}\left((\mathbb{X}, d_{\mathbb{X}}),(\mathbb{Y}, d_{\mathbb{Y}})\right) := \frac{1}{2} \inf_{R\in \mathcal{R}(\mathbb{X},\mathbb{Y})} \sup_{(\bm{x},\bm{y}), (\bm{x}',\bm{y}')\in R}\left|d_{\mathbb{X}}(\bm{x},\bm{x}') - d_{\mathbb{Y}}(\bm{y},\bm{y}') \right|.
	\end{equation}
\end{definition}

We next specialize this definition to the finite ultrametric spaces
associated with two dendrograms $(\mathcal T_G,\mathbb X_n)$ and $(\mathcal T_{G'},\mathbb Y_m)$. Let $u_{\mathbb{X}_n}$ and $u_{\mathbb{Y}_m}$ denote their associated ultrametrics. For mappings  $f:\mathbb{X}_n \to \mathbb{Y}_m$ and $g: \mathbb{Y}_m \to \mathbb{X}_n$, we define the \emph{distortions} of $f$ and $g$ by
\begin{align}
	\label{distortion}
	\begin{split}
		&\mathrm{dis}(f):= \max_{\bm{X}_i,\bm{X}_j\in \mathbb{X}_n} \left|u_{\mathbb{X}_n}\left(\bm{X}_i,\bm{X}_j\right) - u_{\mathbb{Y}_m}\left(f(\bm{X}_i),f(\bm{X}_j)\right) \right|, \\
		&\mathrm{dis}(g):= \max_{\bm{Y}_i,\bm{Y}_j\in \mathbb{Y}_m} \left|u_{\mathbb{X}_n}\left(g(\bm{Y}_i),g(\bm{Y}_j)\right) - u_{\mathbb{Y}_m}\left(\bm{Y}_i,\bm{Y}_j\right) \right|.
	\end{split}
\end{align}
Moreover, the joint distortion of $f$ and $g$ can be defined by
\begin{equation}
	\label{joint_distortion}
	\mathrm{dis}(f,g):= \max_{\bm{X}_i\in \mathbb{X}_n, \bm{Y}_i\in \mathbb{Y}_m} \left|u_{\mathbb{X}_n}\left(\bm{X}_i,g(\bm{Y}_i)\right) - u_{\mathbb{Y}_m}\left(\bm{Y}_i,f(\bm{X}_i)\right) \right|.
\end{equation}
The Gromov-Hausdorff distance between $(\mathbb{X}_n, u_{\mathbb{X}_n})$ and $(\mathbb{Y}_m, u_{\mathbb{Y}_m})$ admits the equivalent representation
\begin{equation}
	\label{Gromov_Hausdorff2}
	\mathrm{GH}\left((\mathbb{X}_n, u_{\mathbb{X}_n}),(\mathbb{Y}_m, u_{\mathbb{Y}_m})\right) = \frac{1}{2} \min_{f,g} \max\left\{\mathrm{dis}(f), \mathrm{dis}(g), \mathrm{dis}(f,g)\right\},
\end{equation}
where the minimum is taken over all mappings $f:\mathbb X_n\to\mathbb Y_m$ and $g:\mathbb Y_m\to\mathbb X_n$.
The joint distortion $\mathrm{dis}(f,g)$ controls the compatibility of the mappings $f,g$, and in particular, penalizes deviations from an approximate inverse relationship between them.

\subsection{Proof of Proposition~\ref{prop:ggdpc_dendro}}
\label{app:ggdpc_single_link}

\begin{proof}[Proof of Proposition~\ref{prop:ggdpc_dendro}]
First, when $\lambda \geq d_{\max} =\max_{1\leq i,j \leq n} \norm{\bm{X}_i - \bm{X}_j}$, $D_{ij} \leq \lambda$ and every pair $(\bm{X}_i,\bm{X}_j)$ is connected. Moreover, since every (directed) edge of the GGDPC graph $G$ has a weight at most $d_{\max}$, every edge is retained in $G_{\lambda}$, so $G_{\lambda}=G$ is also connected. Therefore, both procedures return the single cluster $\mathbb{X}_n$.

Second, when $0\leq \lambda < d_{\max}$, if $\bm{X}_i,\bm{X}_j$ lie in the same cluster in $\mathcal{T}_{\rm SL}(\lambda)$, then there exists a collection of observations $\bm{X}_{i_1}=\bm{X}_i,...,\bm{X}_{i_k}=\bm{X}_j$ such that $\norm{\bm{X}_{i_\alpha} - \bm{X}_{i_{\alpha+1}}} \leq \lambda$ for all $\alpha=1,...,k-1$. By the definition of the (undirected) GGDPC graph $G$, $\bm{X}_{i_\alpha}$ and $\bm{X}_{i_{\alpha+1}}$ are connected in $G$ so that $\bm{X}_i,\bm{X}_j$ are in the same connected component of $G_{\lambda}$. On the other hand, if $\bm{X}_i,\bm{X}_j$ lie in the same cluster in $\mathcal{T}_G(\lambda)$, there is a path $\bm{X}_i \to \cdots \to \bm{X}_{i_k}$ (not necessarily directed) connecting $\bm{X}_i,\bm{X}_j$ in $G$, and $\norm{\bm{X}_{i_\alpha} - \bm{X}_{i_{\alpha+1}}} \leq \lambda$ for all $\alpha=1,...,k-1$. Hence, $\bm{X}_i,\bm{X}_j$ are in the same cluster for the single linkage clustering at threshold $\lambda$.

In summary, $\mathcal{T}_{\rm SL}(\lambda)=\mathcal{T}_G(\lambda)$ for every $\lambda\geq 0$.
\end{proof}

\subsection{A Stability Lemma of the Empirical Modal Projection}
\label{app:lem_stab_emp_proj}

\begin{lemma}[Stability of the empirical modal projection]
\label{lem:emp_mode_proj}
Suppose that Assumptions~\ref{assump:density_reg},
\ref{assump:sample_den_diff}, and
\ref{assump:unique_proj} hold. Assume further that $\eta_n=o(1)$, $\norm{\hat{p}-p}_{\infty}=o_P(1)$, $\norm{\hat{g} - \nabla p}_{\infty} = o_P(1)$, and $\norm{\nabla^2\hat{p} - \nabla^2 p}_{\infty} =o_P(1)$. Let $\hat{\bm{m}}_{j,n}$ be the sample local mode associated with $\bm{m}_j$ in \autoref{thm:modal_const}, and for each non-global (sample) mode, set $\hat{\bm{z}}_{j,n} = \hat{\Phi}_n(\hat{\bm{m}}_{j,n})$. Then, for all non-global (sample) modes, we have that
$$\max_{j\geq 2} \norm{\hat{\bm{z}}_{j,n} - \bm{z}_j} = O_P\left(\left(\frac{\log n}{n}\right)^{\frac{1}{2d}} + \norm{\hat{g} - \nabla p}_{\infty} \right).$$
In particular, $\bm{z}_j$ and $\hat{\bm{z}}_{j,n}$ belong to the same basin of attraction $\mathcal{C}_{\pi(j)}$ for $j\geq 2$ with probability tending to one.
\end{lemma}

\begin{proof}[Proof of Lemma~\ref{lem:emp_mode_proj}]
Let $q_n = \left(\frac{\log n}{n}\right)^{\frac{1}{d}}$.
By Assumption~\ref{assump:density_reg}, $|\mathcal{M}|$ is finite, and all the local modes have non-degenerate Hessian matrices. Then, we can choose a fixed radius $r_j >0$ for each mode $\bm{m}_j \in \mathcal{M}$ such that $2r_j < \psi_j$, $\overline{B(\bm{m}_j, r_j)} \subset \mathcal{C}_j$, and the balls $B(\bm{m}_j,r_j), j=1,...,|\mathcal{M}|$ are pairwise disjoint. As in the proof of \autoref{thm:modal_const}, we obtain that
$$\hat{\bm{m}}_{j,n} \in\argmax_{\bm{X}_i\in \mathbb{X}_n\cap B(\bm{m}_j,r_j)} \hat{p}(\bm{X}_i)$$
with ties resolved by the fixed strict ordering. Additionally, \autoref{thm:modal_const} implies that
\begin{equation}
	\label{mode_rate}
	\max_{j=1,...,|\mathcal{M}|}\norm{\hat{\bm{m}}_{j,n} - \bm{m}_j} = O_P\left(q_n + \min\left\{\norm{\hat{p} - p}_{\infty}^{\frac{1}{2}},\, \norm{\hat{g} - \nabla p}_{\infty}\right\}\right).
\end{equation}
Now, for a given non-global (sample) mode indexed by $j$, we write $\hat{\bm{c}}_{j,n} :=\hat{\bm{m}}_{j,n} + \eta_n\hat{g}(\hat{\bm{m}}_{j,n})$ and define the closed external estimated upper-level set by
$$\widehat{F}_{j,n} :=\left\{\bm{y}\in\mathcal{C}\setminus
B(\bm{m}_j,r_j): \hat{p}(\bm{y})\geq \hat{p}(\hat{\bm{m}}_{j,n})\right\}.$$
Every observation admissible in the definition of
$\hat{\Phi}_n(\hat{\bm{m}}_{j,n})$ belongs to $\widehat{F}_{j,n}$, because
$\hat{\bm{m}}_{j,n}$ maximizes $\hat p$ over the observations in $B(\bm{m}_j,r_j)$.  Thus, $\hat{\bm z}_{j,n}$ is the minimizer of
$\norm{\bm{X}_i - \hat{\bm{c}}_{j,n}}$ over the observed points satisfying
$\hat{p}(\bm{X}_i)> \hat{p}(\hat{\bm m}_{j,n})$.

{\bf Step 1: Continuous projected point.} To obtain an upper bound for $\norm{\hat{\bm{z}}_{j,n} - \bm{z}_j}$, we first study the continuous projection of $\hat{\bm{c}}_{j,n}$ onto $\hat{F}_{j,n}$. At the population level, the optimal condition of $\bm{z}_j$ for solving 
$$\min_{\bm{y}\in \mathcal{C}\setminus B^o(\bm{m}_j,r_j)} \frac{1}{2}\norm{\bm{y}-\bm{m}_j}^2 \quad \text{ subject to } \quad p(\bm{y}) \geq p(\bm{m}_j)$$
yields that
\begin{equation}
	\label{first_order_cond}
	\bm{z}_j - \bm{m}_j -\mu_j \nabla p(\bm{z}_j)=0, \quad p(\bm{z}_j) - p(\bm{m}_j)=0.
\end{equation}
The Jacobian of this system with respect to $(\bm{z},\mu)$ is
\[
\mathcal{J}_j=
\begin{pmatrix}
	I_d - \mu_j\nabla^2 p(\bm{z}_j) & -\nabla p(\bm{z}_j)\\
	\nabla p(\bm{z}_j)^T &  0
\end{pmatrix},
\]
which is nonsingular. Indeed, if $\mathcal{J}_j\begin{pmatrix}
	\bm{v}\\
	s
\end{pmatrix} =0$, then the second row gives that $\bm{v}^T \nabla p(\bm{z}_j)=0$, while the first row gives that $\left[I_d - \mu_j \nabla^2 p(\bm{z}_j)\right]\bm{v} - \nabla p(\bm{z}_j) s=\bm{0}$. Hence, taking the inner product of the equation in the first row with $\bm{v}$ and using Assumption~\ref{assump:unique_proj}(b) lead to $\bm{v}=0$. The non-zero gradient condition $\nabla p(\bm{z}_j) \neq 0$ also implies that $s=0$.

By the conditions $\norm{\hat{g} - \nabla p}_{\infty} = o_P(1)$ and $\norm{\nabla^2\hat{p} - \nabla^2 p}_{\infty} =o_P(1)$, the implicit function theorem applies on an event with probability tending to one. Conditioning on that event, we obtain a unique local optimal solution $(\tilde{\bm{z}}_{j,n}, \tilde{\mu}_{j,n})$ so that
$$\tilde{\bm{z}}_{j,n}-\hat{\bm{c}}_{j,n} -\tilde\mu_{j,n} \nabla \hat{p}(\widetilde{\bm{z}}_{j,n})=0, \qquad \hat{p}(\widetilde{\bm{z}}_{j,n}) =\hat{p}(\hat{\bm{m}}_{j,n}).$$
Since $\nabla p(\bm{m}_j)=0$, the boundedness of $\norm{\nabla^2 p(\bm{x})}_{\max}$ for any $\bm{x}\in \mathcal{C}$ implies that
\begin{align*}
	\norm{\hat{g}(\hat{\bm{m}}_{j,n})} &= \norm{\hat{g}(\hat{\bm{m}}_{j,n}) - \nabla p(\hat{\bm{m}}_{j,n}) + \nabla p(\hat{\bm{m}}_{j,n}) - \nabla p(\bm{m}_j)} \\
	&\leq C_1\left(\norm{\hat{g} - \nabla p}_{\infty} + \norm{\hat{\bm{m}}_{j,n} - \bm{m}_j}\right)
\end{align*}
for some constant $C_1>0$. Thus, 
\begin{align*}
	\norm{\bm{z}_j -\hat{\bm{c}}_{j,n} -\mu_j \nabla \hat{p}(\bm{z}_j)} &= \norm{\underbrace{\bm{z}_j -\bm{m}_j - \mu_j \nabla p(\bm{z}_j)}_{=\bm{0}} +\bm{m}_j - \hat{\bm{c}}_{j,n} + \mu_j\left[\nabla p(\bm{z}_j) - \nabla \hat{p}(\bm{z}_j)\right]} \\
	&\leq C_2 (1+\eta_n) \left(\norm{\hat{g} - \nabla p}_{\infty} + \norm{\hat{\bm{m}}_{j,n} - \bm{m}_j}\right)
\end{align*}
for some constant $C_2>0$, where we recall the definition $\hat{\bm{c}}_{j,n} :=\hat{\bm{m}}_{j,n} + \eta_n\hat{g}(\hat{\bm{m}}_{j,n})$ in the last inequality. The residual of the equality constraint in \eqref{first_order_cond} is bounded as well, because integrating $\hat g-\nabla p$ along the segment from $\hat{\bm m}_{j,n}$ to $\bm z_j$ and using the quadratic modal expansion gives that
\[
\left|\hat p(\bm z_j)-\hat p(\hat{\bm m}_{j,n})\right|
\leq C_3\left[\norm{\hat g-\nabla p}_\infty +\norm{\hat{\bm m}_{j,n}-\bm m_j}^2\right]
\]
for some constant $C_3>0$.
The non-singularity of $\mathcal{J}_j$ and the implicit function theorem thus implies that
\begin{align}
	\label{proj_conv_rate}
	\begin{split}
		\norm{\tilde{\bm{z}}_{j,n} - \bm{z}_j} &\lesssim (1+\eta_n) \left(\norm{\hat{g} - \nabla p}_{\infty} + \norm{\hat{\bm{m}}_{j,n} - \bm{m}_j}\right) \\
		&= O_P\left(q_n + \norm{\hat{g} - \nabla p}_{\infty}\right).
	\end{split}
\end{align}

We next verify that this local solution $\tilde{\bm{z}}_{j,n}$ is the global continuous projection $\hat{\bm{c}}_{j,n}$ onto $\hat{F}_{j,n}$. If not, along a subsequence, there would exist global minimizers
$\bm{y}_n\in \hat{F}_{j,n}$ outside a fixed neighborhood of $\bm{z}_j$
with $\norm{\bm{y}_n-\hat{\bm{c}}_{j,n}} \leq \norm{\widetilde{\bm{z}}_{j,n}-\hat{\bm{c}}_{j,n}}$. By compactness, pass to a further subsequence with $\bm{y}_n\to \bm{y}$. Uniform convergence of $\hat{p}$ and
$\hat{p}(\hat{\bm{m}}_{j,n})\to p(\bm{m}_j)$ imply
$$\bm{y}\in \mathcal{C}\setminus B^o(\bm{m}_j,r_j), \qquad p(\bm{y})\geq p(\bm{m}_j).$$
Moreover, $\hat{\bm{c}}_{j,n}\to\bm{m}_j$ and
\eqref{proj_conv_rate} imply $\norm{\bm{y}-\bm{m}_j} \leq \psi_j$. By the uniqueness of $\bm{z}_j$, it has to be $\bm{y}=\bm{z}_j$, leading to a
contradiction. Thus, $\tilde{\bm{z}}_{j,n}$ is the global continuous projection with probability tending to one.

{\bf Step 2: Uniform quadratic growth.} The second-order condition in Assumption~\ref{assump:unique_proj}(b) implies that for a fixed neighborhood $W_j$ of $\bm{z}_j$ and some constant $A_j>0$,
\begin{equation}
	\label{quadratic_growth}
	\norm{\bm{y}-\hat{\bm{c}}_{j,n}}^2 -\norm{\tilde{\bm{z}}_{j,n}-\hat{\bm{c}}_{j,n}}^2 \geq A_j\norm{\bm{y}-\tilde{\bm{z}}_{j,n}}^2
\end{equation}
for every $\bm{y}\in\widehat F_{j,n}\cap W_j$, simultaneously for all $j\geq 2$. Indeed, for the constrained minimization $\bm{y}\mapsto \frac{1}{2}\norm{\bm{y}-\hat{\bm{c}}_{j,n}}^2$ subject to $\hat{p}(\bm{y})\geq \hat{p}(\hat{\bm{m}}_{j,n})$, its Hessian matrix is uniformly positive definite along both the tangential direction $\bm{v}$ and the normal inward direction $\bm{s}$ in a fixed neighborhood $W_j$ of $\bm{z}_j$, so that $\norm{\bm{y}-\hat{\bm{c}}_{j,n}}^2 -\norm{\tilde{\bm{z}}_{j,n}-\hat{\bm{c}}_{j,n}}^2 \gtrsim \norm{\bm{v} + \bm{s}}^2$ for $\bm{y}\in \hat{F}_{j,n}\cap W_j$. Moreover, $\norm{\bm{y}-\tilde{\bm{z}}_{j,n}} \leq \norm{\bm{v} + \bm{s}}$ under its associated tangential and normal decompositions.

By Assumption~\ref{assump:sample_den_diff} and the fact that $\norm{\nabla \hat{p}(\tilde{\bm{z}}_{j,n})}$ is bounded away from 0 with probability tending to one, we can choose a sufficiently large fixed $C_4>0$ and set $\bm{y}_{j,n} = \tilde{\bm{z}}_{j,n} + C_4q_n \cdot \frac{\nabla \hat{p}(\tilde{\bm{z}}_{j,n})}{\norm{\nabla \hat{p}(\tilde{\bm{z}}_{j,n})}}$. For all sufficiently large $n$, $\bm{y}_{j,n} \in \mathcal{C}$. By the definition of $\zeta_n$ in \eqref{max_1nn_dist}, there exists an observation $\bm{X}_{j,n}$ with $\norm{\bm{X}_{j,n} - \bm{y}_{j,n}} \leq \zeta_n =O_P(q_n)$. By Taylor's expansion of $\hat{p}$ at $\tilde{\bm{z}}_{j,n}$, together with the uniform $O_P(1)$ Hessian bound, we know that
$$\hat{p}(\bm{X}_{j,n}) > \hat{p}(\tilde{\bm{z}}_{j,n})\geq \hat{p}(\hat{\bm{m}}_{j,n})$$
with probability tending to one; see also Lemma~\ref{lem:pop_1nn_bnd}. Thus, $\bm X_{j,n}$ is admissible and $\norm{\bm{X}_{j,n} - \tilde{\bm{z}}_{j,n}}= O_P(q_n)$. By the defining minimality of $\hat{\bm{z}}_{j,n} = \hat{\Phi}_n(\hat{\bm{m}}_{j,n})$,
$$\norm{\hat{\bm{z}}_{j,n} - \hat{\bm{c}}_{j,n}}^2 \leq \norm{\bm{X}_{j,n}-\hat{\bm{c}}_{j,n}}^2 \leq \norm{\tilde{\bm{z}}_{j,n}-\hat{\bm{c}}_{j,n}}^2 + O_P\left(q_n\right).$$
Hence, $\hat{\bm{z}}_{j,n}\in W_j$ with probability tending to one, and \eqref{quadratic_growth} then gives that
$$\norm{\hat{\bm{z}}_{j,n} - \tilde{\bm{z}}_{j,n}} = O_P(\sqrt{q_n}).$$
Together with \eqref{proj_conv_rate} under the triangle's inequality, this proves the asserted rate. 

Finally, Assumption~\ref{assump:unique_proj}(a) and finiteness
of the mode set imply $\min_{j\geq 2} d(\bm{z}_j,\partial\mathcal{C}_{\pi(j)})>0.$
Since $\max_{j\geq 2} \norm{\hat{\bm{z}}_{j,n} - \bm{z}_j}=o_P(1)$, we conclude that $\bm{z}_j,\hat{\bm{z}}_{j,n}$ belong to the same parent basin of attraction $\mathcal{C}_{\pi(j)}$ for $j\geq 2$ with probability tending to one.
\end{proof}

\subsection{Main Proof of \autoref{thm:conv_ggdpc_dendro}}

\begin{proof}[Proof of \autoref{thm:conv_ggdpc_dendro}]
For each $j\geq 2$, we write $\bm{z}_j:=\Pi_{\overline{\mathcal{U}}_j}(\bm{m}_j)$ and $\psi_j=\norm{\bm{m}_j-\bm{z}_j}$. Throughout the proof, $\zeta_n$ denotes the global coverage radius
$\sup_{\bm x\in\mathcal C}\min_i\norm{\bm x-\bm X_i}=O_P(q_n)$ with $q_n=\left(\frac{\log n}{n}\right)^{\frac{1}{d}}$. The proof is divided into four steps.

{\bf Step 1: Identification of the empirical mode tree.} If $|\mathcal{M}|\geq 2$, we let $\psi_{\min} = \min_{2\leq j\leq |\mathcal{M}|} \psi_j >0$ and fix $\lambda_0 \in \left(0, \frac{\psi_{\min}}{2}\right)$. If $|\mathcal{M}|=1$, then we fix any constant $\lambda_0 >0$. In either case, choose a radius $r< \frac{\lambda_0}{2}$.
Recall from the proof of \autoref{thm:modal_const} that we define 
$$\hat{\bm{m}}_{j,n}  \in \argmax_{\bm{X}_i\in \mathbb{X}_n\cap B(\bm{m}_j,r)} \hat{p}(\bm{X}_i)$$ 
as the sample local mode associated with $\bm{m}_j$. Let $\hat{\bm{z}}_{j,n} = \hat{\Phi}_n(\hat{\bm{m}}_{j,n})$. By \autoref{thm:modal_const} with probability tending to one, besides the root $\hat{\bm{m}}_{1,n}$, the vertices whose outgoing edges have weight larger than $\lambda_0$ are exactly $\hat{\bm{m}}_{2,n},...,\hat{\bm{m}}_{|\mathcal{M}|,n}$. Since the undirected version of $G$ is a tree, deleting the $(|\mathcal{M}|-1)$ edges $\hat{\bm{m}}_{j,n}\to \hat{\bm{z}}_{j,n}$ for $j=2,...,|\mathcal{M}|$ produces exactly $|\mathcal{M}|$ connected components. We label them as $V_{1,n},...,V_{|\mathcal{M}|,n}$ so that $\hat{\bm{m}}_{j,n} \in V_{j,n}$ for $j=1,...,|\mathcal{M}|$. 

By Lemma~\ref{lem:emp_mode_proj}, $\max_{j\geq 2} \norm{\hat{\bm{z}}_{j,n} - \bm{z}_j} = O_P\left(\sqrt{q_n} + \norm{\hat{g} - \nabla p}_{\infty} \right)$. Additionally, by Assumption~\ref{assump:unique_proj}(a), $d(\bm{z}_j, \partial \mathcal{C}_{\pi(j)}) >0$. Hence, we can choose a fixed closed ball $\bar{B}_j$ centered at $\bm{z}_j$ and contained in $\mathcal{C}_{\pi(j)}$. With probability tending to one, $\hat{\bm{z}}_{j,n} \in \bar{B}_j$ simultaneously for all $j\geq 2$.

By \eqref{sample_regular_tracking_full} for a fixed $\delta_n>0$,
$$\sup_{0\leq k\leq T_j/\eta_n} \norm{\hat{\bm{Y}}^{(n)}_k-\bm{\gamma}_{\hat{\bm z}_{j,n}}(k\eta_n)} =O_P\left(\frac{q_n}{\eta_n}+ \norm{\hat{g} - \nabla p}_{\infty}+ \eta_n\right)=o_P(\delta_n)=o_P(1),$$
where $q_n=\left(\frac{\log n}{n}\right)^{\frac{1}{d}}$ comes from the term $\zeta_n=O_P(q_n)$ in \eqref{max_1nn_dist}, $\left\{\hat{\bm{Y}}^{(n)}_k\right\}_{k\geq 0}$ is the GGDPC path starting at $\hat{\bm{z}}_{j,n}$, $t\mapsto \bm{\gamma}_{\hat{\bm{z}}_{j,n}}(t)$ is the gradient ascent flow starting at $\hat{\bm{z}}_{j,n}$, and $T_j<\infty$ is the stopping time when $\bm{\gamma}_{\bm{y}}(T_j) \in B(\bm{m}_{\pi(j)}, r_j)$ for all $\bm{y}\in \bar{B}_j$. Here, $r_j>0$ is chosen so that $B(\bm{m}_{\pi(j)}, r_j)$ lies in $\mathcal{C}_{\pi(j)}$.

Thus, with probability tending to one, every edge used by the GGDPC path starting from $\hat{\bm{z}}_{j,n}$ reaches
$\hat{\bm{m}}_{\pi(j),n}$ without using any edge $\hat{\bm{m}}_{j,n} \to \hat{\bm{z}}_{j,n}$ for $j=2,...,|\mathcal{M}|$. Consequently, 
\begin{equation}
	\label{empirical_parent_dendro}
	\hat{\bm z}_{j,n}\in V_{\pi(j),n}, \quad\text{ for } j=2,...,|\mathcal{M}|,
\end{equation}
with probability tending to one. Contracting each $V_{j,n}$ to one vertex turns the empirical tree into exactly the population mode tree, and the edge $\hat{\bm m}_{j,n} \to \hat{\bm z}_{j,n}$ becomes the edge
$\bm{m}_j \to \bm{m}_{\pi(j)}$.

\textbf{Step 2: Convergence of the modal edge weights.} Let $\hat{\psi}_{j,n} = \norm{\hat{\bm{m}}_{j,n} - \hat{\bm{z}}_{j,n}}$ and $\psi_j = \norm{\bm{m}_j - \bm{z}_j}$ for $j=2,...,|\mathcal{M}|$. By triangle's inequality,
\begin{align*}
	\max_{j=2,...,|\mathcal{M}|} \left|\hat{\psi}_{j,n} - \psi_j\right| &\leq \max_{j=2,...,|\mathcal{M}|} \norm{\hat{\bm{m}}_{j,n} - \bm{m}_j} + \max_{j=2,...,|\mathcal{M}|} \norm{\hat{\bm{z}}_{j,n} - \bm{z}_j} \\
	&\stackrel{\text{(i)}}{=} O_P\left(\sqrt{q_n} + \norm{\hat{g} - \nabla p}_{\infty}\right),
\end{align*}
where (i) follows from \autoref{thm:modal_const} and Lemma~\ref{lem:emp_mode_proj}.

\textbf{Step 3: Uniform collapse of all non-modal edges.} Let $\epsilon_n = a_n\left(\frac{q_n}{\eta_n} + \norm{\hat{g}-\nabla p}_{\infty}\right)$ with $a_n\to \infty$ arbitrarily slowly. For any $\bm{X}_i \in \mathbb{X}_n \setminus \left\{\hat{\bm{m}}_{1,n},...,\hat{\bm{m}}_{|\mathcal{M}|,n}\right\}$, we consider three cases. 

$\bullet$ \emph{Case I:} $\bm{X}_i$ is at least $\epsilon_n$-distance away from all the critical points of $p$. Then, by Lemma~\ref{lem:mode_geometry} and $\norm{\hat{g}-\nabla p}_{\infty}=o_P(1)$, we know that $\norm{\nabla p(\bm{X}_i)} \gtrsim \epsilon_n$ and $\norm{\nabla \hat{p}(\bm{X}_i)} \gtrsim \epsilon_n - \norm{\hat{g}-\nabla p}_{\infty} \gtrsim \epsilon_n$. Let $\bm z_i=\bm{\gamma}_{\bm X_i}(\eta_n)\in\mathcal C$, and we choose an observation $\bm Y_i$ with $\norm{\bm Y_i-\bm z_i}\leq \zeta_n$. Since $\frac{q_n}{\eta_n}+\norm{\hat g-\nabla p}_\infty=o_P(\epsilon_n)$ by the definition of $\epsilon_n$, the exact-flow expansion yields that
$$\bm z_i=\bm X_i+\eta_n\nabla p(\bm X_i)+O\left(\eta_n^2\norm{\nabla p(\bm X_i)}\right).$$
By Taylor's expansion of $\hat p$, we know that
\begin{align*}
\hat p(\bm Y_i) - \hat p(\bm X_i) &\geq \nabla \hat{p}(\bm{X}_i)^T\left[\bm{Y}_i -\bm{z}_i + \eta_n\nabla p(\bm X_i)+O\left(\eta_n^2\norm{\nabla p(\bm X_i)}\right)\right] - o(\left[\epsilon_n\cdot \eta_n + q_n\right]^2)\\
&\gtrsim \norm{\nabla p(\bm{X}_i)}^2 \eta_n\left[1 + \frac{q_n}{\eta_n}+\norm{\hat g-\nabla p}_\infty\right] > 0
\end{align*}
uniformly with probability tending to one. Hence, $\bm Y_i$ is admissible and
$$\hat{w}_n(\bm{X}_i) \leq \eta_n\norm{\hat{g}}_{\infty} + C_1 \zeta_n = O_P\left(\eta_n + q_n\right)$$
for some absolute constant $C_1 >0$.

$\bullet$ \emph{Case II:} $\bm{X}_i\in B(\bm{m}_j,\epsilon_n)$ but $\bm{X}_i\neq \hat{\bm{m}}_{j,n}$ for some $\bm{m}_j \in \mathcal{M}$. Since $\hat{p}(\hat{\bm{m}}_{j,n}) > \hat{p}(\bm{X}_i)$ under the fixed tie-breaking ordering in the proof of \autoref{thm:modal_const}, $\hat{\bm{m}}_{j,n}$ is admissible for the GGDPC update $\hat{\Phi}_n(\bm{X}_i)$, and
$$\hat{w}_n(\bm{X}_i) \leq \norm{\hat{\bm m}_{j,n}-\bm X_i}
+2\eta_n\norm{\hat g(\bm X_i)} = O_P\left(\epsilon_n + \eta_n\right).$$

$\bullet$ \emph{Case III:} $\bm{X}_i\in B(\bm{s},\epsilon_n)$ for some non-modal critical point $\bm{s}$ of the density $p$. Since $\bm s$ is non-modal and $\nabla^2 p(\bm{s})$ is nonsingular, it has a positive eigenvalue. Let $\bm{v}_{\bm{s}}$ be a corresponding unit eigenvector. By the condition $\norm{\nabla^2 \hat{p} - \nabla^2 p}_{\infty}=o_P(1)$ under Assumption~\ref{assump:sample_den_diff}, there exist fixed constants $r_{\bm{s}}>0$ and $\rho_{\bm{s}}>0$ such that, with probability tending to one,
\begin{equation}
	\label{positive_hessian_direction_dendro}
	\bm{v}_{\bm s}^{T}\nabla^2\hat p(\bm y)\bm{v}_{\bm s} \geq \rho_{\bm s}
\end{equation}
uniformly over $\bm y\in B(\bm{s},r_{\bm s})$. Additionally, as shown in Lemma~\ref{lem:pop_1nn_tube}, $\sup_{\bm{y}\in B(\bm{s},\epsilon_n)} \norm{\hat{g}(\bm{y})} = O_P(\epsilon_n)$. 

Now, we choose a sufficiently large fixed constant $C_{\bm s}>0$ and set $\bm y_i:=\bm X_i+C_{\bm s}\epsilon_n\bm v_{\bm s}$. For all sufficiently large $n$, the entire segment joining $\bm X_i$ and $\bm y_i$ lies in $B(\bm s,r_{\bm s})$. By Taylor's expansion and \eqref{positive_hessian_direction_dendro},
\begin{align}
	\label{continuous_uphill_rate_dendro}
	\begin{split}
		\hat p(\bm y_i)-\hat p(\bm X_i) &= C_{\bm s}\epsilon_n
		\bm v_{\bm s}^{T}\hat g(\bm X_i) +\frac{C_{\bm s}^2\epsilon_n^2}{2}
		\bm v_{\bm s}^{T}\nabla^2\hat p(\bm{X}_i)\bm v_{\bm s} + o_P(\epsilon_n^2)\\
		&\geq -C_2 C_{\bm s}\epsilon_n^2
		+\frac{\rho_{\bm s}C_{\bm s}^2}{2}\epsilon_n^2\\
		&\stackrel{\text{(ii)}}{\geq} C_3 \epsilon_n^2
	\end{split}
\end{align}
for some absolute constants $C_2,C_3 >0$, where (ii) follows by choosing $C_{\bm s}>0$ sufficiently large.

By the definition of $\zeta_n$, there is an observation $\bm{Y}_i$ with
$\norm{\bm{Y}_i-\bm{y}_i} \leq \zeta_n$. By Taylor's expansion, boundedness of
$\nabla^2 \hat{p}$, and $\sup_{\bm{y}\in B(\bm{s},\epsilon_n)} \norm{\hat{g}(\bm{y})} = O_P(\epsilon_n)$, we know that
$$\hat{p}(\bm Y_i) - \hat{p}(\bm{y}_i) \geq -C_4\epsilon_n\zeta_n -C_4\zeta_n^2 =o_P(\epsilon_n^2).$$
Combining this display with \eqref{continuous_uphill_rate_dendro} implies that $\hat{p}(\bm{Y}_i) > \hat{p}(\bm{X}_i)$ uniformly with probability tending to one. Thus, $\bm{Y}_i$ is admissible, and
\begin{align*}
	\hat{w}_n(\bm X_i) &\leq \norm{\bm Y_i-\bm X_i} +2\eta_n \norm{\hat g(\bm X_i)} \\
	&\leq C_{\bm s}\epsilon_n+\zeta_n +2\eta_n\norm{\hat g(\bm X_i)} =O_P(\epsilon_n)
\end{align*}
uniformly over $B(\bm{s},\epsilon_n)$.

Combining all these cases, we conclude that
$$\max_{j=1,...,|\mathcal{M}|} \max_{\bm{X}_i \in \hat{V}_{j,n}\setminus \left\{\hat{\bm{m}}_{j,n}\right\}} \hat{w}_n(\bm{X}_i) = O_P\left(\eta_n+\epsilon_n\right) = O_P\left(\eta_n + a_n\left[\frac{q_n}{\eta_n} + \norm{\hat{g}-\nabla p}_{\infty}\right]\right)$$
with $a_n\to \infty$ arbitrarily slowly.

\textbf{Step 4: Construction of a low-distortion correspondence.} On the event established in Step 1, we define the surjective map $F_n:\mathbb{X}_n\to \mathcal{M}$ with $F_n(\bm{X}_i)=\bm{m}_j$ if $\bm{X}_i \in V_{j,n}$ and construct the correspondence
$$R_n:=\left\{(\bm{X}_i,F_n(\bm{X}_i)): \bm{X}_i\in \mathbb{X}_n\right\} \subset \mathbb{X}_n \times \mathcal{M}.$$
Now, fix any $\bm{X}_i,\bm{X}_{i'}\in \mathbb{X}_n$. If $F_n(\bm{X}_i) = F_n(\bm{X}_{i'})$, then the unique undirected path joining them in $G$ contains no edges $\hat{\bm{m}}_{j,n}\to \hat{\bm{z}}_{j,n}$ for $j=2,...,|\mathcal{M}|$. By Proposition~\ref{prop:hier_cluster_tree} and \eqref{ultrametric_dendro},
$$u_{\mathbb{X}_n}(\bm{X}_i,\bm{X}_{i'}) \leq \max_{j=1,...,|\mathcal{M}|} \max_{\bm{X}_i \in \hat{V}_{j,n}\setminus \left\{\hat{\bm{m}}_{j,n}\right\}} \hat{w}_n(\bm{X}_i) = O_P\left(\eta_n + a_n\left[\frac{q_n}{\eta_n} + \norm{\hat{g}-\nabla p}_{\infty}\right]\right)$$
and
$$u_{\mathcal{M}}(F_n(\bm{X}_i), F_n(\bm{X}_{i'}))=0.$$

If $F_n(\bm{X}_i) \neq F_n(\bm{X}_{i'})$, then we know from \eqref{empirical_parent_dendro} that the modal edges on the empirical path between $\bm{X}_i$ and $\bm{X}_{i'}$ are indexed by exactly the same set, say $J(i,i')$, as the edges on the population modal tree path between $F_n(\bm{X}_i)$ and $F_n(\bm{X}_{i'})$. Hence,
$$u_{\mathbb{X}_n}(\bm{X}_i,\bm{X}_{i'}) \leq \max\left\{\max_{j=1,...,|\mathcal{M}|} \max_{\bm{X}_i \in \hat{V}_{j,n}\setminus \left\{\hat{\bm{m}}_{j,n}\right\}} \hat{w}_n(\bm{X}_i),\, \max_{j\in J(i,i')} \hat{\psi}_{j,n}  \right\}$$
and
$$u_{\mathcal{M}}(F_n(\bm{X}_i), F_n(\bm{X}_{i'})) = \max_{j\in J(i,i')} \psi_j.$$
Since all the weights are nonnegative, we know that
\begin{align*}
	&\left|u_{\mathbb{X}_n}(\bm{X}_i,\bm{X}_{i'}) - u_{\mathcal{M}}(F_n(\bm{X}_i), F_n(\bm{X}_{i'}))\right| \\
	&\leq \left|\max\left\{\max_{j=1,...,|\mathcal{M}|} \max_{\bm{X}_i \in \hat{V}_{j,n}\setminus \left\{\hat{\bm{m}}_{j,n}\right\}} \hat{w}_n(\bm{X}_i),\, \max_{j\in J(i,i')} \hat{\psi}_{j,n}  \right\} - \max_{j\in J(i,i')} \psi_j \right| \\
	&\leq \max_{j=1,...,|\mathcal{M}|} \max_{\bm{X}_i \in \hat{V}_{j,n}\setminus \left\{\hat{\bm{m}}_{j,n}\right\}} \hat{w}_n(\bm{X}_i) + \max_{j\in J(i,i')} \left|\hat{\psi}_{j,n} - \psi_j \right|\\
	&\stackrel{\text{(iii)}}{=} O_P\left(\sqrt{q_n} + \eta_n + a_n\left[\frac{q_n}{\eta_n} + \norm{\hat{g}-\nabla p}_{\infty}\right]\right),
\end{align*}
where (iii) follows from the results in Steps 2 and 3 as well as $\zeta_n = O_P(q_n)$. The result follows by noting that
\[
\begin{aligned}
	\operatorname{dis}(R_n)
	&:=\sup_{\bm X_i,\bm X_{i'}\in\mathbb X_n}
	\left|u_{\mathbb X_n}(\bm X_i,\bm X_{i'})
	-u_{\mathcal M}(F_n(\bm X_i),F_n(\bm X_{i'}))\right|\\
	&=O_P\left(\eta_n
	+a_n\left[\frac{q_n}{\eta_n}+\norm{\hat g-\nabla p}_\infty\right]\right)
\end{aligned}
\]
and $\sqrt{q_n}$ is absorbed by this rate because $\sqrt{q_n} \lesssim \eta_n + \frac{a_nq_n}{\eta_n}$.
The correspondence formula \eqref{Gromov_Hausdorff} now gives
$\mathrm{GH}((\mathbb X_n,u_{\mathbb X_n}),(\mathcal M,u_{\mathcal M}))\leq\frac{1}{2} \operatorname{dis}(R_n)$.
\end{proof}

\section{Proof of Theorem~\ref{thm:GGDPC_path_full}}
\label{app:ggdpc_path_full_proof}

We begin with technical notation and present the proof of Lemma~\ref{lem:flow_separation_sep} and other useful lemmas, including the invariance results for gradient-flow and oracle GGDPC paths within the basin of attraction. Then, we conclude with the main proof of \autoref{thm:GGDPC_path_full}.

We focus on a basin of attraction $\mathcal{C}_a$ with local mode $\bm{m}^* \in\mathcal{M}$. For $\epsilon>0$, let $\overline{B(\bm{m}^*,\epsilon)} = \left\{\bm{y}\in \mathbb{R}^d: \norm{\bm{y}-\bm{m}^*}\leq \epsilon \right\}$, and we define the hitting time of the population gradient ascent flow starting from $\bm{x}\in \mathcal{C}_a$ to $\overline{B(\bm{m}^*,\epsilon)}$ as:
$$\tau_{\epsilon}(\bm{x}) =\inf\left\{t\geq 0: \bm{\gamma}_{\bm{x}}(t)\in \overline{B(\bm{m}^*,\epsilon)}\right\}.$$
The gradient-flow length up to this neighborhood is $L_\epsilon(\bm{x})=\int_0^{\tau_\epsilon(\bm{x})}\norm{\nabla p(\bm{\gamma}_{\bm{x}}(t))}\,dt$. As a result, $L(\bm{x})= \lim_{\epsilon \to 0}L_{\epsilon}(\bm{x})=\int_0^\infty \norm{\nabla p(\bm{\gamma}_{\bm{x}}(t))}\,dt$ is the population gradient-flow length from $\bm{x}$ to $\bm{m}^*$.

For the oracle GGDPC path $ \bm{Y}_{k+1}^{(n)}=\Phi_n(\bm{Y}_k^{(n)})$ with $\bm{Y}_0^{(n)} =\bm{x}$, we define $$T_{n,\epsilon}=\min\left\{\inf\left\{k\geq 0:\bm{Y}_k^{(n)}\in \overline{B(\bm{m}^*,\epsilon)} \right\}, T_n\right\}$$ 
as the stopping time of the GGDPC path from $\bm{x}$ to the modal neighborhood $\overline{B(\bm{m}^*,\epsilon)}$, where we recall that $T_n$ is the terminal time of the finite GGDPC path within $\mathcal{C}_a$. Let $L_{n,\epsilon}(\bm{x}):=\sum_{k=0}^{T_{n,\epsilon}-1}\norm{\bm{S}_n(\bm{Y}_k^{(n)})}$.

\subsection{Proof of Lemma~\ref{lem:flow_separation_sep}}
\label{app:proof_flow_sep}

\begin{proof}[Proof of Lemma~\ref{lem:flow_separation_sep}]
Let $\mathcal S=\partial\mathcal C_a$. We treat separately the regular
part of $\mathcal S$ and fixed neighborhoods of its finitely many
boundary saddle points.

First, consider a point $\bm{x}\in \mathcal{C}_a$ in a sufficiently small tubular neighborhood of the regular part of $\mathcal S$. Let $\bm{s}=\Pi_{\mathcal S}(\bm{x})$ and write
$\bm{x}=\bm{s}+\norm{\bm{x} - \bm{s}}\nu(\bm{s})$,
where $\nu(\bm{s}) = \frac{\bm{x} - \bm{s}}{\norm{\bm{x} - \bm{s}}}$ points into $\mathcal{C}_a$. Since the separatrix is invariant under the gradient flow, $\nabla p(\bm{s})$ is tangent to $\mathcal S_{\rm reg}$, and hence $\nu(\bm{s})^T\nabla p(\bm{s})=0$. By Taylor's expansion of $p$ and Assumption~\ref{assump:sep_repulsion}, when $\norm{\bm{x} - \Pi_{\mathcal S}(\bm{x})} \leq r_0$ for some small $r_0>0$,
\begin{align*}
	\frac{d}{dt}\norm{\bm{\gamma}_{\bm{x}}(t) - \Pi_{\mathcal S}(\bm{\gamma}_{\bm{x}}(t))}\Big|_{t=0}
	&=\nu(\bm{s})^T\nabla p(\bm{x}) \\
	&=\nu(\bm{s})^T\left[\nabla p(\bm{x})-\nabla p(\bm{s})\right] \\
	&=\norm{\bm{x}-\bm{s}}\nu(\bm{s})^T\nabla^2 p(\bm{s})\nu(\bm{s})+O\left(\norm{\bm{x}-\bm{s}}^2\right) \\
	&\geq \frac{\rho_{\mathcal S}}{2}\norm{\bm{x}-\bm{s}}.
\end{align*}
Thus, if the gradient ascent flow starts at $\bm{x}$ satisfying $d(\bm{x},\partial \mathcal{C}_a) \geq r$, we know that
$$\nu(\bm{s})^T \nabla p(\bm{x}) \geq \frac{\rho_{\mathcal S} r}{2} >0,$$
so the gradient ascent flow cannot decrease its distance to $\mathcal{S}=\partial \mathcal{C}_a$.

Now, fix a saddle point $\bm{s}_j$. By Assumption~\ref{assump:C1_linear}, there is a $C^1$ linearizing diffeomorphism $\Phi_j$ of the flow on a neighborhood $U_j$ of $\bm{s}_j$ with coordinates $(\bm{u},\bm{v})$ such that the local stable manifold is $\{\bm{u}=0\}$ and
\[
\bm{u}'(t)=A_j^u\bm{u}(t),\qquad
\bm{v}'(t)=-A_j^s\bm{v}(t),
\]
where all eigenvalues of $A_j^u$ and $A_j^s$ have positive real parts. Shrinking $U_j$ if necessary, the distance to the local separatrix is comparable with $\norm{\bm{u}}$:
\[
c_j\norm{\bm{u}(\bm{x})}\leq d(\bm{x},\mathcal S)\leq C_j\norm{\bm{u}(\bm{x})},
\qquad \bm{x}\in U_j\cap\mathcal{C}_a.
\]
Moreover, on the compact set of possible saddle passages, there exists $a_j>0$ such that
\[
\norm{e^{A_j^u t}\bm{u}}\geq a_j\norm{\bm{u}}
\]
for every passage time during which the orbit remains in $U_j$. Hence, a trajectory entering $U_j$ at distance at least $r$ from $\mathcal S$ remains at distance at least a constant multiple of $r$ until it exits $U_j$.

The remaining part of the support obtained after removing the regular tubular neighborhood around $\mathcal{S}_{\rm reg}$ and the finitely many saddle neighborhoods is compact and has positive distance from $\mathcal S$. We thus combine the preceding regular boundary and saddle neighborhood estimates, as well as take the minimum over finitely many constants, which yields a constant $C_{\mathcal{S}}\in(0,1)$ such that
$ d(\bm{\gamma}_{\bm{x}}(t),\mathcal S)\geq C_{\mathcal{S}}r$ for all $t\geq0$. Finally, for the tube inclusion, if $\bm{y}\in U_{C_{\mathcal{S}}r/2}(\bm{x})$, then for some $t\geq 0$,
$$d(\bm{y},\mathcal S)\geq d(\bm{\gamma}_{\bm{x}}(t),\mathcal S)-\norm{\bm{y}-\bm{\gamma}_{\bm{x}}(t)}
\geq \frac{C_{\mathcal{S}}r}{2} >0,$$
so $\bm{y}$ lies on the same side of the separatrix as $\bm{x}$, namely in $\mathcal{C}_a$.
\end{proof}

\subsection{Other Supporting Lemmas}

\begin{lemma}[Uniform local sample count]
	\label{lem:uniform_ball_count}
	Suppose that Assumption~\ref{assump:density_reg}(a-b) holds. Let $\epsilon_n\downarrow 0$ be deterministic with $\log\left(\frac{1}{\epsilon_n}\right)=O(\log n)$. Then,
	$$\sup_{\bm z\in\mathbb R^d}\sum_{i=1}^n \mathds{1}\left\{\bm{X}_i \in B(\bm{z},\epsilon_n)\right\} =O_P\left(n\epsilon_n^d+\log n\right).$$
\end{lemma}

\begin{proof}
Similar to the proof of Lemma~\ref{lem:1NN_sup_bnd}, we let $\{\bm z_1,...,\bm z_{N_n}\} \subset \mathcal{C}$ be an $\epsilon_n$-net with $N_n\leq \frac{C_1}{\epsilon_n^d}$ for some constant $C_1>0$. If $B(\bm z,\epsilon_n)\cap\mathcal{C} \neq \emptyset$, we can choose $\bm x\in B(\bm z,\epsilon_n)\cap \mathcal{C}$ and $j$ with $\norm{\bm x-\bm z_j}\leq \epsilon_n$. Then,
$$B(\bm z,\epsilon_n)\cap\mathcal C\subset B(\bm z_j, 3\epsilon_n).$$
Since $p$ is bounded on $\mathcal C$ by Assumption~\ref{assump:density_reg}(a-b),
$$\mathbb{P}\left(\bm{X}_i \in B(\bm z_j,3\epsilon_n)\right)\leq C_2\epsilon_n^d$$
for some constant $C_2>0$
For $N_{j,n}=\sum_{i=1}^n\mathds{1}\left\{\bm X_i\in B(\bm z_j,3\epsilon_n)\right\}$, Bernstein's inequality yields, for every $t>0$,
$$\mathbb{P}\left(N_{j,n}>C_2n\epsilon_n^d+\sqrt{2C_2 n\epsilon_n^dt}+\frac t3\right)\leq e^{-t}.$$
By the union bound, 
$$\mathbb{P}\left(\max_{1\leq j\leq N_n} N_{j,n}>C_2n\epsilon_n^d+\sqrt{2C_2n\epsilon_n^dt}+\frac t3\right)\leq N_n \cdot e^{-t}.$$
Taking $t=2\log N_n+ x$ for $N_n\geq 2$ and $x>0$ shows that the probability of the left-hand side of the above display is at most $\frac{e^{-x}}{N_n}$. Together with $\sqrt{2C_2 n\epsilon_n^d t} \leq C_2n\epsilon_n^d + \frac{t}{2}$, we obtain that
$$\max_{1\leq j\leq N_n} N_{j,n} =O_P\left(n\epsilon_n^d+\log N_n\right) =O_P\left(n\epsilon_n^d+\log n\right).$$
The result thus follows.
\end{proof}

\vspace{2mm}

\begin{lemma}[Local geometry near the local mode; see also Lemma 5 in \citealt{dasgupta2014optimal}]
	\label{lem:mode_geometry}
	Suppose that Assumption~\ref{assump:density_reg} holds. Then, there exist constants $c_0,C_0,c_-,c_+,\epsilon_0>0$ such that, for all $\bm{x}\in \overline{B(\bm{m}^*,\epsilon_0)} \subset \mathcal{C}_a$,
	\begin{align*}
		&c_0\norm{\bm{x}-\bm{m}^*}\leq \norm{\nabla p(\bm{x})}\leq C_0\norm{\bm{x}-\bm{m}^*}, \qquad 
		\nabla p(\bm{x})^T(\bm{x}-\bm{m}^*)\leq -c_0\norm{\bm{x}-\bm{m}^*}^2,\\
		& \text{ and } \qquad c_-\norm{\bm x-\bm m^*}^2 \leq p(\bm m^*)-p(\bm x) \leq c_+\norm{\bm x-\bm m^*}^2.
	\end{align*}
	Moreover, there exists $C_m>1$ such that, for all sufficiently small $\epsilon>0$, 
	$$\left\{\bm x\in \mathcal{C}_a:p(\bm x)\geq p(\bm m^*)-c_+\epsilon^2\right\} \subset B(\bm m^*,C_m \epsilon).$$
	Finally, $\sup_{\bm{x}\in B(\bm{m}^*,\epsilon)} L(\bm{x})=O(\epsilon)$.
\end{lemma}

\begin{proof}
Since $\bm{m}^*$ is a non-degenerate local mode under Assumption~\ref{assump:density_reg}, $\nabla p(\bm{m}^*)=0$ and $\nabla^2p(\bm{m}^*)$ is negative definite. By Taylor expansion of $p$ around $\bm{m}^*$, we obtain that
\begin{equation}
	\label{grad_expand}
	\nabla p(\bm{x})=\nabla^2p(\bm{m}^*)(\bm{x}-\bm{m}^*)+O(\norm{\bm{x}-\bm{m}^*}^2),
\end{equation}
which implies that
$$c_0\norm{\bm{x}-\bm{m}^*}\leq \norm{\nabla p(\bm{x})}\leq C_0\norm{\bm{x}-\bm{m}^*}$$
for $\epsilon_0$ sufficiently small, where $c_0,C_0>0$ are two fixed constants. If $\bm{\gamma}_{\bm{x}}(t)\neq \bm{m}^*$, then
\begin{align*}
	\frac{d}{dt}\norm{\bm{\gamma}_{\bm{x}}(t) - \bm{m}^*} = \frac{\bm{\gamma}_{\bm{x}}'(t)^T \left(\bm{\gamma}_{\bm{x}}(t) - \bm{m}^*\right)}{\norm{\bm{\gamma}_{\bm{x}}(t) - \bm{m}^*}} = \frac{\nabla p(\bm{\gamma}_{\bm{x}}(t))^T \left(\bm{\gamma}_{\bm{x}}(t) - \bm{m}^*\right)}{\norm{\bm{\gamma}_{\bm{x}}(t) - \bm{m}^*}} \leq -c_0 \norm{\bm{\gamma}_{\bm{x}}(t) - \bm{m}^*},
\end{align*}
where the last inequality follows from the negative definiteness of $\nabla^2 p(\bm{m}^*)$ and \eqref{grad_expand}. Additionally, 
$$p(\bm x)-p(\bm m^*)
=\frac12(\bm x-\bm m^*)^T\nabla^2 p(\bm{m}^*)(\bm x-\bm m^*)+O(\norm{\bm x-\bm m^*}^3)$$
implies that 
$$c_-\norm{\bm x-\bm m^*}^2 \leq p(\bm m^*)-p(\bm x) \leq c_+\norm{\bm x-\bm m^*}^2$$
for some fixed constants $c_-, c_+>0$.

Now, for a sufficiently small $\epsilon_1>0$, we also know that
$$\sup_{\bm{x}\in \mathcal{C}_a \setminus B(\bm{m}^*,\epsilon_1)} p(\bm{x}) < p(\bm{m}^*).$$
Hence, there exists a sufficiently small $\epsilon\in (0,\epsilon_1)$ so that
$$\left\{\bm x\in \mathcal{C}_a:p(\bm x)\geq p(\bm m^*)-c_+\epsilon^2\right\} \subset B(\bm m^*,C_m \epsilon).$$
Finally, by $\frac{d}{dt}\norm{\bm{\gamma}_{\bm{x}}(t) - \bm{m}^*} \leq -c_0 \norm{\bm{\gamma}_{\bm{x}}(t) - \bm{m}^*}$ and Gr\"onwall's inequality,
$$\norm{\bm{\gamma}_{\bm{x}}(t)-\bm{m}^*}\leq \norm{\bm{x}-\bm{m}^*}e^{-c_0t}.$$
Using $\norm{\nabla p(\bm{\gamma}_{\bm{x}}(t))}\leq C_0\norm{\bm{\gamma}_{\bm{x}}(t)-\bm{m}^*}$ with $\bm{x}\in B(\bm{m}^*,\epsilon)$ gives
$$L(\bm{x}) = \int_0^{\infty} \norm{\nabla p(\bm{\gamma}_{\bm{x}}(t))}dt \leq C_0\int_0^\infty \norm{\bm{x}-\bm{m}^*} e^{-c_0t}\,dt =O(\norm{\bm{x}-\bm{m}^*}).$$
The result follows.
\end{proof}

\vspace{2mm}

\begin{lemma}[Regularity of the gradient-flow length]
\label{lem:flow_length_regularity}
Suppose that Assumption~\ref{assump:density_reg} holds. For every compact set $\mathbb{K}\subset\mathcal C_a$, there is a neighborhood $U_{\mathbb{K}}$ of the collection of gradient flow trajectories $\{\bm\gamma_{\bm x}(t):\bm x\in \mathbb{K},\ t\geq0\}$ whose closure is contained in $\mathcal C_a$. For all sufficiently small $\epsilon>0$, $\bm{x}\mapsto L_\epsilon(\bm{x})$ is $C^2$ on $U_{\mathbb{K}}\setminus\overline{B(\bm m^*,\epsilon)}$ and satisfies $\nabla L_\epsilon(\bm x)^T\nabla p(\bm x)
=-\norm{\nabla p(\bm x)}$.
Moreover, for constants depending only on $\mathbb{K}$,
\[
\norm{\nabla L_\epsilon(\bm x)}\leq C_{\mathbb{K}},\qquad
\norm{\nabla^2L_\epsilon(\bm x)}_2
\leq \frac{C_{\mathbb{K}}}{\max\{\epsilon,\norm{\bm x-\bm m^*}\}}
\]
whenever $\bm x\in U_{\mathbb{K}}\setminus B(\bm m^*,\epsilon)$.
\end{lemma}

\begin{proof}
Since the gradient flow $\bm{\gamma}_{\bm{x}}(t)$ converges to $\bm{m}^*$ for every $\bm{x}\in \mathbb{K}$ and $\mathbb{K}$ is a compact subset of $\mathcal{C}_a$, there exists $T<\infty$ such that
$\bm{\gamma}_{\bm{x}}(t) \in B(\bm{m}^*,r_1)$ for all $t\geq T$ and some $r_1 >0$. Hence, we can take $U_{\mathbb{K}}$ as a small neighbor of
$$\left\{\bm{\gamma}_{\bm{x}}(t): \bm{x}\in \mathbb{K}, 0\leq t \leq T\right\} \cup \overline{B(\bm{m}^*,r_1)}$$
so that $\overline{U_{\mathbb{K}}}\subset\mathcal{C}_a$.

Consider first $U_{\mathbb{K}}\setminus B(\bm{m}^*,r_1)$.
By compactness and the fact that $\bm{m}^*$ is the only critical point on the relevant trajectories,
$$\inf_{\bm{x}\in U_{\mathbb{K}}\setminus B(\bm{m}^*,r_1)}
\norm{\nabla p(\bm{x})}>0$$
after shrinking $U_{\mathbb{K}}$ if necessary. Let $\tau_{r_1}(\bm{x})
=\inf\left\{t\geq0:
\norm{\bm{\gamma}_{\bm{x}}(t)-\bm{m}^*}=r_1\right\}$. If we take
$$F(t,\bm{x}) :=\norm{\bm{\gamma}_{\bm{x}}(t)-\bm{m}^*}^2-r_1^2,$$
then $F(\tau_{r_1}(\bm{x}),\bm{x})=0$ and
$$\partial_t F(\tau_{r_1}(\bm{x}),\bm{x})
= 2\left(\bm{\gamma}_{\bm{x}}(\tau_{r_1}(\bm{x}))-\bm{m}^*\right)^T
\nabla p\bigl(\bm{\gamma}_{\bm{x}}(\tau_{r_1}(\bm{x}))\bigr)
<0.$$
Thus, the implicit function theorem implies that $\tau_{r_1}$ is $C^2$. Together with the $C^2$ dependence of $\bm{\gamma}_{\bm{x}}(t)$ on $\bm{x}$, this shows that $\bm{x}\mapsto L_\epsilon(\bm{x})$ is $C^2$ on $U_{\mathbb{K}}\setminus \overline{B(\bm{m}^*,r_1)}$
with uniformly bounded first and second derivatives.

It remains to consider $B(\bm{m}^*,r_1)\setminus B(\bm{m}^*,\epsilon)$.
Write
$$\bm{x}=\bm{m}^*+r\bm{\theta}, \qquad r=\norm{\bm{x}-\bm{m}^*}, \qquad
\bm{\theta}\in\mathbb{S}^{d-1}.$$
By Taylor expansion,
$$\nabla p(\bm{m}^*+r\bm{\theta}) = r\nabla^2p(\bm{m}^*)\bm{\theta} +O(r^2),$$
uniformly in $\bm{\theta}$. Therefore, the smooth dependence of the flow and its hitting time gives, uniformly for $0<\epsilon\leq r\leq r_1$,
$$\left|\partial_rL_\epsilon(r,\bm{\theta})\right| + \frac{1}{r}
\norm{\nabla_{\bm{\theta}}L_\epsilon(r,\bm{\theta})} \leq C_{\mathbb{K}},$$
and
$$\norm{\nabla^2L_\epsilon(\bm{m}^*+r\bm{\theta})}_2 \leq \frac{C_{\mathbb{K}}}{r}.$$
If $r\asymp \epsilon$, we know that
$$\norm{\nabla^2L_\epsilon(\bm{x})}_2 \leq \frac{C_{\mathbb{K}}}{\epsilon}.$$
Hence,
$$\norm{\nabla L_\epsilon(\bm{x})}
\leq C_{\mathbb{K}}, \qquad
\norm{\nabla^2 L_\epsilon(\bm{x})}_2 \leq \frac{C_{\mathbb{K}}}{\max\{\epsilon,\norm{\bm{x}-\bm{m}^*}\}}.$$

Finally, for every $t$ before the hitting time of $B(\bm{m}^*,\epsilon)$,
$$L_\epsilon(\bm{\gamma}_{\bm{x}}(t)) =L_\epsilon(\bm{x})-\int_0^t \norm{\nabla p(\bm{\gamma}_{\bm{x}}(s))}\,ds.$$
Differentiating with respect to $t$ at $t=0$ yields that $\nabla L_\epsilon(\bm{x})^T\nabla p(\bm{x})
=-\norm{\nabla p(\bm{x})}$, which completes the proof.
\end{proof}

\vspace{2mm}

\begin{lemma}[Population bound for the 1NN uphill shift outside the modal core]
\label{lem:pop_1nn_bnd}
Suppose that Assumption~\ref{assump:density_reg} holds. There exist fixed constants $C_0',c,\epsilon_0>0$ such that $\epsilon_n:=\frac{C q_n}{\eta_n} \to 0$ for any fixed $C\geq C_0'$ and $q_n = \left(\frac{\log n}{n}\right)^{\frac{1}{d}}$. Then, uniformly over $\bm{x}\in B(\bm{m}^*, \epsilon_0) \setminus B(\bm{m}^*, \epsilon_n)$ and with probability tending to one,
$$\Phi_n(\bm{x}) = \bm{x}+\eta_n\nabla p(\bm{x}) + \bm{u}_n(\bm{x}), \quad \norm{\bm{u}_n(\bm{x})}\leq \zeta_n$$ 
and
$$\norm{\Phi_n(\bm{x}) - \bm{m}^*} \leq (1-c\eta_n) \norm{\bm{x}-\bm{m}^*}.$$
\end{lemma}

\begin{proof}
Recall from \eqref{max_1nn_dist} that $\zeta_n=\sup_{\bm{x}\in \overline{\mathcal{C}}_a}\min_{1\leq i\leq n}\norm{\bm{x}-\bm{X}_i} = O_P\left(\left(\frac{\log n}{n}\right)^{\frac{1}{d}}\right)$. Let $\epsilon=\norm{\bm{x} - \bm{m}^*} \geq \epsilon_n$ and $\bm{x}^+=\bm{x}+\eta_n\nabla p(\bm{x})$. By Lemma~\ref{lem:mode_geometry},
$$p(\bm{x}^+)-p(\bm{x})=\eta_n\norm{\nabla p(\bm{x})}^2 + O\left(\eta_n^2 \norm{\nabla p(\bm{x})}^2 \right)\geq C_1\eta_n \epsilon^2$$
for all large $n$ and some fixed constant $C_1>0$. For any $\bm{y}\in B(\bm{x}^+,\zeta_n)$, by the Lipschitz continuity of $p$, Taylor's theorem gives a fixed constant $C_2>0$ such that
\begin{align*}
	p(\bm{y})&\geq p(\bm{x}^+) - \norm{\bm{y} - \bm{x}^+}\norm{\nabla p(\bm{x}^+)} -C_2\norm{\bm y-\bm x^+}^2\\
	&\geq p(\bm{x}) + C_1\eta_n\epsilon^2 - C_0\norm{\bm{y}-\bm{x}^+} \norm{\bm{x}^+-\bm{m}^*}\\
	&\geq p(\bm{x}) + C_1\eta_n\epsilon^2 - C_0\zeta_n \left(\norm{\bm{x}^+ - \bm{x}} + \norm{\bm{x}-\bm{m}^*}\right)\\
	&\geq p(\bm{x}) + C_1\eta_n\epsilon^2 - C_2\zeta_n^2 -C_2\zeta_n\epsilon\\
	&\geq p(\bm{x}) + \eta_n\epsilon^2\left(C_1 - \frac{C_2}{C} - \frac{C_2\eta_n}{C^2}\right)\\
	&> p(\bm{x})
\end{align*}
with probability tending to one, where the last two inequalities use the condition that $\epsilon\geq \epsilon_n=\frac{Cq_n}{\eta_n}\to 0$, $\zeta_n=O_P(q_n)$, and $C>0$ is chosen to be sufficiently large. By definition of $\zeta_n$ in \eqref{max_1nn_dist}, there exists some observation for the 1NN uphill point $\bm{x}^+$ when $n$ is large. Hence, uniformly over $\bm{x}\in \mathcal{C}$, 
$$\norm{\Phi_n(\bm{x}) - (\bm{x}+\eta_n\nabla p(\bm{x}))}\leq \zeta_n.$$
Finally, by Lemma~\ref{lem:mode_geometry},
\begin{align*}
	\norm{\bm x^+-\bm m^*}^2
	&=\epsilon^2 + 2\eta_n(\bm x-\bm m^*)^T\nabla p(\bm x)+ \eta_n^2\norm{\nabla p(\bm x)}^2\\
	&\leq(1-2c_0\eta_n+C_0^2\eta_n^2)\epsilon^2,
\end{align*}
so $\norm{\bm x^+-\bm m^*}\leq (1-C_3\eta_n) \epsilon$ for all large $n$ and some absolute constant $C_3>0$. Hence,
$$\norm{\Phi_n(\bm x)-\bm m^*} \leq(1-C_3\eta_n)\epsilon +C_\zeta q_n
\leq\left(1-C_3\eta_n+\frac{C_\zeta}{C}\eta_n\right)\epsilon$$
for some constant $C_{\zeta} >0$, where the last inequality again follows from the fact that $\epsilon = \norm{\bm{x}-\bm{m}^*}\geq \epsilon_n=\frac{Cq_n}{\eta_n}$.
Taking $C\geq \frac{2C_\zeta}{C_3}$ proves the final result.
\end{proof}

\vspace{2mm}

\begin{lemma}[Oracle modal-core length]
\label{lem:oracle_modal_core}
Suppose that Assumption~\ref{assump:density_reg} holds. Let $\epsilon_n=\frac{Cq_n}{\eta_n}$ with fixed $C\geq C_0'$ as in Lemma~\ref{lem:pop_1nn_bnd}. Then, uniformly over every oracle GGDPC path that enters $B(\bm m^*,\epsilon_n)$ and remains in $\mathcal C_a$ until its terminal vertex, we have that
$$\sum_{k=T_{n,\epsilon_n}}^{T_n-1}\norm{\bm Y_{k+1}^{(n)}-\bm Y_k^{(n)}} =O_P\left(n \epsilon_n^{d+1}\right) =O_P\left(\frac{q_n\log n}{\eta_n^{d+1}}\right).$$
\end{lemma}

\begin{proof}
By the non-decreasing numerical value of $p$ and strict increase of density rank, for every $k\geq T_{n,\epsilon_n}$ with $k\leq T_n$,
$$p(\bm Y_k^{(n)}) \geq p(\bm Y_{T_{n,\epsilon_n}}^{(n)}) \geq p(\bm m^*)-c_+\epsilon_n^2.$$
Lemma~\ref{lem:mode_geometry} therefore gives
$$\bm Y_k^{(n)}\in B(\bm m^*,C_m \epsilon_n), \qquad T_{n,\epsilon_n}\leq k\leq T_n.$$
Hence, every remaining 1NN uphill shift has length at most $2C_m \epsilon_n$. Since the density values (or more precisely, their ranks) along the path are strictly increasing, no observation is visited twice. Thus,
$$\sum_{k=T_{n,\epsilon_n}}^{T_n-1}\norm{\bm Y_{k+1}^{(n)}-\bm Y_k^{(n)}} \leq C \epsilon_n\left[1+\sum_{i=1}^n\mathds{1}\left\{\bm X_i\in B(\bm m^*,C_m \epsilon_n)\right\}\right].$$
Additionally, Lemma~\ref{lem:uniform_ball_count} gives
$$\sum_{i=1}^n\mathds{1}\left\{\bm X_i\in B(\bm m^*,C_m \epsilon_n)\right\} =O_P(n\epsilon_n^d+\log n).$$
Since $n\epsilon_n^d=\frac{C^d\log n}{\eta_n^d}\geq C^d\log n$ for all large $n$, the result follows.
\end{proof}

\vspace{2mm}

\begin{lemma}[Stability of the fixed-$\epsilon$ oracle GGDPC path length]
\label{lem:GGDPC_path_eps}
Let $\mathcal{C}_{a,\epsilon}:=\mathcal{C}_a\setminus B(\bm{m}^*,\epsilon)$ for a sufficiently small fixed $\epsilon>0$. Suppose that Assumption~\ref{assump:density_reg} holds. If $\eta_n\to 0$ and $\frac{q_n}{\eta_n}\to 0$ with $q_n=\left(\frac{\log n}{n}\right)^{\frac{1}{d}}$,
then 
$$\sup_{\bm{x}\in \mathcal{C}_{a,\epsilon} \ominus \delta}\left|L_{n,\epsilon}(\bm{x})-L_\epsilon(\bm{x})\right|
=O_P\left(\eta_n+\frac{q_n}{\eta_n}\right)$$
for some fixed $\delta >0$, where $\mathcal{C}_{a,\epsilon} \ominus \delta := \left\{\bm{x}\in \mathcal{C}_{a,\epsilon}: \min_{\bm{y}\in \partial \mathcal{C}_a} \norm{\bm{x} - \bm{y}} \geq \delta\right\}$.
\end{lemma}

\begin{proof}
Recall from \eqref{max_1nn_dist} that $\zeta_n=\sup_{\bm{x}\in \overline{\mathcal{C}}_a}\min_{1\leq i\leq n}\norm{\bm{x}-\bm{X}_i} = O_P\left(\left(\frac{\log n}{n}\right)^{\frac{1}{d}}\right)$. Fix any $\bm{x}=\bm{Y}_0^{(n)} \in \mathcal{C}_{a,\epsilon}$. By Lemma~\ref{lem:pop_1nn_bnd}, for every $k<T_{n,\epsilon}$,
\begin{equation}
	\label{1nn_up_rate}
	\bm{S}_n(\bm{Y}_k^{(n)}):= \Phi_n(\bm{Y}_k^{(n)}) - \bm{Y}_k^{(n)} =\eta_n\nabla p(\bm{Y}_k^{(n)})+\bm{u}_k,
	\qquad
	\norm{\bm{u}_k}\leq \zeta_n.
\end{equation}
By Lemma~\ref{lem:flow_length_regularity}, the corresponding population gradient flow trajectories up to $B(\bm{m}^*,\epsilon)$ lie in a fixed compact subset of $\mathcal C_a$. 

Furthermore, the path length function $L_{\epsilon}: \mathbb{R}^d \to \mathbb{R}$ satisfies $\nabla L_\epsilon(\bm{x})^T\nabla p(\bm{x})=-\norm{\nabla p(\bm{x})}$
on $\mathcal{C}_{a,\epsilon}$, and $L_\epsilon$ is twice continuously differentiable on the compact set $\mathcal{C}_{a,\epsilon}$. Thus,
\begin{align}
	\label{L_taylor}
	\begin{split}
		L_{\epsilon}(\bm{Y}_{k+1}^{(n)}) - L_{\epsilon}(\bm{Y}_k^{(n)}) &= \nabla L_{\epsilon}(\bm{Y}_k^{(n)})^T \bm{S}_n(\bm{Y}_k^{(n)}) + O\left(\norm{\bm{S}_n(\bm{Y}_k^{(n)})}^2 \right) \\
		&= -\eta_n \norm{\nabla p(\bm{Y}_k^{(n)})} + \nabla L_{\epsilon}(\bm{Y}_k^{(n)})^T \bm{u}_k + O(\eta_n^2 + \zeta_n^2) \\
		&\stackrel{\text{(i)}}{=} -\eta_n\norm{\nabla p(\bm{Y}_k^{(n)})}+O(\zeta_n+\eta_n^2),
	\end{split}
\end{align}
where (i) follows from \eqref{1nn_up_rate} with $\norm{\bm{u}_k} \leq \zeta_n$ and $\norm{\nabla^2 L_{\epsilon}(\bm{x})} \lesssim \frac{1}{\max\left\{\norm{\bm{x}-\bm{m}^*}, \epsilon\right\}}$.
Meanwhile, $\norm{\bm{S}_n(\bm{Y}_k^{(n)})}
=\eta_n\norm{\nabla p(\bm{Y}_k^{(n)})}+O(\zeta_n)$. 
Therefore,
\begin{align}
	\label{L_path}
	\begin{split}
		\left|L_{n,\epsilon}(\bm{x}) - L_{\epsilon}(\bm{x}) \right| &\stackrel{\text{(ii)}}{=} \left|\sum_{k=0}^{T_{n,\epsilon} -1} \left[\norm{\bm{S}_n(\bm{Y}_k^{(n)})} + L_{\epsilon}(\bm{Y}_{k+1}^{(n)}) - L_{\epsilon}(\bm{Y}_k^{(n)}) \right] \right|\\
		&\stackrel{\text{(iii)}}{\lesssim} \left|\sum_{k=0}^{T_{n,\epsilon} -1} \left[\eta_n \norm{\nabla p(\bm{Y}_k^{(n)})} + O(\zeta_n) -\eta_n \norm{\nabla p(\bm{Y}_k^{(n)})} + O(\zeta_n + \eta_n^2) \right] \right|\\
		&\leq C_1\cdot T_{n,\epsilon} \left(\zeta_n + \eta_n^2\right)
	\end{split}
\end{align}
for some constant $C_1 >0$,
where (ii) uses the fact that $\bm{Y}_0^{(n)}=\bm{x}$ and $\bm{Y}_{T_{n,\epsilon}}^{(n)} \in B(\bm{m}^*,\epsilon)$ with $L_{\epsilon}(\bm{Y}_{T_{n,\epsilon}}^{(n)})=0$ and (iii) plugs in \eqref{L_taylor}. 

Finally, to bound the stopping time $T_{n,\epsilon}$, since $\norm{\nabla p(\bm{x})}$ is bounded away from 0 on $\mathcal{C}_{a,\epsilon}$, \eqref{L_taylor} implies that
$$L_{\epsilon}(\bm{Y}_k^{(n)}) - L_{\epsilon}(\bm{Y}_{k+1}^{(n)}) \geq C_2 \eta_n$$
for some constant $C_2>0$ and all $k< T_{n,\epsilon}$ when $n$ is sufficiently large. Hence, $T_{n,\epsilon}=O_P(\eta_n^{-1})$. The result thus follows by plugging this probabilistic rate into \eqref{L_path} and noting that $\zeta_n = O_P\left(\left(\frac{\log n}{n}\right)^{\frac{1}{d}}\right)$. The uniform statement follows by taking supremum over the fixed compact set $\mathcal{C}_{a,\epsilon}\ominus\delta$.
\end{proof}

\vspace{2mm}

\begin{lemma}[Length stability in a separatrix tube]
\label{lem:L_holder}
Suppose that Assumptions~\ref{assump:density_reg}, \ref{assump:sep_repulsion}, and \ref{assump:C1_linear} hold. Let $\Gamma_\infty(\bm{x})=\{\bm{\gamma}_{\bm{x}}(t):0\leq t<\infty\}$ be the gradient flow trajectory with $\bm{x}\in \mathcal{C}_a$ and suppose that
$$U_\delta(\bm{x})=\{\bm{y}\in \mathbb{R}^d:d(\bm{y},\Gamma_\infty(\bm{x}))\leq\delta\}\subset \mathcal{C}_a,$$
where $d(\bm{y},\Gamma_\infty(\bm{x})) = \inf\left\{\norm{\bm{y}-\bm{z}}: \bm{z}\in \Gamma_\infty(\bm{x})\right\}$ is the distance from $\bm{y}$ to $\Gamma_\infty(\bm{x})$.
Then, there exist constants $C_L,\delta_0,\epsilon_0>0$ such that, for all $0<\delta\leq\delta_0$ and $0<\epsilon\leq\epsilon_0$,
$$\left|L_\epsilon(\bm{y})-L_\epsilon(\bm{z})\right|
\leq \frac{C_L}{\delta}\norm{\bm{y}-\bm{z}}
\quad \text{ for all } \quad \bm{y},\bm{z}\in U_{\delta/2}(\bm{x})\setminus B(\bm{m}^*,\epsilon).$$
\end{lemma}

\begin{proof}
Away from fixed neighborhoods of the finitely many boundary saddle points, the relevant flow segments of $\bm{x}\mapsto L_\epsilon(\bm{x})$ lie in a compact subset on which the gradient flow is $C^1$ with uniformly bounded derivatives. Hence, a $\delta$-dependent factor of $\bm{x}\mapsto L_\epsilon(\bm{x})$ can arise only when the gradient flow passes through neighborhoods of boundary saddle points (or $\partial \mathcal{C}_a$). 
	
Fix such a saddle point $\bm{s}_j$. By Assumption~\ref{assump:sep_repulsion}, its unstable dimension for gradient ascent is one. By Assumption~\ref{assump:C1_linear}, there is a $C^1$ diffeomorphism $\Phi_j$ that linearizes the flow in a neighborhood $U_j$ of $\bm{s}_j$. We write
$$\Phi_j(\bm{x}) = \left(u(\bm{x}), \bm{v}(\bm{x})\right) \in \mathbb{R} \times \mathbb{R}^{d-1},$$
where $u$ denotes the unstable coordinates and $\bm{v}$ the stable coordinates for the gradient ascent flow after linearization. Then,
$$u'(t) = \lambda_j u(t), \quad \bm{v}'(t) = -A_j^s \bm{v}(t),$$	where every eigenvalue of $A_j^s \in \mathbb{R}^{(d-1)\times (d-1)}$ has a positive real part. This also implies that 
\begin{equation}
	\label{lin_exp}
	u(t) = e^{\lambda_j t} u(0) \quad \text{ and } \quad \bm{v}(t) = e^{-A_j^s t} \bm{v}(0).
\end{equation} 
The local stable manifold is $\{(u,\bm{v}): u=0\}$. Shrinking the coordinate neighborhood if necessary, the bi-Lipschitz continuity of $\Phi_j$ gives
$$c_j|u(\bm y)|\leq d(\bm y,\partial\mathcal C_a)\leq C_j|u(\bm y)|$$
for some constants $c_j,C_j > 0$.
Thus, for $\bm y\in U_{\delta/2}(\bm x)$ in this saddle neighborhood,
\begin{equation}
	\label{saddle_u_lower_new}
	|u(\bm y)|\geq c\delta.
\end{equation}
for some constant $c>0$.
Now, choose a rectangular linearizing neighborhood $V_j=\{(u,\bm v):|u|<u_0,\ \norm{\bm v}<v_0\}$
so small that a forward gradient path or orbit with $u\neq0$ cannot exit it through the stable boundary $\norm{\bm v}=v_0$. By \eqref{lin_exp}, its exit time through $|u|=u_0$ is thus given by
\begin{equation}
	\label{saddle_exit_time_exact_new}
	\tau_j(u)=\frac1{\lambda_j}\log\frac{u_0}{|u|},
	\qquad
	|\partial_u\tau_j(u)|=\frac1{\lambda_j|u|}.
\end{equation}
Let $q_j(u,\bm v) =\norm{\nabla p\!\left(\Phi_j^{-1}(u,\bm v)\right)}$.
Along a passage with $u\neq0$, $q_j$ is continuously differentiable and, since $\nabla p(\bm s_j)=0$ and $p\in C^3$,
\begin{equation}
	\label{qj_bounds_new}
	q_j(u,\bm v)\leq C\bigl(|u|+\norm{\bm v}\bigr),
	\qquad
	\norm{Dq_j(u,\bm v)}\leq C
\end{equation}
for some constant $C>0$. From \eqref{lin_exp}, we write
$$\varphi_t(u,\bm v)=\left(e^{\lambda_jt}u,e^{-A_j^st}\bm v\right),
\qquad
\widetilde\ell_j(u,\bm v)
=\int_0^{\tau_j(u)}q_j(\varphi_t(u,\bm v))\,dt.$$
By Leibniz's rule, \eqref{saddle_exit_time_exact_new}, and \eqref{qj_bounds_new}, we know that
\begin{align}
\label{saddle_length_dv_new}
\begin{split}
	|\partial_u\widetilde\ell_j(u,\bm v)|
	&\leq C\int_0^{\tau_j(u)}e^{\lambda_jt}\,dt
	+C|\partial_u\tau_j(u)|
	\leq \frac{C}{|u|},\\
	\norm{D_{\bm v}\widetilde\ell_j(u,\bm v)}
	&\leq C\int_0^\infty\norm{e^{-A_j^st}}_2\,dt
	\leq C.
\end{split}
\end{align}
We also know that the two points $\bm{y},\bm{z}$ in the lemma statement lie on the same side of the local stable manifold, so the line segment between their $u$-coordinates does not cross zero. Combining \eqref{saddle_length_dv_new} with the mean value theorem and the bi-Lipschitz bounds for $\Phi_j$ yields that
$$|\ell_j(\bm y)-\ell_j(\bm z)| \leq \frac{C}{\delta}\norm{\bm y-\bm z}.$$
Adding the uniformly Lipschitz contributions outside the finitely many saddle neighborhoods proves the lemma.
\end{proof}

\vspace{2mm}

\begin{lemma}[Oracle GGDPC path through a boundary saddle point]
\label{lem:discrete_saddle_passage}
Suppose that Assumptions~\ref{assump:density_reg}, \ref{assump:sep_repulsion}, and \ref{assump:C1_linear} hold.  For a fixed boundary saddle point $\bm s_j$ and its sufficiently small fixed
neighborhood $U_j$, we consider a finite segment of the oracle GGDPC path $\{\bm Y_k^{(n)}\}_{k=k_0}^{k_1+1}\subset U_j\cap\mathcal C_a$ satisfying
\begin{equation}
	\label{perturbed_euler_saddle_new}
	\bm Y_{k+1}^{(n)}
	=\bm Y_k^{(n)}+\eta_n\nabla p(\bm Y_k^{(n)})+\bm \xi_k,
	\qquad
	\norm{\bm \xi_k}
	\lesssim q_n.
\end{equation}
Let $\delta_n\downarrow 0$. If $d(\bm Y_{k_0}^{(n)},\partial\mathcal C_a)\geq c_0\delta_n$ and $\eta_n+\frac{q_n}{\eta_n}=o(\delta_n)$ for some constant $c_0>0$, then the path segment stays on the same side of the local stable manifold of $\bm{s}_j$ and, for some fixed constant $c_1>0$,
$$\inf_{k_0\leq k\leq k_1+1}
	d(\bm Y_k^{(n)},\partial\mathcal C_a)
	\geq c_1\delta_n.$$
If this segment traverses $U_j$ once and
$U_j\cap B(\bm m^*,\epsilon_n)=\emptyset$, then
$$\sum_{k=k_0}^{k_1} \left|\norm{\bm Y_{k+1}^{(n)}-\bm Y_k^{(n)}}
	+L_{\epsilon_n}(\bm Y_{k+1}^{(n)}) -L_{\epsilon_n}(\bm Y_k^{(n)}) \right| \lesssim \frac{\eta_n}{\delta_n} + \frac{q_n}{\eta_n\delta_n}.$$
\end{lemma}

\begin{proof}
Since $\nabla p\in C^2$ under Assumption~\ref{assump:density_reg}(b), the $C^2$ stable manifold theorem gives a $C^2$ chart $\Psi_j$ that smoothly transforms the local stable manifold of $\bm{s}_j$ to $\{u=0\}$, \emph{i.e.}, $\Psi_j\bigl(U_j\cap\partial\mathcal C_a\bigr)
=\left\{(u,\bm v):u=0\right\}$. Let $\Psi_j(\bm Y_k^{(n)})=(u_k,\bm v_k)$. 

Consider the representation of the population gradient vector field in this chart as:
$$\widetilde{\bm F}(u,\bm v) := D\Psi_j\!\left(\Psi_j^{-1}(u,\bm v)\right)
\nabla p\!\left(\Psi_j^{-1}(u,\bm v)\right).$$
Let $\widetilde F_1$ denote its first coordinate. Then, by Assumption~\ref{assump:sep_repulsion} and invariance of the local stable manifold $\{u=0\}$, we know that $\widetilde F_1(0,\bm v)=0$. Thus, the first coordinate of the transformed population gradient flow after applying $\Phi_j$ satisfies that
$$u'(t)=a(u(t),\bm v(t)) \cdot u(t), \qquad a(0,\bm{0})=\lambda_j>0.$$
After shrinking $U_j$, we may assume $a(u,\bm v)\geq \frac{\lambda_j}{2}$.
Taylor's expansion of the chart $\Psi_j(\bm Y_k^{(n)})=(u_k,\bm v_k)$ in \eqref{perturbed_euler_saddle_new}
gives
\begin{equation}
\label{discrete_u_recurrence_new}
\begin{split}
u_{k+1} &= u_k + D\Psi_{j,1}(\bm{Y}_k^{(n)}) \left(\bm{Y}_{k+1}^{(n)} - \bm{Y}_k^{(n)}\right) + C_1\norm{\bm{Y}_{k+1}^{(n)} - \bm{Y}_k^{(n)}}^2\\
&= u_k + \eta_n D\Psi_{j,1}(\bm{Y}_k^{(n)}) \nabla p(\bm{Y}_k^{(n)}) + D\Psi_{j,1}(\bm{Y}_k^{(n)})\bm{\xi}_k + C_1\norm{\bm{Y}_{k+1}^{(n)} - \bm{Y}_k^{(n)}}^2\\
&=u_k\left[1+ \eta_n a(u_k,\bm{v}_k)\right] + O\left(q_n
+\eta_n^2\norm{\nabla p(\bm Y_k^{(n)})} \right)
\end{split}
\end{equation}
for some fixed constant $C_1>0$.
The bi-Lipschitz property of the chart $\Psi_j$ gives $|u_{k_0}|\asymp d(\bm Y_{k_0}^{(n)},\partial\mathcal C_a)\gtrsim\delta_n$. Moreover,
$q_n+\eta_n^2 =\eta_n\left(\frac{q_n}{\eta_n}+\eta_n\right) =o(\eta_n\delta_n)$. Consequently, \eqref{discrete_u_recurrence_new} implies inductively that
\begin{align*}
|u_{k+1}| &\geq |u_k| \left[1+ \eta_n a(u_k,\bm{v}_k)\right]  -  C_2\left(q_n+ \eta_n^2\norm{\nabla p(\bm Y_k^{(n)})} \right) \\
&\geq \left(1+\frac{\lambda_j}{4}\eta_n\right)|u_k|
\end{align*}
for some constant $C_2>0$, so $\inf_{k_0\leq k\leq k_1+1}
d(\bm Y_k^{(n)},\partial\mathcal C_a)
\geq c_1\delta_n$ for some fixed constant $c_1>0$.

Since the path leaves $U_j$ when $|u_k|$ exceeds a fixed positive value,
the above display also yields that
\begin{equation}
\label{saddle_step_count_new}
(k_1-k_0+1)\eta_n \leq C_3\left(1+\log\frac1{\delta_n}\right) \leq\frac{C_3}{\delta_n},\qquad \sum_{k=k_0}^{k_1}\frac{\eta_n}{|u_k|} \leq\frac{C}{|u_{k_0}|} \leq\frac{C}{\delta_n}
\end{equation}
for some constant $C_3>0$.

Now, let $\bm Z_k^{(n)}=\bm\gamma_{\bm Y_k^{(n)}}(\eta_n)$ be the exact gradient flow endpoint after time $\eta_n$ starting from $\bm Y_k^{(n)}$ and 
$$\ell_k:=\int_0^{\eta_n} \norm{\nabla p(\bm\gamma_{\bm Y_k^{(n)}}(t))}\,dt.$$ 
Because $\nabla^2p$ is bounded on a fixed neighborhood containing the saddle point $\bm{s}_j$, the gradient vector field is Lipschitz there. Taylor's expansion of the exact flow thus gives that, uniformly for $0\leq t\leq\eta_n$,
$$\norm{\bm\gamma_{\bm Y_k^{(n)}}(t)-\bm Y_k^{(n)}} \leq C_4 t\norm{\nabla p(\bm Y_k^{(n)})}$$
for some constant $C_4>0$. Consequently, 
\begin{align}
	\label{grad_flow_ascent_diff1}
	\begin{split}
		\norm{\bm{Z}_k^{(n)} - \bm{Y}_k^{(n)}-\eta_n\nabla p(\bm{Y}_k^{(n)})} &= \norm{\bm{Y}_k^{(n)} + \int_0^{\eta_n} \nabla p(\bm{\gamma}_{\bm{Y}_k^{(n)}}(t))\, dt - \bm{Y}_k^{(n)} - \eta_n \nabla p(\bm{Y}_k^{(n)})} \\
		&= \norm{\int_0^{\eta_n} \left[\nabla p(\bm{\gamma}_{\bm{Y}_k^{(n)}}(t)) - \nabla p(\bm{Y}_k^{(n)}) \right] dt} \\
		&\leq C_5 \int_0^{\eta_n} \norm{\bm{\gamma}_{\bm{Y}_k^{(n)}}(t) - \bm{Y}_k^{(n)}} dt \\
		&= C_5 \int_0^{\eta_n} \norm{\int_0^t \nabla p(\bm{\gamma}_{\bm{Y}_k^{(n)}}(u))du} dt\\
		&\leq C_5 \eta_n^2 \norm{\nabla p(\bm{Y}_k^{(n)})}
	\end{split}
\end{align}
for some constant $C_5>0$ and by \eqref{perturbed_euler_saddle_new},
\begin{align}
	\label{Y_Z_saddle_new}
\begin{split}
\norm{\bm Y_{k+1}^{(n)}-\bm Z_k^{(n)}}
&\lesssim q_n+\eta_n^2 \norm{\nabla p(\bm Y_k^{(n)})}, \\
\left|\norm{\bm Y_{k+1}^{(n)}-\bm Y_k^{(n)}}-\ell_k\right|
&\lesssim q_n+\eta_n^2 \norm{\nabla p(\bm Y_k^{(n)})}.
\end{split}
\end{align}
Also, $L_{\epsilon_n}(\bm Z_k^{(n)}) -L_{\epsilon_n}(\bm Y_k^{(n)})=-\ell_k$. The proof of Lemma~\ref{lem:L_holder}, applied locally with $|u_k|$ in place of the coarser lower bound $\delta_n$, gives
$$|L_{\epsilon_n}(\bm Y_{k+1}^{(n)})-L_{\epsilon_n}(\bm Z_k^{(n)})|
\leq\frac{\norm{\bm Y_{k+1}^{(n)}-\bm Z_k^{(n)}}}{|u_k|}.$$
Summing these inequalities and using \eqref{saddle_step_count_new} and \eqref{Y_Z_saddle_new} gives
\begin{align*}
	&\sum_{k=k_0}^{k_1}
	\left| \norm{\bm Y_{k+1}^{(n)}-\bm Y_k^{(n)}} +L_{\epsilon_n}(\bm Y_{k+1}^{(n)}) -L_{\epsilon_n}(\bm Y_k^{(n)})
	\right|\\
	&=\sum_{k=k_0}^{k_1}
	\left| \norm{\bm Y_{k+1}^{(n)}-\bm Y_k^{(n)}} +L_{\epsilon_n}(\bm Y_{k+1}^{(n)}) -L_{\epsilon_n}(\bm Z_k^{(n)}) -\ell_k
	\right|\\
	&\lesssim \sum_k\left[q_n+\eta_n^2
	\norm{\nabla p(\bm Y_k^{(n)})}\right]
	+\sum_k\left[\frac{q_n+\eta_n^2 \norm{\nabla p(\bm Y_k^{(n)})}}{|u_k|}\right]\\
	&\leq \frac{1}{\delta_n}\left(\eta_n+\frac{q_n}{\eta_n}\right),
\end{align*}
where we used boundedness of $\nabla p$ on $U_j$ in the last inequality. The result thus follows.
\end{proof}

\vspace{2mm}

\begin{lemma}[Population bound for the 1NN uphill shift in the tube]
	\label{lem:pop_1nn_tube}
Suppose that Assumptions~\ref{assump:density_reg}, \ref{assump:sep_repulsion}, and \ref{assump:C1_linear} hold. Let $\mathcal{A}_{n}(\bm{x}_0)=U_{\delta_n}(\bm{x}_0)\setminus B^o(\bm{m}^*,\epsilon_n)$ with $\delta_n \downarrow 0$, where $U_{\delta_n}(\bm x_0) :=\left\{\bm y\in \mathcal C: d(\bm y,\Gamma_\infty(\bm x_0))\leq\delta_n\right\}$.
If $\eta_n + \frac{q_n}{\eta_n} = o(\delta_n)$ with $q_n = \left(\frac{\log n}{n}\right)^{\frac{1}{d}}$, then, with probability tending to one,
$$\Phi_n(\bm{y}) = \bm{y}+\eta_n\nabla p(\bm{y}) + \bm{u}_n(\bm{y}), \qquad \norm{\bm{u}_n(\bm{y})}\lesssim q_n$$
uniformly over $\bm{y}\in\mathcal{A}_n(\bm{x}_0)$. 
\end{lemma}

\begin{proof}
All the constants denoted by $C_*$ below are fixed and independent of $n$. By Lemma~\ref{lem:flow_separation_sep},
$$d(\bm\gamma_{\bm x_0}(t),\partial\mathcal C_a)
	\geq 2\delta_n,\qquad t\geq 0.$$
Consequently, $U_{\delta_n}(\bm x_0)\subset\mathcal C_a$, and every point of this tube is at distance at least $\delta_n$ from the separatrix.
	
We focus on the event $\zeta_n\leq C_\zeta q_n$, whose probability tends to one by Lemma~\ref{lem:1NN_sup_bnd}. For any fixed $\bm y\in\mathcal A_n(\bm x_0)$, we define the Euler point $\bm y^+:=\bm y+\eta_n\nabla p(\bm y)$. Let $M_p:=\sup_{\bm z\in\mathcal C}\norm{\nabla p(\bm z)}$. Since the distance to $\partial\mathcal C_a$ is $1$-Lipschitz,
$$d(\bm y^+,\partial\mathcal C_a) \geq\delta_n-M_p\eta_n \geq\frac{\delta_n}{2}$$
for all sufficiently large $n$. Thus, $\bm y^+\in\mathcal C_a\subset\mathcal C$.
	
By the definition of $\zeta_n$, we can find an observation $\bm X(\bm y)$ satisfying
\begin{equation}
\label{Euler_candidate_distance_new}
\norm{\bm X(\bm y)-\bm y^+}\leq C_\zeta q_n.
\end{equation}
We now verify that this observation is admissible. Taylor's theorem and the boundedness of $\nabla^2p$ give that
\begin{align}
	p(\bm y^+)-p(\bm y)
	&=\eta_n\norm{\nabla p(\bm y)}^2
	+O\!\left(\eta_n^2\norm{\nabla p(\bm y)}^2\right)\nonumber\\
	&\geq\frac{3}{4}\eta_n\norm{\nabla p(\bm y)}^2
	\label{Euler_population_gain_new}
\end{align}
uniformly for all sufficiently large $n$. The Lipschitz continuity of
$\nabla p$ also gives that
	\[
	\norm{\nabla p(\bm y^+)}
	\leq(1+C_1\eta_n)\norm{\nabla p(\bm y)}
	\leq C_1\norm{\nabla p(\bm y)}.
	\]
	Applying Taylor's theorem at $\bm y^+$ and using
	\eqref{Euler_candidate_distance_new}, we obtain
	\begin{align}
		p(\bm X(\bm y))-p(\bm y^+)
		&\geq-C_2q_n\norm{\nabla p(\bm y^+)}-C_2q_n^2\nonumber\\
		&\geq-C_2C_1 q_n\norm{\nabla p(\bm y)}-C_2q_n^2.
		\label{Euler_candidate_loss_new}
	\end{align}
Since $p$ is Morse and has
	finitely many critical points, local nondegeneracy, the distance lower
	bound from the separatrix, and
	$\bm y\notin B^o(\bm m^*,\epsilon_n)$ imply that
	\begin{equation}
		\label{gradient_lower_tube_shift_new}
		\norm{\nabla p(\bm y)}
		\geq c_g\min\{\delta_n,\epsilon_n\}
	\end{equation}
	for some fixed $c_g>0$, uniformly over the stated tube and all stated
	starting points. Therefore,
	\begin{align*}
		\frac{q_n}{\eta_n\norm{\nabla p(\bm y)}}
		&\leq C_3\max\left\{
		\frac{q_n}{\eta_n\delta_n},
		\frac{q_n}{\eta_n\epsilon_n}
		\right\}\\
		&=C_3\max\left\{
		\frac{q_n}{\eta_n\delta_n},\frac{1}{C}
		\right\}.
	\end{align*}
	The first term converges to zero by our rate condition $\eta_n+\frac{q_n}{\eta_n} = o(\delta_n)$, while the second can be made arbitrarily small by choosing the fixed constant $C$ sufficiently large. In addition,
	\[
	\frac{q_n^2}{\eta_n\norm{\nabla p(\bm y)}^2}
	=\eta_n\left[\frac{q_n}{\eta_n\norm{\nabla p(\bm y)}}
	\right]^2=o(1)
	\]
	uniformly over the same region. Hence, after choosing $C$ sufficiently large and then taking $n$ sufficiently large,
	\[
	C_1C_2q_n\norm{\nabla p(\bm y)}+C_2q_n^2
	\leq\frac12\eta_n\norm{\nabla p(\bm y)}^2.
	\]
	Equations \eqref{Euler_population_gain_new} and
	\eqref{Euler_candidate_loss_new} now give that
	\[
	p(\bm X(\bm y))-p(\bm y)
	\geq\frac14\eta_n\norm{\nabla p(\bm y)}^2>0.
	\]
	Thus, $\bm X(\bm y)$ is admissible in the minimization defining
	$\Phi_n(\bm y)$. By minimality and
	\eqref{Euler_candidate_distance_new},
	\[
	\norm{\Phi_n(\bm y)-\{\bm y+\eta_n\nabla p(\bm y)\}}
	\leq\norm{\bm X(\bm y)-\bm y^+}
	\leq C_\zeta q_n,
	\]
	which establishes the result.
	
	Finally, for every point in the safe region $\left\{\bm y\in\mathcal C_a: d(\bm y,\partial\mathcal C_a)\geq c\delta_n\right\} \setminus B^o(\bm m^*,\epsilon_n)$ with $\epsilon_n=\frac{Cq_n}{\eta_n}$, the same argument applies with $d(\bm y,\partial\mathcal C_a)\geq c\delta_n$. In particular,
	$\bm y^+\in\mathcal C_a$ for all sufficiently large $n$, and
	nondegeneracy gives
	\[
	\norm{\nabla p(\bm y)}
	\geq c_g'\min\{\delta_n,\epsilon_n\},
	\]
	where $c_g'$ may depend on the fixed constant $c>0$. Therefore, the bound also holds uniformly over this region.
\end{proof}

\vspace{2mm}

\begin{lemma}[Basin invariance of the oracle GGDPC path]
	\label{lem:pop_tube_invariance}
	Suppose that Assumptions~\ref{assump:density_reg} and \ref{assump:sep_repulsion} hold. Let $U_{\delta_n}(\bm{x})=\{\bm{y}\in \mathcal{C}:d(\bm{y},\Gamma_\infty(\bm{x}))\leq \delta_n\}\subset\mathcal{C}_a$
	for some sequence $\delta_n\downarrow 0$, and let $\epsilon_n=\frac{Cq_n}{\eta_n}$ for some sufficiently large constant $C>0$ as in the proof of Lemma~\ref{lem:pop_1nn_tube}. If $\eta_n+\frac{q_n}{\eta_n}=o(\delta_n)$ with $q_n = \left(\frac{\log n}{n}\right)^{\frac{1}{d}}$, then there exists a fixed $c_*>0$ such that, uniformly over
	$\bm x\in\mathcal C_a\ominus(2\delta_n/C_{\mathcal S})$, with
	probability tending to one, the oracle GGDPC path $\left\{\bm{Y}_k^{(n)}\right\}_{k\geq 0}$ starting at $\bm{x}$ remains in
	$\mathcal C_a$ until it enters $B(\bm m^*,\epsilon_n)$ and
$$\inf_{0\leq k<T_{n,\epsilon_n}} d(\bm Y_k^{(n)},\partial\mathcal C_a) \geq c_*\delta_n.$$
\end{lemma}

\begin{proof}
By Lemma~\ref{lem:flow_separation_sep}, the population path from every
stated starting point satisfies
$$d(\bm\gamma_{\bm x}(t),\partial\mathcal C_a) \geq2 \delta_n, \qquad t\geq0.$$ 
Let $T^{\rm exit}_n=\inf\{k\geq0:\bm Y_k^{(n)}\notin U_{\delta_n}(\bm{x})\}$ and $T_{n,\epsilon_n}=\inf\{k\geq0:\bm Y_k^{(n)}\in B(\bm{m}^*,\epsilon_n)\}$.
	We will prove that $\mathbb{P}(T^{\rm exit}_n<T_{n,\epsilon_n})\to 0$. 
	
	On the event $k< \min\left\{T^{\rm exit}_n,\, T_{n,\epsilon_n}\right\}$, we have $\bm Y_k^{(n)}\in U_{\delta_n}(\bm{x})\setminus B(\bm{m}^*,\epsilon_n)$, so Lemma~\ref{lem:pop_1nn_tube} gives
	\[
	\bm Y_{k+1}^{(n)}
	=\bm Y_k^{(n)}+\eta_n\nabla p(\bm Y_k^{(n)})+\bm u_k,
	\qquad
	\norm{\bm u_k}\lesssim q_n .
	\]
First, away from neighborhoods of the critical points, we let $\bm Z_k^{(n)}=\bm\gamma_{\bm Y_k^{(n)}}(\eta_n)$ be the exact flow endpoint after time $\eta_n$ starting from $\bm Y_k^{(n)}$. The proof of Lemma~\ref{lem:discrete_saddle_passage} shows that
	$$\norm{\bm Y_{k+1}^{(n)}-\bm Z_k^{(n)}} \lesssim q_n+\eta_n^2 \norm{\nabla p(\bm Y_k^{(n)})}.$$
Also, the gradient vector field is bounded away from zero and has bounded derivatives, so our arguments in Lemma~\ref{lem:GGDPC_path_eps} imply that $\min\left\{T^{\rm exit}_n,\, T_{n,\epsilon_n}\right\}=O_P\left(\frac{1}{\eta_n}\right)$. Hence, the accumulated transverse error of GGDPC from the population gradient flow up to time $\min\left\{T^{\rm exit}_n,\, T_{n,\epsilon_n}\right\}$ is of order $O_P\left(\eta_n+\frac{q_n}{\eta_n}\right) = o(\delta_n)$.
	
Second, near the local mode $\bm{m}^*$, Lemma~\ref{lem:pop_1nn_bnd} gives a contraction, so the same bound does not increase before $B(\bm{m}^*,\epsilon_n)$ is reached. 
	
Third, near a boundary saddle point, Lemma~\ref{lem:discrete_saddle_passage} applies.

Combining these three regions shows that the event $T^{\rm exit}_n>T_{n,\epsilon_n}$ holds with probability tending to one. Hence, the oracle GGDPC path remains in $U_{\delta_n}(\bm{x})\subset\mathcal{C}_a$ until it enters $B(\bm{m}^*,\epsilon_n)$, and $\inf_{0\leq k<T_{n,\epsilon_n}} d(\bm Y_k^{(n)},\partial\mathcal C_a) \geq c_*\delta_n$ follows accordingly.
\end{proof}

\subsection{Main Proof of \autoref{thm:GGDPC_path_full}}

\begin{proof}[Proof of \autoref{thm:GGDPC_path_full}]
Choose $\epsilon_n= \frac{C q_n}{\eta_n}$ with $C>0$ sufficiently large. Let $\mathcal{A}_n=U_{\delta_n}(\bm{x})\setminus B(\bm{m}^*,\epsilon_n)$, where $U_{\delta_n}(\bm{x}) =\left\{\bm{y}\in \mathcal{C}: d(\bm{y}, \Gamma_\infty(\bm{x})) \leq \delta_n\right\}$ and $\Gamma_\infty(\bm{x}) = \left\{\bm{\gamma}_{\bm{x}}(t): 0\leq t < \infty\right\}$. Lemma~\ref{lem:pop_tube_invariance} implies that the oracle GGDPC path remains in $U_{\delta_n}(\bm{x})$ until $T_{n,\epsilon_n}$ with probability tending to one. Hence, Lemma~\ref{lem:pop_1nn_tube} may be applied inductively to every iterate before reaching $B(\bm{m}^*,\epsilon_n)$. We decompose 
\begin{equation}
	\label{L_length_decomp_full}
	\left|L_n(\bm{x}) - L(\bm{x}) \right| \leq \underbrace{\left|L_{\epsilon_n}(\bm{x}) - L(\bm{x}) \right|}_{\textbf{Term I}} + \underbrace{\left|L_{n,\epsilon_n}(\bm{x}) - L_{\epsilon_n}(\bm{x}) \right|}_{\textbf{Term II}} + \underbrace{\left|L_n(\bm{x}) - L_{n,\epsilon_n}(\bm{x}) \right|}_{\textbf{Term III}}.
\end{equation}
We know from Lemmas~\ref{lem:mode_geometry} and \ref{lem:oracle_modal_core} that
$$\textbf{Term I} = O(\epsilon_n), \quad \textbf{Term III} = O_P\left(\frac{q_n\log n}{\eta_n^{d+1}}\right).$$
The major changes lie in the derivations for the rate of convergence for \textbf{Term II}. By Lemma~\ref{lem:pop_1nn_tube}, for $k<T_{n,\epsilon_n}$,
$$\bm{S}_n(\bm{Y}_k^{(n)})=\eta_n\nabla p(\bm{Y}_k^{(n)})+\bm{u}_k,
\qquad \norm{\bm{u}_k}\lesssim q_n + \eta_n^2 \norm{\nabla p(\bm{Y}_k^{(n)})}.$$
Away from saddle neighborhoods, similar to Lemma~\ref{lem:pop_1nn_bnd}, we apply Taylor expansion to $\bm{x}\mapsto L_{\epsilon_n}(\bm{x})$ at $\bm{Y}_k^{(n)}$ as:
\begin{equation}
	\label{L_taylor2}
	\begin{split}
		&L_{\epsilon_n}(\bm{Y}_{k+1}^{(n)}) - L_{\epsilon_n}(\bm{Y}_k^{(n)}) \\
		&= \nabla L_{\epsilon_n}(\bm{Y}_k^{(n)})^T \left[\bm{Y}_{k+1}^{(n)} - \bm{Y}_k^{(n)}\right] + O\left(\norm{\nabla^2 L_{\epsilon_n}(\bm{Y}_k^{(n)})}_2 \norm{\bm{S}_n(\bm{Y}_k^{(n)})}^2\right) \\
		&\stackrel{\text{(i)}}{=} -\eta_n \norm{\nabla p(\bm{Y}_k^{(n)})} + O\left(\zeta_n + \norm{\nabla^2 L_{\epsilon_n}(\bm{Y}_k^{(n)})}_2 \left[ \eta_n^2\norm{\nabla p(\bm{Y}_k^{(n)})}^2 + \zeta_n^2\right]\right) \\
		&\stackrel{\text{(ii)}}{=} -\eta_n \norm{\nabla p(\bm{Y}_k^{(n)})} + O\left(\zeta_n + \eta_n^2 \norm{\bm{Y}_k^{(n)} - \bm{m}^*} +  \frac{\zeta_n^2}{\max\left\{\epsilon_n, \norm{\bm{Y}_k^{(n)} - \bm{m}^*}\right\}} \right),
	\end{split}
\end{equation}
where (i) uses the identity $\nabla L_{\epsilon_n}(\bm{x})^T\nabla p(\bm{x})=-\norm{\nabla p(\bm{x})}$ and the uniform boundedness of $\norm{\nabla L_{\epsilon_n}(\bm{x})}$ on $\mathcal{C}_{a,\epsilon_n}\ominus \delta$ for fixed $\delta >0$, while (ii) leverages Lemma~\ref{lem:mode_geometry} when $\bm{x}$ is near $\bm{m}^*$ and the upper bound on the eigenvalue of $\nabla^2 L_{\epsilon_n}(\bm{x})$ for $\bm{x}$ near $\partial B(\bm{m}^*, \epsilon_n)$ as:
$$\norm{\nabla^2 L_{\epsilon_n}(\bm{x})}_2 \lesssim \frac{1}{\max\left\{\epsilon_n, \norm{\bm{x} -\bm{m}^*}\right\}}.$$
As a result, we derive that
\begin{align}
	\label{L_path2}
	\begin{split}
		\textbf{Term II} &\stackrel{\text{(iii)}}{=} \left|\sum_{k=0}^{T_{n,\epsilon_n} -1} \left[\norm{\bm{S}_n(\bm{Y}_k^{(n)})} + L_{\epsilon_n}(\bm{Y}_{k+1}^{(n)}) - L_{\epsilon_n}(\bm{Y}_k^{(n)}) \right] \right|\\
		&\stackrel{\text{(iv)}}{=} \left|\sum_{k=0}^{T_{n,\epsilon_n} -1} \left[\eta_n \norm{\nabla p(\bm{Y}_k^{(n)})} + O(\zeta_n) - \eta_n \norm{\nabla p(\bm{Y}_k^{(n)})} \right.\right.\\
		&\quad \left.\left. + O\left(\zeta_n + \eta_n^2 \norm{\bm{Y}_k^{(n)} - \bm{m}^*} +  \frac{\zeta_n^2}{\max\left\{\epsilon_n, \norm{\bm{Y}_k^{(n)} - \bm{m}^*}\right\}} \right)\right] \right|\\
		&\stackrel{\text{(v)}}{\leq} C_2\left(T_{n,\epsilon_n} \cdot q_n + \eta_n^2 \sum_{k=0}^{T_{n,\epsilon_n}-1} \norm{\bm{Y}_k^{(n)} - \bm{m}^*}\right)
	\end{split}
\end{align}
for some constant $C_2>0$,
where (iii) uses the fact that $\bm{Y}_0^{(n)}=\bm{x}$ and $\bm{Y}_{T_{n,\epsilon_n}}^{(n)} \in B(\bm{m}^*,\epsilon_n)$ with $L_{\epsilon_n}(\bm{Y}_{T_{n,\epsilon_n}}^{(n)})=0$, (iv) plugs in \eqref{L_taylor2}, and (v) uses our rate condition $\epsilon_n\asymp \frac{Cq_n}{\eta_n}$ and $\zeta_n=O_P(q_n)$ to argue that $\frac{\zeta_n^2}{\max\left\{\epsilon_n, \norm{\bm{Y}_k^{(n)} - \bm{m}^*}\right\}} = o(q_n)$ as $\eta_n \to 0$.

Let $k_0$ be the first index at which the path enters a fixed neighborhood $B(\bm{m}^*,\epsilon_0)$ of the local mode. On the preceding regular segment, the argument of Lemma~\ref{lem:GGDPC_path_eps} contributes $O_P\left(\eta_n+\frac{q_n}{\eta_n}\right)$ to \textbf{Term II}. Inside $B(\bm{m}^*,\epsilon_0) \setminus B(\bm{m}^*,\epsilon_n)$, Lemma~\ref{lem:pop_1nn_bnd} gives
$$\norm{\bm{Y}_{k+1}^{(n)} -\bm{m}^*} \leq (1-C_4\eta_n)\norm{\bm{Y}_k^{(n)} -\bm{m}^*}$$ 
for every $k=k_0,...,T_{n,\epsilon_n}-1$, simultaneously with probability tending to one, where $C_4\in(0,1)$ is fixed. Combining the $O_P(\eta_n^{-1})$ duration of the regular segment with this exponential decay yields
$$T_{n,\epsilon_n} = O_P\left(\frac{1+|\log \epsilon_n|}{\eta_n}\right).$$
Moreover, the modal portion of the last sum in \eqref{L_path2} satisfies
$$\eta_n^2 \sum_{k=k_0}^{T_{n,\epsilon_n}-1} \norm{\bm{Y}_k^{(n)} - \bm{m}^*} \lesssim \eta_n^2 \sum_{r=0}^{\infty} (1-C_4\eta_n)^r \norm{\bm{Y}_{k_0}^{(n)} - \bm{m}^*} = O_P\left(\eta_n\right),$$
while its regular portion is also $O_P(\eta_n)$ because it contains $O_P(\eta_n^{-1})$ bounded summands.
Thus, we conclude that
$$\left|L_{n,\epsilon_n}(\bm{x}) - L_{\epsilon_n}(\bm{x}) \right| = O_P\left(\eta_n + \frac{q_n}{\eta_n}\left|\log \epsilon_n\right|\right).$$
	
Inside a saddle neighborhood, by our arguments in Lemma~\ref{lem:discrete_saddle_passage}, we know that
	$$\left|L_{n,\epsilon_n}(\bm{x}) - L_{\epsilon_n}(\bm{x})\right| =O_P\left(\frac{\eta_n}{\delta_n} + \frac{q_n}{\eta_n\delta_n}\right).$$
	Consequently, we bound \textbf{Term II} as:
	\begin{align*}
		\textbf{Term II} &= O_P\left(\eta_n + \frac{q_n}{\eta_n}\left|\log\left(\frac{q_n}{\eta_n}\right)\right|\right) + O_P\left(\frac{\eta_n}{\delta_n} + \frac{q_n}{\eta_n\delta_n}\right) \\
		&= O_P\left(\frac{q_n}{\eta_n}\left|\log\left(\frac{q_n}{\eta_n}\right)\right| + \frac{\eta_n}{\delta_n} + \frac{q_n}{\eta_n\delta_n}\right)
	\end{align*}
	when $\frac{1}{\delta_n}$ is eventually bounded below by a positive constant and $\epsilon_n=\frac{Cq_n}{\eta_n}$.
	
	In summary, combining our new bounds for \textbf{Term I}, \textbf{Term II}, and \textbf{Term III} with \eqref{L_length_decomp_full}, we obtain that
	\begin{align*}
		\left|L_n(\bm{x}) - L(\bm{x}) \right| &= O_P\left(\frac{q_n}{\eta_n}\left|\log\left(\frac{q_n}{\eta_n}\right)\right| + \frac{\eta_n}{\delta_n} + \frac{q_n}{\eta_n\delta_n} + \frac{q_n\log n}{\eta_n^{d+1}} + \epsilon_n\right) \\
		&= O_P\left(\frac{q_n}{\eta_n}\left|\log\left(\frac{q_n}{\eta_n}\right)\right| + \frac{\eta_n}{\delta_n} + \frac{q_n}{\eta_n\delta_n} + \frac{q_n\log n}{\eta_n^{d+1}} \right).
	\end{align*}
Since $\bm{x}\in \mathcal C_a\ominus \left(\frac{2\delta_n}{C_{\mathcal{S}}}\right)$ is arbitrary, we can take a supremum over this shrinking region.
\end{proof}

\section{Proof of Theorem~\ref{thm:GGDPC_estimated}}
\label{app:proof_ggdpc_estimated}

We begin by stating a key local sample uphill condition that guarantees the sample bound and invariance of the sample GGDPC path within the basin of attraction $\mathcal{C}_a$. Then, we provide two sufficient conditions for this local sample uphill condition through a proposition, one of which resembles Assumption~\ref{assump:sample_den_diff} in the main paper. Finally, we combine these results with other supporting lemmas to present the main proof of \autoref{thm:GGDPC_estimated}.

For the sample GGDPC path $\hat{\bm{Y}}_{k+1}^{(n)} = \hat{\Phi}_n(\hat{\bm{Y}}_k^{(n)})$ with $\hat{\bm{Y}}_0^{(n)}=\bm{x}$, we define
$$\hat T_{n,\epsilon}= \min\left\{ \inf\left\{k\geq 0: \hat{\bm{Y}}_k^{(n)}\in \overline{B(\bm{m}^*,\epsilon)}\right\}, \hat{T}_n \right\},
\qquad
\hat L_{n,\epsilon}(\bm{x}) = \sum_{k=0}^{\hat{T}_{n,\epsilon}-1} \norm{\hat{\bm{S}}_n(\hat{\bm{Y}}_k^{(n)})}.$$
The total sample path length $\hat L_n(\bm{x})$ is defined using the terminal sample mode, as in \autoref{subsec:stability_sample}.

\subsection{Local Sample Uphill Condition}

Let $\Gamma_\infty(\bm{x}_0)=\{\bm{\gamma}_{\bm{x}_0}(t): 0\leq t<\infty\}$ be the gradient flow trajectory with $\bm{x}_0\in \mathcal{C}_a$.

\begin{assump}[Local sample uphill availability]
	\label{assump:sample_uphill}
	Let $U_{\delta_n}(\bm{x}_0)=\{\bm{y}\in \mathcal{C}:d(\bm{y},\Gamma_\infty(\bm{x}))\leq\delta_n\}$ and $\mathcal A_n(\bm{x}_0)=U_{\delta_n}(\bm{x}_0)\setminus B^o(\bm{m}^*,\epsilon_n)$ for sequences $\delta_n\downarrow 0$ and $\epsilon_n\downarrow 0$.
	There exists a fixed constant $C_{\rm up}\geq 1$ such that
	$$\liminf_{n\to \infty}\mathbb{P}\left(\left\{\forall \bm{y}\in \mathcal{A}_n(\bm{x}_0), \exists \bm{X}_i\in B\left(\bm{y}+\eta_n\hat{g}(\bm{y}),\,C_{\rm up}\cdot \zeta_n\right)
	\text{ such that }\hat p(\bm{X}_i)>\hat p(\bm{y})
	\right\}\right)=1.$$
\end{assump}

Assumption~\ref{assump:sample_uphill} ensures that, for sufficiently large $n$, a one-step gradient ascent update moves the current iterate into a neighborhood containing observations of higher density. The following proposition provides two sufficient conditions for this assumption. 

\begin{proposition}[Sufficient conditions for local sample uphill availability]
\label{prop:sample_uphill_sufficient}
Suppose that Assumptions~\ref{assump:density_reg}, \ref{assump:sep_repulsion}, and \ref{assump:C1_linear} hold, as well as $d(\bm{x}_0,\partial\mathcal{C}_a)\geq \frac{2\delta_n}{C_{\mathcal{S}}}$, where $C_{\mathcal{S}}$ is the constant in Lemma~\ref{lem:flow_separation_sep}. Assume further that $\eta_n + \frac{q_n}{\eta_n} = o(\min\{\epsilon_n,\delta_n\})$ and $\norm{\hat g-\nabla p}_\infty=o_P(\min\{\epsilon_n,\delta_n\})$ with $q_n=\left(\frac{\log n}{n}\right)^{\frac{1}{d}}$. 
Then, Assumption~\ref{assump:sample_uphill} holds with $C_{\rm up}=1$ under either of the following conditions.
\begin{enumerate}[label=(\alph*)]
	\item For some fixed constant $C>0$, the density estimator $\hat p$ is twice continuously differentiable on 
	$$\mathcal A_n^+ =
	\left\{\bm{y} \in \mathbb{R}^d: d(\bm{y},\mathcal A_n) \leq C\left(\eta_n+\zeta_n\right)\right\},$$
	$\hat g=\nabla\hat p$ on $\mathcal A_n^+$, and $\sup_{\bm{y}\in\mathcal A_n^+}
	\norm{\nabla^2\hat p(\bm{y})}_{\max}=O_P(1)$.
	
	\item $\sup\limits_{\bm{x}\in\mathcal A_n} \sup\limits_{\substack{1\leq i\leq n:\\ \bm{X}_i\in B(\hat{\bm{x}}^+(\bm{x}),\zeta_n)}}
	\left| \left[\hat p(\bm{X}_i)-\hat p(\bm{x})\right] - \left[p(\bm{X}_i) - p(\bm{x})\right]\right| =o_P\left(\eta_n \min\left\{\epsilon_n^2,\delta_n^2\right\}\right)$.
\end{enumerate}
\end{proposition}

\begin{proof}
	As shown in Lemma~\ref{lem:pop_1nn_tube}, near the boundary saddle point or local mode regions, there exists a constant $C_1>0$ such that, with probability tending to one, $\norm{\nabla p(\bm{x})}\geq C_1 \min\{\delta_n,\epsilon_n\}$ for all $\bm{x}\in\mathcal{A}_n$. Since $\norm{\hat g-\nabla p}_\infty=o_P(\min\{\epsilon_n,\delta_n\})$, then we know that
	$$\norm{\hat g(\bm{x})}\geq \frac{C_1}{2} \min\{\delta_n,\epsilon_n\}$$
	uniformly over $\mathcal A_n$ with probability tending to one. Moreover, by Assumption~\ref{assump:density_reg} and $\norm{\hat g-\nabla p}_\infty=o_P(\min\{\epsilon_n,\delta_n\})$, we have that $\sup_{\bm{x}\in\mathcal A_n}\norm{\hat g(\bm{x})}=O_P(1)$.
	
	Now, by Lemma~\ref{lem:flow_separation_sep} and
	$d(\bm{x}_0,\partial\mathcal{C}_a)\geq \frac{2\delta_n}{C_{\mathcal{S}}}$, we deduce that $\inf_{t\geq 0} d\left(\bm\gamma_{\bm{x}_0}(t),\partial\mathcal{C}_a\right)
	\geq 2\delta_n$.
	Hence, for every $\bm{x}\in\mathcal A_n(\bm{x}_0)\subset U_{\delta_n}(\bm{x}_0)$,
	$$d(\bm{x},\partial\mathcal{C}_a) \geq
	\inf_{t\geq 0} d\left(\bm\gamma_{\bm{x}_0}(t),\partial\mathcal{C}_a\right)
	-d\bigl(\bm{x},\Gamma_\infty(\bm{x}_0)\bigr)
	\geq \delta_n.$$
	Since the distance to the closed set $\partial\mathcal{C}_a$ is $1$-Lipschitz, we have that
	\begin{align*}
		d\bigl(\bm{x} + \eta_n \hat{g}(\bm{x}),\partial\mathcal{C}_a\bigr) \geq \delta_n-\eta_n\norm{\hat g(\bm{x})}.
	\end{align*}
	Because $\eta_n=o(\delta_n)$ and
	$\sup_{\bm{x}\in\mathcal A_n}\norm{\hat g(\bm{x})}=O_P(1)$, we know that
	$\inf_{\bm{x}\in\mathcal A_n}
	d\bigl(\bm{x} + \eta_n \hat{g}(\bm{x}),\partial\mathcal{C}_a\bigr)>C_2 \delta_n$ for some constant $C_2>0$
	with probability tending to one. In particular, the one-step gradient update satisfies that $\bm{x} + \eta_n \hat{g}(\bm{x}) \in \mathcal{C}_a$ for all $\bm{x}\in \mathcal{A}_n$.

	(a) Fix $\bm{x}\in\mathcal A_n$ and let
	$\bm{y}\in B(\bm{x} + \eta_n \hat{g}(\bm{x}),\zeta_n)$. Since $\hat{p}$ is twice continuously differentiable and $\hat{g}=\nabla \hat{p}$, Taylor's expansion of $\hat{p}$ around $\bm{x}$ yields
	\begin{align*}
		\hat{p}(\bm{y}) - \hat{p}(\bm{x}) & \stackrel{\text{(i)}}{\geq} \eta_n\norm{\hat{g}(\bm{x})}^2 - \zeta_n\norm{\hat{g}(\bm{x})} - C_2\eta_n^2 \norm{\hat{g}(\bm{x})}^2 +O_P(\zeta_n^2) \\
		&\stackrel{\text{(ii)}}{\geq} C_3\eta_n\norm{\hat{g}(\bm{x})}^2
	\end{align*}
	for some constants $C_2,C_3>0$, where (i) leverages the condition that $\sup_{\bm{y}\in\mathcal A_n^+} \norm{\nabla^2\hat p(\bm{y})}_{\max}=O_P(1)$ and (ii) uses the condition $\eta_n + \frac{q_n}{\eta_n}=o(\min\{\epsilon_n,\delta_n\})$ while $\norm{\hat g(\bm{x})}\geq \frac{C_1}{2} \min\{\delta_n,\epsilon_n\}$.
	Hence,
	$$\inf_{\bm{x}\in\mathcal A_n}
	\inf_{\bm{y}\in B(\bm{x} + \eta_n\hat{g}(\bm{x}),\zeta_n)}
	\{\hat p(\bm{y})-\hat p(\bm{x})\}>0$$
	with probability tending to one. This in turn implies that, for every $\bm{x}\in\mathcal A_n$, there is an admissible observation
	$\bm{X}_i\in B(\bm{x} + \eta_n\hat{g}(\bm{x}),\zeta_n)$ with $\hat p(\bm X_i)>\hat p(\bm x)$, so Assumption~\ref{assump:sample_uphill} holds with $C_{\rm up}=1$ and probability tending to one.
	
	(b) By the definition of $\zeta_n$, we know that there exists at least one sample point $\bm{X}_i \in B(\bm{x} + \eta_n\hat{g}(\bm{x}), \zeta_n)$. Also, the Taylor's expansion argument in Lemma~\ref{lem:pop_1nn_tube}, with
	$\hat{\bm{x}}^+=\bm{x}+\eta_n\hat g(\bm{x})$ in place of
	$\bm{x}+\eta_n\nabla p(\bm{x})$, gives that
	\begin{align*}
		p(\hat{\bm{x}}^+)-p(\bm{x}) &=\eta_n\norm{\nabla p(\bm{x})}^2 +\eta_n\nabla p(\bm{x})^T\left[\hat{g}(\bm{x}) - \nabla p(\bm{x})\right] + O\left(\eta_n^2\left[\norm{\nabla p(\bm{x})}+\norm{\hat{g} - \nabla p}_{\infty}\right]^2\right)\\
		&\geq \eta_n\norm{\nabla p(\bm{x})}^2 - \eta_n\norm{\nabla p(\bm{x})} \norm{\hat{g} - \nabla p}_{\infty} +O\left(\eta_n^2\left[\norm{\nabla p(\bm{x})}+\norm{\hat{g} - \nabla p}_{\infty}\right]^2\right) \\
		&\stackrel{\text{(iii)}}{\geq} C_4 \eta_n \norm{\nabla p(\bm{x})}^2
	\end{align*}
	for some constant $C_4>0$, where (iii) utilizes the fact that near the boundary saddle point and local mode regions, $\norm{\nabla p(\bm{x})} \gtrsim \min\{\epsilon_n,\delta_n\}$. Additionally, since $\norm{\hat{g} - \nabla p}_{\infty}=o_P(\min\{\epsilon_n,\delta_n\})$, the cross term is $o_P(\norm{\nabla p(\bm{x})}^2)$ uniformly. Then, for $\bm{X}_i \in B(\bm{x} + \eta_n\hat{g}(\bm{x}), \zeta_n)$, by the Lipschitz continuity of $p$, we know that
	\begin{align*}
		\hat{p}(\bm{X}_i) &= p(\bm{X}_i) + \left[\hat{p}(\bm{X}_i) - p(\bm{X}_i)\right] \\
		&\geq p(\hat{\bm{x}}^+) - C_5 \norm{\nabla p(\hat{\bm{x}}^+)} \zeta_n + O(\zeta_n^2) + \left[\hat{p}(\bm{X}_i) - p(\bm{X}_i)\right]\\
		&\geq \hat{p}(\bm{x}) + C_4\eta_n\norm{\nabla p(\bm{x})}^2 - C_5 \norm{\nabla p(\hat{\bm{x}}^+)} \zeta_n + O(\zeta_n^2)+ \left[\hat{p}(\bm{X}_i) - p(\bm{X}_i)\right] - \left[\hat{p}(\bm{x}) - p(\bm{x})\right] \\
		&\stackrel{\text{(iv)}}{\geq} \hat{p}(\bm{x}) + C_6\eta_n\norm{\nabla p(\bm{x})}^2 + \left[\hat{p}(\bm{X}_i) - p(\bm{X}_i)\right] - \left[\hat{p}(\bm{x}) - p(\bm{x})\right]
	\end{align*}
	for some constants $C_5, C_6 >0$, where (iv) uses the rate condition $\frac{q_n}{\eta_n}=o(\min\left\{\delta_n,\epsilon_n\right\})$ with $\zeta_n=O_P(q_n)$ to argue that $C_5 \norm{\nabla p(\hat{\bm{x}}^+)} \zeta_n + O(\zeta_n^2) = o_P(\eta_n\norm{\nabla p(\bm{x})}^2)$ uniformly on $\mathcal{A}_n$. Finally, using the assumption that $\sup\limits_{\bm{x}\in\mathcal A_n} \sup\limits_{\substack{1\leq i\leq n:\\ \bm{X}_i\in B(\hat{\bm{x}}^+(\bm{x}),\zeta_n)}}
	\left| \left[\hat p(\bm{X}_i)-\hat p(\bm{x})\right] - \left[p(\bm{X}_i) - p(\bm{x})\right]\right| =o_P\left(\eta_n \min\left\{\epsilon_n^2,\delta_n^2\right\}\right)$, we conclude that
	$$\left| \left[\hat p(\bm{X}_i)-\hat p(\bm{x})\right] - \left[p(\bm{X}_i) - p(\bm{x})\right]\right| =o_P\left(\eta_n\norm{\nabla p(\bm{x})}^2\right)$$
	and $\hat{p}(\bm{X}_i) > \hat{p}(\bm{x})$ so that the observation $\bm{X}_i \in B(\bm{x} + \eta_n\hat{g}(\bm{x}), \zeta_n)$ is admissible for increasing the estimated density $\hat{p}$ after an estimated gradient update uniformly over $\bm{x}\in \mathcal{A}_n$. Assumption~\ref{assump:sample_uphill} thus follows with $C_{\rm up}=1$.
\end{proof}

\subsection{Supporting Lemmas}

%
%

\begin{lemma}[Sample GGDPC path through a boundary saddle]
\label{lem:sample_discrete_saddle_passage}
Suppose that Assumptions~\ref{assump:density_reg}, \ref{assump:sep_repulsion}, and \ref{assump:C1_linear} hold. For a
boundary saddle point $\bm s_j$ and a sufficiently small fixed
neighborhood $U_j$, we consider a finite path segment $\{\hat{\bm Y}_k^{(n)}\}_{k=k_0}^{k_1+1}\subset U_j\cap\mathcal C_a$ satisfying
\begin{equation}
	\label{sample_perturbed_euler_saddle}
	\hat{\bm Y}_{k+1}^{(n)}
	=\hat{\bm Y}_k^{(n)}
	+\eta_n\nabla p(\hat{\bm Y}_k^{(n)})+\bm R_k,
	\qquad
	\norm{\bm R_k}
	\leq C\{q_n+\eta_n\norm{\hat g-\nabla p}_\infty\}.
\end{equation}
Let $\delta_n\downarrow 0$.  If $d(\hat{\bm Y}_{k_0}^{(n)},\partial\mathcal C_a) \geq c_0\delta_n$ and $\eta_n+\frac{q_n}{\eta_n} +\norm{\hat g-\nabla p}_\infty=o_P(\delta_n)$, then, with probability tending to one, the segment remains on the same side of the local stable manifold and, for some fixed $c_1>0$,
$$\inf_{k_0\leq k\leq k_1+1} d(\hat{\bm Y}_k^{(n)},\partial\mathcal C_a) \geq c_1\delta_n.$$
If the segment traverses $U_j$ once and $U_j\cap B(\bm m^*,\epsilon_n)=\emptyset$, then
$$\sum_{k=k_0}^{k_1}
	\left|\norm{\hat{\bm Y}_{k+1}^{(n)}-\hat{\bm Y}_k^{(n)}}
	+L_{\epsilon_n}(\hat{\bm Y}_{k+1}^{(n)}) -L_{\epsilon_n}(\hat{\bm Y}_k^{(n)})\right|
	=O_P\left(\frac{1}{\delta_n}\left[\eta_n+\frac{q_n}{\eta_n} +\norm{\hat g-\nabla p}_\infty\right]\right).$$
\end{lemma}

\begin{proof}
All the constants denoted by $C_*$ below are fixed and independent of $n$. Let $\Psi_j$ be the $C^2$ chart used in the proof of Lemma~\ref{lem:discrete_saddle_passage}, and write
$$\Psi_j(\hat{\bm Y}_k^{(n)})=(u_k,\bm v_k).$$
	The local stable manifold is $\{u=0\}$ and the first coordinate of the population vector field has the factorization
	\begin{equation}
		\label{sample_saddle_normal_factor}
		D\Psi_{j,1}(\bm y)\nabla p(\bm y)
		=a(\Psi_j(\bm y))u(\bm y),
		\qquad
		a(0,\bm0)=\lambda_j>0.
	\end{equation}
	After decreasing $U_j$, assume $a\geq \frac{\lambda_j}{2}$ and
	$\norm{\nabla p}_{\infty} \leq 1$ on $U_j$. Let
	$\bm\Delta_k=\eta_n\nabla p(\hat{\bm Y}_k^{(n)})+\bm R_k$.
	Taylor's theorem for $\Psi_{j,1}$, \eqref{sample_perturbed_euler_saddle}, and \eqref{sample_saddle_normal_factor} give that
	\begin{equation}
		\label{sample_saddle_u_recurrence}
		u_{k+1}=u_k\left[1+\eta_na(u_k,\bm v_k)\right] +\xi_{k,n},
		\qquad
		|\xi_{k,n}| \leq C_1\left[q_n +\eta_n\norm{\hat g-\nabla p}_\infty
		+\eta_n^2\norm{\nabla p(\hat{\bm Y}_k^{(n)})}\right].
	\end{equation}
	Indeed, the linear term in $\bm R_k$ is bounded by the first two terms
	on the right, while
	$\norm{\bm\Delta_k}^2$ is bounded by the same expression after
	increasing $C_1$.
	
	The chart $\Psi_j$ is bi-Lipschitz, and the only part of
	$\partial\mathcal C_a$ in $U_j$ is the local stable manifold.  Hence
	\begin{equation}
		\label{sample_saddle_chart_distance}
		c|u_k|\leq
		d(\hat{\bm Y}_k^{(n)},\partial\mathcal C_a)
		\leq C|u_k|,
		\qquad
		|u_{k_0}|\geq c'\delta_n.
	\end{equation}
	On an event whose probability tends to one, $\eta_n+\frac{q_n}{\eta_n} +\norm{\hat g-\nabla p}_\infty=o_P(\delta_n)$
	and \eqref{sample_saddle_u_recurrence} imply that $|\xi_{k,n}|
	\leq\frac{\lambda_j}{4}\eta_n|u_k|$
	whenever $|u_k|\geq|u_{k_0}|$.  Induction therefore yields
	\begin{equation}
		\label{sample_saddle_geometric_growth}
		|u_{k+1}|
		\geq\left(1+\frac{\lambda_j}{4}\eta_n\right)|u_k|,
		\qquad
		\operatorname{sign}(u_{k+1})=\operatorname{sign}(u_k).
	\end{equation}
Hence, $\inf_{k_0\leq k\leq k_1+1} d(\hat{\bm Y}_k^{(n)},\partial\mathcal C_a) \geq c_1\delta_n$ follows.  Since a completed passage exits through a section $|u|=u_*>0$, they also give
\begin{equation}
	\label{sample_saddle_step_sums}
	(k_1-k_0+1)\eta_n
	\leq C\left(1+\log\frac1{\delta_n}\right)
	\leq\frac{C}{\delta_n},
	\qquad
	\sum_{k=k_0}^{k_1}\frac{\eta_n}{|u_k|}
	\leq\frac{C}{|u_{k_0}|}
	\leq\frac{C}{\delta_n}.
\end{equation}

Let $\hat{\bm Z}_k^{(n)}
:=\bm\gamma_{\hat{\bm Y}_k^{(n)}}(\eta_n)$ and $\ell_k:=\int_0^{\eta_n}
\norm{\nabla p(\bm\gamma_{\hat{\bm Y}_k^{(n)}}(t))}\,dt$.
The Lipschitz continuity of $\nabla p$ and
\eqref{sample_perturbed_euler_saddle} imply
\begin{align}
	\norm{\hat{\bm Y}_{k+1}^{(n)}-\hat{\bm Z}_k^{(n)}}
	&\leq C D_{k,n},\notag\\
	\left|\norm{\hat{\bm Y}_{k+1}^{(n)}-\hat{\bm Y}_k^{(n)}}
	-\ell_k\right|
	&\leq C D_{k,n},
	\label{sample_saddle_flow_errors}
\end{align}
where $D_{k,n}:=q_n+\eta_n\norm{\hat g-\nabla p}_\infty
+\eta_n^2\norm{\nabla p(\hat{\bm Y}_k^{(n)})}.$ We also know that
$L_{\epsilon_n}(\hat{\bm Z}_k^{(n)})
-L_{\epsilon_n}(\hat{\bm Y}_k^{(n)})=-\ell_k$ by definition.
Moreover, \eqref{sample_saddle_geometric_growth} and the local estimate
in Lemma~\ref{lem:L_holder} give
\begin{equation}
	\label{sample_saddle_L_comparison}
	\left|L_{\epsilon_n}(\hat{\bm Y}_{k+1}^{(n)})
	-L_{\epsilon_n}(\hat{\bm Z}_k^{(n)})\right|
	\leq\frac{C}{|u_k|}
	\norm{\hat{\bm Y}_{k+1}^{(n)}-\hat{\bm Z}_k^{(n)}}.
\end{equation}
Combining \eqref{sample_saddle_flow_errors} and
\eqref{sample_saddle_L_comparison}, and then summing, yields
\begin{align*}
	&\sum_{k=k_0}^{k_1}
	\left|
	\norm{\hat{\bm Y}_{k+1}^{(n)}-\hat{\bm Y}_k^{(n)}}
	+L_{\epsilon_n}(\hat{\bm Y}_{k+1}^{(n)})
	-L_{\epsilon_n}(\hat{\bm Y}_k^{(n)})
	\right| \leq
	C\sum_{k=k_0}^{k_1}D_{k,n}
	+C\sum_{k=k_0}^{k_1}\frac{D_{k,n}}{|u_k|}.
\end{align*}
Since $\nabla p$ is bounded on $U_j$, \eqref{sample_saddle_step_sums}
gives
$$\sum_kD_{k,n} +\sum_k\frac{D_{k,n}}{|u_k|} \leq \frac{C}{\delta_n}\left(\eta_n+\frac{q_n}{\eta_n} +\norm{\hat g-\nabla p}_\infty\right).$$
The result thus follows.
\end{proof}

\vspace{2mm}

\begin{lemma}[Stability of the critical points of $\hat p$]
\label{lem:estimated_critical_stability}
Suppose that Assumptions~\ref{assump:density_reg} and
\ref{assump:sample_den_diff} hold as well as that
$\norm{\hat g-\nabla p}_\infty=o_P(1)$ and
$\norm{\nabla^2\hat p-\nabla^2p}_\infty=o_P(1)$.  Let
$\{\bm s_1,\ldots,\bm s_J\}$ be the critical points of $p$.  With
probability tending to one, $\hat p$ has exactly one critical point
$\widetilde{\bm s}_{j,n}$ in a fixed neighborhood of $\bm s_j$, has no other critical points near $\mathcal C$, and
$$\max_{1\leq j\leq J} \norm{\widetilde{\bm s}_{j,n}-\bm s_j} =O_P(\norm{\hat g-\nabla p}_\infty).$$
Moreover, there exist fixed constants $r_0>0$ and $0<c'<C'<\infty$ such that, uniformly over $j$ and $\bm y\in B(\widetilde{\bm s}_{j,n},r_0)$,
$$c'\norm{\bm y-\widetilde{\bm s}_{j,n}}\leq \norm{\hat g(\bm y)} \leq C'\norm{\bm y-\widetilde{\bm s}_{j,n}},
\qquad |\hat p(\bm y)-\hat p(\widetilde{\bm s}_{j,n})| \leq C'\norm{\bm y-\widetilde{\bm s}_{j,n}}^2.$$
If $\bm s_j$ is a local mode, then $\nabla^2\hat p$ is uniformly
negative definite on this ball and
$$c'\norm{\bm y-\widetilde{\bm s}_{j,n}}^2 \leq\hat p(\widetilde{\bm s}_{j,n})-\hat p(\bm y) \leq C'\norm{\bm y-\widetilde{\bm s}_{j,n}}^2.$$
\end{lemma}

\begin{proof}
Choose disjoint balls $B(\bm s_j,r_0)$ with $r_0>0$ such that
$$\sup_{\bm y\in B(\bm s_j,r_0)}
\norm{\nabla^2p(\bm y)- \nabla^2p(\bm s_j)}_2
\leq\frac{1}{8}\rho_{\min}(\nabla^2p(\bm s_j)),$$
where $\rho_{\min}(\nabla^2p(\bm s_j))$ is the smallest absolute eigenvalue of the Hessian matrix $\nabla^2p(\bm s_j)$.
On their compact complement, $\norm{\nabla p(\bm{x})}$ is bounded below, so the uniform gradient consistency excludes any critical points of $\hat p$ inside the compact complement. Additionally, with probability tending to one,
$$\sup_{\bm y\in B(\bm s_j,r_0)}
\norm{\nabla^2\hat p(\bm y)-\nabla^2p(\bm s_j)}_2
\leq\frac{1}{4} \rho_{\min}(\nabla^2p(\bm s_j))$$
simultaneously for all $j$. On this event, define $\mathcal{T}_{j,n}(\bm{y}) = \bm{y} - \left[\nabla^2 p(\bm{s}_j)\right]^{-1} \hat{g}(\bm{y})$. Then,
$$\sup_{\bm y\in B(\bm s_j,r_0)}\norm{D\mathcal{T}_{j,n}(\bm y)}_2 = \sup_{\bm y\in B(\bm s_j,r_0)}\norm{I_d - \left[\nabla^2 p(\bm{s}_j)\right]^{-1} \nabla^2\hat{p}(\bm{y})}_2
< \frac{1}{2}.$$
Moreover,
$\norm{\mathcal T_{j,n}(\bm s_j)-\bm s_j}
\leq\norm{A_j^{-1}}_2\norm{\hat g-\nabla p}_\infty=o_P(1)$. After decreasing $r_0$ once, $\mathcal{T}_{j,n}$ maps $\overline{B(\bm s_j,r_0)}$ into itself with probability tending to
one. The contraction mapping theorem gives a unique fixed point
$\widetilde{\bm s}_{j,n}$ in this ball, equivalently a unique zero of
$\hat g=\nabla\hat p$. Thus,
$$\norm{\widetilde{\bm s}_{j,n}-\bm s_j}
\leq C_1\norm{\hat g(\bm s_j)}
=O_P(\norm{\hat g-\nabla p}_\infty)$$
for some constant $C_1>0$.

On the compact complement of the union of these balls, $\norm{\nabla p}$ is bounded away from zero, so the uniform gradient consistency $\norm{\hat{g}-\nabla p}_{\infty}=o_P(1)$ excludes any additional critical points of $\hat p$.
Finally, $\nabla^2\hat{p}$ has the same Morse index (positive and negative eigenvalues) as $\nabla^2 p$ within balls $B(\bm{s}_j,r_0)$, respectively. By Taylor's theorem and our arguments in Lemma~\ref{lem:mode_geometry}, the results naturally follow.
\end{proof}

\vspace{2mm}

\begin{lemma}[Sample update before the terminal modal neighborhood]
	\label{lem:sample_admissible}
Suppose that Assumptions~\ref{assump:density_reg}, \ref{assump:sample_den_diff}, \ref{assump:sep_repulsion}, and \ref{assump:C1_linear} hold. Assume that, for a deterministic sequence $\delta_n\downarrow 0$,  $\eta_n+\frac{q_n}{\eta_n}=o(\delta_n)$, $\delta_n\left|\log\left(\frac{q_n}{\eta_n}\right)\right|=O(1)$, $\frac{q_n\log n}{\eta_n^{d+1}}=o(1)$, $\norm{\hat{p} - p}_{\infty}=o_P(1)$, $\norm{\hat{g} - \nabla p}_{\infty} = o_P(\delta_n)$, and $\norm{\nabla^2 \hat{p} - \nabla^2 p}_{\infty}=o_P(1)$ with $q_n = \left(\frac{\log n}{n}\right)^{\frac{1}{d}}$. Let $\epsilon_n=C\left[\frac{q_n}{\eta_n} + \norm{\hat{g}-\nabla p}_{\infty}\right]$, where $C>0$ is a sufficiently large fixed constant.  For every fixed $c>0$, with probability tending to one, uniformly over $\bm y\in\mathcal C_a\setminus B^o(\bm m^*,\epsilon_n)$ and $d(\bm y,\partial\mathcal C_a)\geq c\delta_n$, we have that
$$\hat\Phi_n(\bm y) =\bm y+\eta_n\nabla p(\bm y)+\bm R_n(\bm y),
\qquad \norm{\bm R_n(\bm y)} \leq  C_1 \left[q_n+\eta_n\norm{\hat g-\nabla p}_\infty + \eta_n^2 \norm{\nabla p(\bm{y})}\right]$$
for some fixed constant $C_1>0$.
\end{lemma}

\begin{proof}
Work on the events $\zeta_n\leq C_\zeta q_n$,
$\norm{\nabla^2\hat p}_{\infty}=O(1)$, and $\norm{\hat{g} - \nabla p}_{\infty} = o_P(\delta_n)$. When $\bm{y}$ is close to $B(\bm{m}^*,\epsilon_n)$ with $\epsilon_n=o_P(\delta_n)$, we know that $\norm{\nabla p(\bm y)}\geq c_0\epsilon_n$. This bound also hold uniformly on $\mathcal C_a\setminus B^o(\bm m^*,\epsilon_n)$, since $\norm{\nabla p(\bm y)}$ is lower bounded by $\min\{\delta_n, 1\}$ in other region.  By choosing $C$ sufficiently large,
$$\norm{\hat g(\bm y)}\geq c_1\epsilon_n, \qquad
\frac{q_n}{\eta_n\norm{\hat g(\bm y)}}\lesssim \frac{1}{C}$$
with probability tending to one. The gradient update
$\bm y+\eta_n\hat g(\bm y)$ remains in $\mathcal C_a$ because
$\eta_n=o(\delta_n)$ and $\hat g$ is uniformly bounded.  Choose an
observation $\bm X(\bm y)$ within $\zeta_n$ of this target.  Taylor's theorem gives
\[
\hat p(\bm X(\bm y))-\hat p(\bm y)
\geq\eta_n\norm{\hat g(\bm y)}^2
-\zeta_n\norm{\hat g(\bm y)}
-C_1\left[\eta_n\norm{\hat g(\bm y)}+\zeta_n\right]^2>0
\]
uniformly, after increasing the constant $C_1$ and then taking $n$ sufficiently large. Thus, $\bm X(\bm y)$ is admissible.  Minimality of
$\hat\Phi_n(\bm y)$ yields
\begin{align*}
\norm{\hat\Phi_n(\bm y)-\left[\bm y+\eta_n\hat g(\bm y)\right]}
&\leq \norm{\bm X(\bm y)-\left[\bm y+\eta_n\hat g(\bm y)\right]}\\
&\leq \zeta_n+ \eta_n\norm{\hat g-\nabla p}_\infty +C\eta_n^2\norm{\nabla p(\bm y)}.
\end{align*}
Adding and subtracting $\eta_n\cdot \hat{g}(\bm y)$ proves the final result.
\end{proof}

\vspace{2mm}

\begin{lemma}[Basin invariance of the sample GGDPC path]
	\label{lem:sample_tube}
	Under the conditions of Lemma~\ref{lem:sample_admissible}, let
	$\epsilon_n=C\left[\frac{q_n}{\eta_n} + \norm{\hat{g}-\nabla p}_{\infty}\right]$.  There is a fixed $c_*>0$ such that, uniformly over $\bm x\in\mathcal C_a\ominus\left(\frac{2\delta_n}{C_{\mathcal S}}\right)$, with probability tending to one, the sample path remains in $\mathcal C_a$
	until it enters $B(\bm m^*,\epsilon_n)$ and
$$\inf_{0\leq k<\hat T_{n,\epsilon_n}}
		d(\hat{\bm Y}_k^{(n)},\partial\mathcal C_a)
		\geq c_*\delta_n.$$
\end{lemma}

\begin{proof}
Choose disjoint fixed neighborhoods of the boundary saddles and a fixed modal neighborhood. On a compact regular segment, Lemma~\ref{lem:sample_admissible} and the exact gradient flow expansion give, with $\bm Z_k=\bm\gamma_{\bm y_0}(k\eta_n)$,
	\[
	\norm{\hat{\bm Y}_{k+1}^{(n)}-\bm Z_{k+1}}
	\leq(1+C_1\eta_n)\norm{\hat{\bm Y}_k^{(n)}-\bm Z_k}
	+C_1\left[q_n+\eta_n\norm{\hat g-\nabla p}_\infty+\eta_n^2\right].
	\]
	The population passage time of the segment is uniformly bounded.  Hence, we know that
	\begin{equation}
		\label{sample_regular_tracking_full}
	\begin{split}
	\max_k\norm{\hat{\bm Y}_k^{(n)}-\bm Z_k} &\lesssim \max_k \left[q_n+\eta_n\norm{\hat g-\nabla p}_\infty+\eta_n^2\right] \sum_{j=0}^{k-1} (1+C\eta_n)^j\\
	&=O_P\left(\eta_n+\frac{q_n}{\eta_n}
	+\norm{\hat g-\nabla p}_\infty\right) =o_P(\delta_n).
	\end{split}
	\end{equation}
Since $d(\bm\gamma_{\bm x}(t),\partial\mathcal C_a)
\geq2\delta_n $ for $t\geq0$ by Lemma~\ref{lem:flow_separation_sep}, the GGDPC updates stay inside $\mathcal{C}_a$.
	
At a saddle boundary region, Lemma~\ref{lem:sample_discrete_saddle_passage} applies to the resulting form of Lemma~\ref{lem:sample_admissible}.
	
Inside the fixed modal neighborhood, Lemmas~\ref{lem:mode_geometry} and \ref{lem:sample_admissible} imply, for
	$R_k=\norm{\hat{\bm Y}_k^{(n)}-\bm m^*}\geq\epsilon_n$, that
\begin{align*}
	\norm{\hat{\bm Y}_k^{(n)}-\bm m^* +\eta_n\nabla p(\hat{\bm Y}_k^{(n)})}^2 &=R_k^2
	+2\eta_n\nabla p(\hat{\bm Y}_k^{(n)})^T
	(\hat{\bm Y}_k^{(n)}-\bm m^*)
	+\eta_n^2\norm{\nabla p(\hat{\bm Y}_k^{(n)})}^2\\
	&\quad\leq
	\left(1-2c_0\eta_n+C_0^2\eta_n^2\right)R_k^2.
\end{align*}
Hence,
$$R_{k+1}
	\leq(1-c\eta_n+C_2\eta_n^2)R_k
	+C_2\left[q_n+\eta_n\norm{\hat g-\nabla p}_\infty\right]
	\leq(1-c'\eta_n)R_k$$
	when $C>0$ is sufficiently large. Hence, $c'>0$ is an absolute constant. The result thus follows by combining all these three cases.
\end{proof}

\begin{lemma}[Completion of the sample path in the modal region]
	\label{lem:sample_modal_completion}
	Under the conditions of Lemma~\ref{lem:sample_admissible}, suppose that the
	sample path enters $B(\bm m^*,\epsilon_n)$, where $\epsilon_n = C\left[\frac{q_n}{\eta_n} + \norm{\hat{g}-\nabla p}_{\infty}\right]$.  Let
	\[
	\hat{\bm m}_{a,n}\in
	\argmax_{\bm X_i\in\mathbb X_n\cap B(\bm m^*,r_0)}
	\hat p(\bm X_i),
	\]
	where $r_0>0$ is a sufficiently small fixed modal radius and ties are resolved by the standing deterministic rule.  With probability tending to one, $\hat{\bm m}_{a,n}$ is the terminal vertex of the path within $\mathcal C_a$, and the remaining within-basin path length is
$$O_P\left(
		\frac{q_n}{\eta_n}+\norm{\hat g-\nabla p}_\infty
		+\frac{q_n\log n}{\eta_n^{d+1}}\right).$$
\end{lemma}

\begin{proof}
Let $\widetilde{\bm m}_n^*$ be the critical point of $\hat p$ paired with $\bm m^*$ in Lemma~\ref{lem:estimated_critical_stability}. By triangle's inequality,
\begin{align*}
\norm{\hat{\bm Y}_{\hat T_{n,\epsilon_n}}^{(n)} -\widetilde{\bm m}_n^*} &\leq \norm{\hat{\bm Y}_{\hat T_{n,\epsilon_n}}^{(n)} -\bm{m}^*} + \norm{\bm{m}^*-\widetilde{\bm m}_n^*}\\
&= O_P\left(\epsilon_n +\norm{\hat g-\nabla p}_\infty\right)\\ &=O_P\left(\frac{q_n}{\eta_n}
+\norm{\hat g-\nabla p}_\infty\right).
\end{align*}
On the event of Lemma~\ref{lem:estimated_critical_stability}, $\hat p$ is uniformly strongly concave in a fixed neighborhood of
$\widetilde{\bm m}_n^*$.  The Taylor's expansions in
Lemma~\ref{lem:pop_1nn_bnd}, applied to $\hat p$ and
$\hat g=\nabla\hat p$, gives, whenever
$r=\norm{\bm y-\widetilde{\bm m}_n^*}\geq \frac{C_1q_n}{\eta_n}$,
\begin{equation}
	\label{sample_modal_contraction_estimated}
	\norm{\hat\Phi_n(\bm y)-\widetilde{\bm m}_n^*}
	\leq(1-c\eta_n)r,
	\qquad
	\norm{\hat\Phi_n(\bm y)-\bm y}\leq C\eta_n r,
\end{equation}
for sufficiently large fixed $C_1$. Hence, the length accumulated before the path reaches the core
$B\left(\widetilde{\bm m}_n^*,\frac{C_1q_n}{\eta_n}\right)$ is at most
$$C_2\eta_n\sum_{k\geq0}(1-c\eta_n)^k
O_P\left(\frac{q_n}{\eta_n} +\norm{\hat g-\nabla p}_\infty\right) =O_P\left(\frac{q_n}{\eta_n} +\norm{\hat g-\nabla p}_\infty\right).$$

By Lemma~\ref{lem:estimated_critical_stability}, the strict monotonicity of the density rank keeps all subsequent within-basin vertices in
$B\left(\widetilde{\bm m}_n^*,\frac{C_3q_n}{\eta_n}\right)$ for a fixed $C_3>0$.
No sample vertex can be visited twice.  Lemma~
\ref{lem:uniform_ball_count} therefore gives
\begin{align*}
	\sum_{k:\,\hat{\bm Y}_k^{(n)}\text{ in the core}}
	\norm{\hat{\bm Y}_{k+1}^{(n)}-\hat{\bm Y}_k^{(n)}}
	&\stackrel{\text{(i)}}{\leq} C\frac{q_n}{\eta_n}
	\left[1+
	\sum_{i=1}^n\mathds1\left\{
	\bm X_i\in B\left(\widetilde{\bm m}_n^*,
	C_1\frac{q_n}{\eta_n}\right)\right\}\right]\\
	&=O_P\left(n\left(\frac{q_n}{\eta_n}\right)^{d+1}\right)
	=O_P\left(\frac{q_n\log n}{\eta_n^{d+1}}\right),
\end{align*}
where (i) uses $n\left(\frac{q_n}{\eta_n}\right)^{d}= \frac{\log n}{\eta_n^d} \geq \log n$.
	
	It remains to identify the terminal vertex.  Sample coverage gives an
	observation at distance $O_P(q_n)$ from
	$\widetilde{\bm m}_n^*$.  Strong concavity and
Lemma~\ref{lem:estimated_critical_stability} therefore imply that
$\norm{\hat{\bm m}_{a,n}-\widetilde{\bm m}_n^*}=O_P(q_n)$.
	
For every core vertex $\bm y\neq\hat{\bm m}_{a,n}$,
$\hat{\bm m}_{a,n}$ is admissible and
\[
\norm{\hat{\bm m}_{a,n}-[\bm y+\eta_n\hat g(\bm y)]}
=O_P\left(\frac{q_n}{\eta_n}\right).
\]
Minimality of $\hat\Phi_n(\bm y)$ therefore keeps the subsequent iterations inside the same fixed modal ball. Since density rank increases strictly along the finite path, the path reaches $\hat{\bm m}_{a,n}$. Moreover, compactness and the strict population density gap outside the fixed modal ball, together with uniform consistency of $\hat p$, imply that no observation in $\mathcal C_a$ ranks above $\hat{\bm m}_{a,n}$. Thus, its outgoing edge, if one exists, leaves $\mathcal C_a$, and the result follows.
\end{proof}

\subsection{Main Proof of \autoref{thm:GGDPC_estimated}}

\begin{proof}[Proof of \autoref{thm:GGDPC_estimated}]
All the constants denoted by $C_*$ below are fixed and independent of $n$.
By Lemma~\ref{lem:sample_tube}, the sample path remains in $U_{\delta_n}(\bm{x})$ until it enters $B(\bm{m}^*,\epsilon_n)$ with probability tending to one, where $\epsilon_n=C\left[\frac{q_n}{\eta_n} + \norm{\hat{g}-\nabla p}_{\infty}\right]$ and $C>0$ is a sufficiently large fixed constant. We decompose
\begin{align*}
|\hat L_n(\bm x)-L(\bm x)|
\leq \underbrace{|L(\bm x)-L_{\epsilon_n}(\bm x)|}_{\textbf{Term I}}
+\underbrace{|\hat L_{n,\epsilon_n}(\bm x) -L_{\epsilon_n}(\bm x)|}_{\textbf{Term II}}
+\underbrace{|\hat L_n(\bm x) -\hat L_{n,\epsilon_n}(\bm x)|}_{\textbf{Term III}}.
\end{align*}
By Lemma~\ref{lem:mode_geometry}, \textbf{Term I} is of order $O_P(\epsilon_n) = O_P\left(\frac{q_n}{\eta_n} + \norm{\hat{g}-\nabla p}_{\infty}\right)$.
	
For \textbf{Term II}, let $\hat T_{n,\epsilon_n}=\inf\{k\geq0:\hat{\bm Y}_k^{(n)}\in B(\bm{m}^*,\epsilon_n)\}$. For $k<\hat T_{n,\epsilon_n}$, Lemma~\ref{lem:sample_admissible} implies that
	$$\hat{\Phi}_n(\hat{\bm Y}_k^{(n)})
	= \hat{\bm Y}_k^{(n)} + \eta_n\nabla p(\hat{\bm Y}_k^{(n)})+\bm{R}_k, \qquad \norm{\bm{R}_k}\leq C_1\left[\zeta_n+\eta_n\norm{\hat{g} - \nabla p}_{\infty}\right]$$
for some constant $C_1>0$. Define $\bm Z_k^{(n)}
:=\bm\gamma_{\hat{\bm Y}_k^{(n)}}(\eta_n)$ and $\ell_k:=\int_0^{\eta_n}
\norm{\nabla p(\bm\gamma_{\hat{\bm Y}_k^{(n)}}(t))}\,dt$.
Similar to our arguments for \eqref{Y_Z_saddle_new} in Lemma~\ref{lem:discrete_saddle_passage}, we know that
\begin{align}
	\label{sample_main_exact_flow_error}
\begin{split}
\norm{\hat{\bm Y}_{k+1}^{(n)}-\bm Z_k^{(n)}}
&\leq C_1\left[q_n+\eta_n\norm{\hat g-\nabla p}_\infty
+\eta_n^2\norm{\nabla p(\hat{\bm Y}_k^{(n)})}\right],\\
\left|\norm{\hat{\bm Y}_{k+1}^{(n)}-\hat{\bm Y}_k^{(n)}}
-\ell_k\right|
&\leq C_1\left[q_n+\eta_n\norm{\hat g-\nabla p}_\infty
+\eta_n^2\norm{\nabla p(\hat{\bm Y}_k^{(n)})}\right].
\end{split}
\end{align}
We also know that $L_{\epsilon_n}(\bm Z_k^{(n)})
-L_{\epsilon_n}(\hat{\bm Y}_k^{(n)})=-\ell_k$. By Lemmas~\ref{lem:L_holder} and \ref{lem:sample_tube}, 
$$\left|L_{\epsilon_n}(\hat{\bm Y}_{k+1}^{(n)})
-L_{\epsilon_n}(\bm Z_k^{(n)})\right|
\leq\frac{C_2}{\delta_n}
\norm{\hat{\bm Y}_{k+1}^{(n)}-\bm Z_k^{(n)}}.$$
Combining the last three displays and summing from $k_0$ to $k_1$ in the regular segment yields that
\begin{align*}
	&\sum_{k=k_0}^{k_1}
	\left|
	\norm{\hat{\bm Y}_{k+1}^{(n)}-\hat{\bm Y}_k^{(n)}}
	+L_{\epsilon_n}(\hat{\bm Y}_{k+1}^{(n)})
	-L_{\epsilon_n}(\hat{\bm Y}_k^{(n)})
	\right|\\
	&\leq
	\frac{C_3}{\delta_n}
	\sum_{k=k_0}^{k_1}
	\left[q_n+\eta_n\norm{\hat g-\nabla p}_\infty
	+\eta_n^2\norm{\nabla p(\hat{\bm Y}_k^{(n)})}\right]\\
	&\stackrel{\text{(i)}}{=} O_P\left(\frac{1}{\delta_n}\left[\eta_n+\frac{q_n}{\eta_n}
	+\norm{\hat g-\nabla p}_\infty\right]\right),
\end{align*}
where (i) uses the facts that there are $O_P(\eta_n^{-1})$ iterates on a regular segment and $\nabla p$ is bounded there. Near the saddle point boundary, the same bound follows from Lemma~\ref{lem:sample_discrete_saddle_passage}. Near the modal region but outside $B(\bm{m}^*,\epsilon_n)$, we know from \eqref{sample_modal_contraction_estimated} that the GGDPC path has $O_P(\eta_n^{-1}|\log \epsilon_n|)$ steps, so the above telescoping error is bounded by $O_P\left( \eta_n+\left[\frac{q_n}{\eta_n}+\norm{\hat g-\nabla p}_\infty\right]|\log\epsilon_n|\right)$. Since
$\epsilon_n\geq \frac{Cq_n}{\eta_n}$ and
$\delta_n|\log\left(\frac{q_n}{\eta_n}\right)|=O(1)$, we know that $\eta_n=o\left(\frac{\eta_n}{\delta_n}\right)$ and
$$\norm{\hat g-\nabla p}_\infty|\log\epsilon_n| = O_P\left(\frac{\norm{\hat g-\nabla p}_\infty}{\delta_n}\right).$$
Consequently,
$$\textbf{Term II}
=O_P\left(\frac{q_n}{\eta_n} \left|\log\frac{q_n}{\eta_n}\right|
+ \frac{1}{\delta_n}\left[\eta_n+\frac{q_n}{\eta_n}
+\norm{\hat g-\nabla p}_\infty\right]\right).$$
Finally, Lemma~\ref{lem:sample_modal_completion} gives
$$\textbf{Term III} =O_P\left(\frac{q_n}{\eta_n}+\norm{\hat g-\nabla p}_\infty +\frac{q_n\log n}{\eta_n^{d+1}}\right).$$
Combining all these rates lead to the final conclusion.
\end{proof}

\section{Proof of Theorem~\ref{thm:GGDPC_graph_dist}}
\label{app:GGDPC_graph_dist_proof}

We begin by establishing an essential uniform bound for $\bar{d}_G(\bm{x})$ over $\mathcal{C}$ and then prove \autoref{thm:GGDPC_graph_dist}.

\subsection{Essential Uniform Finiteness of the Population GGDPC Graph Distance}
\label{app:graph_finite_proof}

\begin{proposition}[Essential uniform finiteness of the population GGDPC graph distance]
\label{prop:finite_graph_distance}
Suppose that Assumption~\ref{assump:density_reg} holds and that every gradient ascent flow starting at $\bm{x}\in \mathcal{C}$ remains in $\mathcal{C}$ and converges to a critical point of $p$. 
\begin{enumerate}[label=(\alph*)]
	\item $\sup_{\bm{x}\in \mathcal{C}} L(\bm{x}) <\infty$.
	
	\item If, in addition, Assumption~\ref{assump:unique_proj} holds, then $\sup_{\bm{x}\in\mathcal{C} \setminus \mathcal{N}_{\mathcal{C}}} \bar{d}_G(\bm{x}) <\infty$. Consequently, $$\operatorname*{ess\,sup}\limits_{\bm{x}\in \mathcal{C}} \bar{d}_G(\bm{x})<\infty,$$ 
	where the essential supremum is taken with respect to either Lebesgue measure or $P$.
\end{enumerate}
\end{proposition}

\begin{proof}[Proof of Proposition~\ref{prop:finite_graph_distance}]
(a) For every density value $v \in [p_{\min}, p_{\max}]$ with $p_{\max}=\max_{\bm{x}\in \mathcal{C}} p(\bm{x})$ and the level set $\left\{\bm{y}\in \mathcal{C}:p(\bm{y}) = v\right\}$, we define $h(v) = \inf\left\{\norm{\nabla p(\bm{y})}: \bm{y}\in \mathcal{C}, p(\bm{y}) = v\right\}$. If the level set $\left\{\bm{y}\in \mathcal{C}: p(\bm{y})=v\right\}$ is empty, then we let $h(v)=\infty$ by convention. Whenever $v\neq p(\bm{s})$ for any critical point $\bm{s}$ of $p$, the compactness of $\left\{\bm{y}\in \mathcal{C}:p(\bm{y}) = v\right\}$ implies that $h(v)>0$.

Then, we can prove that for every $u\in \left\{p(\bm{s}): \nabla p(\bm{s})=\bm{0}\right\}$ at the critical point level, there exist constants $C_s, r_s >0$ such that
\begin{equation}
	\label{h_quadra}
	h(v) \geq C_s |v-u|^{\frac{1}{2}}, \quad 0< |u-v| < r_s.
\end{equation}
Indeed, we know from the arguments in Lemma~\ref{lem:mode_geometry} that around small neighborhoods $U_1,...,U_k$ of all the critical points $\bm{s}_1,...,\bm{s}_k$ with critical value $u$,
$$|p(\bm{y}) - u|\leq C_1\norm{\bm{y} - \bm{s}_j}^2, \quad \bm{y}\in U_j, \text{ for }j=1,...,k,$$
and $\norm{\nabla p(\bm{y})} \geq C_2\norm{\bm{y}-\bm{s}_j}$ for $j=1,...,k$ by the non-degeneracy of their Hessian matrices. On the compact complement of $\cup_{j=1}^k U_j$, $\norm{\nabla p(\bm{y})}$ is bounded away from 0 on all levels sufficiently close to $u$. It thus establishes \eqref{h_quadra}. 

There are only finitely many critical values. Integrating the bound in \eqref{h_quadra} over disjoint neighborhoods of those values and using the positive lower bound for $h$ on their compact complement gives that $\int_{p_{\min}}^{p_{\max}} \frac{dv}{h(v)} \leq \max\left\{\int_{p_{\min}}^{p_{\max}} \frac{dv}{C_s |v-u|^{\frac{1}{2}}}, C_3\right\}< \infty$
for some constant $C_3 \in (0,\infty)$.

Now, for any $\bm{x}\in \mathcal{C}$, if $\bm{x}$ is critical, $L(\bm{x})=0<\infty$. Otherwise, along the trajectory, $$\frac{d}{dt}p(\bm\gamma_{\bm x}(t)) =\norm{\nabla p(\bm\gamma_{\bm x}(t))}^2>0$$
until the limiting critical point is reached. Then, we obtain that
\begin{align*}
	L(\bm x)
	&=\int_0^\infty \norm{\nabla p(\bm\gamma_{\bm x}(t))} dt\\
	&\stackrel{\text{(i)}}{=} \int_{p(\bm{x})}^{p(\operatorname{dest}(\bm{x}))}
	\frac{du}{\norm{\nabla p(\bm\gamma_{\bm x}(t(u)))}} \leq \int_{p_{\min}}^{p_{\max}}\frac{du}{h(u)} <\infty,
\end{align*}
where (i) uses the change of variable $u=p(\bm{\gamma}_{\bm{x}}(t))$.

(b) The sequence of modal heights in \eqref{GGDPC_graph_dist_pop} is strictly increasing, so it contains at most $|\mathcal{M}|<\infty$ modes by Assumption~\ref{assump:density_reg}. Since $\norm{\bm{m}_{j_\ell} - \Pi_{\overline{\mathcal{U}}_{j_\ell}}(\bm{m}_{j_\ell})} \leq \mathrm{diam}(\mathcal{C})$,
$$\bar{d}_G(\bm{x}) \leq |\mathcal{M}| \cdot \sup_{\bm{x}\in \mathcal{C}} L(\bm{x}) + \left(|\mathcal{M}| - 1\right) \mathrm{diam}(\mathcal{C})$$
uniformly over $\bm{x}\in \mathcal{C}\setminus \mathcal{N}_{\mathcal{C}}$, where $\mathrm{diam}(\mathcal{C}) = \sup_{\bm{x},\bm{y}\in \mathcal{C}} \norm{\bm{x} -\bm{y}} < \infty$ by the compactness of $\mathcal{C}$. 

Finally, since $p(\bm{x})\in [p_{\min}, p_{\max}]$ for some $0< p_{\min} \leq p_{\max} < \infty$ on $\mathcal C$, $P$ and Lebesgue measure have the same null sets on $\mathcal{C}$. This proves the final essential supremum assertion.
\end{proof}

\subsection{Main Proof of \autoref{thm:GGDPC_graph_dist}}

\begin{proof}[Proof of \autoref{thm:GGDPC_graph_dist}]
Let $\bm{X}_{i^*(\bm{x})}$ be the nearest observation to $\bm{x}$. We know from Lemma~\ref{lem:1NN_sup_bnd} that $\norm{\bm{X}_{i^*(\bm{x})} - \bm{x}}=O_P(q_n)$ with $q_n=\left(\frac{\log n}{n}\right)^{\frac{1}{d}}$. 

The theorem's rate conditions imply that $q_n=o(\delta_n)$. Hence, uniformly over $\mathcal C\setminus\mathcal S_{\rm full}^{2\delta_n/C_{\mathcal S}}$, the nearest observation belongs to the same modal basin as $\bm x$ and remains a distance of order $\delta_n$ from its relative boundary, with probability tending to one.

If $|\mathcal M|=1$, the assertion follows directly from \autoref{thm:GGDPC_estimated}. Hence, we assume $|\mathcal M|\geq2$. By Assumption~\ref{assump:density_reg}, we know that $\psi_j>0$ due to the non-degeneracy of the Hessian matrix around $\bm{m}_j$, and the finiteness of $|\mathcal{M}|$ implies that 
$$\psi_{\min}:= \min_{2\leq j \leq |\mathcal{M}|} \psi_j >0.$$
Choose a fixed $\lambda_0 \in \left(0, \frac{\psi_{\min}}{2}\right)$. By \autoref{thm:modal_const}, with probability tending to one,
$$\hat{\mathcal M}_{\lambda_0} =\left\{\hat{\bm{m}}_{1,n},...,\hat{\bm{m}}_{|\mathcal{M}|,n}\right\},$$
where $\hat{\bm{m}}_{j,n}$ is naturally paired with $\bm{m}_j$ and \eqref{mode_rate} holds. Additionally, distinct modal heights and $\norm{\hat{p} - p}_{\infty}$ imply that the ordering of sample local modes by the estimated density agrees with the ordering of the population local modes by the true density with probability tending to one. To invoke the parent identification argument in Step 1 of the proof of \autoref{thm:conv_ggdpc_dendro}, we choose a deterministic sequence $a_n\to \infty$ sufficiently slowly that $a_n\delta_n\to 0$. Then, $\frac{a_nq_n}{\eta_n}=o(1)$ and $a_n\norm{\hat g-\nabla p}_\infty=o_P(1)$ under the assumptions in the theorem statement. Hence, the empirical modal edge from $\hat{\bm m}_{j,n}$ enters the component associated with $\bm m_{\pi(j)}$ for every $j\geq 2$ with probability tending to one. Thus, the empirical modal sequence agrees with \eqref{mode_seq}.

Now, we decompose the graph distance difference as:
\begin{align*}
	&\left|\hat{d}_G(\bm{x}) - \bar{d}_G(\bm{x}) \right| \\
	&\leq \norm{\bm{X}_{i^*(\bm{x})} - \bm{x}} + \left|L(\bm{X}_{i^*(\bm{x})}) - L(\bm{x})\right| + \left|\hat{L}_n(\bm{X}_{i^*(\bm{x})}) - L(\bm{X}_{i^*(\bm{x})})\right| \\
	&\quad + \sum_{\ell=1}^{T_{\bm{x}}-1} \left[\left|\hat{L}_n(\hat{\bm{z}}_{j_\ell,n}) - L(\hat{\bm{z}}_{j_\ell,n})\right| + \norm{\left(\hat{\bm{m}}_{j_\ell,n} - \hat{\bm{z}}_{j_\ell,n}\right) - \left(\bm{m}_{j_\ell} - \bm{z}_{j_\ell}\right)}  + \left|L(\hat{\bm{z}}_{j_\ell,n}) - L(\bm{z}_{j_\ell}) \right|\right],
\end{align*}
where $\hat{\bm{z}}_{j_\ell,n}=\hat{\Phi}_n(\hat{\bm{m}}_{j_\ell,n})$ is the 1NN uphill update of the sample local mode $\hat{\bm{m}}_{j_\ell,n}$ while $\bm{z}_{j_{\ell}} = \Pi_{\overline{\mathcal{U}}_{j_\ell}}(\bm{m}_{j_\ell})$.
By Lemma~\ref{lem:L_holder}, we know that
$$\left|L(\bm{X}_{i^*(\bm{x})}) - L(\bm{x})\right| \leq \frac{C_L\norm{\bm{X}_{i^*(\bm{x})} - \bm{x}}}{\delta_n} =O_P\left(\frac{q_n}{\delta_n}\right).$$
By Assumption~\ref{assump:unique_proj}, $\bm{z}_{j_\ell}$ and $\hat{\bm{z}}_{j_\ell,n}$ have fixed distances away from the separatrix, so Lemma~\ref{lem:emp_mode_proj} implies that
$$\left|L(\hat{\bm{z}}_{j_\ell,n}) - L(\bm{z}_{j_\ell}) \right| \leq C_1 \norm{\hat{\bm{z}}_{j_\ell,n} - \bm{z}_{j_\ell}} = O_P\left(\sqrt{q_n} + \norm{\hat{g} - \nabla p}_{\infty} \right)$$
for some constant $C_1>0$ and $\ell=1,...,T_{\bm{x}}-1$. Additionally, \eqref{mode_rate} and Lemma~\ref{lem:emp_mode_proj} imply that
$$\norm{\left(\hat{\bm{m}}_{j_\ell,n} - \hat{\bm{z}}_{j_\ell,n}\right) - \left(\bm{m}_{j_\ell} - \bm{z}_{j_\ell}\right)} = O_P\left(\sqrt{q_n} + \norm{\hat{g} - \nabla p}_{\infty} \right)$$
for $\ell=1,...,T_{\bm{x}}-1$. Finally, we apply \autoref{thm:GGDPC_estimated} to argue that both $\left|\hat{L}_n(\bm{X}_{i^*(\bm{x})}) - L(\bm{X}_{i^*(\bm{x})})\right|$ and $\sum_{\ell=1}^{T_{\bm{x}}-1}\left|\hat{L}_n(\hat{\bm{z}}_{j_\ell,n}) - L(\hat{\bm{z}}_{j_\ell,n})\right|$ have the rate $O_P\left(\frac{q_n}{\eta_n}\left|\log\left(\frac{q_n}{\eta_n}\right)\right| +\frac{q_n}{\eta_n\delta_n} + \frac{\eta_n}{\delta_n} + \frac{q_n\log n}{\eta_n^{d+1}} +\frac{\norm{\hat{g} - \nabla p}_{\infty}}{\delta_n}\right)$.

The remaining terms $\frac{q_n}{\delta_n}$, $\sqrt{q_n}$, and $\norm{\hat g-\nabla p}_\infty$ are absorbed by this rate, because for all large $n$, $\eta_n<1$, $\delta_n<1$, and
$\eta_n+\frac{q_n}{\eta_n\delta_n}\geq 2\sqrt{\frac{q_n}{\delta_n}}\geq 2\sqrt{q_n}$.
This proves the asserted uniform rate for $\left|\hat{d}_G(\bm{x}) - \bar{d}_G(\bm{x}) \right|$.
\end{proof}

\section{Proof of Theorem~\ref{thm:wasser_waterfall}}
\label{app:waterfall_proof}

We first derive an integrated graph distance bound near the separatrix and then use it to prove \autoref{thm:wasser_waterfall}.

\subsection{An Integrated Graph Distance Lemma}

\begin{lemma}[Integrated graph distance control near the full separatrix]
\label{lem:int_control_sep}
Suppose that Assumptions~\ref{assump:density_reg}, \ref{assump:sample_den_diff}, \ref{assump:sep_repulsion}, and \ref{assump:C1_linear} hold for every basin of attraction. If $\eta_n+\frac{q_n}{\eta_n} = o(1)$, $\norm{\hat{p} -p}_{\infty}=o_P(1)$, $\norm{\hat{g} - \nabla p}_{\infty} = o_P(1)$, and $\norm{\nabla^2\hat{p} - \nabla^2 p}_{\infty} =o_P(1)$ with $q_n=\left(\frac{\log n}{n}\right)^{\frac{1}{d}}$, then
$$\max_{1\leq i \leq n} \hat{d}_G(\bm{X}_i) = O_P\left(1+ \frac{q_n\log n}{\eta_n^{d+1}}\right).$$
Moreover, for every deterministic sequence $\delta_n \downarrow 0$, if $\frac{q_n\log n}{\eta_n^{d+1}} = o(1)$, then
$$\frac{1}{n}\sum_{i=1}^n \hat{d}_G(\bm{X}_i) \cdot \mathds{1}\left\{d(\bm{X}_i, \mathcal{S}_{\rm full}) \leq \frac{2\delta_n}{C_{\mathcal{S}}}\right\} = O_P\left(\delta_n\right),$$
where $\mathcal{S}_{\rm full} =\bigcup_{a=1}^{|\mathcal{M}|} \partial \mathcal{C}_a$ and $C_{\mathcal{S}}$ is the constant in Lemma~\ref{lem:flow_separation_sep}.
The same results hold with $\hat{d}_G$ replaced by $\hat{d}_{G_{\lambda}}$ for every fixed $\lambda \notin \left\{\psi_j: 1< j\leq |\mathcal{M}|\right\}$.
\end{lemma}

\begin{proof}[Proof of Lemma~\ref{lem:int_control_sep}]
All the constants denoted by $C_*$ or $c_*$ below are fixed and independent of $n$. 

By Assumption~\ref{assump:sep_repulsion}, $\mathcal{S}_{\rm full} =\bigcup_{a=1}^{|\mathcal{M}|} \partial \mathcal{C}_a$ is a
finite union of stable manifolds of saddle points whose Hessian matrices have exactly one positive eigenvalue and $d-1$ negative eigenvalues, and is a $C^2$ hypersurface away from finitely many critical points. Thus, $\mathcal{S}_{\rm full}$ can be covered by $O(\delta^{1-d})$ balls of radius $\delta$ for any $\delta >0$, so its $\delta$-neighborhood has Lebesgue measure $O(\delta)$. When $0< \delta < \delta_0$ for some $\delta_0 >0$, we have that $\mathbb{P}\left(\mathcal{S}_{\rm full}^{\delta}\right) = O(\delta)$ with $\mathcal{S}_{\rm full}^{\delta} = \left\{\bm{x}\in \mathcal{C}: d(\bm{x},\mathcal{S}_{\rm full}) < \delta\right\}$ for all $0< \delta <\delta_0$. In particular, by Markov's inequality, 
\begin{equation}
	\label{mass_around_S}
	\frac{1}{n}\sum_{i=1}^n \mathds{1}\left\{\bm{X}_i \in \mathcal{S}_{\rm full}^{\delta_n}\right\} = O_P\left(\delta_n\right) \quad \text{ when } \delta_n \downarrow 0.
\end{equation}

By \eqref{max_1nn_dist} and Lemma~\ref{lem:1NN_sup_bnd}, we know that $\zeta_n = O_P(q_n)$ with $q_n= \left(\frac{\log n}{n}\right)^{\frac{1}{d}}$. Fix any sufficiently large constant $C_\zeta>0$ and work on the event $\left\{\zeta_n \leq C_{\zeta} q_n \right\}$, whose probability is arbitrarily close to one for all sufficiently large $n$. We condition on the intersection of this event as well as the events on which $\norm{\hat{p} -p}_{\infty}=o(1)$, $\norm{\hat{g} - \nabla p}_{\infty} = o(1)$, and $\norm{\nabla^2\hat{p} - \nabla^2 p}_{\infty} =o(1)$ hold.

Let $\bm s_1,...,\bm s_J$ denote all critical points of $p$ in
$\mathcal C$, and let $\tilde{\bm s}_{1,n},...,\tilde{\bm s}_{J,n}$
be the corresponding critical points of $\hat p$ given by Lemma~\ref{lem:estimated_critical_stability}. Since the population
critical points are finite and lie in the interior of $\mathcal C$,
we may choose a sufficiently small fixed $r_0>0$ such that the balls
$B(\bm s_j, r_0)$ are pairwise disjoint and contained in the interior
of $\mathcal C$. Shrinking $r_0$ if necessary so that the results in
Lemma~\ref{lem:estimated_critical_stability} apply. Moreover, $\sup_{\bm x\in\mathcal C}\norm{\hat g(\bm x)}\leq C_0$, $\sup_{\bm x}\norm{\nabla^2\hat p(\bm x)}_2\leq C_0$, and 
\begin{equation}
	\label{regular_gradient_lower_graph}
	\inf_{\bm x\in\mathcal C:\min_j\norm{\bm x-\widetilde{\bm s}_{j,n}}\geq r_0} \norm{\hat g(\bm x)} \geq c_0 >0.
\end{equation}

We define $b_n=A\frac{q_n}{\eta_n} = o(1)$ a sufficiently large fixed constant $A>0$. Fix an arbitrary starting observation and write its directed GGDPC path as
$$\bm{Y}_0^{(n)} \to \bm{Y}_1^{(n)}\to\cdots \to\bm{Y}_T^{(n)},
\qquad
\bm{Y}_{k+1}^{(n)}=\hat{\Phi}_n(\bm{Y}_k^{(n)}).$$
The estimated density strictly increases along this path, and hence no
vertex can be visited more than once. We bound its total length by
considering three regions.

\textbf{Region 1: Away from all critical points.}
Suppose that $\min_j\norm{\bm x-\widetilde{\bm s}_{j,n}}\geq r_0$. Let $\bm z=\bm\gamma_{\bm x}(\eta_n)$ be the population gradient flow
endpoint after time $\eta_n$, where $\bm{z}\in \mathcal{C}$ under Assumption~\ref{assump:density_reg}(d). By Taylor's expansion,
$$\norm{\bm z-\bm x-\eta_n\nabla p(\bm x)} \leq C_1\eta_n^2.$$
Moreover, by Lemma~\ref{lem:estimated_critical_stability} and $\norm{\hat{g} - \nabla p}_{\infty} = o_P(1)$,
\begin{align*}
	\hat p(\bm z)-\hat p(\bm x)
	&=\int_0^{\eta_n} \hat g(\bm\gamma_{\bm x}(t))^T \nabla p(\bm\gamma_{\bm x}(t))\,dt\\
	&\geq c_1\eta_n
\end{align*}
for all sufficiently large $n$. Choose an observation $\bm X(\bm z)$ satisfying $\norm{\bm X(\bm z)-\bm z}\leq\zeta_n$. Since
$\frac{q_n}{\eta_n}=o(1)$ and $\hat g$ is uniformly bounded,
$$\hat p(\bm X(\bm z))-\hat p(\bm z) \geq-C_2\zeta_n=o(\eta_n),$$
so $\hat{p}(\bm X(\bm z)) > \hat{p}(\bm{x})$ for all sufficiently large $n$. Thus, $\bm{X}(\bm{z})$ is admissible in the definition of $\hat\Phi_n(\bm x)$. By its minimality,
\begin{align*}
\norm{\hat\Phi_n(\bm x)-[\bm x+\eta_n\hat g(\bm x)]}
&\leq \norm{\bm X(\bm z)-[\bm x+\eta_n\hat g(\bm x)]}\\
&\leq \zeta_n+C_3\eta_n^2+\eta_n \norm{\hat{g}-\nabla p}_{\infty} =o(\eta_n).
\end{align*}
Consequently, $\norm{\hat\Phi_n(\bm x)-\bm x}\leq C_4\eta_n$ and 
$$\hat p(\hat\Phi_n(\bm x))-\hat p(\bm x) \geq c_2\eta_n.$$
It implies that a path contains at most $O(\eta_n^{-1})$ updates starting
in this region, so their total length is $O(1)$.

\textbf{Region 2: Annular regions around the critical points.} For $r_{\ell,n}=2^\ell b_n$, we consider the annulus
$$\mathcal A_{j,\ell} = \left\{ \bm x: r_{\ell,n} < \norm{\bm x-\widetilde{\bm s}_{j,n}} \leq \min\{2r_{\ell,n},r_0\} \right\},$$
for those $\ell$ such that $r_{\ell,n}<r_0$. Fix $\bm x\in\mathcal A_{j,\ell}$. By Lemma~\ref{lem:estimated_critical_stability}, $c_0 r_{\ell,n} \leq \norm{\hat g(\bm x)} \leq C_0 r_{\ell,n}$. Let $\bm z=\bm x+\eta_n\hat g(\bm x)$. For all sufficiently large $n$, $\bm z$ remains in
$B(\widetilde{\bm s}_{j,n},2r_0)\subset\mathcal C$. Choose
$\bm X(\bm z)\in\mathbb X_n$ with
$\norm{\bm X(\bm z)-\bm z}\leq\zeta_n$.
Taylor's theorem gives that
$$ \hat p(\bm z)-\hat p(\bm x) \geq \eta_n\norm{\hat g(\bm x)}^2 -C_5\eta_n^2\norm{\hat g(\bm x)}^2 \geq
c_3\eta_n r_{\ell,n}^2.$$
All points on the segment joining $\bm z$ and $\bm X(\bm z)$ remain
within distance $C_6 r_{\ell,n}$ of $\widetilde{\bm s}_{j,n}$, so
Lemma~\ref{lem:estimated_critical_stability} and the bounded Hessian imply that
$$\left|\hat p(\bm X(\bm z))-\hat p(\bm z) \right| \leq C_7 r_{\ell,n} \zeta_n.$$
Since $r_{\ell,n} \geq b_n=A\frac{q_n}{\eta_n}$ and $\zeta_n\leq C_\zeta q_n$,
$$\frac{r_{\ell,n} \zeta_n}{\eta_n r_{\ell,n}^2} = \frac{\zeta_n}{\eta_n r_{\ell,n}} \leq \frac{C_\zeta}{A}.$$
Choosing $A$ sufficiently large therefore gives that $\hat p(\bm X(\bm z))>\hat p(\bm x)$, so $\bm X(\bm z)$ is admissible. Hence,
$$\norm{\hat\Phi_n(\bm x)-[\bm x+\eta_n\hat g(\bm x)]} \leq\zeta_n.$$
It follows that
\begin{equation}
	\label{annulus_shift_graph}
	\norm{\hat\Phi_n(\bm x)-\bm x} \leq C_8\eta_n r_{\ell,n},
\end{equation}
and another Taylor expansion gives that
$$\hat p(\hat\Phi_n(\bm x))-\hat p(\bm x) \geq c_4\eta_n r_{\ell,n}^2.$$

On the other hand, Lemma~\ref{lem:estimated_critical_stability} implies that every starting point
in $\mathcal A_{j,\ell}$ has estimated density in the interval
$$\left[ \hat p(\widetilde{\bm s}_{j,n})-C_9 r_{\ell,n}^2,\, \hat p(\widetilde{\bm s}_{j,n})+C_9 r_{\ell,n}^2 \right].$$
Suppose that a directed path contains $N_{j,\ell}$ vertices whose
outgoing edges start in this annulus. Since the estimated density is
strictly increasing along the path and every such update increases it
by at least $c_4\eta_n r_{\ell,n}^2$, we know that
$$(N_{j,\ell}-1)c_4\eta_n r_{\ell,n}^2 \leq 2C_9 r_{\ell,n}^2.$$
Hence, $N_{j,\ell}\leq \frac{C_{10}}{\eta_n}$. Together with \eqref{annulus_shift_graph},
the total length contributed by this annulus is at most $C_{11} r_{\ell,n}$.
Since the radii are dyadic, 
$$\sum_{\ell:r_{\ell,n}<r_0} r_{\ell,n}
\leq C_{11}r_0.$$
There are only finitely many critical points, so the total contribution
from all critical annuli is $O(1)$, uniformly over the starting
vertex.

\textbf{Region 3: Shrinking critical cores.}
It remains to consider $\mathcal B_{j,n} = B(\tilde{\bm s}_{j,n},b_n)$.
We consider the case when $\bm s_j$ is not a local mode, and the local modal scenario follows similarly. Since it is a
non-degenerate critical point, its Hessian has at least one positive
eigenvalue. Let $\bm e_j$ be a corresponding fixed unit eigenvector of
$\nabla^2p(\bm s_j)$. By shrinking $r_0$ and using $\norm{\nabla^2 \hat{p} - \nabla^2 p}=o_P(1)$, there exists $c_5>0$ such that
\[
\bm e_j^T\nabla^2\hat p(\bm y)\bm e_j\geq c_5
\]
uniformly over
$\bm y\in B(\widetilde{\bm s}_{j,n},2r_0)$.
For $\bm x\in\mathcal B_{j,n}$, set $\bm z=\bm x+A_L b_n\bm e_j$,
where $A_L>0$ is a fixed sufficiently large constant. By Taylor's theorem,
$$\hat p(\bm z)-\hat p(\bm x) \geq c_6 b_n^2.$$
Choose $\bm X(\bm z)$ within $\zeta_n$ of $\bm z$. Since
$\norm{\hat g(\bm y)}\lesssim b_n$ in the relevant neighborhood,
$$\left| \hat p(\bm X(\bm z))-\hat p(\bm z)\right| \leq C_{13}b_n\zeta_n.$$
Moreover, $\frac{\zeta_n}{b_n}
\leq \frac{C_\zeta}{A}\eta_n=o(1)$, so $\hat p(\bm X(\bm z))>\hat p(\bm x)$ for all sufficiently large
$n$. Hence,
\begin{align*}
	\norm{\hat\Phi_n(\bm x)-[\bm x+\eta_n\hat g(\bm x)]}
	&\leq \norm{\bm X(\bm z)-[\bm x+\eta_n\hat g(\bm x)]} \leq C_{14}b_n,
\end{align*}
and therefore,
\begin{equation}
	\label{core_nonmode_edge_graph}
	\norm{\hat\Phi_n(\bm x)-\bm x}
	\leq C_{15}b_n.
\end{equation}

It remains to count the observations contained in the shrinking
critical cores. Since $b_n\geq Aq_n$ for all sufficiently large $n$, we have that $\log(1/b_n)=O(\log n)$. Lemma~\ref{lem:uniform_ball_count} implies that 
\begin{align*}
	\sum_{j=1}^J \sum_{i=1}^n\mathds{1}\left\{\bm X_i\in\mathcal B_{j,n}\right\}&= O_P\left(nb_n^d+\log n\right) =O_P\left(
	\frac{\log n}{\eta_n^d}\right),
\end{align*}
because $nq_n^d=\log n$.
Since a directed path visits every observation at most once,
\eqref{core_nonmode_edge_graph} implies that the total contribution from
all non-exceptional core edges is 
$$O_P\left(b_n\frac{\log n}{\eta_n^d}\right)=O_P\left(\frac{q_n\log n}{\eta_n^{d+1}}\right).$$

Combining the regular-region, annular, critical-core, and exceptional
modal-edge bounds gives, uniformly over all starting observations,
$$\max_{1\leq i\leq n}
\hat d_G(\bm X_i) =O_P\left(1+\frac{q_n\log n}{\eta_n^{d+1}}\right).$$

Combining with \eqref{mass_around_S}, we obtain that 
$$\frac{1}{n}\sum_{i=1}^n \hat{d}_G(\bm{X}_i) \cdot \mathds{1}\left\{d(\bm{X}_i, \mathcal{S}_{\rm full}) \leq \frac{2\delta_n}{C_{\mathcal{S}}}\right\} = O_P\left(\delta_n + \frac{\delta_n q_n \log n}{\eta_n^{d+1}}\right) = O_P(\delta_n)$$
when $\frac{q_n \log n}{\eta_n^{d+1}}=o(1)$.
Truncating the graph can only stop a path earlier, so the same proof applies to $\hat{d}_{G_{\lambda}}$.
\end{proof}

\subsection{Main Proof of \autoref{thm:wasser_waterfall}}

\begin{proof}[Proof of \autoref{thm:wasser_waterfall}]
We first know from Proposition~\ref{prop:finite_graph_distance} that $\sup_{\bm{x}\in \mathcal{C}} \bar{d}_G(\bm{x}) < \infty$. Since $p$ is bounded by Assumption~\ref{assump:density_reg}, we know that the population (global) waterfall measure $\left(\bar{d}_G(\bm{X}), p(\bm{X})\right)$ is supported on a fixed compact set in $\mathbb{R}^2$.
Again, by triangle's inequality, we know that
$$\mathrm{Wass}_1(\hat{Q}_n, Q) \leq \underbrace{\mathrm{Wass}_1(\hat{Q}_n, Q_n^0)}_{\textbf{Term I}} + \underbrace{\mathrm{Wass}_1(Q_n^0, Q)}_{\textbf{Term II}}.$$

\textbf{Term I:} Coupling the $i$-th observation of $\hat{Q}_n$ with the $i$-th observation of $Q_n^0$ gives that
\begin{equation}
	\label{emp_coupling_waterfall}
	\mathrm{Wass}_1(\hat Q_n,Q_n^0) \lesssim \frac{1}{n}\sum_{i=1}^n \left|\hat{d}_G(\bm{X}_i) -\bar{d}_G(\bm{X}_i)\right| + \norm{\hat{p} - p}_{\infty},
\end{equation}
where the universal constant in ``$\lesssim$'' depends only on the choice of norm in $\mathrm{Wass}_1(\cdot,\cdot)$ on $\mathbb{R}^2$.

Let $C_{\mathcal{S}}\in(0,1)$ denote the separation constant in Lemma~\ref{lem:flow_separation_sep}. We separate the support $\mathcal{C}$ of $p$ into two different regions:
$$\mathcal{C}_n^{\rm far} = \mathcal{C} \setminus \mathcal{S}_{\rm full}^{\frac{2\delta_n}{C_{\mathcal{S}}}}, \qquad \mathcal{C}_n^{\rm close} = \mathcal{S}_{\rm full}^{\frac{2\delta_n}{C_{\mathcal{S}}}}.$$
That is, one region is farther away from the separatrices, while the other one is close to them. Let $I_n^{\rm far}:=\left\{1\leq i\leq n:\bm{X}_i\in \mathcal{C}_n^{\rm far}\right\}$ and $I_n^{\rm close}:=\left\{1\leq i\leq n:\bm{X}_i\in \mathcal{C}_n^{\rm close}\right\}$. By \autoref{thm:GGDPC_graph_dist}, we know that
$$\sup_{\bm{x}\in \mathcal{C}_n^{\rm far}} \left|\hat{d}_G(\bm{x}) - \bar{d}_G(\bm{x}) \right| = O_P\left(\frac{q_n}{\eta_n}\left|\log\left(\frac{q_n}{\eta_n}\right)\right| +\frac{q_n}{\eta_n\delta_n} + \frac{\eta_n}{\delta_n} + \frac{q_n\log n}{\eta_n^{d+1}} +\frac{\norm{\hat{g} - \nabla p}_{\infty}}{\delta_n} \right).$$
Consequently,
$$\frac{1}{n} \sum_{i\in I_n^{\rm far}} \left|\hat{d}_G(\bm{X}_i) - \bar{d}_G(\bm{X}_i) \right| = O_P\left(\frac{q_n}{\eta_n}\left|\log\left(\frac{q_n}{\eta_n}\right)\right| +\frac{q_n}{\eta_n\delta_n} + \frac{\eta_n}{\delta_n} + \frac{q_n\log n}{\eta_n^{d+1}} +\frac{\norm{\hat{g} - \nabla p}_{\infty}}{\delta_n} \right).$$
On the other hand, by Proposition~\ref{prop:finite_graph_distance} and Lemma~\ref{lem:int_control_sep} with \eqref{mass_around_S}, 
\begin{align*}
	\frac{1}{n} \sum_{i\in I_n^{\rm close}} \left|\hat{d}_G(\bm{X}_i) - \bar{d}_G(\bm{X}_i) \right| &\leq \frac{1}{n} \sum_{i\in I_n^{\rm close}} \hat{d}_G(\bm{X}_i) + \left(\operatorname*{ess\,sup}_{\bm{x}\in \mathcal{C}} \bar{d}_G(\bm{x})\right)\cdot \frac{\left|I_n^{\rm close}\right|}{n} = O_P\left(\delta_n\right).
\end{align*}
Combining the above results with \eqref{emp_coupling_waterfall} and $\eta_n=o(\delta_n)$ yields that
\begin{align*}
	&\mathrm{Wass}_1(\hat Q_n,Q_n^0) = O_P\left(\delta_n + \frac{q_n}{\eta_n}\left|\log\left(\frac{q_n}{\eta_n}\right)\right| +\frac{q_n}{\eta_n\delta_n} + \frac{\eta_n}{\delta_n} + \frac{q_n\log n}{\eta_n^{d+1}} +\frac{\norm{\hat{g} - \nabla p}_{\infty}}{\delta_n} +\norm{\hat{p} - p}_{\infty}\right).
\end{align*}

\textbf{Term II:} By Theorem 1 in \cite{fournier2015rate}, we know that $\mathrm{Wass}_1(Q_n^0, Q) = O_P\left(\frac{\log n}{\sqrt{n}}\right)$. Indeed, since $\left(\bar{d}_G(\bm{X}_i), p(\bm{X}_i)\right)$ are i.i.d. samples supported on a fixed compact rectangle in $\mathbb{R}^2$, we can rescale this rectangle into $[0,1]^2$ and let $\mathcal{D}_k$ be the dyadic partition of $[0,1]^2$ into squares of side length $2^{-k}$. For each square $A\in \mathcal{D}_k$, $\left|Q_n^0(A) - Q(A)\right|$ measures the mass imbalance between $Q_n^0$ and $Q$, while the diameter of $A$ is at most order $2^{-k}$. Thus, for every integer $J\geq 1$,
$$\mathrm{Wass}_1(Q_n^0, Q) \lesssim 2^{-J} + \sum_{k=0}^J 2^{-k} \sum_{A\in\mathcal{D}_k} \left|Q_n^0(A) - Q(A)\right|.$$
Additionally, for each $A\in \mathcal{D}_k$, by Cauchy-Schwarz inequality,
$$\mathbb{E}\left|Q_n^0(A) - Q(A) \right| \leq \sqrt{\frac{Q(A)\left[1-Q(A)\right]}{n}} \leq \sqrt{\frac{Q(A)}{n}}.$$
Hence,
$$\mathbb{E}\left[\sum_{A\in\mathcal{D}_k} \left|Q_n^0(A) - Q(A) \right|\right] \leq \frac{1}{\sqrt n}\sum_{A\in\mathcal P_k}\sqrt{Q(A)}
\leq \sqrt{\frac{|\mathcal P_k|}{n}} =\frac{2^k}{\sqrt n},$$
where $|\mathcal{D}_k|=4^k$. Thus, $\mathbb{E}\left[\mathrm{Wass}_1(Q_n^0, Q) \right]\lesssim 2^{-J} + \frac{J}{\sqrt{n}}$. Choosing $J=\ceil{\log_2\sqrt{n}}$ and applying Markov's inequality yield that $\mathrm{Wass}_1(Q_n^0,Q) = O_P\left(\frac{\log n}{\sqrt n}\right)$ as expected.

Finally, combining the rates in \textbf{Term I} and \textbf{Term II} gives us that
\begin{align*}
	\mathrm{Wass}_1(\hat{Q}_n, Q) &= O_P\left(\delta_n + \frac{q_n}{\eta_n}\left|\log\left(\frac{q_n}{\eta_n}\right)\right| +\frac{q_n}{\eta_n\delta_n} + \frac{\eta_n}{\delta_n} + \frac{q_n\log n}{\eta_n^{d+1}} +\frac{\norm{\hat{g} - \nabla p}_{\infty}}{\delta_n} +\norm{\hat{p} - p}_{\infty} + \frac{\log n}{\sqrt{n}}\right).
\end{align*}
The result follows.
\end{proof}

\end{document}

%% file: Figures/Cluster_Tree.tikz
\begin{tikzpicture}[
	x=0.55cm,y=0.95cm, >=Latex,
	axis/.style={->,line width=0.6pt},
	density/.style={black,line width=1.05pt},
	mode/.style={circle,fill=blue!65!black,draw=blue!65!black,inner sep=1.45pt},
	proj/.style={regular polygon,regular polygon sides=3,fill=red!75!black,draw=red!75!black,inner sep=1.1pt,rotate=90},
	projline/.style={red!75!black,dashed,line width=0.8pt},
	tree/.style={red!75!black,line width=1.15pt},
	every node/.style={font=\small}
	]
	
	\begin{scope}
		\draw[axis] (0,0) -- (8.85,0) node[below] {$x$};
		\draw[axis] (0,0) -- (0,3.72) node[left] {$p(x)$};
		
		\draw[density]
		(0.12,0.03)
		.. controls (0.50,0.25) and (0.80,1.55) .. (1.15,1.55)
		.. controls (1.27,1.55) and (1.40,0.55) .. (1.52,0.55)
		.. controls (1.62,0.55) and (1.70,1.35) .. (1.80,1.55)
		.. controls (1.93,1.83) and (2.08,2.15) .. (2.25,2.15)
		.. controls (2.70,2.15) and (3.50,0.95) .. (4.45,0.95)
		.. controls (4.82,0.95) and (5.22,1.80) .. (5.55,2.15)
		.. controls (5.78,2.43) and (6.02,2.70) .. (6.25,2.70)
		.. controls (6.48,2.70) and (6.72,2.20) .. (6.95,2.20)
		.. controls (7.05,2.20) and (7.16,2.55) .. (7.25,2.70)
		.. controls (7.34,2.90) and (7.39,3.35) .. (7.48,3.35)
		.. controls (7.78,3.35) and (8.18,0.35) .. (8.58,0.03);
		
		\node[mode,label={[yshift=1pt]above:$m_4$}] at (1.15,1.55) {};
		\node[mode,label={[yshift=1pt]above:$m_3$}] at (2.25,2.15) {};
		\node[mode,label={[yshift=1pt]above:$m_2$}] at (6.25,2.70) {};
		\node[mode,label={[yshift=1pt]above:$m_1$}] at (7.48,3.35) {};
		
		\draw[projline] (1.15,1.55) -- (1.80,1.55);
		\node[proj] at (1.80,1.55) {};
		\node[font=\scriptsize,red!75!black,anchor=north west]
		at (1.83,1.48) {$z_4$};
		
		\draw[projline] (2.25,2.15) -- (5.55,2.15);
		\node[proj] at (5.55,2.15) {};
		\node[font=\scriptsize,red!75!black,anchor=north west]
		at (5.58,2.08) {$z_3$};
		
		\draw[projline] (6.25,2.70) -- (7.25,2.70);
		\node[proj] at (7.25,2.70) {};
		\node[font=\scriptsize,red!75!black,anchor=north east]
		at (7.97,2.6) {$z_2$};
		
		\node[font=\scriptsize,red!75!black,anchor=south] at (4.40,-0.10)
		{$z_j:=\Pi_{\overline{\mathcal U}_j}(m_j),\quad j=2,3,4$};
		
		\node[align=center,font=\small] at (4.42,-0.8)
		{(a) Density $p$ and modal projections};
	\end{scope}
	
	\begin{scope}[xshift=6.05cm]
		\draw[axis] (0.35,0) -- (0.35,3.55) node[left] {$\lambda$};
		
		\coordinate (a4) at (1.30,0.15);
		\coordinate (a3) at (2.85,0.15);
		\coordinate (a2) at (5.95,0.15);
		\coordinate (a1) at (7.50,0.15);
		
		\def\yFour{0.45}
		\def\yTwo{0.80}
		\def\yThree{2.10}
		
		\draw[tree] (a4) -- (1.30,\yFour)
		-- (2.85,\yFour) -- (a3);
		\coordinate (c43) at (2.075,\yFour);
		
		\draw[tree] (a2) -- (5.95,\yTwo)
		-- (7.50,\yTwo) -- (a1);
		\coordinate (c21) at (6.725,\yTwo);
		
		\draw[tree] (c43) -- (2.075,\yThree)
		-- (6.725,\yThree) -- (c21);
		
		\coordinate (cAll) at (4.40,\yThree);
		
		\draw[tree] (cAll) -- (4.40,3);
		
		\draw[red!75!black,dashed,line width=0.65pt]
		(0.35,\yFour) -- (2.075,\yFour);
		\draw[red!75!black,dashed,line width=0.65pt]
		(0.35,\yTwo) -- (6.725,\yTwo);
		\draw[red!75!black,dashed,line width=0.65pt]
		(0.35,\yThree) -- (4.40,\yThree);
		
		\node[font=\scriptsize,red!75!black,anchor=east]
		at (0.30,\yFour) {$\psi_4$};
		\node[font=\scriptsize,red!75!black,anchor=east]
		at (0.30,\yTwo) {$\psi_2$};
		\node[font=\scriptsize,red!75!black,anchor=east]
		at (0.30,\yThree) {$\psi_3$};
		
		\node[below] at (a4) {$m_4$};
		\node[below] at (a3) {$m_3$};
		\node[below] at (a2) {$m_2$};
		\node[below] at (a1) {$m_1$};
		
		\node[align=center,font=\small] at (4.42,-0.8)
		{(b) Modal-distance dendrogram $\mathcal T_{\rm mode}$};
	\end{scope}
	
	\begin{scope}[xshift=12.15cm]
		\draw[axis] (0,0) -- (8.85,0) node[below] {$x$};
		\draw[axis] (0,0) -- (0,3.72) node[left] {$p(x)$};
		
		\draw[density]
		(0.12,0.03)
		.. controls (0.50,0.25) and (0.80,1.55) .. (1.15,1.55)
		.. controls (1.27,1.55) and (1.40,0.55) .. (1.52,0.55)
		.. controls (1.62,0.55) and (1.70,1.35) .. (1.80,1.55)
		.. controls (1.93,1.83) and (2.08,2.15) .. (2.25,2.15)
		.. controls (2.70,2.15) and (3.50,0.95) .. (4.45,0.95)
		.. controls (4.82,0.95) and (5.22,1.80) .. (5.55,2.15)
		.. controls (5.78,2.43) and (6.02,2.70) .. (6.25,2.70)
		.. controls (6.48,2.70) and (6.72,2.20) .. (6.95,2.20)
		.. controls (7.05,2.20) and (7.16,2.55) .. (7.25,2.70)
		.. controls (7.34,2.90) and (7.39,3.35) .. (7.48,3.35)
		.. controls (7.78,3.35) and (8.18,0.35) .. (8.58,0.03);
		
		\node[mode] at (1.15,1.55) {};
		\node[mode] at (2.25,2.15) {};
		\node[mode] at (6.25,2.70) {};
		\node[mode] at (7.48,3.35) {};
		\node[font=\scriptsize,above] at (1.15,1.59) {$m_4$};
		\node[font=\scriptsize,above] at (2.25,2.19) {$m_3$};
		\node[font=\scriptsize,above] at (6.25,2.74) {$m_2$};
		\node[font=\scriptsize,above] at (7.48,3.39) {$m_1$};
		
		\def\hFourThree{0.55}
		\def\hThreeTwo{0.95}
		\def\hTwoOne{2.20}
		
		\draw[tree] (7.48,3.30) -- (7.48,\hTwoOne)
		-- (6.25,\hTwoOne) -- (6.25,2.65);
		\coordinate (c12) at (6.865,\hTwoOne);
		
		\draw[tree] (c12) -- (6.865,\hThreeTwo)
		-- (2.25,\hThreeTwo) -- (2.25,2.10);
		\coordinate (c123) at (4.5575,\hThreeTwo);
		
		\draw[tree] (c123) -- (4.5575,\hFourThree)
		-- (1.15,\hFourThree) -- (1.15,1.50);
		
		
		\coordinate (c1234) at (2.85375,\hFourThree);
		\draw[tree] (c1234) -- (2.85375,0);
		
		\node[align=center,font=\small] at (4.42,-0.8)
		{(c) Density cluster tree $\mathcal T_p$};
	\end{scope}
	
\end{tikzpicture}

%% file: Figures/LM1_figure.tikz
\definecolor{lmregion}{HTML}{7DF9FF}
\definecolor{lmcore}{HTML}{98FB98}
\definecolor{lmink}{HTML}{263238}
\definecolor{lmtriangle}{HTML}{DC143C} 
\definecolor{lmmode}{HTML}{0437F2}     

\begin{tikzpicture}[x=0.01cm,y=-0.0085cm,every node/.style={text=lmink}]
	\path[draw=lmink,line width=0.02cm,line cap=butt,line join=miter]
	(645.119,237.294)
	.. controls (696.592,161.499) and (769.884,102.973) .. (855.839,71.25)
	.. controls (1005.02,16.193) and (1161.69,49.966) .. (1311.008,97.559)
	.. controls (1418.417,131.795) and (1525.91,174.785) .. (1604.811,274.826)
	.. controls (1663.934,349.792) and (1700.182,443.831) .. (1723.504,538.824)
	.. controls (1745.919,630.127) and (1754.538,729.629) .. (1700.818,806.131)
	.. controls (1577.238,982.114) and (1291.438,920.683) .. (1220.19,687.49)
	.. controls (1179.388,553.946) and (1178.451,375.816) .. (1043.303,377.956)
	.. controls (982.122,378.926) and (940.038,431.305) .. (901.153,479.164)
	.. controls (864.546,524.22) and (822.219,568.201) .. (765.069,577.627)
	.. controls (703.729,587.744) and (643.991,555.346) .. (612.466,501.232)
	.. controls (564.525,418.941) and (590.528,317.68) .. (645.119,237.294);
	\path[fill=lmregion,draw=none]
	(669.905,367.857)
	.. controls (676.84,250.123) and (766.113,162.713) .. (872.605,129.506)
	.. controls (984.598,94.584) and (1102.095,118.109) .. (1216.212,142.212)
	.. controls (1294.247,158.694) and (1374.041,176.424) .. (1438.005,227.823)
	.. controls (1478.196,260.12) and (1509.324,309.174) .. (1491.898,357.051)
	.. controls (1467.031,425.372) and (1377.216,428.237) .. (1303.225,392.052)
	.. controls (1260.499,371.157) and (1221.408,343.302) .. (1177.909,324.133)
	.. controls (1117.839,297.662) and (1051.848,289.313) .. (986.286,291.764)
	.. controls (944.735,293.318) and (901.913,299.751) .. (870.719,328.108)
	.. controls (827.319,367.561) and (814.963,442.535) .. (758.351,460.348)
	.. controls (730.833,469.007) and (701.057,459.409) .. (684.138,434.676)
	.. controls (671.18,415.733) and (668.517,391.419) .. (669.905,367.857)
	-- cycle;
	\path[draw=lmink,line width=0.02cm,line cap=butt,line join=miter]
	(669.905,367.857)
	.. controls (676.84,250.123) and (766.113,162.713) .. (872.605,129.506)
	.. controls (984.598,94.584) and (1102.095,118.108) .. (1216.212,142.212)
	.. controls (1294.247,158.694) and (1374.042,176.424) .. (1438.005,227.823)
	.. controls (1478.196,260.12) and (1509.323,309.174) .. (1491.898,357.051)
	.. controls (1467.031,425.372) and (1377.216,428.237) .. (1303.225,392.052)
	.. controls (1260.499,371.157) and (1221.408,343.302) .. (1177.909,324.133)
	.. controls (1117.839,297.662) and (1051.848,289.313) .. (986.286,291.764)
	.. controls (944.735,293.318) and (901.913,299.751) .. (870.719,328.108)
	.. controls (827.319,367.561) and (814.963,442.535) .. (758.351,460.348)
	.. controls (730.833,469.007) and (701.057,459.409) .. (684.138,434.676)
	.. controls (671.18,415.733) and (668.517,391.419) .. (669.905,367.857);
	\path[draw=lmink,line width=0.02cm,line cap=butt,line join=miter]
	(805.167,225.603)
	.. controls (838.138,209.116) and (867.853,186.845) .. (902.448,173.713)
	.. controls (950.076,155.635) and (1002.258,156.045) .. (1053.61,158.936)
	.. controls (1101.079,161.608) and (1148.452,166.418) .. (1195.434,173.87)
	.. controls (1262.18,184.458) and (1329.195,200.156) .. (1385.89,231.652)
	.. controls (1431.68,257.089) and (1462.66,303.434) .. (1436.279,337.07)
	.. controls (1404.747,377.274) and (1337.184,344.314) .. (1280.902,315.478)
	.. controls (1245.294,297.234) and (1208.615,285.1) .. (1170.012,274.748)
	.. controls (1111.938,259.175) and (1049.969,251.489) .. (987.477,246.194)
	.. controls (946.714,242.74) and (903.588,240.491) .. (871.673,259.734)
	.. controls (847.027,274.594) and (832.413,300.256) .. (803.814,309.689)
	.. controls (783.976,316.233) and (761.268,312.152) .. (752.267,295.41)
	.. controls (737.169,267.327) and (770.986,242.694) .. (805.167,225.603);
	\path[draw=lmink,line width=0.019cm,line cap=butt,line join=miter]
	(1386.512,265.283)
	.. controls (1421.296,292.96) and (1413.462,314.93) .. (1369.014,314.354)
	.. controls (1324.567,313.779) and (1260.337,290.876) .. (1225.553,263.2)
	.. controls (1190.769,235.523) and (1198.603,213.553) .. (1243.051,214.128)
	.. controls (1287.499,214.703) and (1351.728,237.606) .. (1386.512,265.283);
	\path[fill=lmregion,draw=none]
	(1606.483,547.285)
	.. controls (1661.949,599.419) and (1661.949,683.946) .. (1606.483,736.08)
	.. controls (1551.017,788.214) and (1461.088,788.214) .. (1405.621,736.08)
	.. controls (1350.155,683.946) and (1350.155,599.419) .. (1405.621,547.285)
	.. controls (1461.088,495.15) and (1551.017,495.15) .. (1606.483,547.285)
	-- cycle;
	\path[draw=lmink,line width=0.02cm,line cap=butt,line join=miter]
	(1606.483,547.285)
	.. controls (1661.949,599.419) and (1661.949,683.946) .. (1606.483,736.08)
	.. controls (1551.016,788.214) and (1461.087,788.214) .. (1405.621,736.08)
	.. controls (1350.154,683.946) and (1350.154,599.419) .. (1405.621,547.285)
	.. controls (1461.087,495.15) and (1551.016,495.15) .. (1606.483,547.285);
	\path[draw=lmink,line width=0.02cm,line cap=butt,line join=miter]
	(1584.135,571.113)
	.. controls (1627.259,610.088) and (1627.259,673.277) .. (1584.135,712.251)
	.. controls (1541.011,751.226) and (1471.093,751.226) .. (1427.969,712.251)
	.. controls (1384.845,673.277) and (1384.845,610.088) .. (1427.969,571.113)
	.. controls (1471.093,532.139) and (1541.011,532.139) .. (1584.135,571.113);
	\path[draw=lmink,line width=0.02cm,line cap=butt,line join=miter]
	(1565.917,586.091)
	.. controls (1598.98,616.793) and (1598.98,666.571) .. (1565.917,697.273)
	.. controls (1532.855,727.975) and (1479.25,727.975) .. (1446.187,697.273)
	.. controls (1413.124,666.571) and (1413.124,616.793) .. (1446.187,586.091)
	.. controls (1479.25,555.39) and (1532.855,555.39) .. (1565.917,586.091);
	\path[fill=lmcore,draw=none]
	(1551.547,599.623)
	.. controls (1576.672,622.852) and (1576.672,660.513) .. (1551.547,683.742)
	.. controls (1526.421,706.971) and (1485.684,706.971) .. (1460.558,683.742)
	.. controls (1435.432,660.513) and (1435.432,622.852) .. (1460.558,599.623)
	.. controls (1485.684,576.394) and (1526.421,576.394) .. (1551.547,599.623)
	-- cycle;
	\path[draw=lmink,line width=0.02cm,line cap=butt,line join=miter]
	(1551.546,599.623)
	.. controls (1576.672,622.852) and (1576.672,660.513) .. (1551.546,683.742)
	.. controls (1526.42,706.971) and (1485.683,706.971) .. (1460.557,683.742)
	.. controls (1435.432,660.513) and (1435.432,622.852) .. (1460.557,599.623)
	.. controls (1485.683,576.394) and (1526.42,576.394) .. (1551.546,599.623);
	\path[
	draw=lmink,
	line width=0.02cm,
	line cap=butt,
	line join=miter,
	scale around={1.18:(818,706)}
	]
	(766.216,661.339)
	.. controls (796.06,628.247) and (845.386,620.113) .. (873.67,650.473)
	.. controls (893.798,672.078) and (901.516,713.112) .. (877.038,730.794)
	.. controls (807.027,781.367) and (769.828,790.598) .. (752.34,741.929)
	.. controls (737.438,700.456) and (745.693,684.096) .. (766.216,661.339);
	\path[
	fill=lmmode,
	draw=none,
	shift={(0,18)},
	scale around={0.80:(1506.052,641.683)}
	]
	(1506.052,620.851)
	-- (1511.913,636.764)
	-- (1530.875,636.765)
	-- (1515.534,646.6)
	-- (1521.394,662.514)
	-- (1506.052,652.679)
	-- (1490.711,662.514)
	-- (1496.57,646.6)
	-- (1481.229,636.765)
	-- (1500.192,636.764)
	-- (1506.052,620.851)
	-- cycle;
	\path[fill=lmmode,draw=none,scale around={0.80:(818.536,681.944)}]
	(818.536,681.944)
	-- (824.396,697.857)
	-- (843.359,697.858)
	-- (828.018,707.693)
	-- (833.877,723.607)
	-- (818.536,713.772)
	-- (803.194,723.607)
	-- (809.053,707.693)
	-- (793.713,697.858)
	-- (812.675,697.857)
	-- (818.536,681.944)
	-- cycle;
	\path[fill=lmmode,draw=none,scale around={0.80:(1332.876,243.41)}]
	(1332.876,243.41)
	-- (1338.737,259.323)
	-- (1357.699,259.324)
	-- (1342.359,269.159)
	-- (1348.218,285.073)
	-- (1332.876,275.238)
	-- (1317.535,285.073)
	-- (1323.394,269.159)
	-- (1308.054,259.324)
	-- (1327.016,259.323)
	-- (1332.876,243.41)
	-- cycle;
	\path[
	fill=lmtriangle,
	draw=none,
	scale around={0.80:(768.008,460.942)}
	]
	(768.008,440.695)
	-- (788.372,471.065)
	-- (747.645,471.065)
	-- (768.008,440.695)
	-- cycle;
	\path[
	fill=lmtriangle,
	draw=none,
	scale around={0.80:(1471.136,593.258)}
	]
	(1471.136,573.012)
	-- (1491.499,603.381)
	-- (1450.772,603.381)
	-- (1471.136,573.012)
	-- cycle;
	\path[draw=lmink,line width=0.016cm,line cap=butt,line join=miter]
	(815.352,674.218)
	-- (775.769,488.212)
	-- (775.353,486.256);
	\path[fill=lmink,draw=none]
	(783.985,486.464)
	-- (772.272,471.78)
	-- (767.553,489.961)
	-- (783.985,486.464)
	-- cycle;
	\path[draw=lmink,line width=0.016cm,line cap=butt,line join=miter]
	(1337.415,280.675)
	-- (1459.422,565.225)
	-- (1460.21,567.064);
	\path[fill=lmink,draw=none]
	(1451.701,568.536)
	-- (1466.042,580.666)
	-- (1467.142,561.915)
	-- (1451.701,568.536)
	-- cycle;
	\path[draw=lmink,line width=0.04cm,dash pattern=on 0.08cm off 0.08cm,line cap=butt,line join=miter]
	(760.922,440.6)
	.. controls (761.618,437.272) and (754.763,414.201) .. (750.321,390.794)
	.. controls (746.303,369.624) and (744.204,347.629) .. (750.058,326.525)
	.. controls (762.702,280.943) and (805.781,253.878) .. (850.097,236.087)
	.. controls (902.273,215.141) and (957.725,203.478) .. (1013.928,201.904)
	.. controls (1071,200.305) and (1127.669,209.147) .. (1183.286,221.922)
	.. controls (1220.6,230.493) and (1257.481,240.839) .. (1293.803,252.926)
	-- (1295.701,253.568);
	\path[fill=lmink,draw=none]
	(1287.134,259.536)
	-- (1305.741,256.967)
	-- (1292.522,243.623)
	-- (1293.806,252.927)
	-- (1287.134,259.536)
	-- cycle;
	\path[draw=lmink,line width=0.03cm,dash pattern=on 0.057cm off 0.057cm,line cap=butt,line join=miter]
	(1475.251,606.68)
	-- (1496.004,645.433)
	-- (1497.418,646.847);
	\path[fill=lmink,draw=none]
	(1487.095,648.403)
	-- (1504.914,654.342)
	-- (1498.974,636.523)
	-- (1496.004,645.433)
	-- (1487.095,648.403)
	-- cycle;
	\path[fill=lmink,draw=none]
	(697.058,797.007)
	.. controls (700.916,801.098) and (700.916,807.732) .. (697.058,811.823)
	.. controls (693.2,815.914) and (686.945,815.914) .. (683.087,811.823)
	.. controls (679.229,807.732) and (679.229,801.098) .. (683.087,797.007)
	.. controls (686.945,792.915) and (693.2,792.915) .. (697.058,797.007)
	-- cycle;
	\path[draw=lmink,line width=0.04cm,dash pattern=on 0.08cm off 0.08cm,line cap=butt,line join=miter]
	(705.63,810.702)
	.. controls (718.41,813.947) and (731.962,811.91) .. (743.217,805.041)
	.. controls (754.297,798.278) and (761.999,787.65) .. (769.476,777.284)
	.. controls (778.299,765.054) and (787.133,752.803) .. (795.985,740.527)
	-- (797.171,738.883);
	\path[fill=lmink,draw=none]
	(800.357,748.825)
	-- (803.371,730.285)
	-- (786.731,738.998)
	-- (796.001,740.505)
	-- (800.357,748.825)
	-- cycle;
	\node[anchor=base west,inner sep=0pt,outer sep=0pt,font=\fontsize{13pt}{13pt}\selectfont] at (825,690) {{\color{lmmode}$\bm{m}_{j_1}$}};
	\node[anchor=base west,inner sep=0pt,outer sep=0pt,font=\fontsize{13pt}{13pt}\selectfont] at (1345,250) {{\color{lmmode}$\bm{m}_{j_2}$}};
	\node[anchor=base west,inner sep=0pt,outer sep=0pt,font=\fontsize{13pt}{13pt}\selectfont] at (1505,710) {\color{lmmode}{$\bm{m}_{j_3}$}};
	\node[anchor=base west,inner sep=0pt,outer sep=0pt,font=\fontsize{11pt}{11pt}\selectfont] at (790,470) {{\color{lmtriangle}$\bm{z}_{j_1}:=\Pi_{\overline{\mathcal{U}}_{j_1}}\!(\bm{m}_{j_1})$}};
	\node[
	anchor=base east,
	inner sep=1.2pt,
	outer sep=0pt,
	font=\fontsize{11pt}{11pt}\selectfont
	] at (1550,597) {{\color{lmtriangle}$\bm{z}_{j_2}$}};
	\node[anchor=base west,inner sep=0pt,outer sep=0pt,font=\fontsize{17.217pt}{17.217pt}\selectfont] at (646.504,785.349) {$\bm{x}$};
\end{tikzpicture}